\documentclass{amsart}
\usepackage{geometry} 
\usepackage{amssymb}
\usepackage{faktor}
\usepackage{tikz}
\usepackage{tikz-cd}
\usepackage[all,cmtip]{xy}
\usepackage{amsmath,amscd}
\usepackage{mathrsfs}
\usepackage{color}
\usepackage{bbm}
\usepackage[mathcal]{euscript}
\usepackage{hyperref}
\hypersetup{colorlinks=true,linkcolor=blue}

\newcommand{\cof}{\mathrm{cof}}
\newcommand{\SW}{\mathcal{S}\mathrm{W}}
\newcommand{\triv}{\mathrm{triv}}
\newcommand{\Inrt}{\mathrm{Inrt}}
\newcommand{\fib}{\mathrm{fib}}
\newcommand{\weq}{\mathrm{weq}}
\newcommand{\Hom}{\mathrm{Hom}}
\newcommand{\Lan}{\mathrm{Lan}}
\newcommand{\Ran}{\mathrm{Ran}}
\newcommand{\sk}{\mathrm{sk}}
\newcommand{\cosk}{\mathrm{cosk}}

\newcommand{\Ho}{\mathrm{Ho}}
\newcommand{\hofib}{\mathrm{hofib}}
\newcommand{\Top}{\mathcal{T}\mathrm{op}}
\newcommand{\Sp}{\mathcal{S}\mathrm{p}}
\newcommand{\hoLan}{\mathrm{hoLan}}
\newcommand{\Emb}{\mathrm{Emb}}

\newcommand{\Conf}{\mathrm{Conf}}

\newcommand{\Sym}{\mathrm{Sym}}

\newcommand{\pt}{\mathrm{pt}}

\newcommand{\Fun}{\mathrm{Fun}}
\DeclareMathOperator*{\holim}{\mathrm{holim}}
\newcommand{\Map}{\mathrm{Map}}
\DeclareMathOperator*{\hocolim}{\mathrm{hocolim}}
\DeclareMathOperator*{\colim}{\mathrm{colim}}

\newcommand \U{\mathcal{U}}

\newcommand\op{\mathcal}

\newcommand\adjunct[4]{\xymatrix{#1\ar @<1.25ex>[rr]^{#3}&\perp&#2\ar @<1.25ex>[ll]^{#4}}}

\newcommand\id{\mathrm{id}}

\newcommand\B{\op B}

\newcommand\D{\mathcal D}

\def\treeof(#1;#2){[#1;#2]}

\newcommand{\comp}{\relax}
\def\comp(#1;#2){#1\circ(#2)}

\newcommand\doubleadjunct[7]{\xymatrix{#1\ar @<1.25ex>[rr]^{#4}&\perp&#2\ar @<1.25ex>[ll]^{#5}\ar @<1.25ex>[rr]^{#6}&\perp&#3\ar @<1.25ex>[ll]^{#7}}}

\theoremstyle{definition}\newtheorem*{blankdefinition}{Definition}
\theoremstyle{definition}\newtheorem{definition}{Definition}[subsection]
\theoremstyle{definition}\newtheorem{warning}[definition]{Warning}
\theoremstyle{theorem}\newtheorem{lemma}[definition]{Lemma}
\theoremstyle{definition}\newtheorem{perspective}[definition]{Perspective}
\theoremstyle{remark}\newtheorem{exercise}[definition]{Exercise}
\theoremstyle{definition}\newtheorem{question}[definition]{Question}
\theoremstyle{theorem}\newtheorem*{vanishingtheorem}{Vanishing Theorem}
\theoremstyle{theorem}\newtheorem*{periodicitytheorem}{Periodicity Theorem}
\theoremstyle{theorem}\newtheorem*{invariantstheorem}{Invariants Theorem}
\theoremstyle{remark}\newtheorem{remark}[definition]{Remark}
\theoremstyle{definition}
\theoremstyle{definition}
\theoremstyle{definition}\newtheorem{recollection}[definition]{Recollection}
\theoremstyle{definition}
\theoremstyle{definition}\newtheorem{example}[definition]{Example}
\theoremstyle{theorem}\newtheorem{proposition}[definition]{Proposition}
\theoremstyle{theorem}\newtheorem{corollary}[definition]{Corollary}
\theoremstyle{theorem}\newtheorem{theorem}[definition]{Theorem}
\theoremstyle{theorem}\newtheorem{hypothesis}[definition]{Hypothesis}
\theoremstyle{definition}\newtheorem{construction}[definition]{Construction}

\title{Configuration spaces in algebraic topology}
\author{Ben Knudsen}
\date{} 

\begin{document}

\maketitle

\begin{abstract}
These expository notes are dedicated to the study of the topology of configuration spaces of manifolds. We give detailed computations of many invariants, including the fundamental group of the configuration spaces of $\mathbb{R}^2$, the integral cohomology of the ordered---and the mod $p$ cohomology of the unordered---configuration spaces of $\mathbb{R}^n$, and the rational cohomology of the unordered configuration spaces of an arbitrary manifold of odd dimension. We also discuss models for mapping spaces in terms of labeled configuration spaces, and we show that these models split stably. Some classical results are given modern proofs premised on hypercover techniques, which we discuss in detail. 
\end{abstract}

\tableofcontents
\newpage
\section{Introduction}

These notes are concerned with the study of the following spaces.

\begin{blankdefinition}
The \emph{configuration space} of $k$ ordered points in the topological space $X$ is \[\Conf_k(X):=\{(x_1,\ldots, x_k)\in X^k: x_i\neq x_j\text{ if } i\neq j\},\] endowed with the subspace topology.\footnote{The reader is warned that, in the literature on configuration spaces, there are almost as many traditions of notation as there are references.} The \emph{unordered} configuration space is the quotient \[B_k(X):=\Conf_k(X)/\Sigma_k.\]
\end{blankdefinition}

Our goal is to attempt to understand the topology of these spaces, in the case of $X$ a manifold, through the lens of homotopy groups and (co)homology. Before beginning our study in earnest, we first mention a few reasons that configuration spaces are particularly interesting objects of study. 

\subsection{Invariants}
The homotopy type of a fixed configuration space is a homeomorphism invariant of the background manifold, and these invariants tend to remember a rather large amount of information. A simple-minded example is provided by Euclidean spaces of different dimension; indeed, as we will see, there is a homotopy equivalence $B_2(\mathbb{R}^m)\simeq B_2(\mathbb{R}^n)$ if and only if $m=n$. In other words, configuration spaces are sensitive to the dimension of a manifold. 

A somewhat more sophisticated example is provided by the fact that $B_2(T^2\setminus \mathrm{pt})\not\simeq B_2(\mathbb{R}^2\setminus S^0)$, which can be shown by a homology calculation. Note that $T^2\setminus \mathrm{pt}$ and $\mathbb{R}^2\setminus S^0$ have the same dimension and homotopy type, having $S^1\vee S^1$ as a common deformation retract. On the other hand, $(T^2\setminus\mathrm{pt})^+\cong T^2\not\simeq S^1\vee S^1\vee S^2\simeq (\mathbb{R}^2\setminus S^0)^+$, so we might conclude from this example that configuration spaces are sensitive to the \emph{proper} homotopy type of a manifold.

In order to discuss the most striking illustration of the sensitivity of configuration spaces, we recall that the \emph{Lens spaces} are a family of compact $3$-manifolds given by \[L(p,q):=S^3/C_p,\] where the cyclic group $C_p$ acts on $S^3\subseteq \mathbb{C}^2$ by multiplication by $(e^{2\pi i/p}, e^{2\pi iq/p})$. It is a classical theorem of Reidemeister that \begin{align*}
L(p,q_1)\simeq L(p, q_2)&\iff q_1q_2\equiv \pm n^2\mod p\\
L(p,q_1)\cong L(p,q_2)&\iff q_1\equiv \pm q_2^{\pm 1}\mod p.
\end{align*} In particular, $L(7,1)$ and $L(7,2)$ are homotopy equivalent but not homeomorphic, and, according to a theorem of Longoni--Salvatore \cite{LongoniSalvatore:CSNHI}, their configuration spaces distinguish them. Thus, configuration spaces are sensitive at least to the \emph{simple} homotopy type of a manifold.

\subsection{Braids}
A point moving in $B_k(\mathbb{R}^2)$ traces out $k$ different paths that weave among one another but can never overlap. For this reason, we think of the fundamental group $\pi_1(B_k(\mathbb{R}^2))$ as the group of geometric \emph{braids} on $k$ strands, with composition given by concatenation of braids. As we shall see, this braid group admits the remarkably simple presentation \[\pi_1(B_k(\mathbb{R}^2))\cong \langle \sigma_1,\ldots, \sigma_{k-1}\mid\sigma_i\sigma_{i+1}\sigma_i=\sigma_{i+1}\sigma_i\sigma_{i+1},\, \sigma_i\sigma_j=\sigma_j\sigma_i \text{ if } |i-j|>1\rangle,\] due originally to Artin \cite{Artin:TB}. The combinatorial, algebraic, and geometric properties of these and related braid groups are of fundamental importance to a vast swath of mathematics that encompasses knot theory, mapping class groups \cite{Birman:BLMCG}, quantum groups \cite{Kassel:QG}, category theory \cite{JoyalStreet:BMC}, and motion planning \cite{Farber:CSRMPA}.

\subsection{Embeddings}
In the company of manifolds with trivialized tangent bundles, it is possible to speak of a \emph{framed embedding}, which is to say an embedding respecting the fixed trivialization, possibly up to a homotopy through bundle maps. The tangent bundle of any Euclidean space is canonically trivialized, and evaluation at the origin determines a homotopy equivalence \[\mathrm{Emb}^\mathrm{fr}(\amalg_k\mathbb{R}^n, \mathbb{R}^n)\xrightarrow{\sim} \Conf_k(\mathbb{R}^n).\] As a consequence of this observation and the fact that framed embeddings compose, we see that the collection $\{\Conf_k(\mathbb{R}^n)\}_{k\geq0}$ of homotopy types is equipped with hidden algebraic structure. This algebraic structure is that of an \emph{operad}, which goes by the name of $E_n$. This perspective has many important descendents, three of which we name.

\subsection{Iterated loop spaces} For any $X$, there is a collection of maps \[\Emb^\mathrm{fr}(\amalg_k\mathbb{R}^n,\mathbb{R}^n)\times (\Omega^nX)^k\to \Omega^n X,\] which arise from a variant of the Pontrjagin--Thom collapse construction. The compatibilities among these maps are summarized by saying that $\Omega^n X$ is an $E_n$-\emph{algebra}; in particular, the homology groups of the configuration spaces of $\mathbb{R}^n$ encode algebraic operations on $H_*(\Omega^nX)$. Remarkably, there is also a partial converse, due to May \cite{May:GILS}, which amounts to an algebraic classification of $n$-fold loop spaces.

\subsection{Factorization homology} The operad $E_n$ is defined in terms of embeddings among disjoint unions of Euclidean spaces, and so, after choosing a coordinate chart, we think of the structure of an $E_n$-algebra $A$ as being ``defined locally'' on a manifold $M$. Patching this local structure together across the elements of an atlas produces a manifold invariant, called the factorization homology of $M$ with coefficients in $A$ and denoted $\int_MA$, which can be thought of as a kind of space of configurations in $M$ labeled by elements of the algebra \cite{AyalaFrancis:FHTM}.

In the example of an $n$-fold loop space discussed above, the \emph{non-Abelian Poincar\'{e} duality} of Salvatore \cite{Salvatore:CSSl} and Lurie \cite{Lurie:HA} supplies the identification \[\int_M\Omega^nX\xrightarrow{\sim} \mathrm{Map}_c(M, X),\] as long as $X$ is $(n-1)$-connected. Later, we will encounter a precursor to this result in the configuration space models for mapping spaces introduced by McDuff \cite{McDuff:CSPNP,Boedigheimer:SSMS}.

In a sense, factorization homology is a method for using configuration spaces to probe manifolds, but it can also be used to study the configuration spaces themselves, reversing this flow of information \cite{Knudsen:BNSCSVFH}. In particular, a theorem of Ayala--Francis \cite{AyalaFrancis:FHTM} shows that configuration spaces enjoy a kind of Mayer--Vietoris property in the existence of a quasi-isomorphism \[C_*(B(M))\simeq C_*(B(M_1))\bigotimes^\mathbb{L}_{C_*(B(N\times\mathbb{R}))}C_*(B(M_2)),\] where $B(X):=\coprod_{k\geq0} B_k(X)$. The fundamental fact underlying this quasi-isomorphism is the contractibility of the unordered configuration spaces of $\mathbb{R}$.

\subsection{Embedding calculus} The embedding calculus of Goodwillie and Weiss \cite{Weiss:EPVIT,GoodwillieWeiss:EPVIT} produces a tower of approximations \[\xymatrix{
&\vdots\ar[d]\\
&T_2\Emb(M,N)\ar[d]\\
\Emb(M,N)\ar[uur]\ar[ur]\ar[r]&T_1\Emb(M,N),
}\] which can be thought of as \emph{algebraic} approximations, where algebra is construed in the operadic sense. Often these approximations become arbitrarily good---in particular, according to a theorem of Goodwillie--Klein \cite{GoodwillieKlein:MDSSE}, this occurs in codimension at least 3---so that one obtains a cofiltration of the space of embeddings. The layers of this cofiltration are described as spaces of sections of certain bundles over configuration spaces, so hard questions about embeddings can sometimes be translated, at least in principle, into softer questions about the ordinary topology of configuration spaces.

\subsection{What is in these notes?} Our guiding mathematical impulse is the basic imperative of the algebraic topologist. Configuration spaces are interesting; therefore, it is worthwhile to compute their homotopy groups and (co)homology. 

From a more pedagogical standpoint, we have attempted to design a set of notes that might help to fulfill three goals of a graduate student hoping to enter the field: first, to become familiar with some of what is already known; second, to acquire a set of modern tools; and, third, to see how these tools might be applied in practice. For this reason, we have opted to explore a range of classical topics with a mix of classical and non-classical techniques, developing the latter along the way, with some old results being given new proofs. In particular, emphasis is placed on the use of hypercover techniques, which the author has found indispensable in his own work.

We now give a linear outline of the contents of these notes, including primary references. 

\begin{enumerate}
\item[\S\ref{section:configuration spaces}] \emph{Configuration spaces and braids}. After exploring examples and developing basic properties, we exploit the Fadell--Neuwirth fibrations \cite{FadellNeuwirth:CS} to deduce a number of easy computations of homotopy groups. Afterward, we turn to the identification of $\pi_1(B_k(\mathbb{R}^2))$ with the braid group \cite{Artin:TB, FoxNeuwirth:BG}. We  follow one of the approaches outlined in \cite{Birman:BLMCG}, reviewing the Reidemeister--Schreier method along the way from the topological viewpoint \cite{ZieschangVogtColdeway:SPDG}. We defer the proof of Fadell--Neuwirth, which will eventually be our first use of hypercover techniques.
\item[\S\ref{section:cohomology}] \emph{(Co)homology of $\Conf_k(\mathbb{R}^n)$}. The main result of this section is the computation of the integral cohomology ring of $\Conf_k(\mathbb{R}^n)$ \cite{Arnold:CRCBG, CohenLadaMay:HILS}. Our argument is essentially that of Cohen. As it turns out, the natural basis in homology is more convenient for some purposes, and we develop this viewpoint in detail following \cite{Sinha:HLDO}. We close by using these computations to dedue the rational homology of $B_k(\mathbb{R}^n)$.
\item[\S\ref{section:covering theorems}] \emph{Covering theorems}. We give a detailed development of the machinery of \v{C}ech covers, hypercovers, and complete covers following \cite{DuggerIsaksen:THAR}. The main result is that the weak homotopy type of a space may be recovered as the homotopy colimit over a complete cover---a topological basis, for example.
\item[\S\ref{section:deferred proofs}] \emph{Deferred proofs}. We use the covering techniques of the previous section, together with Quillen's Theorem B, to prove Fadell--Neuwirth, which was the main tool in \S\ref{section:configuration spaces}. We then use the same techniques to construct the Serre spectral sequence following \cite{Segal:CSSS}, from which we deduce the Leray--Hirsch theorem, which was used in \S\ref{section:cohomology}.
\item[\S\ref{section:mapping space models}] \emph{Mapping space models}. We introduce configuration spaces with labels in a pointed space \cite{McDuff:CSPNP, Boedigheimer:SSMS}, which we motivate as a type of homology theory for manifolds following \cite{AyalaFrancis:FHTM}. We give a new proof of the analogue of the exactness axiom for homology, namely the existence of a certain homotopy pullback square, using hypercovers and Quillen's Theorem B. Using scanning techniques following \cite{Boedigheimer:SSMS}, we show that labeled configuration spaces model mapping spaces, an analogue of Poincar\'{e} duality.
\item[\S\ref{section:homology from splittings}] \emph{Homology calculations from stable splittings}. The main result is the computation of the rational homology of $B_k(M)$ for odd-dimensional $M$ following \cite{BoedigheimerCohenTaylor:OHCS}. The key ingredient is to show that labeled configuration split stably \cite{Boedigheimer:SSMS, CohenMayTaylor:SCSC}. 
\item[\S\ref{section:mod p cohomology}] \emph{Mod $p$ cohomology}. The main result is the computation of the mod $p$ cohomology of $B_p(\mathbb{R}^n)$. We follow the argument of \cite[III]{CohenLadaMay:HILS} in the large, giving simplified and corrected proofs of several key steps using the homology basis developed in \S\ref{section:cohomology}.
\item[\S\ref{section:lie algebra methods}] \emph{Postscript: Lie algebra methods}. We give an informal discussion of the main result of \cite{Knudsen:BNSCSVFH}, which expresses the rational homology of $B_k(M)$ for arbitrary $M$ as Lie algebra homology. We use this expression to give a proof of homological stability for configuration spaces. 
\item[\ref{appendix:split simplicial spaces}] \emph{Split simplicial spaces}. In this appendix, following \cite{DuggerIsaksen:THAR}, we develop a criterion guaranteeing that degreewise weak equivalence between simplicial spaces induces a weak equivalence after geometric realization. We then verify that this criterion is satisfied in the cases of interest.
\item[\ref{section:homotopy colimits}] \emph{Homotopy colimits}. We motivate the bar construction model for the homotopy colimit and discuss its justification through the theory of derived functors and model categories \cite{Quillen:HA, Hirschhorn:MCL, Dugger:PHC}. We give a detailed proof that the bar construction is cofibrant in the projective model structure on functors. We discuss homotopy left Kan extensions and prove Quillen's Theorem B following \cite{GoerssJardine:SHT}.
\item[\ref{appendix:Spanier--Whitehead}] \emph{The Spanier--Whitehead category}. We discuss a very rudimentary form of stable homotopy theory, which is necessary for the theorem on stable splittings \cite{SpanierWhitehead:FAHT, DellAmbrogio:SWCAT}. We avoid the need to introduce the full machinery of spectra with the notion of a filtered stable weak equivalence.
\item[\ref{appendix:periodicity}] \emph{Tate cohomology and periodicity}. We review some facts about Tate cohomology \cite{CartanEilenberg:HA} and give a proof of the periodicity theorem used in \S\ref{section:mod p cohomology} following \cite{Swan:PPFG}.
\end{enumerate}

\subsection{What is not in these notes?} The title of these notes could---and someday, perhaps, will---be the title of a book, or of a book series. That being the case, there are vast and important swaths of the literature that we do not touch. What follows is a non-exhaustive list of omissions, in no particular order. 
\begin{itemize}
\item Iterated loop space theory \cite{May:GILS, Segal:CSILS} and the computation of the mod $p$ homology of $B_k(\mathbb{R}^n)$ with all of the induced structure \cite[III]{CohenLadaMay:HILS}.
\item Factorization homology \cite{AyalaFrancis:FHTM} and embedding calculus \cite{Weiss:EPVIT}.
\item The Cohen--Taylor--Totaro spectral sequence \cite{Totaro:CSAV} and representation stability \cite{ChurchEllenbergFarb:FIMSRSG}.
\item The Fulton--MacPherson compactification \cite{FultonMacPherson:CCS,Sinha:MTCCS} and formality \cite{Kontsevich:OMDQ, LambrechtsVolic:FLNDO}.
\item The failure of homotopy invariance \cite{LongoniSalvatore:CSNHI}.
\item Configuration spaces of graphs \cite{Abrams:CSBGG, Ghrist:CSBGGR, AnDrummond-ColeKnudsen:SSGBG}.
\item Topological combinatorics \cite{Farber:TBPI, BlagojevichLueckZiegler:ETCS}.
\end{itemize}

The author hopes to remedy some of these omissions---as well as the unforgivable absence of pictures---in future versions of these notes.

\subsection{Prerequisites} Sections \ref{section:configuration spaces}-\ref{section:cohomology} of these notes should be accessible to anyone with a first course in algebraic topology \cite{Hatcher:AT}. In particular, facts about homology, cohomology, homotopy groups, covering spaces, and the like are used freely and without comment. In the remainder of the notes, particularly in \S\ref{section:covering theorems}, it will be helpful to have a working knowledge of basic category theory \cite{MacLane:CWM} and simplicial methods \cite{GoerssJardine:SHT}. We make some use of the Serre spectral sequence, which we review at the appropriate time, albeit briefly, and increasingly heavy use of homotopy colimits, which are discussed in some detail in Appendix \ref{section:homotopy colimits}. We make do with a very primitive form of stable homotopy theory, which is reviewed in Appendix \ref{appendix:Spanier--Whitehead}. No knowledge of spectra is assumed or used (although it never hurts).

\subsection{Conventions} For convenience, we take manifolds to be smooth of finite type. In dealing with graded vector spaces, we make use of the usual Koszul sign convention. For $V$ a graded vector space, $\Sym(V)=\bigoplus_{k\geq0}\Sym^k(V)$ is the tensor product of a polynomial algebra on the even part of $V$ with an exterior algebra on the odd part of $V$. We grade homologically and write $V[r]$ for the homological suspension of $V$---that is, the degree $n$ part of $V[r]$ is the degree $n-r$ part of $V$.

\subsection{Acknowledgments} These notes grew out of a course of the same name that I taught during the fall semester of 2017 at Harvard University. I am grateful to the members of the class for their engagement and to Harvard for the opportunity to design a course with complete freedom. Extra thanks are due to Sander Kupers for catching innumerable errors while they were still on the blackboard (if he missed any, we can blame him).

I wish to acknowledge a few key influences on the development of the course and notes. First and most important of these is the mathematics and metamathematics that I learned from my advisor, John Francis, without whom, I am certain, these notes would not exist. Second, the theory of hypercovers and complete covers developed by Dugger--Isaksen \cite{DuggerIsaksen:THAR} plays a central role here, as it has in most of what I have done as a mathematician. Third, the geometric description of the homology of configuration spaces in terms of planetary systems given by Sinha \cite{Sinha:HLDO} is not only very beautiful but also very useful, as it allows for greatly simplified arguments in some cases. Fourth, I would like to thank Haynes Miller for his sustained encouragement in the teaching of the course and the writing of these notes.

Finally, I thank everyone whose enthusiastic conversation has helped to foster my love for the ideas discussed here. In alphabetical order, with many omissions, this group includes Byung Hee An, Lauren Bandklayder, Mark Behrens, Lukas Brantner, Tom Church, Gabriel Drummond-Cole, Jordan Ellenberg, Elden Elmanto, Benson Farb, Ezra Getzler, Paul Goerss, Owen Gwilliam, Kathryn Hess, Sander Kupers, Pascal Lambrechts, Daniel L\"{u}tgehetmann, Jeremy Miller, Sam Nariman, Martin Palmer, Dan Petersen, Dev Sinha, Joel Specter, Alex Suciu, Hiro Tanaka, Nathalie Wahl, Brian Williams, Dylan Wilson, and Jesse Wolfson.

\section{Configuration spaces and braids}\label{section:configuration spaces}

\subsection{Basic examples} In order to begin to develop a feel for how configuration spaces behave, we begin with a few examples.

\begin{example}[$\varnothing$, $\mathbb{R}^0$]
It is usually best to begin with the trivial cases. In the case of the empty manifold, we have \[\Conf_k(\varnothing)=\begin{cases}
\mathrm{pt}&\quad k=0\\
\varnothing&\quad\text{ else,}
\end{cases}\] while \[\Conf_k(\mathbb{R}^0)=\begin{cases}
\mathrm{pt}&\quad k=0,1\\
\varnothing&\quad \text{else.}
\end{cases}\] Notice that $\Conf_0(X)$ is a singleton for any space $X$.
\end{example}

\begin{example}[$\mathbb{R}\cong (0,1)$]
From the definition $\Conf_2((0,1))=(0,1)^2\setminus\{(x,y):x=y\}$, it is clear that the ordered configuration space two points in $(0,1)$ is a disjoint union of two open 2-simplices, and that $\Sigma_2$ acts by permuting these components. In particular, $B_2((0,1))\cong \mathring{\Delta}^2\simeq \mathrm{pt}$. This description generalizes readily to higher $k$. Note that the natural orientation of $(0,1)$ induces a second ordering on the coordinates of any configuration, which is to say a permutation of $\{1,\ldots, k\}$, that any such permutation is possible, and that the assignment of configuration to permutation is locally constant by the intermediate value theorem. Thus, we have a $\Sigma_k$-equivariant bijection $\pi_0(\Conf_k((0,1)))\cong\Sigma_k$, and the unordered configuration space is naturally identified with the path component of the identity permutation. For this space, we define a map \begin{align*}
B_k((0,1))&\to\mathring{\Delta}^k:=\left\{(t_0,\ldots, t_k)\in \mathbb{R}^{k+1}: t_i>0,\, \sum_{i=0}^kt_i=1\right\}\\
\{x_1,\ldots, x_k\}&\mapsto (x_1, x_2-x_1, \ldots, 1-x_k),
\end{align*} where the set $\{x_1,\ldots, x_k\}$ is ordered so that $x_1<\cdots< x_k$. This map is a homeomorphism; in particular, the unordered configuration spaces of $\mathbb{R}$ are all contractible. 
\end{example}

\begin{example}[$\mathbb{R}^n$]
In the configuration space of two points in $\mathbb{R}^n$, there are exactly three pieces of data---the direction, the center of mass, and the distance---and only one of these is not a contractible choice. Precisely, the map \begin{align*}\Conf_2(\mathbb{R}^n)&\to S^{n-1}\times\mathbb{R}_{>0}\times\mathbb{R}^n\\
(x_1,x_2)&\mapsto \left(\frac{x_2-x_1}{\|x_2-x_1\|}, \|x_2-x_1\|,\frac{x_1+x_2}{2}\right)
\end{align*} is a homeomorphism, and, in particular, the \emph{Gauss map} \[\Conf_2(\mathbb{R}^n)\to S^{n-1}\] given by the first coordinate of this homeomorphism is a homotopy equivalence. Since this map is also $\Sigma_2$-equivariant for the antipodal action on the target, we conclude that $B_2(\mathbb{R}^n)\simeq \mathbb{RP}^{n-1}$. The Gauss map will play a fundamental role in our later study; in particular, by pulling back a standard volume form, this map furnishes us with our first example of a nonzero higher dimensional class in the cohomology of configuration spaces.

We cannot be as explicit about $\Conf_k(\mathbb{R}^n)$ for higher $k$, but it should be noted that this space is an example of the complement of a \emph{hyperplane arrangement}, since \[\Conf_k(\mathbb{R}^n)=\mathbb{R}^{nk}\setminus \bigcup_{1\leq i<j\leq k}\{(x_1,\ldots, x_k)\in\mathbb{R}^{nk}: x_i=x_j\},\] which is an interesting and classical type of mathematical object in its own right with its own literature \cite{OrlikTerao:AH}.
\end{example}

\begin{example}[$S^1$] By definition $\Conf_2(S^1)$ is the $2$-torus with its diagonal removed, which is homeomorphic to a cylinder, and, by cutting and pasting, one can see directly that the corresponding unordered configuration space is the open M\"{o}bius band \cite{Tuffley:FSSS}. In the general case \cite[p. 292]{CrabbJames:FHT}, make the identification $S^1=\mathbb{R}/\mathbb{Z}$, and consider first the subspace $A\subseteq \Conf_k(S^1)$ of configurations $(x_1,\ldots, x_k)$ whose ordering coincides with the cyclic ordering induced by the standard orientation of $S^1$. For $1\leq i\leq k$, the difference $t_i=x_{i+1}-x_i\in (0,1)$ is well-defined, where we set $t_{k+1}=t_1+1$, and $\sum_{i=1}^kt_i=1$. Recording the normalized first coordinate and the $t_i$ defines a $C_k$-equivariant homeomorphism $A\cong S^1\times\mathring{\Delta}^{k-1}$, which induces a $\Sigma_k$-equivariant homeomorphism \[(S^1\times\mathring{\Delta}^{k-1})\times_{C_k}\Sigma_k\xrightarrow{\simeq} \Conf_k(S^1).\]  In particular, $B_k(S^1)\simeq S^1/C_k\cong S^1$ for $k>0$.
\end{example}

\begin{example}[$S^n$]
In the case of two points in the $n$-sphere, as with Euclidean space, the choice of the direction in $S^n$ between the two points is the fundamental parameter. Precisely, the projection onto the first coordinate defines a map $\Conf_2(S^n)\to \Conf_1(S^n)$, and recording the unit tangent vector in the direction of the second coordinate produces a commuting diagram \[\xymatrix{
\Conf_2(S^n)\ar[d]\ar[r]&\mathrm{Sph}(TS^n)\ar[d]\\
\Conf_1(S^n)\ar@{=}[r]&S^n
}\] in which the top map is a homotopy equivalence. The projection map is a special case of a family of maps, the \emph{Fadell--Neuwirth fibrations} \cite{FadellNeuwirth:CS}, which relate different cardinalities of configuration space in any background manifold. The existence of this family of maps, which will be one of our most important tools in what follows, places very strong constraints on configuration spaces; for example, it ultimately implies that the $i$th Betti number of $B_k(M)$ is constant for large $k$ \cite{Church:HSCSM,RandalWilliams:HSUCS}. This phenomenon of \emph{homological stability}, and the corresponding \emph{representation stability} in the ordered case, has since become an active area of study in its own right \cite{Farb:RS}.
\end{example}

\subsection{First properties}
We begin our formal study of configuration spaces by cataloguing a few of their basic features. First, we note that, if $f:X\to Y$ is an injective continuous map, we have the $\Sigma_k$-equivariant factorization \[\xymatrix{
X^k\ar[rr]^-{f^k}&&Y^k\\
\Conf_k(X)\ar[u]\ar@{-->}[rr]^-{\Conf_k(f)}&&\Conf_k(Y)\ar[u]
}\] and thus an induced map $B_k(f):B_k(X)\to B_k(Y)$. This construction respects composition and identities by inspection, so we have functors \begin{align*}
\Conf_k:\Top^{\mathrm{inj}}&\to \Top^{\Sigma_k}\\
B_k:\Top^{\mathrm{inj}}&\to \Top,
\end{align*} where $\Top^\mathrm{inj}$ denotes the category of topological spaces and injective continuous maps, $\Top^{\Sigma_k}$ the category of $\Sigma_k$-spaces and equivariant maps, and $\Top$ the category of topological spaces.

\begin{proposition}
If $f:X\to Y$ is an open embedding, then $\Conf_k(f):\Conf_k(X)\to \Conf_k(Y)$ and $B_k(f):B_k(X)\to B_k(Y)$ are also open embeddings.
\end{proposition}
\begin{proof}
From the definition of the product topology, $f^k:X^k\to Y^k$ is an open embedding, and the first claim follows. The second claim follows from the first and the fact that $\pi:\Conf_k(X)\to B_k(X)$ is a quotient map.
\end{proof}

This functor also respects a certain class of weak equivalences, for which we now introduce some (very nonstandard) terminology.

\begin{definition}
Let $f,g:X\to Y$ be injective continuous maps. 
\begin{enumerate}
\item We say that a map $H:X\times[0,1]\to Y$ is a \emph{monotopy} between $f$ and $g$ if it is a homotopy between $f$ and $g$ and if $H_t$ is injective for every $t\in [0,1]$. 
\item We say that $f$ and $g$ are \emph{monotopic} if there is such an injectopy. 
\item We say that an injective continuous map $f:X\to Y$ is an \emph{monotopy equivalence} if there is an injective continuous map $g:Y\to X$ such that $g\circ f$ is monotopic to $\id_X$ and $f\circ g$ is monotopic to $\id_Y$.
\end{enumerate}
\end{definition}

\begin{proposition}\label{prop:monotopy}
If $f$ and $g$ are monotopic, then $\Conf_k(f)$ and $\Conf_k(g)$ are homotopic (in fact monotopic).
\end{proposition}
\begin{proof}
In the solid commuting diagram
\[\xymatrix{
X^k\times[0,1]^k\ar@{=}[r]^-\sim& (X\times[0,1])^k\ar[r]^-{H^k}&Y^k\\\\
X^k\times[0,1]\ar[uurr]^-{H^\Delta}\ar[uu]^-{\id_{X^k}\times\Delta_k}\\
\Conf_k(X)\times[0,1]\ar@{-->}[rr]\ar[u]&&\Conf_k(Y),\ar[uuu]
}\] the diagonal composite is given by the formula \[H_t^\Delta(x_1,\ldots, x_k)=(H_t(x_1), \ldots, H_t(x_k)).\] By assumption, $H_t$ is injective for each $t\in[0,1]$, so the dashed filler exists. By construction, this map is injective for each $t$ and restricts to $\Conf_k(f)$ at $t=0$ and to $\Conf_k(g)$ at $t=1$.
\end{proof}

\begin{corollary}
If $f$ is a monotopy equivalence, then both $\Conf_k(f)$ and $B_k(f)$ are homotopy (in fact monotopy) equivalences.
\end{corollary}
\begin{proof}
The claim for $\Conf_k(f)$ is immediate from Proposition \ref{prop:monotopy}. For $B_k(f)$, we note that the homotopies constructed in the proof of that result are homotopies through $\Sigma_k$-equivariant maps, so they descend to the unordered configuration space.
\end{proof}

\begin{remark}
Another point of view on the previous two results is provided by the fact (which we will not prove here) that $\Conf_k$ and $B_k$ are \emph{enriched} functors, where the space of injective continuous maps from $X$ to $Y$ is given the subspace topology induced by the compact-open topology on $\mathrm{Map}(X,Y)$. Taking this fact for granted, these results follows immediately, since a monotopy is simply a path in this mapping space.
\end{remark}

The most important examples of monotopy equivalences will be isotopy equivalences of manifolds, as in the following example.

\begin{example}
If $M$ is a manifold with boundary, then $\partial M$ admits a collar neighborhood $U\cong \partial M\times (0,1]$. We define a map $r:M\to M$ by setting $r|_{\partial M\times(0,1]}(x,t)=(x,\frac{t}{2})$ and extending by the identity. This map is injective, and dilation defines monotopies from $r\circ i:\mathring{M}\to \mathring{M}$ and $i\circ r:M\to M$ to the respective identity maps. It follows that the induced map $\Conf_k(\mathring{M})\to \Conf_k(M)$ is a homotopy equivalence.
\end{example}

These functors also interact well with the operation of disjoint union.

\begin{proposition}
Let $X$ and $Y$ be topological spaces. The natural map \[\coprod_{i+j=k}\Conf_i(X)\times\Conf_j(Y)\times_{\Sigma_i\times\Sigma_j}\Sigma_k\to \Conf_k(X\amalg Y)\] is a $\Sigma_k$-equivariant homeomorphism. In particular, the natural map \[\coprod_{i+j=k}B_i(X)\times B_j(Y)\to B_k(X\times Y)\] is a homeomorphism.
\end{proposition}
\begin{proof}
From the definitions, the dashed filler exists in the commuting diagram
\[\xymatrix{\displaystyle \coprod_{i+j=k}X^i\times Y^j\times_{\Sigma_i\times\Sigma_j}\Sigma_k\ar[r]^-\simeq&(X\amalg Y)^k\\
\displaystyle\coprod_{i+j=k}\Conf_i(X)\times\Conf_j(Y)\times_{\Sigma_i\times\Sigma_j}\Sigma_k\ar@{-->}[r]\ar[u]& \Conf_k(X\amalg Y)\ar[u]}\] and is easily seen to be a bijection, which implies the first claim, since the vertical arrows are inclusions of subspaces. The second claim follows from the first after taking the quotient by the action of $\Sigma_k$.
\end{proof}

Thus, we may restrict attention to connected background spaces whenever it is convenient to do so.

Our next goal is to come to grips with the local structure of configuration spaces. We assume from now on that $X$ is locally path connected, and we fix a basis $\B$ for the topology of the space $X$ consisting of connected subsets. We define two partially ordered sets as follows.
\begin{enumerate}
\item We write $\B_k=\{U\subseteq X: U\cong \amalg_{i=1}^k U_i,\, U_i\in\B\}$, and we impose the order relation \[U\leq V\iff U\subseteq V \text{ and } \pi_0(U\subseteq V) \text{ is surjective}.\]
\item We write $\B_k^\Sigma=\{(U,\sigma): U\in \B_k, \,\sigma:\{1,\ldots, k\}\xrightarrow{\simeq}\pi_0(U)\}$, and we impose the order relation \[(U,\sigma)\leq (V,\tau)\iff U\leq V\text{ and } \tau=\sigma\circ\pi_0(U\subseteq V).\]
\end{enumerate} Denoting the poset of open subsets of a space $Y$ by $\op O(Y)$, there is an inclusion $\B_k^\Sigma\to \op O(\Conf_k(X))$ of posets defined by \[U\mapsto \Conf_k^0(U,\sigma):=\left\{(x_1\ldots, x_k)\in \Conf_k(U): x_i\in U_{\sigma(i)}\right\}\subseteq\Conf_k(X)\] and similarly an inclusion $\B_k\to \op O(B_k(X))$ defined by \[ U\mapsto B_k^0(U):=\left\{\{x_1,\ldots, x_k\}\in B_k(U): \{x_1,\ldots, x_k\}\cap U_i\neq \varnothing,\, 1\leq i\leq k\right\}\subseteq B_k(X).\] Note that these subsets are in fact open, since $U\subseteq X$ is open and configuration spaces respect open embeddings.

\begin{lemma}
For any $U\in \B_k$ and $\sigma:\{1,\ldots, k\}\xrightarrow{\simeq} \pi_0(U)$, there are canonical homeomorphisms \[B_k^0(U)\cong\Conf_k^0(U,\sigma)\cong\prod_{i=1}^kU_{\sigma(i)}.\]
\end{lemma}
\begin{proof}
It is easy to see from the definitions that the dashed fillers in the commuting diagram \[\xymatrix{X^k&\Conf_k(X)\ar[l]\ar[r]& B_k(X)\\
\displaystyle\prod_{i=1}^kU_{\sigma(i)}\ar[u]&\Conf_k^0(U,\sigma)\ar@{-->}[l]\ar@{-->}[r]\ar[u]&B_k^0(U)\ar[u]
}\] exist and are bijections. Since the lefthand map is the inclusion of a subspace and the righthand map is a quotient map, the claim follows.
\end{proof}

\begin{proposition}\label{prop:conf basis} Let $X$ be a locally path connected Hausdorff space and $\B$ a topological basis for $X$ consisting of connected subsets.
\begin{enumerate}
\item The collection $\{\Conf_k^0(U,\sigma): (U,\sigma)\in \B_k^\Sigma\}\subseteq \op O(\Conf_k(X))$ is a topological basis.
\item The collection $\{B_k^0(U): U\in \B_k\}\subseteq \op O(B_k(X))$ is a topological basis.
\end{enumerate}
\end{proposition}
\begin{proof}
By the definition of the product and subspace topologies, it will suffice for the first claim to show that, given $(x_1,\ldots, x_k)\in V\subseteq X^k$ such that \begin{itemize}
\item $V\cong \prod_{i=1}^k V_i$ for open subsets $x_i\in V_i\subseteq X$, and
\item $V\subseteq \Conf_k(X)$,
\end{itemize} there exists $(U,\sigma)\in \B_k^\Sigma$ with $(x_1,\ldots, x_k)\in \Conf_k^0(U,\sigma)\subseteq V$. Now, since $\B$ is a topological basis, we may find $U_i\in \B$ with $x_i\in U_i\subseteq V_i$. The second condition and the assumption that $X$ is Hausdorff imply that the $V_i$ are pairwise disjoint, so we may set $U=\coprod_{i=1}^kU_i$ and take $\sigma(i)=[U_i]$. With these choices \[(x_1,\ldots, x_k)\in \Conf_k^0(U,\sigma)\cong\prod_{i=1}^k U_i\subseteq \prod_{i=1}^k V_i=V,\] as desired.

The second claim follows from first, the fact that $\pi:\Conf_k(X)\to B_k(X)$ is a quotient map, and the fact that $\pi(\Conf_k^0(U,\sigma))=B_k^0(U)$ for every $\sigma$.
\end{proof}

These and related bases will be important for our later study, when we come to hypercover methods. For now, we draw the following consequences.

\begin{corollary}
Let $X$ be a locally path connected Hausdorff space. The projection $\pi:\Conf_k(X)\to B_k(X)$ is a covering space.
\end{corollary}
\begin{proof}
For $U\in \B_k$, we have $\Sigma_k$-equivariant identifications \[\pi^{-1}(B_k^0(U))=\bigcup_{\sigma:\{1,\ldots, k\}\cong \pi_0(U)} \Conf_k^0(U,\sigma)\cong B^0_k(U)\times\Sigma_k,\] where the second is induced by a choice of ordering of $\pi_0(U)$.
\end{proof}

As the example of the line with two origins shows, one cannot in general remove the hypothesis that $X$ be Hausdorff (we learned this example from Sander Kupers).

\begin{corollary}
If $M$ is an $n$-manifold, then $\Conf_k(M)$ and $B_k(M)$ are $nk$-manifolds.
\end{corollary}
\begin{proof}
We take $\B$ to be the set of Euclidean neighborhoods in $M$, in which case \[B_k^0(U)\cong \Conf_k^0\cong \mathbb{R}^{nk}\] for any $U\in \B_k$ and $\sigma:\{1,\ldots, k\}\xrightarrow{\simeq}\pi_0(U)$.
\end{proof}

\begin{exercise}
When is $B_k(M)$ orientable?
\end{exercise}

\subsection{Forgetting points}

From now on, unless otherwise specified, we take our background space to be a manifold $M$. In this case, we have access to a poweful tool relating configuration spaces of different cardinalities. 

The starting point is the observation that the natural projections from the product factor through the configuration spaces as in the following commuting diagram: \[\xymatrix{\Conf_\ell(M)\ar@{-->}[d]\ar[r]&M^\ell\ar[d]&(x_1,\ldots, x_\ell)\ar@{|->}[d]\\
\Conf_k(M)\ar[r]&M^k&(x_1,\ldots, x_k)
}\] (we take the projection to be on the last $\ell-k$ coordinates for simplicity, but it is not necessary to make this restriction). Clearly, the fiber over a configuration $(x_1,\ldots, x_k)$ in the base is the configuration space $\Conf_{\ell-k}(M\setminus\{x_1,\ldots, x_k\})$. Our first theorem asserts that the situation is in fact much better than this.

\begin{recollection}
Recall that, if $f:X\to Y$ is a continuous map, the \emph{mapping path space} of $f$ is the space of paths in $Y$ out of the image of $f$. In other words, it is the pullback in the diagram \[\xymatrix{E_f\ar[r]\ar[d]&Y^{[0,1]}\ar[d]&p\ar@{|->}[d]\\
X\ar[r]^-f&Y&p(0).
}\] The inclusion $X\to E_f$ given by the constant paths is a homotopy equivalence, and evaluation at $1$ defines a map $\pi_f:E_f\to Y$, which is a fibration \cite[7.3]{May:CCAT}. The \emph{homotopy fiber} of $f$ is the fiber \[\hofib(f):=\pi_f^{-1}(y)\] of this fibration, where $y\in Y$ is some basepoint. The construction of $E_f$, and hence the homotopy fiber, is functorial, and we say that a diagram \[\xymatrix{
X\ar[d]_-f\ar[r]&X'\ar[d]^-{f'}\\
Y\ar[r]&Y'
}\] is \emph{homotopy Cartesian}, or a \emph{homotopy pullback square}, if the induced map $\hofib(f)\to \hofib(f')$ is a weak equivalence. The primary benefit of knowing that a square is homotopy Cartesian is the induced Mayer--Vietoris long exact sequence in homotopy groups.
\end{recollection}

\begin{theorem}[Fadell--Neuwirth \cite{FadellNeuwirth:CS}]\label{thm:Fadell--Neuwirth}
Let $M$ be a manifold and $0\leq k\leq \ell<\infty$. The diagram \[\xymatrix{
\Conf_{\ell-k}(M\setminus\{x_1,\ldots, x_k\})\ar[d]\ar[r]&\Conf_\ell(M)\ar[d]\\
(x_1,\ldots, x_k)\ar[r]&\Conf_k(M)
}\] is homotopy Cartesian.
\end{theorem}

\begin{remark}
In fact, Fadell--Neuwirth prove that this map is a locally trivial fiber bundle, and they give an identification of its structure group. We will not need this full statement, and our alternate proof of this weaker form will allow us to illustrate the efficacy of hypercover methods at a later point.
\end{remark}

The proof is a debt that we will return to pay after having developed a few more advanced homotopy theoretic techniques. For the time being, we concentrate on exploiting this result.

\begin{corollary}
If $M$ is a simply connected $n$-manifold with $n\geq3$, then $\Conf_k(M)$ is simply connected for every $k\geq0$. In particular, $\pi_1(B_k(M))\cong\Sigma_k$.
\end{corollary}
\begin{proof}
The case $k=0$ is trivial and the case $k=1$ is our assumption. The Van Kampen theorem and our assumption on $n$ imply that $M\setminus\{\pt\}$ is simply connected, so the first claim follows by induction using the exact sequence \[\xymatrix{\pi_1(M\setminus \{\pt\})\ar[r]&\pi_1(\Conf_{k}(M))\ar[r]&\pi_1(\Conf_{k-1}(M)).}\] The second claim follows from the observation that $\Conf_k(M)\to B_k(M)$ is a $\Sigma_k$-cover with simply connected total space and hence the universal cover.
\end{proof}

\begin{corollary}\label{cor:surface aspherical}
If $M$ is a connected surface different from $S^2$ or $\mathbb{RP}^2$, then $\Conf_k(M)$ is aspherical for every $k\geq0$. In particular, $B_k(M)$ is aspherical.
\end{corollary}
\begin{proof}
The case $k=0$ is obvious and the case $k=1$ follows from our assumption on $M$. This assumption further guarantees that $M\setminus\{\pt\}$ is also aspherical, so the first claim follows by induction using the exact sequence \[\xymatrix{\pi_i(M\setminus \{\pt\})\ar[r]&\pi_i(\Conf_{k}(M))\ar[r]&\pi_i(\Conf_{k-1}(M))}\] with $i\geq2$. The second claim follows from the first and the fact that $\pi:\Conf_k(M)\to B_k(M)$ is a covering space.
\end{proof}

In order to proceed further, it will be useful to have a criterion for splitting these exact sequences.

\begin{proposition}
If $M$ is the interior of a manifold with non-empty boundary, then the map $\pi_{k,\ell}:\Conf_\ell(M)\to \Conf_k(M)$ admits a section up to homotopy.
\end{proposition}
\begin{proof}
Write $M=\mathring{N}$, and fix a collar neighborhood $\partial N\subseteq U$ and an ordered set $\{x_{k+1},\cdots, x_\ell\}$ of distinct points in $U$. By retracting along the collar, we obtain an embedding $\varphi:M\to M$ that is isotopic to the identity and misses our chosen points. The assignment $(x_1,\ldots, x_k)\mapsto (\varphi(x_1),\ldots, \varphi(x_k), x_{k+1},\cdots, x_\ell)$ defines a continuous map $s:\Conf_k(M)\to \Conf_\ell(M)$ such that $\pi_{k,\ell}\circ s=\Conf_k(\varphi)\simeq \id_{\Conf_k(M)}$, since $\varphi$ is isotopic to the identity.
\end{proof}

\begin{corollary}
For $n\geq 3$, $k\geq0$, and $i\geq0$, there is an isomorphism \[\pi_i(\Conf_k(\mathbb{R}^n))\cong \prod_{j=1}^{k-1} \pi_i\left(\bigvee_jS^{n-1}\right).\]
\end{corollary}
\begin{proof}
For $k\in \{0,1\}$ the claim is obvious, as is the claim for $\pi_0$, and the claim for $\pi_1$ has already been established. In the generic case, we proceed by induction using the exact sequence \[\xymatrix{
\pi_{i+1}(\Conf_{k-1}(\mathbb{R}^n))\ar[r]& \pi_i(\mathbb{R}^n\setminus\{x_1,\ldots, x_{k-1}\})\ar[r]&\pi_i(\Conf_k(\mathbb{R}^n))\ar[r]&\pi_i(\Conf_{k-1}(\mathbb{R}^n)).
}\] The section up to homotopy constructed above induces a section at the level of homotopy groups, so the lefthand map is trivial and the righthand map is surjective. The result now follows from the homotopy equivalence $\mathbb{R}^n\setminus\{x_1,\ldots, x_{k-1}\}\simeq \bigvee_{k-1}S^{n-1}$ and the fact that all groups in sight are Abelian.
\end{proof}

The higher homotopy groups of bouquets of spheres being very complicated objects \cite{Hilton:HGUS}, this result is a striking contrast to the situation in dimension 2 as characterized in Corollary \ref{cor:surface aspherical}.

\begin{remark}
It should be noted that the product decomposition of the previous corollary is additive only. Viewed as shifted Lie algebra via the Whitehead bracket, $\pi_*(\Conf_k(\mathbb{R}^n))$ has a very rich structure---see \cite[II]{FadellHusseini:GTCS}.
\end{remark}

The following result is proved by essentially the same argument.

\begin{corollary}
The fundamental group of $\Conf_k(\mathbb{R}^2)$ is an iterated extension of free groups.
\end{corollary}

\subsection{Braid groups}

We turn our attention now to the fundamental group of the unordered configuration space $B_k(\mathbb{R}^2)$. We fix the basepoints $\{(2i,0)\}_{i=1}^k\in B_k(\mathbb{R}^2)$ and $(2,\ldots, 2k)\in \Conf_k(\mathbb{R}^2)$ and employ them implicitly throughout. 

An element in $\pi_1(B_k(\mathbb{R}^2))$ is determined by the data of a permutation $\tau\in\Sigma_k$ and a path $p:[0,\pi]\to (\mathbb{R}^2)^k$ such that
\begin{enumerate}
\item $p(0)_r=(2r,0)$ for $1\leq r\leq k$,
\item $p(\pi)_r=(2\tau(r),0)$ for $1\leq r\leq k$, and
\item $p(t)_r\neq p(t)_{r'}$ for $1\leq r\neq r'\leq k$ and $t\in[0,\pi]$.
\end{enumerate} Composition is determined by composition of permutations and concatenation of paths; in particular, there is a canonical group homomorphism \[\pi_1(B_k(\mathbb{R}^2))\to \Sigma_k,\] which we will shortly see to be surjective. We have the following simple observation regarding this homomorphism.

\begin{lemma}
The subgroup $\pi_1(\Conf_k(\mathbb{R}^2))\leq \pi_1(B_k(\mathbb{R}^2))$ coincides with the kernel of the homomorphism $B_k\to \Sigma_k$.
\end{lemma}
\begin{proof}
It is obvious that $\pi_1(\Conf_k(\mathbb{R}^2))\leq \ker(B_k\to \Sigma_k)$, and both subgroups have the same index in $B_k$, since $\Conf_k(\mathbb{R}^2)\to B_k(\mathbb{R}^2)$ is $k!$-fold cover.
\end{proof}

To see verify surjectivity, we will exhibit elements $\sigma_i\in\pi_1(B_k(\mathbb{R}^2))$ lifting the respective transpositions $\tau_i=(i,i+1)$.

\begin{construction}
For $1\leq i\leq k$, define a path $p_i:[0,\pi]\to (\mathbb{R}^2)^k$ by the formula \[p_i(t)_r=\begin{cases}
(2r,0)&\quad r\notin\{i,i+1\}\\
c_{2i+1,1}(t+\pi)&\quad r=i\\ 
c_{2i+1,1}(t)&\quad r=i+1,
\end{cases}\] where $c_{a,b}(t)=(a+b\cos t,b\sin t)$ is the standard parametrization of the circle of radius $b$ centered at $(a,0)$.

The dashed factorization exists in the commuting diagram \[\xymatrix{[0,\pi]\ar[d]_-{p_i}\ar@{-->}[dr]\ar[r]&B_k(\mathbb{R}^2)\\
(\mathbb{R}^2)^k&\Conf_k(\mathbb{R}^2)\ar[l]\ar[u]_-\pi.
}\] Indeed, $\sin t=-\sin t$ if and only if $t\in \{0,\pi\}$, and in both of these cases we have $\cos t\neq -\cos t$. We write $\sigma_i$ for the homotopy class of the top horizontal map relative to $\{0,\pi\}$. Since $p_i(0)=((2,0),\ldots, (2i,0), (2i+2,0), \ldots, (2k,0))$ and $p_i(\pi)=((2,0),\ldots, (2i+2,0), (2i,0),\ldots, (2k,0))$, the class $\sigma_i$ defines an element of $\pi_1(B_k(\mathbb{R}^2))$.
\end{construction}

These elements satisfy some easy relations.

\begin{lemma}
If $|i-j|>1$, then $\sigma_i\sigma_j=\sigma_j\sigma_i$.
\end{lemma}
\begin{proof}
Without loss of generality, we may assume that $j>i$. Define $H:[0,\pi]^2\to (\mathbb{R}^2)^k$ by the formula \[H(s,t)_r=\begin{cases}
(2r,0)&\quad r\notin\{i,i+1,j,j+1\}\\
p_i(s)_r&\quad r\in\{i,i+1\}\\
p_j(t)_r&\quad r\in \{j,j+1\}.
\end{cases}\] For $r\in\{i,i+1\}$, the image of $H(s,t)_{r}$ is contained in a closed disk of radius $1$ around $(2i+1,0)$, and, for $r\in\{j,j+1\}$, in a closed disk of radius $1$ around $(2j+1,0)$. Since $|j-i|>1$, these sets are disjoint, and we conclude that $H$ factors through $\Conf_k(\mathbb{R}^2)$. The lemma now follows from the observation that $\sigma_j\sigma_i$ (resp. $\sigma_i\sigma_j$) is represented by the path obtained by composing $H$ with the counterclockwise (resp. clockwise) path around the boundary of $[0,\pi]^2$.
\end{proof}

\begin{lemma}
For $1\leq i\leq k$, $\sigma_i\sigma_{i+1}\sigma_i=\sigma_{i+1}\sigma_i\sigma_{i+1}$.
\end{lemma}
\begin{proof}
Define maps $q_1,q_2:[0,\pi]\to (\mathbb{R}^2)^k$ by the formulas \[
q_1(t)_r=\begin{cases}
(2r,0)&\quad r\notin\{i,i+1, i+2\}\\
c_{2i+2,2}(t+\pi)&\quad r=i\\
c_{2i+1,1}(2t)&\quad r=i+1\\
c_{2i+2,2}(t)&\quad r=i+2
\end{cases}
\]
\[
q_2(t)_r=\begin{cases}
(2r,0)&\quad r\notin\{i,i+1,i+2\}\\
\bar c_{2i+2,2}(t)&\quad r=i\\
\bar c_{2i+3,1}(2t+2\pi)&\quad r=i+1\\
\bar c_{2i+2,2}(t+\pi)&\quad r=i+2,
\end{cases}\] where a bar indicates reversal of parametrization. We claim that the concatenation of $q_1$ followed by $q_2$ represents $\sigma_{i+1}^{-1}\sigma_i^{-1}\sigma_{i+1}^{-1}\sigma_i\sigma_{i+1}\sigma_i$. Assuming this claim, the lemma folllows, for the three nonconstant coordinate functions of this concatenation are nullhomotopic via homotopies with pairwise disjoint images.

To see why the claim is true, consider the representative $\tilde q$ of $\sigma_{i+1}^{-1}\sigma_i^{-1}\sigma_{i+1}^{-1}\sigma_i\sigma_{i+1}\sigma_i$ that is given by concatenating the representatives $p_i$ and $p_{i+1}$ and their inverses. The path $\tilde q(t)_i$ first follows the lower half-circle of radius $1$ centered at $(2i+1,0)$, then follows the lower half-circle of radius $1$ centered at $(2i+3,0)$, then retraces these steps exactly. Thus, $\tilde q(t)_i$ is homotopic by a straight-line homotopy to a path that first follows the lower half-circle of radius 2 centered at $2i+2$ and then retraces this path, and the image of this homotopy away from $\tilde q(t)_i$ is disjoint from the images of $\tilde q(t)_r$ for $r\neq i$. In this way, we obtain a homotopy in the configuration space, and similar considerations apply to the $(i+2)$nd coordinate and the respective upper half-circles. Thus, $\sigma_{i+1}^{-1}\sigma_i^{-1}\sigma_{i+1}^{-1}\sigma_i\sigma_{i+1}\sigma_i$ is represented by a path in which the $i$th and $(i+2)$nd coordinate trace and then retrace these larger half-circles while the $(i+1)$st coordinate traverses the circle of radius 1 centered at $2i+1$ followed by the circle of radius 1 centered at $2i+3$ in reverse. Since the concatenation of $q_1$ followed by $q_2$ is such a path, this establishes the claim.
\end{proof}

\begin{definition}[Artin \cite{Artin:TB}]
The \emph{braid group} on $k$ strands is the group $B_k$ defined by the presentation \[B_k=\left\langle \sigma_1,\ldots, \sigma_{k-1}\mid\sigma_i\sigma_{i+1}\sigma_i=\sigma_{i+1}\sigma_i\sigma_{i+1},\, \sigma_i\sigma_j=\sigma_j\sigma_i \text{ if } |i-j|>1\right\rangle.\]
\end{definition}

The previous two lemmas imply the existence of a canonical group homomorphism from the abstract group $B_k$ to $\pi_1(B_k(\mathbb{R}^2))$.

\begin{theorem}[Fox--Neuwirth \cite{FoxNeuwirth:BG}]\label{thm:fox-neuwirth}
The homomorphism $B_k\to \pi_1(B_k(\mathbb{R}^2))$ is an isomorphism.
\end{theorem}

Before proceeding with the proof, we first recall that the symmetric group on $k$ letters admits the Coxeter presentation \[\Sigma_k=\langle \tau_1,\ldots, \tau_k\mid \tau_i^2=1,\,(\tau_i\tau_{i+1})^3=1,\,(\tau_i\tau_j)^2=1\text{ if }|i-j|>1\rangle.\] It follows that the assignment $\sigma_i\mapsto \tau_i$ determines a group homomorphism $B_k\to \Sigma_k$.

\begin{definition}
The \emph{pure braid group} on $k$ strands is the subgroup $P_k=\ker(B_k\to \Sigma_k)\leq B_k$.
\end{definition}

\begin{proof}[Proof of Theorem \ref{thm:fox-neuwirth}]
Consider the following diagram of group homomorphisms: \[\xymatrix{
1\ar[r]&P_k\ar[r]\ar@{-->}[d]&B_k\ar[r]\ar[d]&\Sigma_k\ar@{=}[d]\ar[r]&1\\
1\ar[r]&\pi_1(\Conf_k(\mathbb{R}^2))\ar[r]&\pi_1(B_k(\mathbb{R}^2))\ar[r]&\Sigma_k\ar[r]&1.
}\] The top row is exact by definition and the bottom was shown to be exact above. The righthand square commutes by the construction of the $\sigma_i$, and it follows that the dashed factorization exists making the lefthand square commute. Thus, by the five lemma, it will suffice to verify that the induced map $P_k\to \pi_1(\Conf_k(\mathbb{R}^2))$ is an isomorphism for each $k$. 

For this claim, we proceed by induction on $k$, the cases $k=0$ and $k=1$ being trivial. Consider the following diagram of group homomorphisms:
\[\xymatrix{
1\ar[r]&K\ar@{-->}[d]\ar[r]& P_k\ar[d]\ar@{-->}[r]& P_{k-1}\ar[d]^-{\mathrel{\rotatebox[origin=c]{-90}{$\simeq$}}}\ar[r]&1\\
1\ar[r]&\pi_1(\mathbb{R}^2\setminus \{x_1,\ldots, x_{k-1}\})\ar[r]&\pi_1(\Conf_k(\mathbb{R}^2))\ar[r]&\pi_1(\Conf_{k-1}(\mathbb{R}^2))\ar[r]&1.
}\] The middle and righthand vertical maps are the maps constructed above, the upper righthand horizontal map is defined by the inductive hypothesis and the requirement that the righthand square commute, and the subgroup $K\leq P_k$ is defined to be the kernel of this map. Since the top sequence is exact by definition and the bottom by Theorem \ref{thm:Fadell--Neuwirth} and the fact that the three spaces shown are all aspherical, the dashed vertical factorization exists making the lefthand square commute. 

As we will see below in Lemma \ref{lem:kernel isomorphism}, $K$ is generated by $k-1$ elements, which map to a set of $k-1$ free generators under the lefthand vertical map, and assuming this fact, we may complete the proof. Indeed, these $k-1$ elements determine a homomorphism from the free group, which is a section of the map in question. This section is injective, since it is a section, as well as surjective, since its image contains a generating set. Since the section is an isomorphism, the map itself is so as well, and the five lemma implies the claim.
\end{proof}

In this way, we are led to the following question, which will next occupy our attention: given a presentation of a group $G$, how can we find a presentation for a subgroup $H\leq G$?

\subsection{The Reidemeister--Schreier method} We now discuss a technique from combinatorial group theory for producing presentations of subgroups---see \cite[2.3]{MagnusKarrassSolitar:CGT} for a standard account. We take the topological viewpoint expounded in \cite{ZieschangVogtColdeway:SPDG}.

\begin{definition}
A \emph{graph} is a 1-dimensional CW complex $\Gamma$ with at most countably many cells. The 0-cells of $\Gamma$ are its \emph{vertices} and the 1-cells its \emph{edges}. An \emph{edge path} is a sequence of oriented edges $\{e_i\}_{i=1}^n$ such that the head of $e_i$ is the tail of $e_{i+1}$ for $1\leq i<n$. An edge path is \emph{reduced} if it contains no subpath of the form $ee^{-1}$, where $e^{-1}$ denotes the edge $e$ with the opposite orientation. A \emph{tree} is a contractible graph. A \emph{spanning tree} for a graph $\Gamma$ is a subgraph $T\subseteq \Gamma$ such that $T$ is a tree and $T$ contains every vertex of $\Gamma$.
\end{definition}

\begin{lemma}
In a tree $T$, any pair of distinct vertices are connected by a unique reduced edge path.
\end{lemma}
\begin{proof}
After perhaps discarding redundant edges, the existence of two such edge paths produces an embedding of $S^1$ into $T$, a contradiction.
\end{proof}

\begin{lemma}
If $\Gamma$ is a connected graph, then $\Gamma$ admits a spanning tree.
\end{lemma}
\begin{proof}
Fix a vertex $v\in \Gamma$, set $T_0=\{v\}$, and define $T_i\subseteq \Gamma$ for $i>0$ recursively by adding to $T_{i-1}$ each vertex $w\in \Gamma$ that is separated from $T_{i-1}$ by a exactly one edge, together with one such edge for each such $w$. Then $T=\bigcup_{i=0}^\infty T_i$ is a spanning tree.
\end{proof}

This observation implies the following familiar consequence.

\begin{corollary}
If $\Gamma$ is a connected graph, $v\in \Gamma$ is a vertex, and $T\subseteq \Gamma$ is a spanning tree, then $\pi_1(\Gamma,v)$ is freely generated by a set of cardinality $|\pi_0(\Gamma\setminus T)|$.
\end{corollary}
\begin{proof}
Extend a contraction of $T$ to the vertex $v$ to obtain a deformation of $\Gamma$ onto a wedge of $|\pi_0(\Gamma\setminus T)|$-many circles, and apply the Van Kampen theorem.
\end{proof}

In particular, after fixing a vertex $v\in\Gamma$ and a spanning tree $T\subseteq\Gamma$, a set of free generators for $\pi_1(\Gamma, v)$ is given by the collection of loops given by concatenating an edge $e$ in $\Gamma\setminus T$ on either side with the unique reduced edge paths in $T$ from $v$ to the endpoints of $e$.

If the free group $F(S)$ on the set $S$ corresponds to the graph $\Gamma_S:=\bigvee_S S^1$, then a subgroup $H\leq F(S)$ corresponds to a covering space $\pi:\widetilde \Gamma\to \Gamma_S$. Since any covering space has unique lifting for maps with simply connected domains, a covering space of a graph is again a graph in a canonical way; in particular, we conclude the following classical fact.

\begin{corollary}[Nielsen--Scheier]
A subgroup of a free group is free.
\end{corollary}

We now consider the problem of finding generators for $H$. We begin by fixing a basepoint $\tilde v\in\widetilde\Gamma$ lying over the unique vertex $v\in \Gamma_S$ and a spanning tree $\tilde v\in T\subseteq \widetilde\Gamma$. The edges of $\widetilde\Gamma$ correspond via $\pi$ to edges of $\Gamma_S$, which is to say generators of $G$, and the vertices correspond bijectively to the set $F(S)/H$ of cosets. For any vertex $w\in T$, there is a unique reduced edge path in $T$ from $\tilde v$ to $w$, and the word corresponding to this edge path represents the coset corresponding to $w$. This edge path clearly has the property that each of its initial segments is again such a path. 

\begin{definition}
A set $A=\{g_\ell\}_{\ell\in F(S)/H}$ of coset representatives is a \emph{Schreier set} for $H$ if, after writing each $g_\ell$ as a reduced word in $S$, every initial subword of $g_\ell$ is again an element of $A$.
\end{definition}

As a matter of notation, we write $g\mapsto \bar g$ for the composite $F(S)\to F(S)/H\cong A\to F(S)$.

\begin{theorem}[Reidemeister--Schreier, part I]
Let $H\leq F(S)$ be a subgroup and $A$ a Schreier set for $H$. The nontrivial elements of the form $g_\ell s(\overline{g_\ell s})^{-1}$ for $\ell\in F(S)/H$ and $s\in S$ freely generate $H$.
\end{theorem}

\begin{proof}
Fix a basepoint $\tilde v$. By path lifting, the Schreier set $A$ arises geometrically from a spanning tree $T$. The element $g_\ell s(\overline{g_\ell s})^{-1}$ corresponds to the path given by following the unique reduced edge path from $\tilde v$ to the vertex corresponding to $\ell$, crossing the edge $e$ corresponding to $s$, and then following the unique reduced edge path back to $\tilde v$. If $e$ lies in $T$, then $\overline{g_\ell s}= g_\ell s$, and the element is trivial; otherwise, we recognize $g_\ell s(\overline{g_\ell s})^{-1}$ as one of our previously established free generators for $H\cong\pi_1(\widetilde\Gamma, \tilde v)$.
\end{proof}

In order to accommodate the presence of relations, we extend our definition of a Schreier set as follows. 

\begin{definition}
Let  $G=\langle s\in S\mid r\in R\rangle$ be a group with a presentation and $H\leq G$ a subgroup. A set $A=\{g_\ell\}_{\ell\in G/H}$ of coset representatives is a \emph{Schreier set} for $H$ if it lifts to a Schreier set for $F(S)\times_GH\leq F(S)$.
\end{definition}

By the Van Kampen theorem, the group $G$ with the specified presentation is the fundamental group of the CW complex obtained by attaching 2-cells to $\Gamma_S$ according to the words $r\in R$. This cell structure lifts canonically to the covering space by attaching cells along the conjugates $g_\ell rg_\ell^{-1}$. Thus, we conclude the following generalization.

\begin{theorem}[Reidemeister--Schreier, part II]
Let $G=\langle s\in S\mid r\in R\rangle$ be a group with a presentation, $H\leq G$ a subgroup, and $A$ a Schreier set for $H$. Then $H$ is generated by the nontrivial elements $g_\ell s(\overline{g_\ell s})^{-1}$ for $\ell\in G/H$ and $s\in S$, with defining relations $g_\ell rg_\ell^{-1}$ for $\ell\in G/H$ and $r\in R$, written in the nontrivial generators.
\end{theorem}

\subsection{Back to braid groups} Before leveraging this result in completing the proof of Theorem \ref{thm:fox-neuwirth}, we pause to establish some notation. For $1\leq i<j\leq k$, we write $A_{ij}$ for the braid in which the $j$th strand winds once around the $i$th strand, passing in front of the intervening strands, and then returns to its starting position. In symbols, \[A_{ij}=\sigma_{j-1}\cdots\sigma_{i+1}\sigma_i^2\sigma_{i+1}^{-1}\cdots \sigma_{j-1}^{-1}.\] Note that the braids $A_{ij}$ are all pure. 

We write $U_k\leq B_k$ for the subgroup generated by the elements $A_{ik}$ for $1\leq i<k$. It is easy to verify geometrically that $U_k\leq \ker(P_k\to P_{k-1})$ and that the induced homomorphism $U_k\to \pi_1(\mathbb{R}^2\setminus\{x_1,\ldots, x_{k-1}\})$ sends $\{A_{ik}\}_{i=1}^{k-1}$ to a set of free generators; in particular, $U_k$ is free on these generators. Our goal is to prove the following, which is the missing ingredient in the proof of Theorem \ref{thm:fox-neuwirth}.

\begin{lemma}\label{lem:kernel isomorphism}
The inclusion $U_k\leq \ker(P_k\to P_{k-1})$ is an isomorphism.
\end{lemma}

This result will be an easy consequence of the following.

\begin{proposition}\label{prop:pure semidirect}
There is an isomorphism $P_k\cong U_k\rtimes P_{k-1}$.
\end{proposition}

\begin{proof}[Proof of Lemma \ref{lem:kernel isomorphism}]
Inverting isomorphisms in the diagram
\[\xymatrix{
U_k\ltimes P_{k-1}\ar[d]_-\wr\ar[r]&P_{k-1}\ar@{=}[d]\\
P_k\ar[d]&P_{k-1}\ar[d]^-\wr\\
\pi_1(\Conf_k(\mathbb{R}^2))\ar[r]&\pi_1(\Conf_{k-1}(\mathbb{R}^2)),
}\] of group homomorphisms, we obtain two maps $P_k\to P_{k-1}$, the kernels of which are $U_k$ and the kernel in question; therefore, to conclude that these two coincide, it suffices to verify that the diagram commutes. By inspection of the element in homotopy corresponding to $A_{i k}$, it is clear that both composites annihilate $U_k$, so it suffices to verify that the composite \[P_{k-1}\to P_k\to \pi_1(\Conf_k(\mathbb{R}^2))\to \pi_1(\Conf_{k-1}(\mathbb{R}^2))\] is the canonical map. By the previous corollary and induction on $k$, $P_{k-1}$ is generated by $\{A_{ij}\}_{1\leq i<j\leq k-1}$, and the claim follows from the geometric interpretation of the $A_{ij}$.
\end{proof}

Our proof of Proposition \ref{prop:pure semidirect} will be somewhat indirect, proceeding through the intermediate group $D_k\leq B_k$ of braids that do not permute the last strand. In other words, $D_k$ is defined as the pullback in the diagram \[\xymatrix{
D_k\ar[r]\ar[d]&B_k\ar[d]\\
\Sigma_{k-1}\ar[r]&\Sigma_k.
}\] The proposition is an immediate consequence of the following result.

\begin{proposition}\label{prop:D semidirect}
There is an isomorphism $D_k\cong U_k\rtimes B_{k-1}$.
\end{proposition}

We prove this result by applying the Reidemeister--Schreier method to $D_k\leq B_k$. Define elements $g_\ell=\sigma_{k-1}\cdots \sigma_\ell$ for $1\leq \ell\leq k$. 

\begin{lemma}\label{lem:schreier set}
The set $\{g_\ell\}_{\ell=1}^k$ is a Schreier set for $D_k$ in $B_k$.
\end{lemma}
\begin{proof}
Reading off the last entry of a permutation defines a bijection $\Sigma_k/\Sigma_{k-1}\cong \{1,\ldots, k\}$, and the last entry of the permutation $\tau_{k-1}\cdots \tau_\ell$ is $\ell$. Thus, the composite \[\{g_\ell\}_{\ell=1}^k\to B_k/D_k\to \Sigma_k/\Sigma_{k-1}\to \{1,\ldots, k\}\] is surjective and hence bijective. Since the rightmost two maps are also bijective, the first is as well, and we conclude that $\{g_\ell\}_{\ell=1}^k$ is a set of coset representatives for $D_k$ in $B_k$. Since any initial word in $g_\ell$ is obviously of the form $g_{\ell'}$ for some $\ell'$, the claim follows.
\end{proof}

In carrying out our computations below, we will need to know the following relations. 

\begin{lemma}\label{lem:g sigma relations}
The following relations hold in $B_k$:\begin{align*}
&\sigma_i^{g_\ell}=\sigma_i,  &1\leq i<\ell-1< k \\
&\sigma_i^{g_\ell}=\sigma_{i-1}, &1\leq \ell<i<k\\
&g_i\sigma_ig_{i+1}^{-1}=A_{ik}, &1\leq i< k\\
&g_i\sigma_{i-1}g_{i-1}^{-1}=1, &1<i\leq k.
\end{align*}
\end{lemma}
\begin{proof}
The third and fourth relations are obvious from the definitions, and the first is immediate from the commutativity relations in $B_k$. For the second, we have that \begin{align*}
g_\ell\sigma_ig_\ell^{-1}&=\sigma_{k-1}\cdots\sigma_\ell\sigma_i\sigma_\ell^{-1}\cdots \sigma_{k-1}^{-1}\\
&=\sigma_{k-1}\cdots \sigma_i\sigma_{i-1}\sigma_i\sigma_{i-2}\cdots \sigma_\ell\sigma_\ell^{-1}\cdots \sigma_{k-1}^{-1}\\
&=\sigma_{k-1}\cdots \sigma_i\sigma_{i-1}\sigma_i\sigma_{i-1}^{-1}\cdots \sigma_{k-1}^{-1}\\
&=\sigma_{k-1}\cdots \sigma_{i-1}\sigma_i\sigma_{i-1}\sigma_{i-1}^{-1}\cdots \sigma_{k-1}^{-1}\\
&=\sigma_{k-1}\cdots\sigma_{i+1}\sigma_{i-1}\sigma_{i+1}^{-1}\cdots\sigma_{k-1}^{-1}\\
&=\sigma_{j-1},
\end{align*} where the second equality follows from the commutativity relations, the third by cancellation, the fourth by the braid relations, the fifth by cancellation, and the last by commutativity.
\end{proof}

\begin{proof}[Proof of Proposition \ref{prop:D semidirect}]
We have that, modulo $\Sigma_{k-1}$, \[\tau_{k-1}\cdots \tau_\ell\tau_i\equiv \begin{cases}
\tau_{k-1}\cdots \tau_\ell&\quad 1\leq i<\ell-1< k\text{ or }1\leq \ell<i< k\\
\tau_{k-1}\cdots \tau_{\ell-1}&\quad 1\leq \ell-1=i<k\\
\tau_{k-1}\cdots \tau_{\ell+1}&\quad 1\leq \ell=i<k.
\end{cases}\] Thus, we have \[\overline{g_\ell\sigma_i}=\begin{cases}
g_\ell&\quad1\leq i<\ell-1< k\text{ or }1\leq \ell<i< k\\
g_{\ell-1}&\quad 1\leq \ell-1=i<k\\
g_{\ell+1}&\quad 1\leq \ell=i<k,
\end{cases}\] and Lemma \ref{lem:g sigma relations} now implies that \begin{align*}g_\ell\sigma_{\ell-1}(\overline{g_\ell\sigma_{\ell-1}})^{-1}&=g_\ell\sigma_{\ell-1}g_{\ell-1}^{-1}=1\\
g_\ell\sigma_{\ell}(\overline{g_\ell\sigma_{\ell}})^{-1}&=g_\ell\sigma_\ell g_{\ell+1}^{-1}=A_{\ell k}\\
g_\ell\sigma_{i}(\overline{g_\ell\sigma_{i}})^{-1}&=\sigma_i^{g_\ell}=\begin{cases}
\sigma_i&\quad 1\leq i<\ell-1< k \\
\sigma_{i-1}&\quad 1\leq \ell<i<k.
\end{cases}
\end{align*}
We conclude by the Reidemeister--Schreier method that $D_k$ is generated by the collection $\{A_{\ell k}, \sigma_i: 1\leq \ell\leq k,\, 1\leq i< k-1\}$, which is to say by the subgroups $U_k$ and $B_{k-1}$. 

In order to determine the relations, we conjugate the relations in $B_k$ by the $g_\ell$ and express the result in terms of our chosen generators using Lemma \ref{lem:g sigma relations}. We begin with the commutativity relations, for which there are six cases.
\begin{enumerate}
\item For $1\leq \ell<i<j-1< k-1$, \begin{align*}
[\sigma_i,\sigma_j]^{g_\ell}&= \sigma_i^{g_\ell}\sigma_j^{g_\ell}\sigma_i^{-g_\ell}\sigma_j^{-g_\ell}\\
&=[\sigma_{i-1},\sigma_{j-1}]
\end{align*} 
\item For $1\leq i=\ell <j-1<k-1$, \begin{align*}
[\sigma_i,\sigma_j]^{g_\ell}&=g_\ell\sigma_ig_{\ell+1}^{-1}g_{\ell+1}\sigma_jg_{\ell+1}^{-1}g_{\ell+1}\sigma_i^{-1}g_{\ell}^{-1}g_\ell\sigma_j^{-1}g_{\ell}^{-1}\\
&=[A_{ik},\sigma_{j-1}].
\end{align*}
\item For $1\leq i=\ell-1<j-1<k-1$, \begin{align*}
[\sigma_i,\sigma_j]^{g_\ell}&=g_\ell\sigma_ig_{\ell-1}^{-1}g_{\ell-1}
\sigma_jg_{\ell-1}^{-1}g_{\ell-1}\sigma_i^{-1}g_\ell^{-1}g_\ell\sigma_j^{-1}g_\ell^{-1}\\
&=[1, \sigma_{j-1}]\\
&=1.
\end{align*}
\item For $1\leq i <\ell-1=j-1<k-1$, \begin{align*}
[\sigma_i,\sigma_j]^{g_\ell}&=g_\ell\sigma_ig_\ell^{-1}g_\ell\sigma_jg_{\ell+1}^{-1}g_{\ell+1}\sigma_i^{-1}g_{\ell+1}^{-1}g_{\ell+1}\sigma_j^{-1}g_\ell^{-1}\\
&=[\sigma_i,A_{jk}].
\end{align*}
\item For $1< i+1<j=\ell-1<k-1$, \begin{align*}
[\sigma_i,\sigma_j]^{g_\ell}&=g_\ell\sigma_ig_\ell^{-1}g_\ell\sigma_jg_{\ell-1}^{-1}g_{\ell-1}\sigma_i^{-1}g_{\ell-1}^{-1}g_{\ell-1}\sigma_j^{-1}g_\ell\\
&=[\sigma_i,1]\\
&=1.
\end{align*}
\item For $1< i+1<j<\ell-1<k-1$, \begin{align*}
[\sigma_i,\sigma_j]^{g_\ell}&=\sigma_i^{g_\ell}\sigma_j^{g_\ell}\sigma_i^{-g_\ell}\sigma_j^{-g_\ell}\\
&=[\sigma_i,\sigma_j].
\end{align*}
\end{enumerate} We turn now to the braid relations, for which there are five further cases.
\begin{enumerate}
\item[(7)] For $1\leq \ell<i< k-1$, \begin{align*}
(\sigma_i\sigma_{i+1}\sigma_i\sigma_{i+1}^{-1}\sigma_i^{-1}\sigma_{i+1}^{-1})^{g_\ell}&=\sigma_i^{g_\ell}\sigma_{i+1}^{g_\ell}\sigma_i^{g_\ell}\sigma_{i+1}^{-g_\ell}\sigma_i^{-g_\ell}\sigma_{i+1}^{-g_\ell}\\
&=\sigma_{i-1}\sigma_{i}\sigma_{i-1}\sigma_{i}^{-1}\sigma_{i-1}^{-1}\sigma_{i}^{-1}. 
\end{align*}
\item[(8)] For $1\leq \ell=i<k-1$, \begin{align*}
(\sigma_i\sigma_{i+1}\sigma_i\sigma_{i+1}^{-1}\sigma_i^{-1}\sigma_{i+1}^{-1})^{g_\ell}&=g_\ell\sigma_ig_{\ell+1}^{-1}g_{\ell+1}\sigma_{i+1}g_{\ell+2}^{-1}g_{\ell+2}\sigma_ig_{\ell+2}^{-1}g_{\ell+2}\sigma_{i+1}^{-1}g_{\ell+1}^{-1}g_{\ell+1}\sigma_i^{-1}g_\ell^{-1}g_\ell\sigma_{i+1}^{-1}g_\ell^{-1}\\
&=A_{ik}A_{i+1,k}(A_{ik}A_{i+1,k})^{-\sigma_i}.
\end{align*}
\item[(9)] For $1\leq i=\ell-1\leq k-1$, \begin{align*}
(\sigma_i\sigma_{i+1}\sigma_i\sigma_{i+1}^{-1}\sigma_i^{-1}\sigma_{i+1}^{-1})^{g_\ell}&=g_\ell\sigma_ig_{\ell-1}^{-1}g_{\ell-1}\sigma_{i+1}g_{\ell-1}^{-1}g_{\ell-1}\sigma_ig_{\ell}^{-1}g_{\ell}\sigma_{i+1}^{-1}g_{\ell+1}^{-1}g_{\ell+1}\sigma_i^{-1}g_{\ell+1}^{-1}g_{\ell+1}\sigma_{i+1}^{-1}g_\ell^{-1}\\
&=A_{ik}^{\sigma_i}A_{i+1,k}^{-1}.
\end{align*}
\item[(10)] For $1< i+1=l-1\leq k-1$, \begin{align*}
(\sigma_i\sigma_{i+1}\sigma_i\sigma_{i+1}^{-1}\sigma_i^{-1}\sigma_{i+1}^{-1})^{g_\ell}&=g_\ell\sigma_ig_{\ell}^{-1}g_\ell \sigma_{i+1}g_{\ell-1}^{-1}g_{\ell-1}\sigma_ig_{\ell-2}^{-1}g_{\ell-2}\sigma_{i+1}^{-1}g_{\ell-2}^{-1}g_{\ell-2}\sigma_i^{-1}g_{\ell-1}^{-1}g_{\ell-1}\sigma_{i+1}^{-1}g_\ell^{-1}\\
&=[\sigma_i,1]\\
&=1.
\end{align*}
\item[(11)] For $1<i+1<\ell-1\leq k-1$, \begin{align*}
(\sigma_i\sigma_{i+1}\sigma_i\sigma_{i+1}^{-1}\sigma_i^{-1}\sigma_{i+1}^{-1})^{g_\ell}&=\sigma_i^{g_\ell}\sigma_{i+1}^{g_\ell}\sigma_i^{g_\ell}\sigma_{i+1}^{-g_\ell}\sigma_i^{-g_\ell}\sigma_{i+1}^{-g_\ell}\\
&=\sigma_i\sigma_{i+1}\sigma_i\sigma_{i+1}^{-1}\sigma_i^{-1}\sigma_{i+1}^{-1}
\end{align*}
\end{enumerate}

Relations (3), (5), and (10) are vacuous; relations (1), (6), (7), and (11) yield the defining presentation for $B_{k-1}$; and relations (2), (4), (8), and (9) imply that $U_k$ is a normal subgroup. We recognize the resulting presentation as a presentation of a semidirect product of $B_{k-1}$ with the free group on the set $\{A_{\ell k}\}_{\ell=1}^k$.
\end{proof}

\section{(Co)homology of $\Conf_k(\mathbb{R}^n)$}\label{section:cohomology}

Our present goal is to understand the cohomology ring $H^*(\Conf_k(\mathbb{R}^n))$ with integer coefficients. Since the cases $n\in\{0,1\}$ are exceptional and easy, we assume throughout that $n\geq0$.

\subsection{Additive structure} We begin with the task of understanding the cohomology as a graded Abelian group.

\begin{definition}
Let $V$ be a degreewise finitely generated non-negatively graded Abelian group $V$. The \emph{Poincar\'{e} polynomial} of $V$ is the polynomial \[P(V)=\sum_{i\geq0}\mathrm{rk}(V_i)t^i.\] If $X$ is a space of finite type, the \emph{Poincar\'{e} polynomial} of $X$ is the $P(X):= P(H_*(X))$.
\end{definition}

The Poincar\'{e} polynomial for graded Abelian groups is additive under direct sum and multiplicative under tensor product, so the Poincar\'{e} polynomial for spaces is additive under disjoint union and, with appropriate torsion-freeness assumptions in place, multiplicative under Cartesian product.

\begin{theorem}[Leray--Hirsch]\label{thm:Leray--Hirsch}
Suppose that the diagram \[\xymatrix{
F\ar[r]\ar[d]&E\ar[d]\\
\pt\ar[r]&B
}\] is homotopy Cartesian and that \begin{enumerate}
\item $F$ and $B$ are path connected,
\item $H^*(F)$ is free Abelian,
\item $H^*(F)$ or $H^*(B)$ is of finite type, and
\item $H^*(E)\to H^*(F)$ is surjective.
\end{enumerate} There is an isomorphism $H^*(E)\cong H^*(B)\otimes H^*(F)$ of $H^*(B)$-modules. In particular, we have the equation \[P(E)=P(B)P(F).\]
\end{theorem}

The proof, which is premised on a few basic properties of the Serre spectral sequence, is deferred to a later point in the notes, at which we will discuss this tool in some detail.

In order to apply the Leray--Hirsch theorem, we must verify point (4). In doing so, we employ the \emph{Gauss maps} \begin{align*}
\Conf_k(\mathbb{R}^n)&\xrightarrow{\gamma_{ab}} S^{n-1}\\
(x_1,\ldots, x_k)&\mapsto \frac{x_b-x_a}{\|x_b-x_a\|}.
\end{align*}

\begin{lemma}\label{lem:cohomology surjective}
The natural map $H^*(\Conf_k(\mathbb{R}^n))\to H^*(\mathbb{R}^n\setminus\{x_1,\ldots, x_{k-1}\})$ is surjective.
\end{lemma}
\begin{proof}
We first construct a collection of maps $\varphi_a:S^{n-1}\to \mathbb{R}^{n}\setminus\{x_1,\ldots, x_{k-1}\}$ inducing a homotopy equivalence from the bouquet $\bigvee_{k-1}S^{n-1}$. We will then show that each composite \[\xymatrix{S^{n-1}\ar[r]^-{\varphi_a}&\mathbb{R}^n\setminus\{x_1,\ldots, x_{k-1}\}\subseteq \Conf_k(\mathbb{R}^n))\ar[r]^-{\gamma_{ka}}& S^{n-1}
}\] is the identity, implying that $H^*(\mathbb{R}^n\setminus\{x_1,\ldots, x_{k-1}\})$ is generated by classes of the form $\gamma_{ka}^*(\alpha)$ for $\alpha\in H^*(S^{n-1})$.

For the construction, we set $x_a=(3^a,0,\ldots, 0)$ for $1\leq a\leq k-1$ and define $\varphi_a:S^{n-1}\to\mathbb{R}^n\setminus\{x_1,\ldots, x_k\}$ by \[\varphi_a(v)_r=(3^a,0,\ldots, 0)+ 3^av\] where $v\in S^{n-1}$ is regarded as a unit vector in $\mathbb{R}^n$. Then $\varphi_a(-1,\ldots, 0)=(0,\ldots, 0)$ for $1\leq a\leq k-1$, and the induced map from $\bigvee_{k-1}S^{n-1}$ is clearly a homotopy equivalence. Finally, we have \[\gamma_{ka}\circ\varphi_a(v)=\frac{(3^a,0,\ldots, 0)+3^av-(3^a,0,\ldots, 0)}{\|(3^a,0,\ldots, 0)+3^av-(3^a,0,\ldots, 0)\|}=\frac{3^av}{\|3^av\|}=v.\]
\end{proof}

Write $\alpha_{ab}\in H^{n-1}(\Conf_k(\mathbb{R}^n))$ for the pullback along $\gamma_{ab}$ of the standard volume form on $S^{n-1}$. We are now able to give a complete additive description of the cohomology ring in terms of these generators.

\begin{corollary}
For any $k\geq0$, $H^*(\Conf_k(\mathbb{R}^n))$ is free with basis \[S_k=\{\alpha_{a_1b_1}\alpha_{a_2b_2}\cdots \alpha_{a_{m}b_m}: m\geq0,\,1\leq b_1<\cdots<b_m \leq k,\, a_\ell<b_\ell\}.\] In particular, the Poincar\'{e} polynomial is given by \[P(\Conf_k(\mathbb{R}^n))=\prod_{j=1}^{k-1}(1+jt^{n-1}).\]
\end{corollary}
\begin{proof}
Since $\mathbb{R}^n\setminus\{x_1,\ldots, x_{k-1}\}$ and $\Conf_{k-1}(\mathbb{R}^n)$ are path connected and the cohomology of the former is free, Theorem \ref{thm:Fadell--Neuwirth} and the Lemma \ref{lem:cohomology surjective} allow us to apply the Leray--Hirsch theorem. The first and third claim now follow by induction and the observation that $H^*(\mathbb{R}^n\setminus\{x_1,\ldots, x_{k-1}\})$ is free with Poincar\'{e} polynomial $1+(k-1)t^{n-1}$. For the third claim, we note that, by induction, the Leray--Hirsch theorem gives the additive isomorphism \[H^*(\Conf_k(\mathbb{R}^n))\cong \mathbb{Z}\langle S_{k-1}\rangle\otimes \mathbb{Z}\langle1,\, \alpha_{ak},\, 1\leq a\leq k-1\rangle,\] and it is easy to check that the map \begin{align*}
S_{k-1}\times \{1,\, \alpha_{ak},\,1\leq a\leq k-1\}\to S_k
\end{align*} given by concatenation on the right is a bijection.
\end{proof}

\subsection{Multiplicative structure}

In order to obtain a multiplicative description, we require information about the relations among the various $\alpha_{ab}$. We begin with a few easy but useful observations. Recall that $\Sigma_k$ has a right action on $\Conf_k(\mathbb{R}^n)$ given by \[(x_1,\ldots, x_k)\cdot \sigma=(x_{\sigma(1)},\ldots, x_{\sigma(k)}).\] This action induces a left action on cohomology.

\begin{proposition} The following relations hold in $H^*(\Conf_k(\mathbb{R}^n))$ for $1\leq a\neq b\leq k$.
\begin{enumerate}
\item $\alpha_{ab}=(-1)^n\alpha_{ba}$
\item $\alpha_{ab}^2=0$
\item $\sigma^*\alpha_{ab}=\alpha_{\sigma(a)\sigma(b)}$ for $\sigma\in\Sigma_k$.
\end{enumerate}
\end{proposition}
\begin{proof}
The first claim follows from the observation that the diagram \[\xymatrix{
&\Conf_2(\mathbb{R}^n)\ar[dl]_-{\gamma_{ab}}\ar[dr]^-{\gamma_{ba}}\\
S^{n-1}\ar[rr]^-{x\mapsto -x}&&S^{n-1}
}\] commutes, and that the degree of the antipodal map on $S^{n-1}$ is $(-1)^n$. The second follows form the fact that the volume form on $S^{n-1}$ squares to zero. For the third claim, we have the commuting diagram \[\xymatrix{
\Conf_k(\mathbb{R}^n)\ar[dr]_-{\gamma_{\sigma(a)\sigma(b)}}\ar[rr]^-\sigma&&\Conf_k(\mathbb{R}^n)\ar[dl]^-{\gamma_{ab}}\\
&S^{n-1}.
}\]
\end{proof}

We come now to the fundamental relation.

\begin{proposition}[Arnold relation]
For $	1\leq a<b<c\leq k$, \[\alpha_{ab}\alpha_{bc}+\alpha_{bc}\alpha_{ca}+\alpha_{ca}\alpha_{ab}=0.\]
\end{proposition}

\begin{remark}
The Arnold relation holds without the assumption $a<b<c$. Indeed, the general form follows from the form given here and the antipodal relation $\alpha_{ab}=(-1)^n\alpha_{ba}$.
\end{remark}

We will discuss three proofs of this relation. For the time being, we concentrate on exploiting it. We write $\op A$ for the quotient of the free graded commutative algebra on generators $\{\alpha_{ab}\}_{1\leq a\neq b\leq k}$ by the ideal generated by $\{\alpha_{ab}+(-1)^{n+1}\alpha_{ba},\, \alpha_{ab}^2, \,\alpha_{ab}\alpha_{bc}+\alpha_{bc}\alpha_{ca}+\alpha_{ca}\alpha_{ab}\}$. 

\begin{theorem}[Arnold \cite{Arnold:CRCBG}, Cohen \cite{CohenLadaMay:HILS}]
The natural map of graded commutative algebras \[\op A\to H^*(\Conf_k(\mathbb{R}^n))\] is an isomorphism.
\end{theorem}
\begin{proof}
The map is surjective, since its image contains a generating set, so it will suffice to show that the basis $S_k$ exhibited above, thought of as lying in $\op A$, spans. Indeed, it follows that $S_k$ is a basis for $\op A$, since any relation would map to a relation in $H^*(\Conf_k(\mathbb{R}^n))$, and this fact implies the claim.

To verify that $S_k$ spans, it suffices to show that the element $\alpha_{a_1b_1}\cdots\alpha_{a_m b_m}$ may be written as a linear combination of elements in $S_k$. By the antipodal relation and graded commutativity, we may assume that $a_\ell<b_\ell$ for $1\leq \ell\leq m$ and that $1\leq b_1\leq\cdots\leq b_m\leq k$. We proceed by downward induction on the largest value of $\ell$ such that $b_\ell=b_{\ell-1}=:b$. By the Arnold and antipodal relations, we have \begin{align*}
\alpha_{a_1b_1}\cdots\alpha_{a_{\ell-1}b}\alpha_{a_\ell b}\cdots\alpha_{a_m b_m}&=(-1)^{n+1}\alpha_{a_1b_1}\cdots(\alpha_{a_\ell a_{\ell-1}}\alpha_{a_{\ell-1}b}+\alpha_{ba_\ell}\alpha_{a_\ell a_{\ell-1}})\cdots\alpha_{a_mb_m}\\
&=(-1)^{n+1} \alpha_{a_1b_1}\cdots \alpha_{a_\ell a_{\ell-1}}\alpha_{a_{\ell-1}b}\cdots \alpha_{a_mb_m}\\&\qquad+(-1)^{2n+1+(n-1)^2}\alpha_{a_1b_1}\cdots \alpha_{a_\ell a_{\ell-1}}\alpha_{a_\ell b}\cdots \alpha_{a_mb_m}
\end{align*} After a second induction on the number of values of $\ell$ such that $b_\ell=b$, we may assume that $b_{l-2}<b$. Then, using our assumption that $a_\ell<b$ and $a_{\ell-1}<b$, these two monomials lie in the span of $S_k$ by the first induction. 
\end{proof}

\subsection{The Arnold relation}

We now describe several approaches to proving the Arnold relation. The first reduction in all three cases is the observation that, using the projection $\Conf_k(\mathbb{R}^n)\to \Conf_3(\mathbb{R}^n)$ sending $(x_1,\ldots, x_k)$ to $(x_a,x_b,x_c)$, it suffices to verify the relation $\alpha_{12}\alpha_{23}+\alpha_{23}\alpha_{31}+\alpha_{31}\alpha_{12}=0$ in $H^*(\Conf_3(\mathbb{R}^n))$. 

The first argument, due to Cohen \cite{CohenLadaMay:HILS}, is the most elementary, and we will be able to give a complete account, since it uses only techniques that we have already encountered.

\begin{proof}[Cohen's proof of the Arnold relation]
We have already seen that $H^{2n-2}(\Conf_3(\mathbb{R}^n))$ is free of rank 2 with basis $\{\alpha_{12}\alpha_{23}, \alpha_{31}\alpha_{12}\}$, where we have used the antipodal relation and graded commutativity to rearrange indices in the second case. Note that another basis for this module is $\{\alpha_{12}\alpha_{23},\alpha_{23}\alpha_{31}\}$, since the permutation $\binom{123}{321}$ interchanges this set with our known basis up to sign. 

We conclude the existence of a relation \[x\alpha_{12}\alpha_{23}+y\alpha_{23}\alpha_{31}+z\alpha_{31}\alpha_{12}=0.\] Applying $\tau_{12}$ to this relation, we obtain the relation \begin{align*}
0&=x\alpha_{21}\alpha_{13}+y\alpha_{13}\alpha_{32}+z\alpha_{32}\alpha_{21}\\
&=(-1)^{(n-1)^2+2n}(x\alpha_{31}\alpha_{12}+y\alpha_{23}\alpha_{31}+z\alpha_{12}\alpha_{23}).
\end{align*} Canceling the sign and subtracting the result from the known relation yields \[(x-z)\alpha_{12}\alpha_{23}+(z-x)\alpha_{31}\alpha_{12}=0,\] whence $x=z$ by linear independence. Repeating the same process with $\tau_{23}$ shows that \[(x-y)\alpha_{12}\alpha_{23}+(y-x)\alpha_{23}\alpha_{31}=0,\] whence $x=y$ by linear independence. We conclude that the expression in question is $x$-torsion and therefore zero, since $H^*(\Conf_3(\mathbb{R}^n))$ is torsion-free.
\end{proof}

The original proof, due to Arnold \cite{Arnold:CRCBG}, is of a very different flavor, but is only valid in its original form in dimension 2.

\begin{proof}[Arnold's proof of the Arnold relation ($n=2$)]
Since there is no torsion, it suffices to prove the relation holds in cohomology with coefficients in $\mathbb{C}$. Make the identification $\mathbb{R}^2\cong\mathbb{C}$. The class $\alpha_{ab}$ is obtained by pulling back a standard generator of $H^1(S^1)$ along the composite \[\xymatrix{\Conf_k(\mathbb{C})\ar[r]^-{(z_a,z_b)}&\Conf_2(\mathbb{C})\ar[r]^-{z_2-z_1}&\mathbb{C}^\times\ar[r]^-{\frac{z}{\|z\|}}&S^1. 
}\] A representative for this generator in $H^1(\mathbb{C}^\times)$ is given by the differential form $dz/z$, since \[\frac{1}{2\pi i}\int_{S^1}\frac{dz}{z}=1\] by the residue theorem. Therefore, we may represent $\alpha_{ab}$ by the differential form \[\omega_{ab}=\frac{dz_b-dz_a}{z_b-z_a}.\] The claim now follows from the easy observation that the differential forms $\omega_{ab}$ satisfy the Arnold relation.
\end{proof}

This line of argument actually yields the far stronger result of a quasi-isomorphism \[H^*(\Conf_k(\mathbb{C}))\xrightarrow{\sim} \Omega^*(\Conf_k(\mathbb{C});\mathbb{C})\] of differential graded algebras, where the cohomology is regarded as a chain complex with zero differential; in jargon, $\Conf_k(\mathbb{C})$ is \emph{formal}.

In higher dimensions, the corresponding differential forms satisfy the relation only up to a coboundary, i.e., we have the equation \[\omega_{12}\omega_{23}+\omega_{23}\omega_{31}+\omega_{31}\omega_{12}=d\beta.\] Roughly, the differential form $\beta$ is obtained by integrating the form $\omega_{14}\omega_{24}\omega_{34}$ along the fibers of the projection $\pi:\Conf_4(\mathbb{R}^n)\to \Conf_3(\mathbb{R}^n)$ onto the first three coordinates. To see why this might be the case, we imagine that a fiberwise version of Stokes' theorem should imply that the boundary of the fiberwise integral should be the fiberwise integral along the ``boundary'' of the fiber, which in turn should be a sum of four terms: the first three terms are the loci where $x_i=x_4$ for $1\leq i\leq 3$, and the fourth lies at infinity, where $x_4$ is very far away. We might imagine that the three terms in the Arnold relation arise from these first three terms and that the term at infinity vanishes.

Of course, the fiber of this projection is non-compact, so, in order to make this kind of reasoning precise, one must replace the configuration space $\Conf_k(\mathbb{R}^n)$ with its \emph{Fulton-MacPherson compactification} $\Conf_k[\mathbb{R}^n]$, which is defined as the closure of the image of $\Conf_k(\mathbb{R}^n)$ under the maps \[\Conf_k(\mathbb{R}^n)\to (\mathbb{R}^n)^k\times(S^{n-1})^{\binom{n}{2}}\times[0,\infty]^{\binom{n}{3}}\] given by the inclusion in the first factor, the Gauss maps $\gamma_{ab}$ for $1\leq a<b\leq k$ in the second, and the relative distance functions $\delta_{abc}(x_1,\ldots, x_k)=\frac{\|x_a-x_b\|}{\|x_a-x_c\|}$ for $1\leq a<b<c\leq k$ in the third---see the original references \cite{FultonMacPherson:CCS, AxelrodSinger:CSPT} or the detailed account \cite{Sinha:MTCCS}. It turns out that $\Conf_k[\mathbb{R}^n]$ is a manifold with corners on which the integration described above can actually be carried out.

Using this compactification, Kontsevich \cite{Kontsevich:OMDQ} was able to carry out an analogue of Arnold's program from above. The basic observation is that the construction of $\beta$ is an example of a more systematic method for generating differential forms from graphs, which, when pursued fully, yields a zig-zag of quasi-isomorphisms \[\xymatrix{H^*(\Conf_k(\mathbb{R}^n))& ?\ar[r]^-\sim\ar[l]_-\sim& \Omega^*(\Conf_k[\mathbb{R}^n]),}\] where the unspecified middle term is a certain ``graph complex.'' Thus, in higher dimensions, too, configuration spaces are formal. See \cite{LambrechtsVolic:FLNDO} for a detailed proof of the formality theorem.

\subsection{Planetary systems}

The third proof of the Arnold relation will proceed through a geometric, intersection-theoretic analysis following \cite{Sinha:HLDO}. In order to pursue this direction, we will need to understand something about the homology of $H_*(\Conf_k(\mathbb{R}^n))$. We begin by introducing a systematic method for generating homology classes.

\begin{definition}
Fix a subset $S\subseteq\{1,\ldots, k\}$. 
\begin{enumerate}
\item An $S$-\emph{tree} $T$ is a pair of an ordering of $S$ and a binary parenthesization of $S$ with its ordering. 
\item A $k$-\emph{forest} is an ordered partition $\{1,\ldots, k\}\cong \coprod_i S_i$ and an $S_i$-tree for each $i$.
\end{enumerate}
\end{definition}

\begin{example}
With $S=\{1,3,4,7,8\}\subseteq \{1,\ldots, 9\}$, the expression $((48)((17)3))$ is an $S$-tree. 
\end{example}

\begin{example}
There is a unique $S$-tree with $S=\{i\}\subseteq\{1,\ldots, k\}$ given by the expression $i$.
\end{example}

The terminology is motivated by the observation that the data of an $S$-tree is equivalent to an isotopy class of planar tree $T$ with the following features:
\begin{enumerate}
\item $T$ has only univalent and trivalent vertices, called \emph{external} and \emph{internal}, respectively;
\item $T$ has a distinguished external \emph{root} vertex, and its other external vertices are \emph{leaves};
\item the leaves of $T$ are labeled by elements of $S$.
\end{enumerate}

The internal vertices in the geometric picture correspond bijectively to the pairs of matching open and closed parentheses in the combinatorial picture. We write $V(T)$ for the set of internal vertices of $T$. Note that an internal vertex lies on the path from the leaf $i$ to the root if and only if its parentheses enclose $i$; in this case, we write $v<i$. We define the \emph{height} $h(v)$ of a vertex $v$ to be the number of edges between $v$ and the root. In the combinatorial picture, the height of an internal vertex corresponds to the depth of the corresponding pair of parentheses.

Fix $S\subseteq \{1,\ldots, k\}$, and let $T$ be an $S$-tree. For each $i\in S$, define a map \begin{align*}
(S^{n-1})^{V(T)}&\xrightarrow{P_{T,i}} \mathbb{R}^n\\
(u_v)_{v\in V(T)}&\mapsto \displaystyle \sum_{v<i}(-1)^{\delta(i,v)}\epsilon^{h(v)}u_v,
\end{align*} where $\epsilon$ is a small positive real number, and $\delta(i,v)$ takes the value $1$ if the path from $i$ to the root passes through the left edge at $v$ and $0$ if the right edge. 

\begin{lemma}
For $1\leq i\neq j\leq k$, $P_{T,i}\left((u_v)_{v\in V(T)}\right)\neq P_{T,j}\left((u_v)_{v\in V(T)} \right)$.
\end{lemma}
\begin{proof}
Let $w$ denote the highest internal vertex with $w<i$ and $w<j$, and assume without loss of generality that $\delta(j,w)=0$ and $\delta(i,w)=1$. Then, cancelling terms involving $v<w$, we have
\begin{align*}
P_{T,j}\left((u_v)_{v\in V(T)}\right)-P_{T,i}\left((u_v)_{v\in V(T)}\right)&=\displaystyle \sum_{v<j}(-1)^{\delta(j,v)}\epsilon^{h(v)}u_v-\displaystyle \sum_{v<i}(-1)^{\delta(i,v)}\epsilon^{h(v)}u_v\\
&=\epsilon^{h(w)}\left(2u_w+\epsilon(\cdots)\right).
\end{align*} For $\epsilon$ sufficiently small, this expression does not vanish.
\end{proof}

Thus, taking the coordinates of $\{1,\ldots, k\}\setminus S$ to be fixed at some large, distinct values, we obtain a map $P_T:(S^{n-1})^{V(T)}\to \Conf_k(\mathbb{R}^n)$. We refer to $P_T$ and to the image of the fundamental class under $P_T$ interchangeably as the \emph{planetary system} associated to $T$. A forest $F=\{T_i\}$ also defines a planetary system $P_F$ by taking products of translates of the planetary systems of its component trees.

\begin{definition}
A tree is \emph{tall} if it is of the form $(\cdots(i_1i_2)\cdots i_m)$ with $i_1$ minimal considered as a natural number. A forest is \emph{tall} if 
\begin{enumerate}
\item each component tree is tall, and
\item the induced ordering on the minimal leaves of the component trees is the natural ordering.
\end{enumerate}
\end{definition}

\begin{proposition}
Planetary systems of tall trees form a basis for $H_{(k-1)(n-1)}(\Conf_k(\mathbb{R}^n))$.
\end{proposition}
\begin{proof}
From Leray--Hirsch and Fadell--Neuwirth, we know that $H^{(k-1)(n-1)}(\Conf_k(\mathbb{R}^n))$ is free Abelian of rank $(k-1)!$, so the homology group of interest has these properties as well. On the other hand, the set of tall trees on $\{1,\ldots, k\}$ is put into bijection with the set of permutations $\sigma\in \Sigma_{k}$ fixing $1$ by associating to $\sigma$ the tree \[T_\sigma=((\cdots(\sigma^{-1}(1)\sigma^{-1}(2))\cdots \sigma^{-1}(k))).\] Since this set has cardinality $(k-1)!$, we conclude that it suffices to show that the corresponding planetary systems are linearly independent. For this task, we define a map $\gamma_\sigma:\Conf_k(\mathbb{R}^n)\to (S^{n-1})^{k-1}$ by putting $\gamma_{\sigma^{-1}(i)\sigma^{-1}(i+1)}$ in the $i$th coordinate, and we set \[\alpha_\sigma=\gamma_\sigma^*(\mathrm{vol}_{S^{n-1}})=\alpha_{\sigma^{-1}(1)\sigma^{-1}(2)}\cdots\alpha_{\sigma^{-1}(k-1)\sigma^{-1}(k)}.\] Then the proof will be complete upon verifying that \[\left\langle P_{T_\sigma}, \alpha_{\tau}\right\rangle=\delta_{\sigma\tau}.\]

Write $\{v_1,\ldots, v_{k-1}\}$ for the vertices of $T_\sigma$, where $v_i$ is the unique internal vertex with $h(v_i)=k-i$. For $i<j$, we compute that $\gamma_{\sigma^{-1}(i)\sigma^{-1}(j)}(P_{T_\sigma}(u_{v_1},\ldots, u_{v_{k-1}}))$ is the unit vector in the direction of \begin{align*}
\left(\epsilon^{k-j+1}u_{v_{j-1}}-\sum_{\ell=j}^{k-1}\epsilon^{k-\ell}u_{v_\ell}\right)-\left(\epsilon^{k-i+1}u_{v_{i-1}}-\sum_{\ell=i}^{k-1}\epsilon^{k-\ell}u_{v_\ell}\right)
&=\epsilon^{k-j+1}\left(2u_{v_{j-1}}+\epsilon(\cdots)\right).
\end{align*} Letting $\epsilon$ tend to zero defines a homotopy between this map and the map \[(u_{v_1},\ldots, u_{v_{k-1}})\mapsto u_{v_{j-1}}.\] Therefore, $\gamma_\sigma\circ P_{T_\sigma}$ is homotopic to the identity, whence $\langle P_{T_\sigma},\alpha_\sigma\rangle=1$. 

Assume now that $\tau\neq \sigma$. Then there is some $1<i<k$ such that $\sigma\tau^{-1}(i)$ is greater than both $\sigma\tau^{-1}(i-1)$ and $\sigma\tau^{-1}(i+1)$, for otherwise, using the fact that $\sigma\tau^{-1}(1)=1$, we conclude that $\sigma\tau^{-1}$ is order-preserving and hence the identity, a contradiction. But then, by the previous calculation, the two maps \begin{align*}\gamma_{\tau^{-1}(i-1)\tau^{-1}(i)}&=\gamma_{\sigma^{-1}(\sigma\tau^{-1}(i-1))\sigma^{-1}(\sigma\tau^{-1}(i))}\\ \gamma_{\tau^{-1}(i)\tau^{-1}(i+1)}&=\gamma_{\sigma^{-1}(\sigma\tau^{-1}(i))\sigma^{-1}(\sigma\tau^{-1}(i+1))}
\end{align*} differ by the antipodal map up to homotopy, so $\gamma_\tau$ factors through a submanifold of $(S^{n-1})^k$ of positive codimension. It follows that $\langle P_{T_\sigma},\alpha_\tau\rangle=0$.
\end{proof}

\begin{corollary}
Planetary systems of tall forests form a basis for $H_*(\Conf_k(\mathbb{R}^n))$.
\end{corollary}
\begin{proof}
Let $F=\{T_{\sigma_i}\}$ be a tall forest, where $T_{\sigma_i}$ has $k_i$ leaves. We apply the same reasoning to the diagram
\[\xymatrix{(S^{n-1})^{V(F)}\ar@{=}[d]\ar[r]^-{P_F}&\Conf_k(\mathbb{R}^n)\ar[r]& \displaystyle\prod_i \Conf_{k_i}(\mathbb{R}^n)\ar[r]^-{(\gamma_{\sigma_i})}&\displaystyle\prod_{i}(S^{n-1})^{V(T_i)}\\
\displaystyle\prod_i(S^{n-1})^{V(T_i)}\ar[urr]_-{(P_{T_{\sigma_i}})},
}\] which commutes up to homotopy.
\end{proof}

\begin{recollection}
One version of Poincar\'{e} duality for oriented, connected, boundaryless, possibly non-compact $n$-manifolds of finite type is the isomorphism \[\widetilde H_i(M^+)\cong H^{n-i}(M),\] where $M^+$ denotes the one-point compactification of $M$ and we reduce with respect to the point at infinity. In particular, such a manifold has a fundamental class $[M]\in \widetilde H^n(M^+)$, defined as the preimage of $1\in H^0(M)$ under this isomorphism. This duality can sometimes be interpreted geometrically.
\begin{enumerate}
\item If $N\subseteq M$ is a proper submanifold of dimension $r$ and $P\subseteq M$ is a compact submanifold of dimension $n-r$, we may contemplate the composite \[\xymatrix{
\widetilde H_r(N^+)\otimes H_{n-r}(P)\ar[r]&\widetilde H_r(M^+)\otimes H_{n-r}(M)\cong H^{r}(M)\otimes H_{n-r}(M)\ar[r]^-{\langle-,-\rangle}&\mathbb{Z}. 
}\] (note that the existence of the first map uses the fact that $N$ is properly embedded). If $N$ and $P$ intersect transversely, then the value of this composite on $[N]\otimes [P]$ is the signed intersection number of $N$ and $P$.
\item Since cohomology is a ring, we may likewise contemplate the composite \[\xymatrix{
\widetilde H_r(N_1^+)\otimes \widetilde H_s(N_2^+)\ar[r]& H^{n-r}(M)\otimes H^{n-s}(M)\ar[r]^-\smile& H^{2n-r-s}(M)\cong \widetilde H^{r+s-n}(M^+),
}\] where $N_1$ and $N_2$ are proper submanifolds of dimension $r$ and $s$, respectively. If $N_1$ and $N_2$ intersect transversely, then the value of this composite on $[N_1]\otimes[N_2]$ is $[N_1\cap N_2]$.
\end{enumerate}
\end{recollection}

Now, consider the submanifold of $\Conf_3(\mathbb{R}^n)$ defined by requiring that $x_1$, $x_2$, and $x_3$ be collinear. This manifold has three connected components, and we let $C_a$ denote the component in which $x_a$ lies between $x_b$ and $x_c$. Then the map \begin{align*}
C_a&\to\mathbb{R}^n\times\mathbb{R}_{>0}\times\mathbb{R}_{>0}\times S^{n-1}\\
(x_1, x_2, x_3)&\mapsto \left(x_a,\, |x_b-x_a|,\, |x_c-x_a|, \,\frac{x_c-x_b}{|x_c-x_b|}\right)
\end{align*} is a homeomorphism. In particular, $\dim C_a=2n+1$. Note that $C_a$ is closed as a subspace of $\Conf_3(\mathbb{R}^n)$ and hence proper as a submanifold.

\begin{proof}[Sinha's proof of the Arnold relation]
Pushing forward $[C_1]$ and applying Poincar\'{e} duality as above, we obtain an element of $H^{n-1}(\Conf_3(\mathbb{R}^n))$. By our homology calculation, this class is determined by evaluating it on $P_{(12)}$ and $P_{(13)}$. These values are given by the respective intersection numbers with $C_1$, which are $\pm 1$ with opposite signs. Thus, with the appropriate choice of orientation, $C_1$ is Poincar\'{e} dual to $\alpha_{12}-\alpha_{13}$. Similar remarks apply to $C_2$, and, since $C_1\cap C_2=\varnothing$, we conclude that \begin{align*}
0&=(\alpha_{12}-\alpha_{13})(\alpha_{23}-\alpha_{21})\\
&=\alpha_{12}\alpha_{23}-\alpha_{12}\alpha_{21}-\alpha_{13}\alpha_{23}+\alpha_{13}\alpha_{21}\\
&=\alpha_{12}\alpha_{23}+(-1)^{n(n-1)}\alpha_{23}\alpha_{31}+(-1)^{2n}\alpha_{31}\alpha_{12}\\
&=\alpha_{12}\alpha_{23}+\alpha_{23}\alpha_{31}+\alpha_{31}\alpha_{12}.
\end{align*}
\end{proof}

\subsection{The Jacobi identity and little cubes}

The Arnold relation has its reflection in homology. For trees $T_1$ and $T_2$, we write $[T_1,T_2]$ for the tree obtained by grafting the roots of $T_1$ and $T_2$ to the leaves of $(12)$, in this order. 

\begin{proposition}[Jacobi identity]\label{prop:jacobi}
The relation $[[T_1,T_2], T_3]+[[T_2,T_3], T_1]+[[T_3,T_1],T_2]=0$ holds in $H_*(\Conf_k(\mathbb{R}^n))$. More generally, if $R$ is any tree, then the trees resulting from grafting the roots of $[[T_1,T_2], T_3]$, $[[T_2,T_3], T_1]$, and $[[T_3,T_1],T_2]$ to any fixed leaf of $R$ sum to zero.
\end{proposition}

It is possible to give a geometric derivation of the Jacobi identity---see \cite{Sinha:HLDO}---but we will pursue an alternate route. We begin by observing that the most basic case of the identity, in which $T_1$, $T_2$, $T_3$, and $R$ are all trivial trees with no internal vertices, is essentially immediate from what we have already done. 

\begin{proof}[Proof of Proposition \ref{prop:jacobi}, trivial case] We calculate that \begin{align*}
\left\langle ((23)1), \alpha_{12}\alpha_{23}\right\rangle&=\left\langle ((23)1), -\alpha_{23}\alpha_{31}-\alpha_{31}\alpha_{12}\right\rangle\\
&=-\left\langle((23)1), \alpha_{23}\alpha_{31}\right\rangle+(-1)^{1+2n+(n-1)^2}\left\langle ((23)1), \alpha_{21}\alpha_{13}\right\rangle\\
&=-\left\langle((13)2), \alpha_{13}\alpha_{32}\right\rangle+(-1)^n\left\langle((13)2), \alpha_{12}\alpha_{23}\right\rangle\\
&=-1,
\end{align*} where we have applied the permutation $\tau_{12}$ in going from the second to the third line, and the last equality follows from the perfect pairing between tall trees and the corresponding cohomology classes. A similar calculation shows that $\left\langle((23)1), \alpha_{31}\alpha_{12}\right\rangle=-1$, and it follows that \[((23)1)=-((31)2)-((12)3),\] as desired. 
\end{proof}

The general form of the Jacobi identity follows from this basic case once we are assured that grafting of trees is linear. In order to see why this linearity might hold, we turn to an alternative model for the homotopy types of configuration spaces---for original references, see \cite{BoardmanVogt:HIASTS, May:GILS}.

\begin{definition}
A \emph{little $n$-cube} is an embedding $f:(0,1)^n\to (0,1)^n$ of the form $f(x)=Dx+b$, where $b\in(0,1)^n$ and $D$ is a diagonal matrix with positive eigenvalues.
\end{definition}

We write $\op C_n(k)$ for the space of $k$-tuples of little $n$-cubes with pairwise disjoint images, topologized either as a subspace of $\mathrm{Map}(\amalg_k (0,1)^n, (0,1)^n)$. Note that, since little cubes are closed under composition, we have a collection of maps of the form \[\op C_n(m)\times \op C_n(k_1)\times\cdots\op C_n(k_m)\to \op C_n(k)\] whenever $k_1+\cdots +k_m=k$. These maps furnish the collection $\{\op C_n(k)\}_{k\geq0}$ of spaces with the structure of an \emph{operad} \cite{May:GILS}, but we will not need to make use of the full strength of this notion.

\begin{proposition}
The map $\op C_n(k)\to \Conf_k((0,1)^n)\cong\Conf_k(\mathbb{R}^n)$ given by evaluation at $(1/2,\ldots, 1/2)$ is a homotopy equivalence.
\end{proposition}
\begin{proof}[Sketch proof]
A section of the map in question is defined by sending a configuration $x$ to the unique $k$-tuple of little cubes $(f_1,\ldots, f_k)$ with the following properties:
\begin{enumerate}
\item $f_i(1/2, \ldots, 1/2)=x_i$ for $1\leq i \leq k$; 
\item all sides of each $f_i$ have equal length, and all $f_i$ have equal volume;
\item the images of the $f_i$ do not have pairwise disjoint closures.
\end{enumerate}
One checks that this map is continuous, so that we may view the configuration space as a subspace of $\op C_n(k)$. Scaling defines a deformation retraction onto this subspace.
\end{proof}

For further details, see \cite[4.8]{May:GILS}.

\begin{proof}[Proof of Proposition \ref{prop:jacobi}, general case]
By considering planetary systems of little cubes rather than configurations, one obtains the dashed lifts depicted in the diagram \[\xymatrix{
&\op C_n(k)\ar[d]\\
(S^{n-1})^{V(F)}\ar@{-->}[ur]^-{P_F^\Box}\ar[r]^-{P_F}&\Conf_k(\mathbb{R}^n).
}\] With these maps in hand, the combinatorics of grafting trees becomes the combinatorics of composing little cubes; that is, the tree $[[T_1, T_2], T_3]$ is the image of $\left( ((12)3), T_1, T_2, T_3\right)$ under the composition map \[\op C_n(3)\times \op C_n(3)\times \op C_n(k_1)\times \op C_n(k_2)\times\op C_n(k_3)\to \op C_n(k),\] and similar remarks pertain to grafting roots of trees onto a fixed leaf of a tree $R$. Thus, grafting, as the map induced on homology by a map of spaces, is linear, and the general identity follows from the basic case proven above.
\end{proof}

 A similar argument as in our earlier cohomology calculation, using the Jacobi identity to rebracket forests into sums of long forests, proves the following.

\begin{theorem}[Cohen]
The graded Abelian group $H_*(\Conf_k(\mathbb{R}^n))$ is isomorphic to the quotient of the free Abelian group with basis the set of $k$-forests by the Jacobi relations and signed antisymmetry.
\end{theorem}

\begin{remark}
This isomorphism may be promoted to an isomorphism of the operad of graded Abelian groups given by the collection $\{H_*(\op C_n(k))\}_{k\geq0}$ with the operad controlling $(n-1)$-shifted Poisson algebras. 
\end{remark}

\subsection{The unordered case}
We close this section with a calculation in the unordered case.

\begin{proposition}\label{prop:unordered rational}
For $k\geq2$ and $n\geq1$, there is an isomorphism \[
H_i(B_k(\mathbb{R}^n);\mathbb{Q})\cong\begin{cases}
\mathbb{Q}&\quad\text{if either $i=0$ or $i=n-1$ is odd}\\
0&\quad\text{otherwise.}
\end{cases}
\]
\end{proposition}

\begin{remark}
Note the vast difference in size and complexity between the rational homology of $B_k(\mathbb{R}^n)$ and that of $\Conf_k(\mathbb{R}^n)$. This disparity, which may at first seem surprising, is characteristic of the relationship between ordered and unordered configuration spaces in characteristic zero. In finite characteristic, as we will see, this relationship is reversed, and it is the homology in the unordered case that is by far more complex.

One obvious indicator of the rational difference between ordered and unordered is the fact that the $i$th Betti number of $\Conf_k(\mathbb{R}^n)$ tends to infinity with $k$, while that of $B_k(\mathbb{R}^n)$ quickly stabilizes to a fixed value. This observation is a simple example of the general phenomenon of \emph{homological stability} for configuration spaces of manifolds \cite{Church:HSCSM, RandalWilliams:HSUCS}. Although the Betti numbers in the ordered case do not stabilize, the analogous phenomenon of \emph{representation stability}, which takes the action of $\Sigma_k$ into account, does occur \cite{Farb:RS}.
\end{remark}

In making this calculation, we will use the following basic fact.

\begin{lemma}\label{lem:transfer}
Let $\pi: E\to B$ be a finite regular cover with deck group $G$. If $\mathbb{F}$ is a field in which $|G|$ is invertible, then the natural map \[\bar\pi_*:H_*(E; \mathbb{F})_G\to H_*(B;\mathbb{F})\] is an isomorphism.
\end{lemma}

This result is a consequence of the existence and basic properties of the \emph{transfer map}. Recall that the transfer is a wrong-way map on homology \[\mathrm{tr}:H_*(B)\to H_*(E)\] defined by sending a singular chain to the sum over its $|G|$ lifts to $E$, which is clearly a chain map. It is obvious from the definition that $\pi_*(\mathrm{tr}(\alpha))=|G|\alpha$.

\begin{proof}[Proof of Lemma \ref{lem:transfer}]
We claim that the composite \[\xymatrix{f:H_*(B;\mathbb{F})\ar[r]^-{\frac{1}{|G|}\mathrm{tr}}&H_*(E;\mathbb{F})\ar[r]&H_*(E;\mathbb{F})_{G}}\] is an inverse isomorphism to $\bar\pi_*$. Note that we have used the assumption that $|G|$ is invertible in $\mathbb{F}$ in defining $f$. In one direction, we compute that \[\bar\pi_*(f(\alpha))=\pi_*\left(\frac{1}{|G|}\mathrm{tr}(\alpha)\right)=\frac{1}{|G|}\pi_*(\mathrm{tr}(\alpha))=\alpha,\] and in the other we have \[f(\bar\pi_*([\beta]))=f(\pi_*(\beta))=\frac{1}{|G|}\left[\mathrm{tr}(\pi_*(\beta))\right]=\frac{1}{|G|}\left[\sum_{g\in G}g\cdot\beta\right]=\frac{1}{|G|}\left[\sum_{g\in G}\beta\right]=\beta.\]
\end{proof}

With the identification $H_*(B_k(\mathbb{R}^n);\mathbb{Q})\cong H_*(\Conf_k(\mathbb{R}^n);\mathbb{Q})_{\Sigma_k}$ in hand, we proceed by first identifying the coinvariants in top degree.

\begin{lemma}\label{lem:top homology}
For $k>1$, there is an isomorphism \[H_{(n-1)(k-1)}(\Conf_k(\mathbb{R}^n);\mathbb{Q})_{\Sigma_k}\cong\begin{cases}
\mathbb{Q}&\quad k=2 \text{ and $n$ even}\\
0&\quad\text{otherwise.}
\end{cases}\]
\end{lemma}
\begin{proof}
If $n$ is odd, then any tall tree $T$ is equal to the additive inverse of the tree obtained by switching the labels of the first two leaves of $T$. Since this operation may be achieved by the action of the symmetric group, it follows that $2[T]=0$ at the level of coinvariants, whence $[T]=0$. Since tall trees span the top homology, their images span the coinvariants, and the claim follows in this case.

Assume now that $n$ is even. If $k\geq3$, then the Jacobi identity applied to the bottom three leaves of a tall tree $T$ shows that $3[T]=0$, and so $[T]=0$, and we conclude as before. In the remaining case $k=2$, we note that $H_{n-1}(\Conf_2(\mathbb{R}^n))\cong\mathbb{Z}\langle P_{(12)}\rangle$, and that $\Sigma_2$ acts trivially.
\end{proof}

\begin{proof}[Proof of Proposition \ref{prop:unordered rational}]
As a consequence of our description in terms of tall forests, we have the following calculation:
\begin{align*}
H_*(\Conf_k(\mathbb{R}^n))&\cong \bigoplus_{\text{partitions of [k]}}\bigotimes_i H_{(n-1)(k_i-1)}(\Conf_{k_i}(\mathbb{R}^n))\\
&\cong \bigoplus_{r\geq0}\left(\bigoplus_{k_1+\cdots+k_r=k}\bigotimes_{i=1}^rH_{(n-1)(k_i-1)}(\Conf_{k_i}(\mathbb{R}^n))\otimes_{\Sigma_{k_1}\times\cdots\times\Sigma_{k_r}}\mathbb{Z}[\Sigma_k]\right)_{\Sigma_r}\\
&\cong\bigoplus_{r\geq0}\left(\bigoplus_{k_1+\cdots+k_r=k}\bigotimes_{i=1}^rH_{(n-1)(k_i-1)}(\Conf_{k_i}(\mathbb{R}^n))\otimes\mathbb{Z}[\Sigma_k]\right)_{\Sigma_r\ltimes \Sigma_{k_1}\times\cdots\times\Sigma_{k_r}}.
\end{align*} Thus, tensoring with $\mathbb{Q}$, forming the $\Sigma_k$-coinvariants, and using that $k!$ is invertible, we find that \begin{align*}
H_*(B_k(\mathbb{R}^n);\mathbb{Q})&\cong \bigoplus_{r\geq0}\left(\bigoplus_{k_1+\cdots+k_r=k}\bigotimes_{i=1}^rH_{(n-1)(k_i-1)}(\Conf_{k_i}(\mathbb{R}^n);\mathbb{Q})\right)_{\Sigma_r\ltimes \Sigma_{k_1}\times\cdots\times\Sigma_{k_r}}\\
&\cong\bigoplus_{r\geq0}\left(\bigoplus_{k_1+\cdots+k_r=k}\bigotimes_{i=1}^rH_{(n-1)(k_i-1)}(\Conf_{k_i}(\mathbb{R}^n);\mathbb{Q})_{\Sigma_{k_i}}\right)_{\Sigma_r}.
\end{align*} The claim now follows easily from Lemma \ref{lem:top homology}, since the only nonvanishing terms up to the action of $\Sigma_r$ are $(k_1,\dots, k_m)=(1,\ldots, 1)$ and possibly $(k_1,\ldots, k_m)=(2,1,\ldots, 1)$.
\end{proof}

With a few more definitions in hand, this calculation may be packaged in a more succinct form.

\begin{definition}
A \emph{symmetric sequence} of graded Abelian groups is a collection $\{V(k)\}_{k\geq0}$ where $V(k)$ is a graded Abelian group equipped with an action of $\Sigma_k$. 
\end{definition}

Thus, a symmetric sequence is equivalent to the data of a functor from the category $\Sigma$ of finite sets and bijections to graded Abelian groups. There is a notion of tensor product of symmetric sequences, which is given by the formula \begin{align*}(V\otimes W)(k)&=\bigoplus_{i+j=k}V(i)\otimes W(j)\otimes_{\Sigma_i\times\Sigma_j}\mathbb{Z}[\Sigma_k].
\end{align*} Defining a symmetric sequence by $H_*(\Conf(\mathbb{R}^n))(k)=H_*(\Conf_k(\mathbb{R}^n))$, we now recognize the identification \[H_*(\Conf(\mathbb{R}^n))\cong \Sym(H_\mathrm{top}(\Conf(\mathbb{R}^n)))\] with the symmetric algebra for this tensor product. 

Now, a symmetric sequence $V$ determines a bigraded Abelian group $V_\Sigma$ by the formula \[V_\Sigma=\bigoplus_{k\geq0}V(k)_{\Sigma_k},\] and it is immediate from the formula that \[(V\otimes W)_\Sigma\cong V_\Sigma\otimes W_\Sigma.\] Thus, we have an isomorphism of bigraded vector spaces \begin{align*}
\textstyle\bigoplus_{k \geq0}H_*(B_k(\mathbb{R}^n);\mathbb{Q})&\cong H_*(\Conf(\mathbb{R}^n))_\Sigma\\
&\cong \Sym(H_\mathrm{top}(\Conf(\mathbb{R}^n)))_\Sigma\\
&\cong \Sym(H_\mathrm{top}(\Conf(\mathbb{R}^n))_\Sigma)\\
&\cong \Sym(\mathbb{Q}[0,1]\oplus \mathbb{Q}[n-1, 2]).
\end{align*}

\begin{remark}
From the operadic point of view, this bigraded Abelian group is the free shifted Poisson algebra on one generator.
\end{remark}

This calculation illustrates a valuable lesson, namely that configuration spaces tend to exhibit more structure when taken all together. This insight will be indispensable to us in our future investigations. Before pursuing this direction, however, we will need to invest in some new tools.

\section{Covering theorems}\label{section:covering theorems}

Having exploited it at length, our next long-term goal is to circle back and prove our version of the Fadell--Neuwirth theorem asserting that the diagram \[\xymatrix{
\Conf_{\ell-k}(M\setminus\{x_1,\ldots, x_k\})\ar[r]\ar[d]&\Conf_\ell(M)\ar[d]\\
(x_1,\ldots,x_k)\ar[r]&\Conf_k(M)
}\] is homotopy Cartesian. This type of statement is about the local homotopy type of the configuration space, while, through the topological basis that we exhibited in Proposition \ref{prop:conf basis}, we have fine control over the local topology. Of course, in some sense, everything about a space $X$ is determined by a basis, since $X$ can be reconstructed from the basis by gluing. On the other hand, as the following classic example illustrates, gluing is not a homotopically well-behaved operation.

\begin{example}
The two diagrams
\[\xymatrix{
S^{n-1}\times(0,1)\ar[d]\ar[r]&\mathring{D}^n\ar[d]&&S^{n-1}\ar[r]\ar[d]&\pt\ar[d]\\
\mathring{D}^n\ar[r]&S^n&&\pt\ar[r]&\pt
}\] are pushout squares, and there is a map from the left square to the right that is a homotopy equivalence on all but the bottom right corner.
\end{example}

This example illustrates that ordinary gluing, which is a colimit construction, is insufficient for the kind of homotopy theoretic questions we wish to pursue. The replacement will be the \emph{homotopy colimit}, which we review at length in Appendix \ref{section:homotopy colimits}. 

The kind of result that we aim to prove is the following.

\begin{theorem}\label{thm:basis recovery}
Let $\mathcal{B}$ be a topological basis for $X$, regarded as a poset under inclusion and thereby as a category. The natural map \[\hocolim_{U\in \mathcal{B}}U\to X\] is a weak equivalence.
\end{theorem}

This result will be an immediate consequence of Theorem \ref{thm:complete cover recovery}. First, we explore the intermediary concepts of \v{C}ech covers and hypercovers, which are interesting in their own right. These tools all permit the reconstruction of the weak homotopy type of $X$ from various forms covering data. A general reference for this material is \cite{DuggerIsaksen:THAR}.

\subsection{\v{C}ech covers}

 \begin{recollection}
We write $\Delta_+$ for the category of finite ordered sets and $\Delta\subseteq \Delta_+$ for the full subcategory of sets that are nonempty. Thus, up to isomorphism, the objects of $\Delta_+$ are the sets $[n]=\{0,\ldots, n\}$ for $n\geq-1$. A \emph{simplicial space} is a functor $\op X:\Delta^{op}\to \Top$ (resp. \emph{augmented simplicial space}, $\Delta_+^{op}$). 

Using the traditional notation $\op X_n:=\op X([n])$, we write $d_i:\op X_n\to \op X_{n-1}$ for map induced by the inclusion $[n-1]\to [n]$ that misses the element $i$ (the $i$th \emph{face map}), and $s_i:\op X_{n}\to \op X_{n+1}$ for the map induced by the surjection $[n+1]\to [n]$ that sends $i$ and $i+1$ to $i$ (the $i$th \emph{degeneracy}). If $\op X$ is augmented, we refer to the induced map $\op X_0\to \op X_{-1}$ as the \emph{augmentation}.

The \emph{geometric realization} of the simplicial space $\op X$ is the quotient \[|\op X|:=\faktor{\coprod_{n\geq0}\op X_n\times\Delta^n}{\sim}=\mathrm{coeq}\left(\coprod_{\Delta([\ell],[m])}\op X_m\times\Delta^\ell\rightrightarrows \coprod_{n\geq0} \op X_n\times\Delta^n\right),\] where the two arrows are given by the covariant functoriality of $\Delta^{(-)}$ and the contravariant functoriality of $\op X$, respectively. If $\op X$ is augmented, there results a canonical map $|\op X|\to \op X_{-1}$.
\end{recollection}

\begin{definition}
Let $\U=\{U_\alpha\}_{\alpha\in A}$ be an open cover of $X$. The \emph{\v{C}ech nerve} of $\U$ is the augmented simplicial space $\check{C}(\U):\Delta^{op}_+\to \Top$ specified as follows.
\begin{enumerate}
\item In nonnegative simplicial degree, we have \[\check{C}(\U)_n=\coprod_{A^{n+1}}U_{\alpha_0}\cap\cdots\cap U_{\alpha_n},\] and $\check{C}(\U)_{-1}=X$.
\item The face maps $d_i$ is induced by the inclusions \[U_{\alpha_0}\cap\cdots\cap U_{\alpha_n}\subseteq U_{\alpha_0}\cap\cdots \widehat{U}_{\alpha_i}\cdots \cap U_{\alpha_n}.\]
\item The degeneracy $s_i$ is induced by the identifications \[U_{\alpha_0}\cap\cdots\cap U_{\alpha_n}=U_{\alpha_0}\cap \cdots\cap U_{\alpha_i}\cap U_{\alpha_i}\cap \cdots\cap U_{\alpha_n}\]
\item The augmentation is induced by the inclusions $U_\alpha\subseteq X$.
\end{enumerate}
\end{definition}

Applying $H_0$ to $\check{C}(\U)$ in each simplicial degree and taking the alternating sum of the face maps as a differential, we obtain the classical \emph{\v{C}ech complex} \[\cdots \to \bigoplus_{A^{n+1}}H_0(U_{\alpha_0}\cap \cdots \cap U_{\alpha_n})\to \cdots\to \bigoplus_AH_0(U_\alpha),\] which computes the homology of $X$ if $\U$ is a sufficiently good cover. In fact, this result can be strengthened to a recovery of the full weak homotopy type.

\begin{theorem}[Segal, Dugger-Isaksen]\label{thm:cech recovery}
For any topological space $X$ and any open cover $\U$ of $X$, the augmentation \[|\check{C}(\U)|\to X\] is a weak homotopy equivalence.
\end{theorem}

The proof will make use of a wonderful local-to-global principle. In establishing this principle, we will employ a little machinery, but see \cite[16.24]{Gray:HT} for a more elementary argument premised on subdivision.

\begin{proposition}\label{prop:mayer-vietoris}
Let $f:Y\to Z$ be a continuous map and $\U=\{U_\alpha\}_{\alpha\in A}$ an open cover of $Z$. If the induced map $f^{-1}(U_{\alpha_0}\cap\cdots\cap U_{\alpha_n})\to U_{\alpha_0}\cap\cdots\cap U_{\alpha_n}$ is a weak homotopy equivalence for every $(\alpha_0,\ldots, \alpha_n)\in A^{n+1}$, then $f$ is also a weak homotopy equivalence.
\end{proposition}
\begin{proof}
Suppose first that $\U=\{U,V\}$. Using the appropriate versions of the Van Kampen \cite{BrownRazekSalleh:VKTUNS} and Mayer-Vietoris \cite[5.13]{DavisKirk:LNAT} theorems, we conclude that $f$ induces isomorphisms on fundamental groupoids and on homology with arbitrary local coefficients, so the claim follows in this case from the Whitehead theorem \cite[6.71]{DavisKirk:LNAT}.

In the general case, we choose an ordinal $\lambda$ and a bijection $\varphi:\lambda\cong \U$ and proceed by transfinite induction, assuming that the claim is known for all covers with the cardinality of $\mu<\lambda$. 

Suppose first that $\lambda$ is a successor ordinal. Setting $U=U_\lambda$ and $V=\bigcup_{\mu<\lambda}U_\mu$, it will suffice by the previous argument to verify that the restrictions of $f$ to $U$, to $V$, and to $U\cap V$ are all weak homotopy equivalences. The case of $U$ is a special case of our hypothesis, and the cases of $V$ and $U\cap V$ follow from the inductive assumption applied to the covers $\{U_\mu\}_{\mu<\lambda}$ and $\{U_\mu\cap U_\lambda\}_{\mu<\lambda}$, respectively, each of which is in bijection with $\lambda-1$ and satisfies the hypothesis of the proposition.

Suppose now that $\lambda$ is a limit ordinal. By compactness, any map $(D^{n+1}, S^n)\to (Z,Y)$ factors as in the solid commuting diagram 
\[\xymatrix{
S^n\ar[d]\ar[d]\ar[r]&\ar[d]\displaystyle f^{-1}\left(\bigcup_{\mu<\mu_0}U_\mu\right)\ar[r]&Y\ar[d]^-f\\
D^{n+1}\ar[r]\ar@{-->}[ur]&\displaystyle\bigcup_{\mu\leq \mu_0}U_\mu\ar[r]& Z
}\] for some $\mu_0<\lambda$. The inductive hypothesis applied to the cover $\{U_\mu\}_{\mu\leq \mu_0}$ implies that the middle map is a weak homotopy equivalence, so the dashed lift exists making the upper triangle commute and the lower triangle commute up to homotopy. It follows that $\pi_n(f)=0$ for every $n\geq0$, and consideration of the long exact sequence in homotopy associated to $f$ completes the proof.
\end{proof}

\begin{proof}[Proof of Theorem \ref{thm:cech recovery}] For $(\alpha_0,\ldots, \alpha_n)\in A^{n+1}$, set $\U_{\alpha_0,\ldots, \alpha_n}=\{U_\alpha\cap U_{\alpha_0}\cap\cdots\cap U_{\alpha_n}\}_{\alpha\in A}$, and note that the diagram
\[\xymatrix{
|\check{C}(\U_{\alpha_0,\ldots, \alpha_n})|\ar[r]\ar[d]&|\check{C}(\U)|\ar[d]\\
U_{\alpha_0}\cap\cdots\cap U_{\alpha_n}\ar[r]& X
}\] is a pullback. Therefore, by Proposition \ref{prop:mayer-vietoris}, it suffices to prove the claim under the assumption that $X$ is a member of $\U$. In this case, $\check{C}(\U)$ has an extra degeneracy given by forming the intersection with $X$. 
\end{proof}

\subsection{Hypercovers}
In light of Theorem \ref{thm:cech recovery}, it is natural to wonder what makes \v{C}ech nerves of covers special. The first step toward answering this question is to notice that \v{C}ech nerves may be completely characterized.

\begin{definition}
Let $f:Y\to Z$ be a continuous map. We say that $f$ is a \emph{covering map} if, up to homeomorphism, it is of the form $\coprod_{\alpha\in A} U_\alpha\to Z$ for some open cover $\U=\{U_\alpha\}_{\alpha\in A}$ of $Z$.
\end{definition}

\begin{definition}
We say that an augmented simplicial space $\op X$ is a \emph{\v{C}ech cover} if \begin{enumerate}
\item the augmentation $\op X_0\to \op X_{-1}$ is a covering map, and
\item for every $n>0$, the natural map \[\op X_n\to  M_n(\op X):=\mathrm{eq}\left(\prod_{0\leq i\leq n}\op X_{n-1}\rightrightarrows\prod_{0\leq i<j\leq n}\op X_{n-2}\right)\] induced by the face maps is a homeomorphism.
\end{enumerate}
\end{definition}

The space $M_n(\op X)$ is called the $n$th \emph{matching space} of $\op X$.

\begin{example}
In low degrees, we have $M_0(\op X)=\op X_{-1}$ and $M_1(\op X)=\op X_0\times_{\op X_{-1}}\op X_0$.
\end{example}

\begin{exercise}
The \v{C}ech nerve of an open cover is a \v{C}ech cover, and, conversely, a \v{C}ech cover $\op X$ is the \v{C}ech nerve of some open cover $\U$---for example, we may take $\U$ to be the collection of connected components of $\op X_0$---but different open covers may have isomorphic \v{C}ech nerves.
\end{exercise}

Thus, a \v{C}ech cover is exactly what we obtain by deforming the constant simplicial object at $X$ by replacing the 0-simplices by a cover and filling in the rest of the object in the canonical way with matching objects. This observation identifies \v{C}ech covers as the first stage in an obvious hierarchy of covering notions.

\begin{definition}
We say that an augmented simplicial space $\op X:\Delta^{op}_+\to \Top$ is a \emph{hypercover} if the canonical map $\op X_n\to M_n(\op X)$ is a covering map for all $n\geq0$. We say that the hypercover $\op X$ is \emph{bounded} if there is some $N$ such that this map is an isomorphism for $n>N$, the smallest such $N$ being the \emph{height} of $\op X$. If $\op X_{-1}=X$, then we say that $\op X$ is a hypercover of $X$.
\end{definition}

A hypercover of height 0 is precisely a \v{C}ech cover, while a hypercover of height $-1$ is isomorphic to the constant simplicial space with value $X$.

\begin{theorem}[Dugger-Isaksen]\label{thm:hypercover recovery}
For any topological space $X$, and any hypercover $\op X$ of $X$, the augmentation map \[|\op X|\to X\] is a weak homotopy equivalence.
\end{theorem}

The proof will make use of some formal machinery from the theory of simplicial spaces, which identifies the matching object $M_n(\op X)$ itself as the degree $n$ entry of a simplicial space. First, a few categorical reminders.

\begin{recollection}
Let $\op C$ be a category and $\D$ a category with small limits and colimits, and let $\iota:\op C_0\to \op C$ be a functor. Then the restriction functor $\iota^*:\op D^{\op C}\to \op D^{\op C_0}$ admits both a left and a right adjoint, the so-called \emph{Kan extension} functors, as depicted in the following diagram \[\doubleadjunct{\op D^{\op C_0}}{\op D^{\op C}}{\op D^{\op C_0}}{\iota_!}{\iota^*}{\iota^*}{\iota_*}.\] We may also use the notation $\Lan_\iota=\iota_!$ and $\Ran_\iota=\iota_*$ for the left and right Kan extensions, respectively. The value of the left Kan extension is given by the formula \[\iota_!F(C)\cong \colim\left((\op C_0\downarrow C)\to \op C\xrightarrow{F} \op D\right),\] where $(\op C_0\downarrow C)$ denotes the \emph{overcategory} whose objects are morphisms $f:C'\to C$ with $C'$ an object of $\op C_0$, and whose morphisms are commuting triangles \[\xymatrix{
C'\ar[rr]\ar[dr]&& C''\ar[dl]\\
&C.
}\] Dually, the right Kan extension $\iota_*F$ is computed as the corresponding limit over the undercategory $(C\downarrow\op C_0)$.
\end{recollection}

Taking $i:(\Delta^{op}_+)_{\leq n}\to  \Delta^{op}_+$ to be the inclusion of the full subcategory of finite ordered sets of cardinality at most $n+1$, we obtain adjunctions \[\doubleadjunct{\Top^{(\Delta^{op}_+)_{\leq n}}}{\Top^{\Delta^{op}_+}}{\Top^{(\Delta^{op}_+)_{\leq n}}}{\iota_!}{\iota^*}{\iota^*}{\iota_*}.\] We write $\tau_{\leq n}(\op X)=\iota^*\op X$, $\sk_n(\op X)=\iota_!\iota^*\op X$, and $\cosk_n(\op X)=\iota_*\iota^*\op X$ and refer to the $n$-\emph{truncation}, $n$-\emph{skeleton}, and $n$-\emph{coskeleton}, respectively, of $\op X$.

An exercise in the manipulation of limits shows that matching objects and coskeleta are related as \[M_n(\op X)\cong \cosk_{n-1}(\op X)_n,\] and the map $\op X_n\to M_n(\op X)$ discussed above is a component of the unit transformation of the appropriate adjunction. Note that $\cosk_{n-1}(\op X)_k\cong \op X_k$ for $k<n$, so the matching object is in some sense the primary measure of the difference between $\op X$ and its $n-1$-coskeleton. 

We will need a few facts about the behavior of coskeleta of hypercovers. First, we note that, by elementary properties of Kan extensions, we have \[
\cosk_m(\cosk_n(\op X))\cong\begin{cases}
\cosk_n(\op X)&\quad m\geq n\\
\cosk_m(\op X)&\quad m\leq n
\end{cases}\] for any augmented simplicial space $\op X$. Second, we have the following pullback diagram for any $n$ and $m$
\[\xymatrix{
\cosk_{m}(\op X)_n\ar[d]\ar[r]&\displaystyle\prod_{[m]\subseteq [n]} \op X_m\ar[d]\\
\cosk_{m-1}(\op X)_n\ar[r]&\displaystyle\prod_{[m]\subseteq [n]} M_m(\op X),
}\] where the products are indexed on the set of injective order-preserving maps. This pullback square, which is dual to the inductive formation of skeleta by cell attachments, may be derived as an exercise in the manipulation of limits.

\begin{remark}\label{remark:mnemonic}
A mnemonic for this pullback diagram, which is rigorous at the level of simplicial sets and can be made rigorous for simplicial spaces with the correct interpretation of the symbols, is the following. By adjunction, we should think that \[\cosk_m(\op X)_n\cong\Hom(\Delta^n, \cosk_m(\op X))\cong \Hom(\tau_{\leq m}(\Delta^n), \tau_{\leq m}(\op X))\cong \Hom(\sk_m(\Delta^k), \op X),\] and there is a pushout diagram \[\xymatrix{
\displaystyle\coprod_{[m]\subseteq [n]}\partial \Delta^n\ar[r]\ar[d]&\sk_{m-1}(\Delta^n)\ar[d]\\
\displaystyle\coprod_{[m]\subseteq [n]}\Delta^n\ar[r]&\sk_m(\Delta^n).
}\]
\end{remark}

\begin{lemma}\label{lem:coskeleta}
Let $\op X$ be a hypercover.
\begin{enumerate}
\item The map $\op X\to \cosk_N(\op X)$ is an isomorphism if and only if $\op X$ is of height at most $N$.
\item For every $N\geq0$, $\cosk_{N-1}(\op X)$ is a hypercover, which is of height at most $N-1$ if $\op X$ is of height at most $N$.
\item  If $\op X$ is of height at most $N$, then the map $\op X\to \cosk_{N-1}(\op X)$ is a degreewise covering map.
\end{enumerate}
\end{lemma}
\begin{proof}
For the first claim, assume first that $\op X$ is of height at most $N$. Since $\op X_n\cong\cosk_N(\op X)_n$ whenever $n\leq N$, it suffices to demonstrate the isomorphism for $n>N$. Since $\op X$ is of height at most $N$, the righthand map in the pullback diagram above is an isomorphism whenever $m>N$, so the lefthand map is an isomorphism in these cases as well. Thus, \[\op X_n\cong\cosk_n(\op X)_n\cong \cosk_{n-1}(\op X)_n\cong\cdots\cong \cosk_N(\op X)_n,\] as desired. On the other hand, suppose that $\op X\cong\cosk_N(\op X)$. Then for $m\geq N$, we have \[\cosk_m(\op X)\cong\cosk_m(\cosk_N(\op X))\cong \cosk_N(\op X)\cong \op X,\] which implies the claim after setting $m=n-1$ and evaluating at $n$. 

For the second claim, it suffices by point (1) and the isomorphism $\cosk_{N-1}(\cosk_{N-1}(\op X))\cong\cosk_{N-1}(\op X)$ to show that $\cosk_{N-1}(\op X)$ is a hypercover. For $n\geq N$, we have \[M_n(\cosk_{N-1}(\op X))=\cosk_{n-1}(\cosk_{N-1}(\op X))_n\cong\cosk_{N-1}(\op X)_n,\] as desired, while, for $n\leq N$, we have the commuting diagram \[\xymatrix{
\cosk_{N-1}(\op X)_n\ar[r]&M_n(\cosk_{N-1}(\op X))=\cosk_{n-1}(\cosk_{N-1}(\op X))_n\\
\op X_n\ar[r]\ar@{=}[u]_-\wr&M_n(\op X)=\cosk_{n-1}(\op X)_n.\ar@{=}[u]^-\wr
}\] Since $\op X$ is a hypercover, the bottom map, and hence the top map, is a covering map, as desired.

The third claim is immediate for simplicial degree $n<N-1$. For $n=N-1$, we appeal to the pullback square above with $m=N$. Since $\op X$ is a hypercover, each of the maps $\op X_m\to M_m(\op X)$ is a covering map. Since covering maps are preserved under finite products and pullback, it follows that the lefthand map is a covering map, as desired.
\end{proof}

In the proof of Theorem \ref{thm:hypercover recovery}, we will thrice use that a degreewise weak homotopy equivalence between simplicial spaces induces a weak homotopy equivalence after geometric realization. This implication does not hold in general, but it does hold for a certain class of \emph{split} simplicial spaces, as shown in Appendix \ref{appendix:split simplicial spaces}, which includes the examples we need.

\begin{proof}[Proof of Theorem \ref{thm:hypercover recovery}]
We proceed by induction on the height $N$ of $\op X$, the base case $N=0$ being the case of a \v{C}ech cover. By Lemma \ref{lem:coskeleta}, the natural map $\op X\to \cosk_{N-1}(\op X)=:\op Y$ is a covering map in each degree, so we may extend this map to the horizontally augmented bisimplicial space \[\op W:=\left(\cdots\rightrightarrows\op X\times_{\op Y}\op X\rightrightarrows\op X\to \op Y\right)\] in which the $n$th row is the \v{C}ech nerve of the covering map $\op X_n\to \op Y_n$. By Theorem \ref{thm:cech recovery}, the map \[|\op W|_h\xrightarrow{\sim}\op Y\] is a degreewise weak homotopy equivalence, where $|-|_h$ denotes the geometric realization in the horizontal direction, and we conclude that \[|d^*\op W|\cong \big||\op W|_h\big|\xrightarrow{\sim}|\op Y|\xrightarrow{\sim} X,\] where $d:\Delta^{op}_+\to \Delta^{op}_+\times\Delta^{op}_+$ is the diagonal functor. Here the isomorphism follows from the general fact that the diagonal coincides with either iterated geometric realization \cite[p. 86]{Quillen:HAKTI}, the rightmost weak equivalence follows from the inductive hypothesis and the fact that $\op Y$ is a hypercover of height at most $N-1$, and the middle weak equivalence follows from Proposition \ref{prop:split criterion} in light of the fact that both $|\op W|_h$ and $\op Y$ are split by Corollaries \ref{cor:hypercovers are split} and \ref{cor:horizontal cech cover split}. Thus, in order to conclude the result for bounded $\op X$, it suffices to show that the natural map of simplicial spaces \[i:\op X\to d^*\op W=\left(\cdots\rightrightarrows \op X_1\times_{\op Y_1} \op X_1\rightrightarrows \op X_0\right)\] given by the diagonal in each degree, admits a retraction over the constant simplicial space with value $X$. Indeed, assuming this fact, it follows that the map $|\op X|\to X$ is a retract of a weak homotopy equivalence, and the claim follows.

To define the putative retraction $r:d^*\op W\to \op X\cong\cosk_{N}(\op X),$ it suffices by adjunction to exhibit a map $\bar r:\tau_{\leq N}(d^*\op W)\to \tau_{\leq N}(\op X)$, for which we take the isomorphism \[\bar r_n:\op X_n\times_{\op Y_n}\cdots\times_{\op Y_n}\op X_n=\op X_n\times_{\op X_n}\cdots\times_{\op X_n}\op X_n\cong\op X_n\] in degrees $n<N$. In degree $N$, we take the map \[\bar r_N:\op X_N\times_{\op Y_N}\cdots \times_{\op Y_N}\op X_N\to \op X_N\] to be any of the projections. One checks that, with these choices, $\bar r$ is a simplicial map, and that $\tau_{\leq N}(r\circ i)=\id_{\tau_{\leq N}(\op X)}$, implying the claim.

To conclude the claim for $\op X$ not necessarily bounded, it will suffice to show that the natural map \[\pi_n(|\op X|)\to \pi_n(|\cosk_{n+1}(\op X)|)\] is an isomorphism for every $n\geq0$, since $\cosk_{n+1}(\op X)$ is a bounded hypercover, so $\pi_n(|\cosk_{n+1}(\op X)|\cong\pi_n(X)$ by the argument above. For this claim, we note that \[|\op X| \xleftarrow{\sim}\big||\mathrm{Sing}(\op X)|_h\big|\cong |d^*\mathrm{Sing}(\op X)|,\] and similarly for $\cosk_{n+1}(\op X)$---these are the second and third cases in which we use that a degreewise weak equivalence induces a weak equivalence on realizations, for which we again invoke Proposition \ref{prop:split criterion} in light of Corollary \ref{cor:hypercovers are split} and Lemma \ref{lem:realization split}. From this equivalence, we conclude that \begin{align*}
\pi_n(|X|)&\cong \pi_n(|d^*\mathrm{Sing}(\op X)|)\\
&\cong\pi_n(|\sk_{n+1}(d^*\mathrm{Sing}(\op X))|)\\
&\cong \pi_n(|\sk_{n+1}(d^*\mathrm{Sing}(\cosk_{n+1}(\op X)))|)\\
&\cong \pi_n(|d^*\mathrm{Sing}(\cosk_{n+1}(\op X))|)\\
&\cong \pi_n(|\cosk_{n+1}(\op X)|),
\end{align*} where the second and fourth isomorphism follow as usual from the fact that $\pi_n(S^{m})=0$ for $n<m$, and the third follows from the fact that $\op X\to \cosk_{n+1}(\op X)$ is an isomorphism through simplicial degree $n+1$
\end{proof}

\subsection{Complete covers} We arrive now at the desired result.

\begin{definition}
We say that an open cover $\U$ of $X$ is \emph{complete} if $\U$ contains an open cover of $\bigcap_{U\in \U_0}U$ for every finite subset $\U_0\subseteq \U$.
\end{definition}

We regard $\U$ as a partially ordered set and thereby as a category.

\begin{theorem}[Dugger-Isaksen]\label{thm:complete cover recovery}
If $\U$ is a complete cover of $X$, then the natural map \[\hocolim_{U\in \U}U\to X\] is a weak equivalence.
\end{theorem}

\begin{remark}
We adopt the standard abuse of referring to the homotopy colimit as a space rather than an object of the homotopy category.
\end{remark}

Any open cover containing a basis for the topology of $X$ is complete, so Theorem \ref{thm:basis recovery} is a special case of this result. 

The strategy of the proof is to reduce the result to Theorem \ref{thm:hypercover recovery} by relating the bar construction modeling the homotopy colimit in question to a certain hypercover.

\begin{construction}
We write $P_n$ for the set of nonempty subsets of $[n]$, regarded as partially ordered under inclusion and thereby as a category. Given a functor $F:I\to \Top$, we write \[\mathrm{Bar}_n^\#(F)=\coprod_{f:P_n^{op}\to I}F(f([n]))\] and extend this to a simplicial space $\mathrm{Bar}_\bullet^\#(F)$ by letting the structure map associated to $h:[m]\to [n]$ be given by the restriction along the induced map $P_m^{op}\to P_n^{op}$. There is a canonical map $\mathrm{Bar}_\bullet(F)\to\mathrm{Bar}_\bullet^\#(F)$ of simplicial spaces induced by restriction along (the opposites of) the functors $P_n\to [n]$ sending a subset to its maximal element. Both simplicial spaces are naturally augmented over $\colim_IF$, and this map is compatible with the augmentations.
\end{construction}

Following Theorem \ref{thm:hypercover recovery}, the theorem will be an immediate consequence of the following two results.

\begin{proposition}\label{prop:subdivided comparison}
For any functor $F:I\to \Top$, the induced map $|\mathrm{Bar}_\bullet(F)|\to |\mathrm{Bar}_\bullet^\#(F)|$ is a weak homotopy equivalence.
\end{proposition}

We will take up the proof of this proposition momentarily.

\begin{lemma}
Let $\U$ be an open cover and $F:\U\to \Top$ the tautological functor. If $\U$ is complete, then the augmented simplicial space $\mathrm{Bar}_\bullet^\#(F)$ is a hypercover.
\end{lemma}
\begin{proof}
Writing $\overline P_n\subseteq P_n$ for the subcategory of proper nonempty subsets, we compute that \begin{align*}
M_n(\mathrm{Bar}_\bullet^\#(F))&=\mathrm{eq}\left(\prod_{S\subsetneq[n]}\coprod_{f:P_S^{op}\to \U}F(f(S))\rightrightarrows \prod_{S_1\subseteq S_2\subsetneq[n]}\coprod_{f:P_{S_1}^{op}\to \U}F(f(S_1))\right)\\
&\cong \mathrm{eq}\left(\coprod_{f:\overline{P}_n^{op}\to \U}\left(\prod_{S\in \overline P_n}F(f(S))\rightrightarrows\prod_{S_1\subseteq S_1\in \overline P_n}F(f(S)\right)\right)\\
&\cong \coprod_{f:\overline{P}_n^{op}\to \U}\bigcap_{S\in \overline P_n}F(f(S)),
\end{align*} and the canonical map to the matching object is the map \[\coprod_{f:P_n^{op}\to \U}F(f([n]))\to \coprod_{f:\overline{P}_n^{op}\to \U}\bigcap_{S\in \overline P_n}F(f(S))\] with components given by the inclusions $F(f([n]))\subseteq F(f(S))$ in $\U$ induced by the inclusions $S\subseteq [n]$ in $P_n$. Fixing $f:\overline P_n^{op}\to \U$, the inverse image under this map of the subspace $\bigcap_{S\in \overline P_n}F(f(S))$ is the disjoint union over all extensions of $f$ to $P_n^{op}$ of the value on $[n]$, which is to say the disjoint union of the elements of $\U$ contained in the intersection. Since the cover is complete, this collection of opens forms an open cover of the intersection and the claim follows.
\end{proof}

In order to prove Proposition \ref{prop:subdivided comparison}, we reinterpret the bar construction in a way that will generalize in parallel to the variant $\mathrm{Bar}_\bullet^\#(F)$.

\begin{definition}
Let $I$ be a category. The \emph{category of simplices} of $I$ is the category $\Delta^{op}I$ specified as follows.
\begin{enumerate}
\item An object of $\Delta^{op}I$ is a functor $f:[n]^{op}\to I$.
\item A morphism from $f:[n]^{op}\to I$ to $g:[m]^{op}\to I$ is a morphism $h:[m]\to [n]$ in $\Delta$ making the diagram \[\xymatrix{
[n]^{op}\ar[dr]_-{f}&&[m]^{op}\ar[dl]^-{g}\ar[ll]_-{h^{op}}\\
&I
}\] commute.
\item Composition is given by composition in $\Delta$. 
\end{enumerate} The \emph{category of subdivided simplices} of $I$ is the category $\Delta_\#^{op}I$ with objects functors $f:P_n^{op}\to I$ and morphisms given as in $\Delta^{op}I$.
\end{definition}

\begin{remark}
The reader is warned that our terminology is nonstandard.
\end{remark}

\begin{remark}
In keeping track of the variance, it is helpful to think of the morphism from $f$ to $g$ as being given by the pullback functor $(h^{op})^*$.
\end{remark}

These categories come equipped with a number of functors, which we summarize in the following commuting diagram \[\xymatrix{
&I\\
\Delta^{op}I\ar[rr]^-\rho\ar[dr]_-\delta\ar[ur]^-\lambda&&\Delta^{op}_\#I\ar[ul]_{\lambda_\#}\ar[dl]^-{\delta_\#}\\
&\Delta^{op}
}\]

Here the functor $\delta$ records the domains of functors, and the functor $\lambda$ is given on objects by the formula \[\lambda\left(f:[n]^{op}\to I\right)=f(n).\] We use our choice of variance in extending $\lambda$ to a functor, for, since $h(m)\leq n$, there is a canonical map $n\to h(m)$ in $[n]^{op}$, and the functor $f$ provides a map $\lambda(f)=f(n)\to f(h(m))= g(m)=\lambda(g)$. Similarly, $\delta_\#(f:P_n^{op}\to I)=[n]$ and $\lambda_\#(f:P_n^{op}\to I)=f([n])$, and the functor $\rho$ is defined by the restrictions along the functors $P_n\to [n]$ mentioned above.

\begin{lemma}\label{lem:domain finality}
For every $n\geq0$, the natural functor $\delta^{-1}([n])\to (\delta\downarrow[n])$ is homotopy final, and similarly for $\delta_\#$.
\end{lemma}
\begin{proof}
We prove the first claim, the argument for the second being essentially identical. The overcategory $(\delta\downarrow[n])$ has objects the pairs of a functor $f:[m]^{op}\to I$ and a morphism $r:[n]\to [m]$ in $\Delta$, and a morphism is a commuting diagram \[\xymatrix{
&[n]^{op}\ar[dl]_-{r^{op}}\ar[dr]^-{(r')^{op}}\\
[m]^{op}\ar[dr]_-{f}&&[m']^{op}\ar[dl]^-{f'}\ar[ll]\\
&I.
}\] Note that the variance is such that this diagram represents a morphism from $(f, r)$ to $(f', r')$. Thus, denoting the natural inclusion by $\iota:\delta^{-1}([n])\to (\delta\downarrow[n])$, we see from the definitions that $\left((f,r)\downarrow\iota)\right)$ is the category of commuting diagrams of the form \[\xymatrix{
&[n]^{op}\ar[dl]_-{r^{op}}\ar@{=}[dr]\\
[m]^{op}\ar[dr]_-{f}&&[n]^{op}\ar[dl]\ar[ll]\\
&I.
}\] This category is isomorphic to the discrete category with one object, which is certainly contractible.
\end{proof}

\begin{corollary}\label{cor:bar kan extension}
There are natural degreewise weak equivalences \[\hoLan_\delta(\lambda^*F)\xrightarrow{\sim}\mathrm{Bar}_\bullet(F)\] and \[\hoLan_{\delta_\#}(\lambda_\#^*F)\xrightarrow{\sim}\mathrm{Bar}^\#_\bullet(F)\] of simplicial spaces.
\end{corollary}
\begin{proof}
We calculate that
\begin{align*}
\hoLan_\delta(\lambda^*F)_n&\simeq \hocolim\left((\delta\downarrow[n])\to \Delta^{op}I\xrightarrow{\lambda} I\xrightarrow{F}\Top\right)\\
&\simeq \hocolim\left(\delta^{-1}([n])\to\Delta^{op}I\xrightarrow{\lambda} I\xrightarrow{F}\Top\right)\\
&\simeq \coprod_{f:[n]^{op}\to I}F(f(n))\\
&\cong\coprod_{i_n\to \cdots \to i_0}F(i_n)\\
&=\mathrm{Bar}_n(F),
\end{align*} where Lemma \ref{lem:domain finality} and Proposition \ref{prop:hocolim facts}(1) are used in obtaining the second equivalence. The calculation in the subdivided case is essentially identical.
\end{proof}

The final ingredient that we will need is the following.

\begin{lemma}\label{lem:last value finality}
The functors $\lambda$ and $\lambda_\#$ are each homotopy final.
\end{lemma}
\begin{proof}
The claim for $\lambda$ will follow after verifying for each $i\in I$, first, that the category $\lambda^{-1}(i)$ has a final object, and, second, that the canonical functor $\iota:\lambda^{-1}(i)\to (i\downarrow\lambda)$ is homotopy initial. For the first claim, we note that the functor $[0]^{op}\to I$ with value $i$ is final in $\lambda^{-1}(i)$. For the second claim, we observe that an object of $(i\downarrow\lambda)$ is simply a composable tuple $i\to f(n)\to \cdots\to f(0)$, which determines a canonical functor $\bar f:[n+1]^{op}\to I$ such that $\lambda(\bar f)=i$, together with a universal map $\bar f\to f$. The argument for $\lambda_\#$ is essentially the same, the only difference being that we extend $f:P_n^{op}\to I$ to $\bar f:P_{n+1}^{op}\to I$ by the prescription \[
\bar f(S)=\begin{cases}
f(S)&\quad n+1\notin S\\
i&\quad n+1\in S.
\end{cases}
\]
\end{proof}

\begin{proof}[Proof of Proposition \ref{prop:subdivided comparison}]
The claim will follow after verifying that each of the marked arrows in the commuting diagram
\[\xymatrix{
\displaystyle\hocolim_IF\ar@{=}[d]&\displaystyle\hocolim_{\Delta^{op}I}\lambda^*F\ar[l]_-{(1)}\ar[d]\ar[r]^-{(3)}&\displaystyle\hocolim_{\Delta^{op}}\hoLan_\delta(\lambda^*F)\ar[d]\ar[r]^-{(5)}&|\mathrm{Bar}_\bullet(F)|\ar[d]\\
\displaystyle\hocolim_IF&\displaystyle\hocolim_{\Delta^{op}_\#I}\lambda_\#^*F\ar[l]_-{(2)}\ar[r]^-{(4)}&\displaystyle\hocolim_{\Delta^{op}}\hoLan_{\delta_\#}(\lambda_\#^*F)\ar[r]^-{(6)}&|\mathrm{Bar}^\#_\bullet(F)|
}\] is a weak equivalence. The first and second follow from Proposition \ref{prop:hocolim facts}(1) and Lemma \ref{lem:last value finality}, the third and fourth are formal, since left Kan extensions compose, and the fifth and sixth follow from Corollary \ref{cor:bar kan extension} and Proposition \ref{prop:hocolim facts}(2).
\end{proof}

\section{Deferred proofs}\label{section:deferred proofs}

\subsection{Fadell--Neuwirth fibrations}

We are now able to repay the first of our long outstanding debts, namely the proof of Theorem \ref{thm:Fadell--Neuwirth}, which asserts that the diagram 
\[\xymatrix{
\Conf_{\ell-k}(M\setminus\{x_1,\ldots, x_k\})\ar[d]\ar[r]&\Conf_\ell(M)\ar[d]\\
(x_1,\ldots, x_k)\ar[r]&\Conf_k(M)
}\] is homotopy Cartesian. Recall that $\Conf_\ell(M)$ has a topological basis indexed by the partially ordered set \[\op B(M)_\ell^\Sigma=\left\{(U,\sigma):(\mathbb{R}^n)^{\amalg \ell}\cong U\subseteq M,\, \sigma:\{1,\ldots, \ell\}\xrightarrow{\simeq} \pi_0(U)\right\}\] consisting of the open sets $\Conf_\ell^0(U,\sigma)=\{x\in \Conf_\ell(M): x_i\in U_{\sigma(i)}\}.$ Thus, we have a functor \[\Conf_\ell^0:\op B(M)_\ell^\Sigma\to \Top\] whose homotopy colimit is canonically equivalent to $\Conf_\ell(M)$, and similarly for $\Conf_k(M)$. We also have a functor $\pi:\op B(M)_\ell^\Sigma\to \op B(M)_k^\Sigma$ defined by \[\pi(U,\sigma)=\left(\coprod_{i=1}^kU_{\sigma(i)},\, \sigma|_{\{1,\ldots, k\}}\right)\] and a natural transformation fitting into the commuting digram \[\xymatrix{
\Conf_\ell^0\ar[d]\ar@{-->}[r]&\pi^*\Conf_k^0\ar[d]\\
\Conf_\ell(M)\ar[r]&\Conf_k(M),
}\] where the bottom map is the Fadell--Neuwirth map given by projection onto the first $k$ factors. 

Now, since each $\Conf_\ell^0(U,\sigma)$ is contractible, we obtain the lefthand set of weak equivalences in the commuting diagram \[\xymatrix{
B(\op B(M)_\ell^\Sigma)\ar[d]_-{B\pi}&\hocolim_{\op B(M)_\ell^\Sigma}\Conf_\ell^0\ar[d]\ar[r]^-\sim\ar[l]_-\sim&\Conf_\ell(M)\ar[d]\\
B(\op B(M)_k^\Sigma)&\hocolim_{\op B(M)_k^\Sigma}\Conf_k^0\ar[r]^-\sim\ar[l]_-\sim&\Conf_k(M).
}\] Thus, understanding the homotopy fiber of the map rightmost map is tantamount to understanding the homotopy fiber of the map between classifying spaces induced by the functor $\pi$.

\begin{proof}[Proof of Theorem \ref{thm:Fadell--Neuwirth}]
We wish to apply Corollary \ref{cor:quillen b} to obtain the homotopy pullback square \[\xymatrix{
B((U,\sigma)\downarrow\pi)\ar[d]\ar[r]&B(\op B(M)_\ell^\Sigma)\ar[d]^-{B\pi}\\
\pt\ar[r]^-{(U,\sigma)}&B(\op B(M)_k^\Sigma). 
}\] In order to verify that the hypotheses hold in this case, we note that $((U,\sigma)\downarrow\pi)$ is the category of $(W,\tau)\in \op B(M)_\ell^\Sigma$ such that $U_{\sigma(i)}\subseteq W_i$ for $1\leq i\leq k$. It is easy to see that the inclusion of the subcategory with $U_{\sigma(i)}=W_{\tau(i)}$ is homotopy initial, and this subcategory is isomorphic to $\op B(M\setminus U)_{\ell-k}^\Sigma$. Since homotopy initial functors induce weak equivalences on classifying spaces, we have the weak equivalences in the diagram \[\xymatrix{
B((U,\sigma)\downarrow\pi)&B(\op B(M\setminus U)_{\ell-k}^\Sigma)\ar[l]_-\sim&\hocolim_{\op B(M\setminus U)_{\ell-k}^\Sigma}\Conf_{\ell-k}^0\ar[l]_-\sim\ar[r]^-\sim&\Conf_{\ell-k}(M\setminus U)\\
B((U',\sigma')\downarrow\pi)\ar[u]&B(\op B(M\setminus U')_{\ell-k}^\Sigma)\ar[u]\ar[l]_-\sim&\hocolim_{\op B(M\setminus U')_{\ell-k}^\Sigma}\Conf_{\ell-k}^0\ar[u]\ar[l]_-\sim\ar[r]^-\sim&\Conf_{\ell-k}(M\setminus U')\ar[u]}\] for any $(U,\sigma)\leq (U',\sigma')$ in $\op B(M)_k^\Sigma$. Since $M\setminus U'\subseteq M\setminus U$ is a monotopy equivalence, the rightmost map is a weak equivalence, so all of the vertical arrows are weak equivalences. Therefore, by Corollary \ref{cor:quillen b} and what we have already shown, we have the homotopy pullback \[\xymatrix{
B(\op B(M\setminus U)_{\ell-k}^\Sigma)\ar[d]\ar[r]&B(\op B(M)_\ell^\Sigma)\ar[d]^-{B\pi}\\
\pt\ar[r]^-{(U,\sigma)}&B(\op B(M)_k^\Sigma). 
}\] The proof is concluded upon noting that the inclusion \[\Conf_{\ell-k}(M\setminus U)\to \Conf_{\ell-k}(M\setminus \{x_1,\ldots, x_k\})\] is a weak equivalence for any $x_i\in U_{\sigma(i)}$.
\end{proof}

\subsection{Spectral sequences}
In proving the Leray--Hirsch theorem, and in much of what will follow, we will make use of the Serre spectral sequence. We begin with a few reminders on spectral sequences---for a general reference, see \cite{McCleary:UGSS}.

\begin{definition}
A (cohomological) \emph{spectral sequence} is a collection $\{E_r\}_{r\geq0}$ of bigraded $R$-modules, called the \emph{pages} of the spectral sequence, equipped with 
\begin{enumerate}
\item differentials $d_r:E_r\to E_r$ of bidegree $(r,1-r)$, and
\item isomorphisms $H(E_r,d_r)\cong E_{r+1}$ of bigraded $R$-modules.
\end{enumerate}
\end{definition}

 A \emph{map of spectral sequences} is a collection of bigraded chain maps $f_r:E_r\to \widetilde E_r$ compatible with these isomorphisms. We say that the spectral sequence $\{E_r\}$ is \emph{multiplicative} if each $E_r$ is equipped with the structure of a bigraded $R$-algebra for which $d_r$ is a bigraded derivation and the isomorphisms of (2) are algebra isomorphisms (note that this language is abusive, since multiplicativity is a structure rather than a property). Recall that $d_r$ is bigraded derivation if it satisfies the \emph{Leibniz rule} \[d_r(ab)=d_r(a)b+(-1)^{p+q}ad_r(b),\qquad |a|=(p,q).\]

\begin{remark}
Our statements and definitions are cohomological, since that is the nature of the application we have in mind, but the obvious dual notions are valid and often better behaved.
\end{remark}

In a typical situation, the $E_2$-page is something identifiable and relatively computable; each module $E_r^{p,q}$ is independent of $r$ for sufficiently large, and the identification of this common module $E_\infty^{p,q}$ is the goal; and the differentials are mysterious. We will not define $E_\infty$ or discuss convergence here---see \cite[3]{McCleary:UGSS} for details.

\begin{example}
If $(V,\partial,\delta)$ is a bicomplex satisfying mild boundedness conditions, there is a spectral sequence with $E_1\cong H(V,\partial)$ and \[E_2\cong H(H(V,\partial),\delta)\implies H(V,\partial+\delta).\] See \cite[2.4]{McCleary:UGSS} for a construction of this spectral sequence.

An alternative perspective on the spectral sequence of a bicomplex is offered by the \emph{homotopy transfer theorem} \cite{Vallette:AHO}. The starting observation is that the differential $\delta$ can be viewed as an algebraic structure on the chain complex $(V,\partial)$, in the form of an action of the dual numbers $k[\epsilon]/\epsilon^2$ (we work over a field $k$ for simplicity). After choosing representatives for homology and extending this choice to a chain deformation retraction of $(V,\partial)$ onto $(H(V,\partial),0)$, the homotopy transfer theorem endows $H(V,\partial)$ with the structure of a \emph{homotopy} $k[\epsilon]/\epsilon^2$-module, which amounts to the induced differential $\delta$ together with countably many higher operations $\{\delta_n\}$ satisfying certain relations (classically, this structure is known as a \emph{multicomplex}). These higher operations are direct analogues of the Massey products on the cohomology of a space, and they induce the differentials in the spectral sequence for $(V,\partial,\delta)$.
\end{example}

\begin{example}
If $X=\bigcup_{p\geq0} F_pX$ is a filtered space satisfying mild completeness conditions, there is a spectral sequence \[E_1^{p,q}\cong H^{p+q}(F_pX,F_{p-1}X)\implies H^{p+q}(X),\] i.e., there is an isomorphism of the $E_\infty$-page with the associated graded of the induced filtration on $H^*(X)$. The differential $d_1$ is the connecting homomorphism in the long exact sequence for the triple $(F_pX,F_{p-1}X, F_{p-2}X)$, and a filtration preserving map between spaces induces a map between the associated spectral sequences. The same construction goes through with integral cohomology replaced by an arbitrary cohomology theory. See \cite[2.2]{McCleary:UGSS} for details.
\end{example}

Applying this example to the skeletal filtration of the geometric realization of a simplicial space yields the following result. As a matter of notation, for a simplicial $R$-module $V$, we write $\mathrm{Alt}(V)$ for the associated chain complex of $R$-modules, whose differential is the alternating sum of the face maps of $V$.

\begin{corollary}[{\cite[5.1,\,5.3,\,5.4]{Segal:CSSS}}]\label{cor:simplicial spectral sequence}
Let $\op X:\Delta^{op}\to \Top$ be a simplicial space and $\op H$ a cohomology theory. There is a spectral sequence \[E_2^{p,q}\cong H^p(\mathrm{Alt}(\op H^q(\op X)))\implies \op H^{p+q}(|\op X|),\] which is natural for simplicial maps and multiplicative if $\op H$ is.
\end{corollary}

Using this basic construction and our results on \v{C}ech nerves, we may associate a spectral sequence to a general map.

\begin{theorem}[Leray, Segal]\label{thm:spectral sequence of a map}
Let $f:X\to Y$ be a map between topological spaces with $Y$ paracompact and $\op H$ a cohomology theory. There is a spectral sequence \[E_2^{p,q}\cong H^p(Y; \op H^q(f^{-1}))\implies \op H^{p+q}(X),\] where $\op H^q(f^{-1})$ is the sheaf associated to the presheaf $U\mapsto\op H^q(f^{-1}(U))$. Moreover, this spectral sequence is natural on the arrow category and multiplicative if $\op H$ is.
\end{theorem}
\begin{proof}
For each open cover $\op U$ of $Y$, the collection $f^{-1}\op U=\{f^{-1}(V): V\in \U\}$ is an open cover of $X$, so we have a canonical weak homotopy equivalence \[|\check{C}(f^{-1}\op U)|\xrightarrow{\sim} X.\] We now apply Corollary \ref{cor:simplicial spectral sequence} to obtain a spectral sequence with  \begin{align*}E_2^{p,q}&\cong H^p(\mathrm{Alt}(\op H^q(\check{C}(f^{-1}\op U))))\\
&= H^p\left(\cdots \to\bigoplus_{\op U^2}\op H^q(f^{-1}(V_0\cap V_1))\to \bigoplus_{\op U}H^q(f^{-1}(V))\right)\\
&\cong \check{H}(\op U;\op H^q(f^{-1}))\implies \op H^{p+q}(X),\end{align*} where $\check{H}$ denotes \v{C}ech cohomology. A refinement of open covers induces a map at the level of \v{C}ech nerves and therefore a map of spectral sequences, and, forming the colimit over the partially ordered set of open covers of $Y$, we obtain at last the spectral sequence \begin{align*}E_2^{p,q}&\cong \colim_{\op U}\check{H}^p(U;\op H^q(f^{-1}))\\
&\cong H^p(Y;\op H^q(f^{-1}))\implies \op H^{p+q}(X),
\end{align*} where the last isomorphism uses the assumption that $Y$ is paracompact.
\end{proof}

This spectral sequence specializes to two well-known spectral sequences.

\begin{corollary}[Atiyah-Hirzebruch]
Let $X$ be a paracompact space and $\op H$ a cohomology theory. There is a spectral sequence \[E_2^{p,q}\cong  H^p(X; \op H^q(\pt))\implies \op H^{p+q}(X),\] which is natural and multiplicative if $\op H$ is.
\end{corollary}
\begin{proof}
We apply Theorem \ref{thm:spectral sequence of a map} to the map $\id_X$, in which case the sheaf in question is the constant sheaf with value $\op H^q(\pt)$.
\end{proof}

The spectral sequence of greatest interest to us is the \emph{Serre spectral sequence} associated to a fibration \cite{Serre:HSEF}.

\begin{corollary}[Serre]\label{cor:serre ss}
Let $R$ be a ring and $\pi:E\to B$ a fibration with fiber $F$, and assume that $B$ is path-connected and paracompact. There is a spectral sequence \[E_2^{p,q}\cong H^p(Y; \underline{H^q(F; R)})\implies H^{p+q}(X;R),\] which is multiplicative and natural for maps of fibrations, where $\underline{H^q(F; R)}$ denotes the local coefficient system induced by the homotopy action of $\pi_1(B)$ on $F$.
\end{corollary}

\begin{remark}
Since an arbitrary map may be approximated by a fibration over a paracompact base, there is a Serre spectral sequence for an arbitrary map with connected target, which involves the homotopy fiber rather than the point-set fiber.
\end{remark}

\begin{corollary}\label{cor:ss of a cover}
Let $\pi:P\to X$ be a connected principal $G$-bundle. There is a spectral sequence  \[E_2^{p,q}\cong H^p(G;H^q(P))\implies H^{p+q}(X),\] which is multiplicative and natural for maps of principal $G$-bundles. 
\end{corollary}
\begin{proof} Forming the Borel construction on $P$, we obtain the homotopy pullback square \[\xymatrix{
P\ar[d]\ar[r]&EG\times_G X\ar[d]\\
\pt\ar[r]&BG,
}\] and, since $G$ acts freely on $X$, the natural map $EG\times_G P\to X$ is a weak equivalence. The claim follows from Corollary \ref{cor:serre ss} after identifying the group cohomology of $G$ with the twisted cohomology of $BG$.
\end{proof}

\subsection{The Leray--Hirsch theorem}

We turn now to the proof of the Leray--Hirsch theorem. We make use of the following basic observation, which identifies the \emph{edge maps} in the Serre spectral sequence.

\begin{lemma}\label{lem:edge maps}
Let $\{E_r\}$ be the spectral sequence for the fibration $F\xrightarrow{i} E\xrightarrow{\pi}B$. There is a commuting diagram \[\xymatrix{
H^q(E;R)\ar@{->>}[d]\ar[r]^-{i^*}&H^q(F;R)\ar@{=}[d]^-\wr\\
E_\infty^{0,q}\ar@{^{(}->}[r]&E_2^{0,q}
}\] for every $q\geq0$. Moreover, if $F$ is connected, then there is a commuting diagram \[\xymatrix{
H^p(B;R)\ar@{=}[d]_-{\wr}\ar[r]^-{\pi^*}&H^p(E;R)\\
E^{p,0}_2\ar@{->>}[r]&E^{p,0}_\infty\ar@{^{(}->}[u]
}\] for every $p\geq0$.
\end{lemma}
\begin{proof}
Both claims follow by naturality from the commuting diagram \[\xymatrix{
F\ar@{=}[d]\ar@{=}[r]&F\ar[d]^-i\ar[r]&\pt\ar[d]\\
F\ar[d]\ar[r]^-i&E\ar[r]^-\pi\ar[d]^-\pi&B\ar@{=}[d]\\
\pt\ar[r]&B\ar@{=}[r]&B,
}\] in which the bottom set of vertical arrows are all fibrations. The assumption that $F$ is connected permits the identification \[E_2^{p,0}\cong H^p(B;\underline{H^0(F;R)})\cong H^p(B;R).\]
\end{proof}

\begin{proof}[Proof of Theorem \ref{thm:Leray--Hirsch}]
The assumption that $B$ is connected permits the use of the Serre spectral sequence. The assumption that $F$ is connected and the surjectivity assumption together imply that the local system $\underline{H^q(F)}$ is trivial; indeed, the former implies that $\pi_1(E)$ surjects onto $\pi_1(B)$, so the action of $\pi_1(B)$ on $H^*(F)$ may be computed after lifting both to $E$. This action is trivial, since the action of $\pi_1(E)$ on $H^*(E)$ is so. The assumptions on the cohomology of $F$ and $B$ now permit us to apply the K\"{u}nneth theorem in cohomology to obtain the isomorphism \[E_2^{p,q}\cong H^p(B)\otimes H^q(F).\] 

We claim that this spectral sequence collapses at $E_2$. To see why this is so, we first note that surjectivity of $i^*$ and Lemma \ref{lem:edge maps} imply the isomorphism $E_2^{0,q}\cong E_\infty^{0,q}$, so $d_r|_{E_r^{0,q}}=0$ for all $r\geq2$ and $q\geq0$. Assume for induction that we have established the isomorphism $E_2\cong E_r$, the base case being $r=2$. Since a bihomogeneous element in $E_r$ may be written as $a\otimes b$ for $a\in E_r^{p,0}$ and $b\in E_r^{0,1}$, it suffices by the Leibniz rule to show that $d_r(a)=0$ and that $d_r(b)=0$ separately. The latter equality was shown above, and the former holds for degree reasons.

Thus, we have an isomorphism of $H^*(B)$-modules between $H^*(B)\otimes H^*(F)$ and the associated graded for the induced filtration on $H^*(E)$. Since $H^*(F)$ is free Abelian, $i^*$ admits a section $s$, and the map \begin{align*}
H^*(B)\otimes H^*(F)&\to H^*(E)\\
a\otimes b&\mapsto s(a)\smile \pi^*(b)
\end{align*} is an isomorphism, since it becomes an isomorphism after passing to the associated graded.
\end{proof}

\section{Mapping space models}\label{section:mapping space models}

Our next goal is to explore the perhaps surprising connection between configuration spaces and mapping spaces, following the seminal work of McDuff \cite{McDuff:CSPNP} (see also \cite{Boedigheimer:SSMS}). A motivating idea due to Segal \cite{Segal:CSILS} is that of the \emph{electric field map} \[\coprod_{k\geq0}B_k(\mathbb{R}^n)\to \Omega^nS^n,\] which assigns to a configuration, viewed as a collection of point charges, the corresponding electric field. As a vector field, this electric field is a priori a map from $\mathbb{R}^n$ to itself; however, since it becomes infinite at the location of a point charge, and since it tends to zero at infinity, it naturally extends to a map between the respective one-point compactifications. Alternatively, we can understand this map as a kind of Pontrjagin-Thom construction, sending a configuration of $k$ points to the composite $S^n\to \vee_k S^n\to S^n$ of the Thom collapse map for the normal bundle of the configuration followed by the fold map.

The electric field map is not a homotopy equivalence; for example, the induced map on $\pi_0$ is the inclusion $\mathbb{N}\to \mathbb{Z}$. We would like to understand whether this failure can be rectified, as well as whether the same idea may be adapted to more general background manifolds and target spaces.

\begin{remark}
In a sense, this ``group completion'' discrepancy on connected component is the only obstruction to the map being an equivalence---see \cite{Segal:CSILS}.
\end{remark}

\subsection{Labeled configuration spaces}

We begin by identifying the type of combinatorics at play.

\begin{definition}
A \emph{pointed finite set} is a finite set $I$ together with a distinguished element, called the \emph{basepoint} and denoted $*$. We write $I^\circ$ for the pointed finite set $I\setminus \{*\}$. A map $f:I\to J$ is \emph{inert} if 
\begin{enumerate}
\item $f(*)=*$, and 
\item $f|_{f^{-1}(J^\circ)}$ is a bijection onto $J^\circ$.
\end{enumerate}
\end{definition}

We write $\Inrt$ for the category pointed finite sets and inert maps. Note that this category is isomorphic to the opposite of the category of finite sets and injective maps.

\begin{construction}
A manifold $M$ defines a functor from $\Inrt$ to $\Top$ by sending $I$ to $\Conf_{I^\circ}(M)$ and the inert map $f:I\to J$ to the projection \[\xymatrix{
\Conf_{I^\circ}(M)\ar[d]\ar@{-->}[r]&\Conf_{\pi^{-1}(J^\circ)}(M)\ar[r]^-{\simeq}\ar[d]&\Conf_{J^\circ}(M)\ar[d]\\
M^{I^0}\ar[r]^-{\pi_f}&M^{\pi^{-1}(J^\circ)}\ar[r]^-\simeq&M^{J^\circ}
}\] given by the formula \[\pi_f\left((m_i)_{i\in I^\circ}\right)=(m_{f^{-1}(j)})_{j\in J^\circ}.\] 
\end{construction}

\begin{construction}
A based space $(X,x_0)$ determines a functor from $\Inrt^{op}$ to $\Top$ by sending $I$ to $\Map_*(I, X)\cong X^{I^\circ}$ and the inert map $f$ to the inclusion \[X^{J^\circ}\cong X^{f^{-1}(J^\circ)}\times \{x_0\}^{I^\circ\setminus f^{-1}(J^\circ)}\subseteq X^{f^{-1}(J^\circ)}\times X^{I^\circ\setminus f^{-1}(J^\circ)}\cong X^{I^\circ}.\]
\end{construction}

\begin{definition}
The \emph{configuration space of $M$ with labels in $X$} is the coequalizer \[\Conf_X(M)=\mathrm{coeq}\left(
\coprod_{J\to K}\Conf_{J^\circ}(M)\times X^{K^\circ}\rightrightarrows\coprod_{I}\Conf_{I^\circ}(M)\times X^{I^\circ}
\right),\] where the coproducts are indexed on the morphisms and objects of $\Inrt$, respectively. 
\end{definition}

\begin{remark}
Note that this construction is sensible without the assumption that $M$ be a manifold.
\end{remark}

Thus, a point in $\Conf_X(M)$ is a finite formal sum $\sum m_ax_a$ with $m_a\in M$ distinct and $x_a\in X$, and the following relation holds \[\textstyle\sum m_ax_a\sim \sum m_ax_a+mx_0.\] We refer to the point $x_a$ as the \emph{label} of $m_a$. The topology is such that a point vanishes if its label moves to the basepoint of $X$; thus, if $x_k\to x_0$ in $X$, for example, then $mx_k\to \varnothing$ for any $m\in M$.

\begin{example}
For any $M$, $\Conf_\mathrm{pt}(M)=\{\varnothing\}$.
\end{example}

\begin{example}
For any $M$, there is a homeomorphism $\Conf_{S^0}(M)\cong \coprod_{k\geq0}B_k(M).$
\end{example}

We will also have use for a relative version of this construction. If $M_0\subseteq M$ is a closed subspace, we write $\Conf_X(M,M_0)$ for the quotient of $\Conf_X(M)$ by the further relation \[\textstyle\sum m_ax_a\sim \sum m_ax_a+mx,\qquad m\in M_0.\] We refer to this space as the labeled configuration space with \emph{annihilation in $M_0$}. In this space, a point vanishes also if it collides with the annihilation subspace $M_0$; thus, if $m_k\to m\in M_0$, for example, then $m_kx\to \varnothing$ for any $x\in X$.

\begin{example}
For any $M$ and $X$, $\Conf_X(M,M)=\{\varnothing\}$
\end{example}

\begin{example}
If either $X$ or $(M,M_0)$ is path connected, then so is $\Conf_X(M,M_0)$ (recall that a pair is path connected if the map on path components is surjective).
\end{example}

\begin{definition}
The \emph{support} of the configuration $\sum m_ax_a$ is the (finite) subset \[\textstyle\mathrm{Supp}\left(\sum m_ax_a\right)=\left\{m_a\mid m_a\notin M_0 \text{ and } x_a\neq x_0\right\}\subseteq M.\]
\end{definition}

The space $\Conf_X(M,M_0)$ is filtered by the closed subspaces \[\Conf_X(M,M_0)_{\leq k}:=\left\{\textstyle \sum m_ax_a\mid  |\mathrm{Supp}(\sum m_ax_a)|\leq k\right\}.\] Moreover, both the successive quotients and the successive complements of this filtration are comprehensible, as \[\faktor{\Conf_X(M,M_0)_{\leq k}}{\Conf_X(M,M_0)_{\leq k-1}}\cong \Conf_k(M,M_0)\wedge_{\Sigma_k}X^{\wedge k},\] where $\Conf_k(M,M_0)$ is the quotient of $\Conf_k(M)$ by the subspace of configurations intersecting $M_0$ non-vacuously, while \[\Conf_X(M,M_0)_{\leq k}\setminus \Conf_X(M,M_0)_{\leq k-1}\cong \Conf_k(M\setminus M_0)\times_{\Sigma_k}(X\setminus x_0)^k.\]

\begin{example}
The filtration quotients of $\Conf_{S^r}(M)$ are given by the Thom spaces of the vector bundles $\Conf_k(M)\times_{\Sigma_k} \mathbb{R}^{rk}\to B_k(M),$ so, by the Thom isomorphism, we may compute the homology of the configuration spaces of $M$ from knowledge of this filtration. In fact, as we will show, this filtration splits at the level of homology, making this type of computation often feasible in practice. 

Notice that, for $r>0$, $\Conf_{S^r}(M)$ is connected. In particular, the analogous electric charge map $\Conf_{S^r}(\mathbb{R}^n)\to \Omega^nS^{n+r}$ is a bijection on $\pi_0$, unlike in the case $r=0$ considered above. In fact, as we will show, this map is a homotopy equivalence. 
\end{example}

We close this section by advancing the thesis that the construction $\Conf_X$ should be thought of as a kind of homology theory for manifolds (see \cite{AyalaFrancis:FHTM} for a detailed elaboration on this idea). It is easy to see that analogues of some of the Eilenberg-Steenrod axioms hold. For example, the construction is functorial in an obvious way for embeddings of pairs $(M, M_0)\to (N, N_0)$ and for baseed maps $X\to Y$; the map induced by an isotopy equivalence in the former case or a homotopy equivalence in the latter is a homotopy equivalence; we have a homeomorphism \[\Conf_X(M\amalg N, M_0\amalg N_0)\cong \Conf_X(M, M_0)\times\Conf_X(N, N_0),\] which is an analogue of the additivity axiom; and an analogue of excision is supplied by the homeomorphism \[\Conf_X(M\setminus U, M_0\setminus U)\xrightarrow{\simeq} \Conf_X(M,M_0)\] for $U\subseteq M_0$ open in $M$.

The basic building blocks of manifolds being disks rather than points, an appropriate analogue of the dimension axiom is supplied by the following result:

\begin{proposition}
Fix $X$ and $n\geq0$.
\begin{enumerate}
\item The inclusion $\Conf_X(D^n,\partial D^n)_{\leq 1}\subseteq \Conf_X(D^n,\partial D^n)$ is a weak homotopy equivalence, and
\item there is a homeomorphism $\Conf_X(D^n,\partial D^n)_{\leq 1}\cong \Sigma^nX$.
\end{enumerate}
\end{proposition}  
\begin{proof}
For (1), we note that, by radial expansion, any pointed map from a compact space $K$ to $\Conf_X(D^n,\partial D^n)$ factors up to pointed homotopy through $\Conf_X(D^n,\partial D^n)_{\leq 1}$. For (2), we calculate that \begin{align*}
\Conf_X(D^n,\partial D^n)_{\leq 1}&\cong \frac{\Conf_0(D^n)\times X^0\amalg \Conf_1(D^n)\times X^1}{\sim}\\
&\cong \frac{(D^n\times X)_+}{(\partial D^n\times X)_+\cup (D^n\times\{x_0\})_+}\\
&\cong \frac{D^n_+\wedge X}{\partial D^n_+\wedge X}\\
&\cong S^n\wedge X,
\end{align*} as desired.
\end{proof}

Finally, we have the following analogue of exactness.

\begin{theorem}[McDuff]\label{thm:exactness}
Let $M_0\subseteq M$ be a closed submanifold of codimension 0, possibly with boundary. If $X$ is connected, then the diagram \[
\xymatrix{
\Conf_X(M_0)\ar[d]\ar[r]&\Conf_X(M)\ar[d]\\
\pt\ar[r]&\Conf_X(M,M_0)
}
\] induced by the maps $(M_0,\varnothing)\to (M,\varnothing)\to (M,M_0)$ is homotopy Cartesian.
\end{theorem}

We will take up the proof of this theorem presently. We conclude this section with an obvious question.

\begin{question} If $\Conf_X$ is a homology theory, what is Poincar\'{e} duality?
\end{question}

\subsection{Local topology} Our strategy in proving Theorem \ref{thm:exactness} will be to combine covering technology and Quillen's Theorem B as in our proof of Fadell--Neuwirth. In order to do so, we must come to grips with the local topology of $\Conf_X(M,M_0)$.

\begin{definition}
An open subset $U\subseteq M$ is \emph{pointed} with respect to $M_0$ if some union of components $U_*\subseteq U$ contains $M_0$ as an isotopy retract.
\end{definition}

\begin{remark}
The assumption that $M_0\subseteq M$ is a codimension 0 submanifold guarantees a ready supply of open sets in $M$ that are pointed with respect to $M_0$; indeed, we may take the union of $M_0$ with any tubular neighborhood of $\partial M_0$ in $M\setminus \mathring{M_0}$.
\end{remark}

Note that, if $U$ is pointed with respect to $M$, then $\pi_0(U, U_*)$ is canonically pointed by the class of $U_*$.

\begin{construction}
Fix $M_0\subseteq M$ and $x_0\in X$ as above. We define a category $\op B(M,M_0;X)$ by the following specifications.
\begin{enumerate}
\item An object of $\op B(M,M_0;X)$ is a pair $(U,\sigma)$ with $U\subseteq M$ a proper open subset pointed with respect to $M_0$ and $\sigma:\pi_0(U, U_*)\to\mathrm{Op}(X)$ a function with $x_0\in \sigma(U_*)$, such that, for $i\in \pi_0(U, U_*)^\circ$
\begin{enumerate}
\item $U_i\cong\mathbb{R}^n$,
\item $\sigma(U_i)\simeq \pt$, and
\item $\sigma(U_i)\cap \sigma(U_*)=\varnothing$.
\end{enumerate}
\item A morphism $(U,\sigma)\to (V,\tau)$ is an inert map $f:\pi_0(U, U_*)\to \pi_0(V, V_*)$ such that 
\begin{enumerate}
\item $U_{i}\subseteq V_{f(i)}$ and $\sigma(U_{i})\subseteq \tau(V_{f(i)})$ for $i\notin f^{-1}(*)\setminus\{*\}$, and
\item either $U_i\subseteq V_*$ or $\sigma(U_i)\subseteq \tau(V_*)$ for $i\in f^{-1}(*)\setminus \{*\}$.
\end{enumerate}
\end{enumerate}
We write $\Conf_X^0(U,\sigma)\subseteq \Conf_X(M,M_0)$ for the subspace consisting of labeled configurations $\sum m_a x_a$ such that \begin{enumerate}
\item for each $i\in \pi_0(U)^\circ$, there is a unique $a$ with $m_a\in U_i$ and $x_a\in \sigma(U_i)$; and
\item otherwise, either $m_a\in U_*$ or $x_a\in \sigma(U_*)$.
\end{enumerate}
\end{construction}

We summarize the relevant facts about this category and these subspaces in a series of lemmas. First, a preliminary concept.

\begin{definition}
Let $\sum m_a x_a$ be a configuration lying in $\Conf_X^0(U,\sigma)$. We say that the point $m_{a}$ is \emph{essential} for $(U,\sigma)$ if $m_a\in U_i$ and $x_a \in \sigma(U_i)$ for some (necessarily unique) $i\in\pi_0(U,U_*)$.
\end{definition}

Equivalently, $m_a$ is essential\[\textstyle\sum_{a'\neq a}m_{a'}x_{a'}\notin \Conf_X^0(U,\sigma).\]

Recall that a category is a poset if and only if, first, each hom set contains at most one element, and, second, every isomorphism is an identity.

\begin{lemma}\label{lem:labeled poset}
The category $\op B(M,M_0;X)$ is a poset, and the inequality $(U,\sigma)\leq (V,\tau)$ holds if and only if the containment $\Conf_X^0(U,\sigma)\subseteq \Conf_X^0(V,\tau)$ holds.
\end{lemma}
\begin{proof}
Suppose that the inert maps $f$ and $g$ each determine a morphism $(U,\sigma)\to (V,\tau)$ in $\op B(M,M_0;X)$. Then there exists $i\in \pi_0(U,U_*)^\circ$ such that $f(i)=*\neq g(i)$; otherwise, $f$ and $g$ differ by a permutation of $\pi_0(V,V_*)$, and the conditions on morphisms in $\op B(M,M_0;X)$ force this permutation to be the identity. Now, since $f(i)=*$, either $U_i\subseteq V_*$ or $\sigma(U_i)\subseteq \tau(V_*)$; however, since $g(i)\neq *$, $U_i\subseteq V_{g(i)}$ and $\sigma(U_i)\subseteq \tau(V_{g(i)}$. It follows that either $V_{g(i)}\cap V_*\neq \varnothing$ or $\tau(V_{g(i)})\cap\tau(V_*)\neq \varnothing$, a contradiction. We conclude that hom sets in $\op B(M,M_0;X)$ have cardinality at most 1. 

Now, if $(U,\sigma)\cong (V,\tau)$, then the associated inert map is also an isomorphism, which is to say a pointed bijection, so $U_i\subseteq V_{f(i)}\subseteq U_i$ and $\sigma(U_i)\subseteq\tau(V_{f(i)})\subseteq \sigma(U_i)$ for every $i\in \pi_0(U,U_*)$, and it follows that $(U,\sigma)=(V,\tau)$. Thus, $\op B(M,M_0;X)$ is a poset.

For the second claim, we verify the ``if'' implication, the converse being essentially obvious. Suppose, then, that $\Conf_X^0(U,\sigma)\subseteq \Conf_X^0(V,\tau)$, and choose $\sum m_ax_a\in\Conf_X^0(U,\sigma)$. We note that \begin{align*}
\pi_0(U,U_*)^\circ&\cong\left\{a\mid m_a\text{ is essential for $(U,\sigma)$}\right\}\\
&\supseteq\left\{a\mid m_a\text{ is essential for $(V,\tau)$}\right\}\\
&\cong\pi_0(V,V_*)^\circ,
\end{align*} where the inclusion follows from the definition of an essential point and our containment assumption. Since the data of an inert map is equivalent to that of an inclusion in the opposite direction, this observation defines an inert map $f:\pi_0(U,U_*)\to \pi_0(V,V_*)$. It is easy to see that two labeled configurations that may be joined by a path give rise to the same inert map; therefore, since $\Conf_X^0(U,\sigma)$ is path connected, $f$ is independent of the choice of $\sum m_ax_a$.

To see that $f$, so defined, witnesses the inequality $(U,\sigma)\leq (V,\tau)$, there are two points to verify.
\begin{enumerate}
\item For $i\in f^{-1}(\pi_0(V,V_*)^\circ)$, we note that $U_i\cap V_{f(i)}\neq \varnothing$, since there is some point $m_a$ that is essential both for $(U,\sigma)$ and for $(V,\tau)$. The existence of a path in $U_i$ from $m_a$ to a point lying outside of this intersection leads to a contradiction of the fact that $m_a$ is essential for $(V,\tau)$, so $U_i\subseteq V_{f(i)}$. On the other hand, since $U_*$ and $V_*$ both contain $M_0$ as an isotopy retract, we have the bijection $\pi_0(U_*)\cong\pi_0(M_0)\cong \pi_0(V_*)$. Thus, since $V_*\cap V_j=\varnothing$ for $j\neq*$, it follows that $U_*\cap V_j=\varnothing$ for $j\neq *$. Allowing a point to range over $U_*$ with a fixed label lying outside of $\tau(V_*)$ now shows that $U_*\subseteq V_*$.\\
\item For $*\neq i\in f^{-1}(*)$, if neither containment holds, then there exist $m\in U_i$ and $x\in \sigma(U_i)$ with $m\notin V_*$ and $x\notin \tau(V_*)$. Thus, there is a configuration of the form $\sum m_ax_a+mx\in \Conf_X^0(U,\sigma)$; however, since $m$ is not essential in $(V,\tau)$ by the definition of $f$, we must also have $\sum m_ax_a+mx\notin \Conf_X^0(V,\tau)$, contradicting our assumption.
\end{enumerate}
\end{proof}

\begin{lemma}\label{lem:labeled open}
Each $\Conf_X^0(U,\sigma)$ is open in $\Conf_X(M,M_0)$.
\end{lemma}
\begin{proof}
From the definitions, a subset of $\Conf_X(M, M_0)$ is open if and only if its preimage under each of the maps \[q_r:\Conf_r(M)\times X^r\to \Conf_X(M,M_0)\] is so. For $r<k:=|\pi_0(U)^\circ|$, we have $q^{-1}_r(\Conf_X^0(U,\sigma)=\varnothing$; otherwise, the inverse image is the $\Sigma_r$-orbit of the subspace \[\bigcup_{k+\ell+m=r} A_{k,\ell,m}\times\prod_{i=1}^k\sigma(U_i)\times X^\ell\times \sigma(U_*)^m,\] where $A_{k,\ell,m}\subseteq \Conf_r(M)$ is defined by requiring that each of the diamonds in the commuting diagram
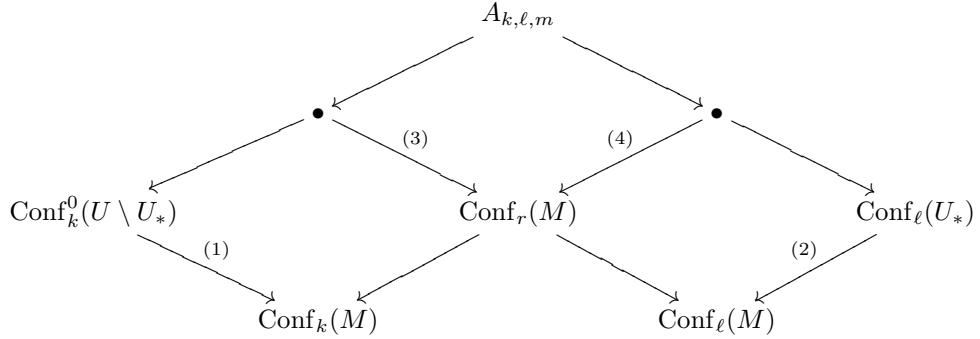
\[\xymatrix{
&&A_{k,\ell,m}\ar[dr]\ar[dl]\\
&\bullet\ar[dl]\ar[dr]^-{(3)}&&\bullet\ar[dl]_-{(4)}\ar[dr]\\
\Conf^0_k(U\setminus U_*)\ar[dr]^-{(1)}&&\Conf_r(M)\ar[dr]\ar[dl]&&\Conf_\ell(U_*)\ar[dl]_-{(2)}\\
&\Conf_k(M)&&\Conf_\ell(M)
}\] be a pullback; in other words, $A_{k,\ell,m}$ is the subspace in which the first $k$ points lie $\pi_0$-surjectively in $U\setminus U_*$, the next $\ell$ points lie in $U_*$, and the remaining $m$ points are arbitrary. Since the first and second marked arrows are inclusions of open subspaces, so are the third and fourth marked arrows. Thus, $A_{k,\ell,m}$ is the intersection of two open subspaces and hence itself open, implying the claim.
\end{proof}

\begin{lemma}\label{lem:labeled complete cover}
The collection $\left\{\Conf_X^0(U,\sigma): (U,\sigma)\in B(M,M_0; X)\right\}$ is a complete cover of $\Conf_X(M,M_0)$.
\end{lemma}
\begin{proof}
We first verify that $\U$ is a cover. We begin by fixing $\sum m_ax_a\in \Conf_X(M,M_0)$ with $m_a\notin M_0$ and $x_a\neq x_0$ for all $a$. We now choose \begin{enumerate}
\item disjoint Euclidean neighborhoods $m_a\in U_a$;
\item a tubular neighborhood $M_0\subseteq U_*$ disjoint from $\bigcup_a U_a$;
\item contractible neighborhoods $x_a\in \sigma(U_a)$ disjoint from $x_0$; and
\item a contractible neighborhood $x_0\in \sigma(U_*)$ disjoint from $\bigcup_a \sigma(U_a)$.
\end{enumerate} Setting $U=U_*\cup\bigcup_a U_a$, we clearly have $\sum m_ax_a\in\Conf_X^0(U,\sigma)$.

For completeness, suppose that $\sum m_ax_a\in\bigcap_{r=1}^N\Conf_X^0(V_r,\tau_r)$. We make the same sequence of choices as in the previous step, while requiring that
\begin{enumerate}
\item if $m_a$ is contained in some $V_{r,i}$, then $U_a\subseteq \bigcap_{m_a\in V_{r,i}} V_{r,i}$;\item $U_*\subseteq \bigcap_{r=1}^N V_{r,*}$; 
\item if $x_a$ is contained in some $\tau_r(V_{r,i})$, then $x_0\notin\sigma(U_a)\subseteq \bigcap_{x_a\in \tau_r(V_{r,i})}\tau_r(V_{r,i})$; and
\item $\sigma(U_a)\subseteq \bigcap_{r=1}^N \tau_r(V_{r,*})$.
\end{enumerate} We define a function $f:\pi_0(U,U_*)\to \pi_0(V_r, V_{r,*})$ by sending $U_a$ to the component $V_{r,i}$ such that $m_a\in V_{r,i}$, provided $m_a$ is essential for $(V_r,\tau_r)$, and to the basepoint otherwise. It is easy to check that $f$ is a well-defined inert map, and the proof is complete upon verifying that $f$ witnesses the inequality $(U,\sigma)\leq (V_r,\tau_r)$. There are two points to check.
\begin{enumerate}
\item If $f(a)\neq *$ or if $a=*$, then $U_a\subseteq V_{r,f(a)}$ and $\sigma(U_a)\subseteq \tau_r(V_{r,f(a)})$ by construction.
\item If $*\neq a\in f^{-1}(*)$, then $m_a$ is not essential for $(V_r, \tau_r)$, so either $m_a\in V_{r,*}$ or $x_a\in \tau_r(V_{r,*})$. By construction, then, either $U_a\subseteq V_{r,*}$ or $\sigma(U_a)\subseteq \tau_r(V_{r,*})$.
\end{enumerate}
\end{proof}

\begin{lemma}\label{lem:labeled contractible}
Each $\Conf_X^0(U, \sigma)$ is contractible.
\end{lemma}
\begin{proof}
A choice of contraction of $\sigma(U_*)$ onto $x_0$ defines a homotopy equivalence \[\Conf_X^0(U,\sigma)\simeq\prod_{\pi_0(U,U_*)^\circ}(U_i\times \sigma(U_i))\times \Conf_X(U_*, M_0).\] The second factor is contractible in view of the isotopy equivalence of pairs $(M_0, M_0)\xrightarrow{\sim} (U_*, M_0)$, and each $U_i$ and each $\sigma(U_i)$ is contractible by assumption.
\end{proof}

We write $q:\Conf_X(M)\to \Conf_X(M,M_0)$ for the map induced by $(M,\varnothing)\to (M,M_0)$.

\begin{lemma}\label{lem:labeled fiber}
For every $(U,\sigma)$, there is a homotopy equivalence $q^{-1}\Conf_X^0(U,\sigma)\simeq \Conf_X(M_0)$.
\end{lemma}
\begin{proof}
The inverse image $q^{-1}\Conf_X(U,\sigma)$ is the subspace of labeled configurations in $\Conf_X(M)$ with 1) exactly one essential point in each $U_i$, 2) all other points not lying in $U_*$ labeled by points of $\sigma(U_*)$, and 3) an arbitrary subconfiguration lying in $U_*$. Thus, a choice of contraction of $\sigma(U_*)$ onto $x_0$ defines a homotopy equivalence \[q^{-1}\Conf_X^0(U,\sigma)\simeq \prod_{\pi_0(U,U_*)^\circ}(U_i\times \sigma(U_i))\times\Conf_X(U_*).\] Since each $U_i$ and each $\sigma(U_i)$ is contractible, the claim follows from the isotopy equivalence $M_0\xrightarrow{\sim} U_*$.
\end{proof}

\begin{lemma}\label{lem:labeled locally constant}
If $X$ is connected, then the inclusion \[q^{-1}\Conf_X^0(U,\sigma)\to q^{-1}\Conf_X^0(V,\tau)\] is a homotopy equivalence whenever $(U,\sigma)\leq(V,\tau)$.
\end{lemma}

Before turning to the proof, we consider an example illustrating the necessity of the hypothesis on $X$.

\begin{example}
Set $M=D^n_6(0)$, $M_0=D_6^n(0)\setminus\mathring{D}_5^n(0)$, $X=S^0$. We define a pair $(U,\sigma)\leq (V,\tau)$ by setting \begin{align*}
U_1&=\mathring{D}^n_1(0)\\
U_2&=\mathring{D}_{1/3}(7/2,0,\ldots, 0)\\
U_*&=D^n_6(0)\setminus D_4^n(0)\\
V_1&=\mathring{D}_2^n(0)\\
V_*&=D^n_6(0)\setminus D_3^n(0).
\end{align*} Since $X=S^0$, $\sigma$ and $\tau$ are determined, and, since $U_1\subseteq V_1$ and $U_*\cup U_2\subseteq V_*$, the inert map $f:\{1,2,*\}\to \{1,*\}$ with $f(1)=1$ and $f(2)=*$ witnesses the claimed inequality. Now, in the commuting diagram
\[\xymatrix{
q^{-1}\Conf_{S^0}(U,\sigma)\ar[d]\ar@{=}[r]^-\sim& U_1\times U_2\times \Conf_{S^0}(U_*)\ar[d]&\{m_1\}\times\{m_2\}\times\displaystyle\coprod_{k\geq0}B_k(U_*)\ar[d]\ar[l]_-{\sim}\\
q^{-1}\Conf_{S^0}(V,\tau)\ar@{=}[r]^-\sim&V_1\times\Conf_{S^0}(V_*)&\{m_1\}\times\displaystyle\coprod_{k\geq0}B_k(V_*),\ar[l]_-{\sim}
}\] the point $\{m_1\}\times\{\varnothing\}$ does not lie in the image of the rightmost map; therefore, this map fails to surject on $\pi_0$.
\end{example}

\begin{definition}
Let $P$ be connected and $N\subsetneq P$ an isotopy retract. For $p\in P\setminus N$ and $x\neq x_0$, the \emph{stabilization map} with respect to $p$ and $x$ is the map \[\Conf_X(N)\cong \{px\}\times \Conf_X(N)\subseteq \Conf_X(D_{\epsilon}(p))\times\Conf_X(N)\to \Conf_X(M),\] where the rightmost map is the composite of the homeomorphism $\Conf_X(D_\epsilon(p))\times\Conf_X(N)\cong \Conf_X(D_\epsilon(p)\amalg N)$ with the structure map for the inclusion.
\end{definition}

Thus, in the homotopy category of spaces, there is a well-defined stabilization map $\Conf_X(P)\to \Conf_X(P)$ depending only on the class of $x$ in $\pi_0(X)$ and a choice of element in $\pi_0(P,N)$.

\begin{proposition}
Stabilization with respect to $p$ and $x$ is an equivalence if and only if $x$ lies in the path component of $x_0$.
\end{proposition}
\begin{proof}
If $x$ lies in the path component of $x_0$, then a path joining the two defines a homotopy from the stabilization map to the structure map for the isotopy equivalence $N\subseteq P$. Conversely, suppose that $X=X_1\amalg X_2$ with $x_0\in X_1$ and $x\in X_2$. The subspace \[\left\{{\textstyle\sum m_ax_a}\mid \{x_a\}\cap X_2\neq \varnothing\right\}\subset\Conf_X(P)\] is closed, open, proper, and contains the image of stabilization with respect to $p$ and $x$. It follows that stabilization does not surject on $\pi_0$.
\end{proof}

\begin{proof}[Proof of Lemma \ref{lem:labeled locally constant}]
Suppose that $f:\pi_0(U,U_*)\to \pi_0(V,V_*)$ witnesses the inequality $(U,\sigma)\leq(V,\tau)$, and write \[\pi_0(U,U_*)\cong *\sqcup f^{-1}(\pi_0(V,V_*)^\circ)\sqcup I\sqcup I',\] where $I=\{*\neq i\in f^{-1}(*)\mid\sigma(U_i)\subsetneq \tau(V_*)\}$ and $I'=\{*\neq i\in f^{-1}(*)\mid\sigma(U_i)\subseteq \tau(V_*)\}$. Choosing contractions of $\sigma(U_*)$ and $\tau(V_*)$ onto $x_0$ and points in $U_i\times\sigma(U_i)$ for $i\in\pi_0(U,U_*)^\circ$ induces the equivalences in the commuting diagram \[\xymatrix{
q^{-1}\Conf_X(U,\sigma)\ar[d]&\displaystyle\prod_I(U_i\times \sigma(U_i))\times\Conf_X(U_*)\ar[d]\ar[l]_-\sim&\Conf_X(U_*)\ar[l]_-\sim\ar[d]\\
q^{-1}\Conf_X(V,\tau)&\Conf_X(V_*)\ar[l]_-\sim&\Conf_X(V_*),\ar@{=}[l]
}\] where the righthand map is a composite of $|I|$ stabilization maps, each of which is an equivalence by our assumption on $X$.
\end{proof}

We now prove the theorem.

\begin{proof}[Proof of Theorem \ref{thm:exactness}]
Lemma \ref{lem:labeled poset} supplies the top arrow and Lemma \ref{lem:labeled open} the dashed filler in the commuting diagram \[\xymatrix{
\op B(M,M_0;X)\ar@{-->}[dr]\ar[rr]^-{\Conf_X^0}&&\Top\\
&\mathrm{Op}(\Conf_X(M,M_0)).\ar[ur]
}\] By Lemma \ref{lem:labeled complete cover}, then, \[\Conf_X(M,M_0)\simeq \hocolim_{\op B(M,M_0;X)}\Conf_X^0\simeq B(\op B(M,M_0;X)),\] where the second equivalence follows from Lemma \ref{lem:labeled contractible}. 

We similarly have the commuting diagram \[\xymatrix{
\op B(M,M_0;X)\ar@{-->}[dr]\ar[rr]^-{q^{-1}\Conf_X^0}&&\Top\\
&\mathrm{Op}(\Conf_X(M)).\ar[ur]
}\] Lemma \ref{lem:labeled complete cover} implies that the collection $\left\{q^{-1}(\Conf_X^0(U,\sigma): (U,\sigma)\in \op B(M, M_0;X)\right\}$ is likewise a complete cover of $\Conf_X(M)$, so \[\Conf_X(M)\simeq \hocolim_{\op B(M, M_0;X)}q^{-1}\Conf^0_X.\] Since $X$ is connected, Lemma \ref{lem:labeled locally constant} and Corollary \ref{cor:hocolim quasifibration} grant that the diagram \[\xymatrix{
q^{-1}\Conf_X^0(U,\sigma)\ar[r]\ar[d]&\displaystyle\hocolim_{\op B(M,M_0;X)}q^{-1}\Conf_X^0\ar[d]\\
\pt\ar[r]^-{(U,\sigma)}&B(\op B(M,M_0;X))
}\] is homotopy Cartesian, and the desired conclusion follows from Lemma \ref{lem:labeled fiber} and the identifications already established.
\end{proof}

We close by pointing out that the same reasoning provides an analogue of the long exact sequence for a triple.

\begin{theorem}[McDuff]
Let $M_0$, $M_1$, and $M_0\cap M_1$ be closed submanifolds of codimension 0 in $M$. If either $X$ or $(M_0,M_0\cap M_1)$ is connected, then the diagram \[\xymatrix{
\Conf_X(M_0, M_0\cap M_1)\ar[r]\ar[d]&\Conf_X(M, M_1)\ar[d]\\
\pt\ar[r]&\Conf_X(M, M_0\cup M_1)
}\] is homotopy Cartesian.
\end{theorem}

\subsection{Duality for labeled configuration spaces}

We now disuss an analogue of Poincar\'{e} duality in the context of the ``homology theory'' of labeled configuration spaces. One approach to the classical isomorphism $H_*(M)\cong H_c^{n-*}(M)$ is through the following steps.
\begin{enumerate}
\item By taking a compact exhaustion, reduce to the statement $H_*(N)\cong H_c^{n-*}(N,\partial N)$ for compact manifolds with boundary.
\item Observe the local calculation $\mathbb{Z}\cong H_*(D^n)\cong H^{n-*}(D^n,\partial D^n)\cong \widetilde H^{n-*}(S^n)\cong\mathbb{Z}$.
\item Deduce the general case inductively using Mayer-Vietoris and/or exactness.
\end{enumerate}
Our approach will proceed along these same lines; first, however, we must identify the object that is to play the role of compactly supported cohomology.

\begin{recollection}
A Riemannian $n$-manifold $M$ has a principal right $O(n)$-bundle of \emph{orthonormal frames} \[\xymatrix{
\mathrm{Isom}(\mathbb{R}^n, T_pM)\ar[d]\ar[r]&\mathrm{Fr}_M\ar[d]\\
\{p\}\ar[r]&M,
}\] with fiber over $p$ the set of isometries from $\mathbb{R}^n$ with the standard inner product to the tangent fiber over $p$. The orthogonal group $O(n)$ acts on this space of isometries by precomposition. Note the $O(n)$-equivariant but non-canonical isomorphism $\mathrm{Isom}(\mathbb{R}^n, T_pM)\cong O(n)$.
\end{recollection}

We equip $M$ with a metric, whose features will be further specified in time.

\begin{construction}
Define a bundle $E_X=E_X(M)$ by \[E_X:=\mathrm{Fr}_M\times_{O(n)}\Conf_X(D^n,\partial D^n),\] where the action of $O(n)$ on $\Conf_X(D^n,\partial D^n)$ arises from the action by diffeomorphisms of $O(n)$ on the pair $(D^n,\partial D^n)$.
\end{construction}

The projection $\pi: E_X\to M$ has a canonical section defined by $s_0(p)=\varnothing$, and we may identify the fiber $\pi^{-1}(p)$ with $\Conf_X(D_1(T_p(M)), \partial D_1(T_pM))$. Moreover, the inclusion of the subbundle $\mathrm{Fr}_M\times_{O(n)}\Conf_X(D^n,\partial D^n)_{\leq 1}$ is a fiberwise homotopy equivalence, and, using the isomorphism $\Conf_X(D^n,\partial D^n)\cong\Sigma^n X$, which is $O(n)$-equivariant with $O(n)$ acting on the righthand side via the suspension coordiantes, we may identify this subbundle with the fiberwise smash product of $X$ with the fiberwise one point compactification of $TM$. Under this identification, the section $s_0$ is given by the fiberwise basepoint.

\begin{definition}
For a map $f:E\to B$, the \emph{space of sections} of $f$ over $A\subseteq B$ is the pullback in the diagram \[\xymatrix{
\Gamma(A;E)\ar[r]\ar[d]&\Map(A,E)\ar[d]^-{f_*}\\
\{A\subseteq B\}\ar[r]&\Map(A,B).
}\] Fixing a section $s_0$ we say that such a section $s$ has \emph{compact support} if $s|_{A\setminus K}= s_0|_{A\setminus K}$ for some compact subset $K\subseteq A$. We write $\Gamma_c(A;E)\subseteq\Gamma(A;E)$ for the subspace of compactly supported sections.
\end{definition}

Note that we have the identification \[\Gamma_c(A;E)\cong \colim_K \Gamma(A, A\setminus K;E),\] where the colimit is taken over the category of inclusions among compact subsets of $A$, and the space of relative sections $\Gamma(A,A_0;E)$ for $A_0\subseteq A$ is defined as the pullback in the diagram \[\xymatrix{
\Gamma(A,A_0;E)\ar[d]\ar[r]&\Gamma(A;E)\ar[d]^-{(A_0\subseteq A)^*}\\
\{s_0|_{A_0}\}\ar[r]&\Gamma(A_0;E).
}\]

The space $\Gamma_c(M;E_X)$ will play the role of compactly supported cohomology in our version of Poincar\'{e} duality. In order to make the comparison to the labeled configuration space $\Conf_X(M)$, we require a map.

\begin{recollection}
Given a Riemannian manifold $M$, a point $p\in M$, and a sufficiently small $\epsilon>0$, there is an \emph{exponential map} \[\exp_p^\epsilon:D_\epsilon(T_pM)\to M,\] which is an embedding.
\end{recollection}

We assume that $M$ is equipped with a metric such that $\exp_p^1$ is an embedding for all $p\in M$; for example, such a metric exists if $M$ is the interior of a compact manifold with boundary.

\begin{definition}
The \emph{scanning map} for $M$ and $X$ is the map \[s:\Conf_X(M)\to \Gamma(M;E_X)\] defined by letting $s\left(\sum m_a x_a\right)(p)$ be the image of $\sum m_ax_a$ under the composite map \begin{align*}\Conf_X(M)&\to\Conf_X(M,M\setminus \exp_p^1(\mathring{D}_1(T_pM)))\\
&\cong \Conf_X(\exp_p^1(D_1(T_pM)), \exp_p^1(\partial D_1(T_pM)))\\
&\cong \Conf_X(D_1(T_pM),\partial D_1(T_pM))\\
&\cong\pi^{-1}(p),
\end{align*} where the first map is induced by $(M,\varnothing)\to (M,M\setminus \exp_p^1(\mathring{D}_1(T_pM)))$, the second is excision, and the third is induced by the exponential map. 
\end{definition}

Note that, since there are only finitely many $m_a$ with $x_a\neq x_0$, the section $s\left(\sum m_a x_a\right)$ always has compact support.

\begin{theorem}[McDuff]\label{thm:duality noncompact}
Scanning induces a weak equivalence \[\Conf_X(M)\to \Gamma_c(X;E_X)\] provided $X$ is connected.
\end{theorem}

As in our approach to classical Poincar\'{e} duality outlined above, we deduce this result from a version for compact manifolds with boundary $N$. In order to accommodate boundary into our scanning technique, we set $W=N\amalg_{\partial N}[0,1)\times\partial N$ and arrange as before that $\exp_p^1$ is an embedding for $p\in W$. Using a collar neighborhood of $\partial N$ in $N$, we choose an isotopy retract $N'\subseteq N$ such that $\exp_p^1(D_1(T_pN))\cap N'=\varnothing$ for $p\in \partial N\times[0,1)$, implying that the dashed filler exists in the commuting diagram \[\xymatrix{\Conf_X(N)\ar[r]^-{N\subseteq W}& \Conf_X(W) \ar[r]^-s& \Gamma(W;E_X)\\
\Conf_X(N')\ar[u]^-\wr\ar@{-->}[rr]&&\Gamma(W,\partial N\times[0,1);E_X)\cong \Gamma(N,\partial N;E_X).\ar[u]
}\]

\begin{theorem}[McDuff]\label{thm:duality with boundary}
Let $N$ be a compact manifold with boundary. Scanning induces a weak equivalence \[\Conf_X(N)\simeq \Gamma(N,\partial N;E_X)\] provided $X$ is connected.
\end{theorem}

Before turning to the proof of this result, we use it to deduce Theorem \ref{thm:duality noncompact}.

\begin{proof}[Proof of Theorem \ref{thm:duality noncompact}]
Realize $M$ as the interior of a compact manifold with boundary and choose a compact exhaustion $M=\colim_k N_k$ with each $N_k$ an isotopy retract along a collar neighborhood of the boundary. In the commuting diagram \[\xymatrix{
\Conf_X(M)\ar[r]&\Gamma_c(M;E_X)\\
\colim_k \Conf_X(N_k)\ar[u]\\
\colim_k\Conf_X(N_k'\ar[u])\ar[r]&\colim_k\Gamma(N_k,\partial N_k;E_X),\ar[uu]
}\] the upper left arrow is a homeomorphism; the lower left arrow is a obtained from a levelwise weak equivalence between diagrams of relatively $T_1$ inclusions, since $N_k'\subseteq N_k$ is an isotopy equivalence; the same holds for the bottom arrow by Theorem \ref{thm:duality with boundary}; and the righthand arrow is a homeomorphism, since $\Gamma(N_k,\partial N_k;E_X)\cong \Gamma(M,M\setminus N_k;E_X)$ and the diagram $\{N_k\}_{k\geq0}$ is final in the category of compact subsets of $M$. It follows that the top arrow is a weak equivalence, as desired.
\end{proof}

As in our proof schematic from before, our strategy in proving Theorem \ref{thm:duality with boundary} will be induction on a local calculation. We first clarify the nature of the induction in question.

\begin{recollection}
If $M$ is an $n$-manifold with boundary and $\varphi: \partial D^i\times D^{n-i}\to \partial M$ an embedding, then we obtain a new manifold with boundary $\overline M:= M\cup_\varphi (D^i\times D^{n-i})$, which we refer to as the result of \emph{attaching an $i$-handle} to $M$ along $\varphi$. A compact manifold with boundary may be built from the empty manifold by a finite sequence of handle attachments.
\end{recollection}

Our strategy, then, will be to proceed by induction on such a handle decomposition. Specifically, a handle attachment gives rise a homotopy pullback square \[\xymatrix{
\Conf_X(M)\ar[r]\ar[d]&\Conf_X(\overline M)\ar[d]\\
\pt\ar[r]&\Conf_X(D^i\times D^{n-i}, \partial D^i\times D^{n-i})
}\] provided $X$ is connected. Thus, if the theorem is known for $M$, then the theorem for $\overline M$ will follow by examining $\Conf_X(D^i\times D^{n-i}, \partial D^i\times D^{n-i})$. This logic leads us to formulate a third, relative version of the theorem.

\begin{theorem}[McDuff]\label{thm:duality relative}
Let $M$ be a compact manifold with boundary and $M_0\subseteq M$ a closed submanifold. Scanning induces a weak equivalence \[\Conf_X(M,M_0)\simeq \Gamma(M\setminus M_0, \partial M\setminus M_0; E_X)\] provided either $X$ or $(M,M_0)$ is connected.
\end{theorem}

We note first that it suffices to consider only those pairs $(M,M_0)$ with $M_0\subseteq \partial M$. Indeed, if $M_0\subseteq N$ is a closed tubular neighborhood, then \begin{align*}
\Conf_X(M,M_0)&\xrightarrow{\sim} \Conf_X(M, N)\\
&\cong \Conf_X(M\setminus \mathring{N}, \partial N),
\end{align*} and, similarly, \begin{align*}
\Gamma(M\setminus M_0, \partial M\setminus M_0)&\xrightarrow{\sim}\Gamma(M\setminus N, \partial M\setminus N)\\
&=\Gamma\left((M\setminus \mathring{N})\setminus \partial N, \partial (M\setminus \mathring{N})\setminus \partial N\right),
\end{align*} so the case of $(M,M_0)$ follows from the case of $(M\setminus \mathring{N}, \partial N)$.

In order to proceed, we must accommodate the relative case into our scanning setup.

\begin{construction} Assuming that $M_0\subseteq\partial M$, we write $\partial M=M_0\cup_{\partial M_1} M_1$, set $W=M\sqcup_{\partial M}\partial M\times[0,1)$, and let $M'$ be a retract along a collar of $M_1$ relative to $M_0$. Finally, choosing a tubular neighborhood $M_0\subseteq N\subseteq W$ such that $\exp_p^1(\mathring{D}_1(T_pM))\cap M_0=\varnothing$ for $p\notin N$, we obtain, up to homotopy, a scanning map
\[\Conf_X(M,M_0)\xleftarrow{\sim}\Conf_X(M', M_0)\to \Gamma(M\setminus N, \partial M\setminus N)\xleftarrow{\sim}\Gamma(M\setminus M_0, \partial M\setminus M_0),\] where the middle arrow is defined by the same formula as before. 
\end{construction}

This construction is possible because $M_0\subseteq M\setminus \exp_p^1(\mathring{D}_1(T_pM))$ for $p\notin N$, so the required map \[\Conf_X(M, M_0)\to \Conf_X(M,M\setminus \exp_p^1(D_1(T_pM)))\cong \Conf_X(D_1(T_pM), \partial D_1(T_pM))\] is defined.

The key step in the proof of Theorem \ref{thm:duality relative} is the following result.

\begin{lemma}\label{lem:handle case}
For $0<i\leq n$ and any $X$, the conclusion of Theorme \ref{thm:duality relative} holds for the pair $(M,M_0)=(D^i\times D^{n-i}, \partial D^i\times D^{n-i})$. Moreover, if $X$ is connected, the conclusion also holds in the case $i=0$.
\end{lemma}
\begin{proof}
We proceed by downward induction on $i$. For the base case $i=n$, it suffices to show that the dashed restriction in the commuting diagram \[\xymatrix{
\Conf_X(D_1^n, \partial D_1^n)_{\leq 1}\ar[d]_-\wr\ar@{=}[r]^-\sim&\Sigma^nX\ar@{-->}[d]   \\
\Conf_X(D_1^n,\partial D_1^n)\ar[r]&\Gamma(\mathring{D}_{1/2}^n;E_X)&\Gamma(\mathring{D}_{1}^n;E_X)\ar[l]_-\sim
}\] is a weak equivalence. For this claim, we note that radial expansion of the scanning neighborhood defines a homotopy from this restriction to the map $\Sigma^nX\to \Gamma(\mathring{D}^n_{1/2}; E_X)\cong \Map(\mathring{D}^n_{1/2}, \Sigma^nX)$ sending a point in $\Sigma^nX$ to the constant map to that point. 

For the induction step, we write $\partial_\epsilon D^i$ for a closed collar neighborhood of $\partial D^i$, and we write $\partial_\epsilon D^i=\partial_\epsilon^+ D^i\cup \partial_\epsilon^- D^i$ with $\partial_\epsilon^+D^i\cap\partial_\epsilon^- D^i\cong \partial_\epsilon D^{i-1}\times D^1$. By exactnes, the diagram \[\xymatrix{\Conf_X(\partial^+_\epsilon D^i\times D^{n-i}, \partial_\epsilon D^{i-1}\times D^1\times D^{n-i})\ar[r]\ar[d]&\Conf_X(D^i\times D^{n-i}, \partial_\epsilon^- D^i\times D^{n-i})\ar[d]\\
\pt\ar[r]&\Conf_X(D^i\times D^{n-i}, \partial_\epsilon D^i\times D^{n-i})
}\] is homotopy Cartesian provided $X$ is connected or $i>0$. The conclusion holds for the bottom right pair by induction, since \[\Conf_X(D^i\times D^{n-i}, \partial_\epsilon D^i\times D^{n-i})\simeq \Conf_X(D^i\times D^{n-i}, \partial D^i\times D^{n-i}),\] while for the upper right pair both the labeled configuration space and the section space are contractible. It follows that the conclusion also holds for the upper left pair, but \[\Conf_X(\partial^+_\epsilon D^i\times D^{n-i}, \partial_\epsilon D^{i-1}\times D^1\times D^{n-i})\cong \Conf_X(D^{i-1}\times D^{n-(i-1)}, \partial D^{i-1}\times D^{n-(i-1)}),\] so the claim follows.
\end{proof}

We now prove the theorem following \cite{Boedigheimer:SSMS}.

\begin{proof}[Proof of Theorem \ref{thm:duality relative}]
We first prove the claim in the special case $M_0=\partial M$. We proceed by induction on a handle decomposition of $M$, which we may take to involve no $n$-handles if no component of $M$ has empty boundary, which is the condition that $(M,\partial M)$ be connected. The base case of $M=D^n$ is known. For the induction step, we write $\overline M=M\cup_\varphi D^i\times D^{n-i}$ and assume the claim for $(M,\partial M)$. By exactness, the diagram \[\xymatrix{
\Conf_X(D^i\times D^{n-i}, D^i\times\partial D^{n-i})\ar[d]\ar[r]&\Conf_X(\overline M, \partial\overline M)\ar[d]\\
\pt\ar[r]&\Conf_X(\overline M, \partial \overline M\cup D^i\times D^{n-i})
}\] is homotopy Cartesian if $X$ is connected or if $i<n$. By excision, \[\Conf_X(\overline M,\partial \overline M\cup D^i\times D^{n-i})\cong\Conf_X(M,\partial M),\] so the claim is known for this pair by induction. Since it is also known for the upper left pair by Lemma \ref{lem:handle case}, the proof is complete in this case.

In the general case, we write $\partial M=M_0\cup_{\partial M_1} M_1$, set $\overline W=M\sqcup_{\partial M}\partial M\times[0,1]$, and note that the diagram \[\xymatrix{
\Conf_X(M_1\times[0,1], \partial M_1\times[0,1])\ar[r]\ar[d]&\Conf_X(\overline W, M_0\times[0,1])\ar[d]\\
\pt\ar[r]&\Conf_X(\overline W, \partial M\times[0,1]).}\] is homotopy Cartesian if $X$ is connected or $(M_1,\partial M_1)$ is connected. Assuming one of these conditions to hold, excision and isotopy invariance show that \begin{align*}\Conf_X(\overline W, \partial M\times[0,1])&\cong\Conf_X(M, \partial M)\\
\Conf_X(\overline W, M_0\times[0,1])&\simeq \Conf_X(M, M_0),
\end{align*} so it suffices to prove the claim for the pair $(M_1\times[0,1],\partial M_1\times[0,1])$. To establish this last case, we use exactness once more, invoking our assumed connectivity of $X$ or $(M_1,\partial M_1)$, to obtain the homotopy pullback square \[\xymatrix{
\Conf_X(M_1\times[0,1], \partial M_1\times[0,1])\ar[d]\ar[r]&\Conf_X(M_1\times[0,2], \partial (M_1\times[0,2])\setminus \mathring{M_1}\times\{0\})\ar[d]\\
\pt\ar[r]&\Conf_X(M_1\times [0,2], M_1\times[0,2]\setminus \mathring{M}_1\times(1,2)).
}\] The upper right entry and the corresponding section space are contractible, and the case of a manifold relative to its boundary shows that the theorem holds for the bottom right pair, since \[\Conf_X(M_1\times [0,2], M_1\times[0,2]\setminus \mathring{M}_1\times(1,2))\cong\Conf_X(M_1\times[1,2],\partial (M_1\times [1,2]))\] by excision. Thus, the proof is complete under the assumption that $X$ or $(M_1,\partial M_1)$ is connected.

In case $X$ and $(M_1,\partial M_1)$ are both disconnected, we let $N$ denote the union of the components of $\partial M$ not intersecting $M_0$, and we set $\widetilde M=M\cup_N M$. Exactness implies that the diagram \[\xymatrix{
\Conf_X(M,M_0)\ar[d]\ar[r]&\Conf_X(\widetilde M, M_0\sqcup M_0)\ar[d]\\
\pt\ar[r]&\Conf_X(\widetilde M, M\sqcup M_0)
}\] is homotopy Cartesian, since $(M,M_0)$ is connected by assumption. Since the theorem holds for the two righthand pairs by what has already been shown, it also holds for $(M,M_0)$.
\end{proof}

\subsection{Postmortem}Before moving on from the subject of Poincar\'{e} duality for labeled configuration spaces, we pause to give an informal discussion of several generalizations and continuations of this story.

\begin{perspective}
We have shown that, for connected $X$, there is a weak equivalence \[\Conf_X(D^n)\xrightarrow{\sim}\Omega^n\Sigma^nX.\] In fact, this equivalence may be upgraded to an equivalence of $E_n$-\emph{algebras} \cite{May:GILS}. Since the domain of this map is the free $E_n$-algebra on $X$ in the category of pointed spaces under Cartesian product, monadicity considerations show that the homotopy theory of connected $E_n$-algebras is equivalent to that of connected $n$-fold loop spaces. This equivalence can be further upgraded to an equivalence of the homotopy theory of all $n$-fold loop spaces with that of \emph{grouplike} $E_n$-algebras, i.e., those whose connected components form a group.
\end{perspective}

\begin{perspective}
For connected $X$, we have the equivalence \[\Conf_X(M)\xrightarrow{\sim}\Gamma_c(M;E_X)=\Gamma(M^+,\infty; E_X)\] as well as the dual equivalence \[\Conf_X(M^+,\infty):=\colim_K\Conf_X(K,\partial K)\xrightarrow{\sim}\colim_K\Gamma(\mathring{K};E_X)\cong \Gamma(M;E_X).\] Thus, the duality of the ``manifolds'' $M$ and $M^+$ mirrors the duality of the ``coefficient systems'' $\Omega^n\Sigma^nX$ and $\Sigma^nX$, a type of algebra and coalgebra, respectively, which is an algebraic phenomenon known as \emph{Koszual duality.} This observation has a generalization in terms of factorization (co)homology in the form of a ``Poincar\'{e}/Koszul duality'' map \[\int_Z A\to \int^{Z^\neg}B^nA,\] where $Z$ is a \emph{zero-pointed manifold} and $Z^\neg$ its \emph{negation}---in our case, $(M_+)^\neg=M^+$ and $(M^+)^\neg=M_+$. Like the scanning map it generalizes, this map is an equivalence under connectivity assumptions \cite{AyalaFrancis:PKD}.
\end{perspective}

\begin{perspective}
It is natural to wonder what can be said about the scanning map \[\Conf_X(M)\to \Gamma(M,\partial M; E_X)\] for disconnected $X$, particularly since our motivating example was the case $X=S^0$. The key to answering this question is the observation that, in the local case $M=D^n$, both the labeled configuration space and the section space carry the structure of a monoid up to homotopy---indeed, as we noted above, both are $E_n$-algebras---and the scanning map respects this structure. Like a group, a monoid has a classifying space, and, at least in the case of connected $X$, the homotopy pullback square \[\xymatrix{\Conf_X(D^n)\ar[d]\ar[r]&\Conf_X(D^{n-1}\times(D^1,\{1\}))\simeq\pt\ar[d]\\
\pt\ar[r]&\Conf_X(D^{n-1}\times (D^1,\partial D^1))
}\] quaranteed by exactness shows that \[B\Conf_X(D^n)\simeq \Conf_X(D^{n-1}\times(D^1,\partial D^1))\overset{?}{\simeq} \Conf_{\Conf_X(D^1,\partial D^1)}(D^{n-1})\simeq\Conf_{\Sigma X}(D^{n-1}),\] and, in fact, this equivalence holds for any $X$ \cite[2.1]{Segal:CSILS}. Since $\Gamma(D^n, \partial D^n;E_X)\cong\Omega^n\Sigma^nX$ is grouplike, it follows from the group completion theorem \cite[1]{McDuffSegal:HFGCT} that we have the induced isomorphism \[H_*(\Conf_X(D^n))[\pi_0^{-1}]\cong H_*(\Gamma(D^n,\partial D^n;E_X)),\] giving a precise answer to the question of how far the scanning map is from a weak equivalence, at least in this local case. Note that $\pi_0$ acts on homology by stabilization, so, taking $X=S^0$, this isomorphism may be interpreted as the statement that \[\colim_k H_i(B_k(D^n))\cong H_i(\Omega^n_0\Sigma^n)\] for $i>0$, where $\Omega^n_0\Sigma^n\subseteq \Omega^n\Sigma^n$ is the degree zero component.

In the general case, assuming $\partial M\neq\varnothing$, we proceed by choosing a small disk $D$ sharing part of its boundary with $M$ and considering the diagram \[\xymatrix{
\Conf_X(D)\ar[r]\ar[d]& \Gamma(D,\partial D;E_X)\ar[d]\\
\Conf_X(M)\ar[r]\ar[d]_-q& \Gamma(M,\partial M;E_X)\ar[d]\\
\Conf_X(M,D)\ar[r]^-\sim&\Gamma(M\setminus D, \partial M\setminus D; E_X).
}\] Taking the homotopy colimit over stabilization maps corresponding to various components of $X$, the top arrow becomes an isomorphism on homology, and, while the map from $\Conf_X(D)$ to the homotopy fiber is not a weak equivalence, it does becomes a homology isomorphism in the limit. 

One way to see this is by applying the following result, which is proved in exactly the same manner as Corollary \ref{cor:hocolim quasifibration}.

\begin{proposition}
If $F:I\to \mathrm{Top}$ is a functor sending each morphism in $I$ to a homology isomorphism, then the diagram \[\xymatrix{
F(i)\ar[r]\ar[d]&\hocolim_IF\ar[d]\\
\pt\ar[r]^-i&BI
}\] of spaces is homology Cartesian for every object $i\in I$.
\end{proposition} 

In the scenario at hand, we apply this result to the functor \begin{align*}
\op B(M, D; X)&\to \Top\\
(U,\sigma)&\mapsto(\hocolim q)^{-1}\Conf^0_X(U,\sigma).
\end{align*} The previous local group completion calculation shows that the hypothesis is satisfied. Finally, the Serre spectral sequence implies that scanning induces an isomorphism \[H_*(\hocolim \Conf_X(M))\cong H_*(\Gamma(M,\partial M; E_X)).\] In particular, the homology of $B_k(M)$ converges to that of a component of $\Gamma(M,\partial M; E_{S^0}),$ and, with more care, one can show that each $H_i(B_k(M))$ eventually stabilizes; one says that configuration spaces of open manifolds exhibit \emph{homological stability}. What is more surprising is that configuration spaces of closed manifolds also exhibit homological stability, at least rationally \cite{Church:HSCSM, RandalWilliams:HSUCS}. 
\end{perspective}

\section{Homology calculations from stable splittings}\label{section:homology from splittings}

The goal of this section is the computation of the rational homology of the unordered configuration spaces of a large class of manifolds. 

\begin{theorem}\label{thm:odd homology}
If $M$ is an odd dimensional compact manifold with boundary and $\mathbb{F}$ is a field of characteristic 0, then \[H_*(B_k(M))\cong \Sym^k\left(H_*(M)\right).\]
\end{theorem}

\begin{remark}
Since configuration spaces are isotopy invariant, this computation is valid for any odd-dimensional manifold of finite type.
\end{remark}

\begin{remark}
With slightly more care, it is not hard to show that this isomorphism is induced by the canonical inclusion $B_k(M)\to \Sym^k(M)$.
\end{remark}

\begin{remark}
Similar, but more complicated, statements are available for $\mathbb{F}_p$, and may be derived in essentially the same way, given the requisite knowledge of the mod $p$ homology of the configuration spaces of $\mathbb{R}^n$. In the case $p=2$, one can eliminate the requirement that $n$ be odd.
\end{remark}

Our plan is to exploit the natural filtration $\Conf_X(M)=\bigcup_{k\geq0}\Conf_X(M)_{\leq k}$, whose successive quotients are identified as
\[\faktor{\Conf_X(M)_{\leq k}}{\Conf_X(M)_{\leq k-1}}\cong \Conf_k(M)_+\wedge_{\Sigma_k}X^{\wedge k}.\] In the case $X=S^m$, these quotients are Thom spaces of vector bundles on the configuration spaces $B_k(M)$, so, by the Thom isomorphism, the primary obstruction to extracting the homology of $B_k(M)$ from the homology of $\Conf_{S^m}(M)\simeq \Gamma_c(M; E_{S^m})$ is the difference between $\Conf_{S^m}(M)$ and the graded space associated to the cardinality filtration. Happily, a result of \cite{CohenMayTaylor:SCSC,Boedigheimer:SSMS} guarantees that, at the level of homology, this difference is no difference at all.

\subsection{Stable splitting}\label{section:stable splittings}

In this section, we prove the following result. See Appendix \ref{appendix:Spanier--Whitehead} for an explanation of unfamiliar terms.

\begin{theorem}[B\"{o}digheimer, Cohen-May-Taylor]\label{thm:stable splitting}
There is a filtered stable weak equivalence \[\Conf_X(M,M_0)\xrightarrow{\sim_s}\bigvee_{k\geq1}\Conf_k(M,M_0)\wedge_{\Sigma_k}X^{\wedge k}\] for any manifold $M$, submanifold $M_0$, and pointed CW complex $X$.
\end{theorem}

Since filtered stable weak equivalences induce isomorphisms on homology by Corollary \ref{cor:filtered stable weak equivalence homology}, this result will be sufficient to allow us to proceed with our computation.

There are other lenses through which to view this result aside from this application to configuration spaces; indeed, one of its appealing features is that, through the scanning map, it allows information to flow equally in the other direction, implying stable splittings for many familiar mapping spaces.

\begin{example}[Snaith]
For $X$ connected, we have \[\Omega^n\Sigma^nX\xleftarrow{\sim} \Conf_X(D^n)\xrightarrow{\sim_s} \bigvee_{k\geq1}\Conf_k(D^n)_+\wedge_{\Sigma_k}X^{\wedge k}.\]
\end{example}

\begin{example}[Goodwillie]
For $X$ connected, we have \[\Lambda\Sigma X\xleftarrow{\sim} \Conf_X(S^1)\xrightarrow{\sim_s} \bigvee_{k\geq1}S^1_+\wedge_{C_k}X^{\wedge k}.\]
\end{example}

\begin{example}[\cite{Boedigheimer:SSMS}]
Let $K$ be a finite CW complex. For connected $X$, the mapping space $\Map(K,\Sigma^mX)$ splits stably for sufficiently large $m$. Indeed, choosing $M\simeq K$ with $M$ a parallelizable $m$-manifold, we have \begin{align*}
\Map(K,\Sigma^mX)&\simeq \Map(\mathring{M}, \Sigma^mX)\\
&\simeq\Gamma(\mathring{M}; E_X)\\
&\simeq \Conf_X(M,\partial M)\\
&\simeq_s \bigvee_{k\geq1}\Conf_k(M,\partial M)\wedge_{\Sigma_k} X^{\wedge k}.
\end{align*}
\end{example}

For the remainder of this section, we set $V_k=\bigvee_{i=1}^k\Conf_i(M,M_0)\wedge_{\Sigma_i}X^{\wedge i}$. To begin with, we seek to find a non-decreasing sequence $\{r_k\}_{k\geq1}$ of non-negative integers and a collection of maps $\Sigma^{r_k}\Conf_X(M,M_0)_{\leq k}\to \Sigma^{r_k} V_k$ making each of the diagrams \[\xymatrix{
\Sigma^{r_{k}}\Conf_X(M,M_0)_{\leq k}\ar[r]& \Sigma^{r_{k}} V_{k}\\
\Sigma^{r_{k}}\Conf_X(M,M_0)_{\leq k-1}\ar[u]\ar[r]& \Sigma^{r_{k}} V_{k-1}\ar[u]
}\] commute up to homotopy. Using the adjunction between loops and suspension and the fact that, by scanning, \[\Omega^r\Sigma^rA\simeq \Conf_A(\mathbb{R}^r)\implies \Omega^\infty\Sigma^\infty A\simeq \Conf_A(\mathbb{R}^\infty),\] it will suffice to produce the commuting diagram \[\xymatrix{
\Conf_X(M,M_0)\ar[r]^-\Phi&\Conf_V(\mathbb{R}^\infty)\\
\vdots\ar[u]&\vdots\ar[u]\\
\Conf_X(M,M_0)_{\leq k}\ar[u]\ar[r]^-{\Phi_k}&\Conf_{V_k}(\mathbb{R}^{r_k})\ar[u]\\
\Conf_X(M,M_0)_{\leq k-1}\ar[r]^-{\Phi_{k-1}}\ar[u]&\Conf_{V_{k-1}}(\mathbb{R}^{r_{k-1}})\ar[u]\\
\vdots\ar[u]&\vdots\ar[u]
}\]

\begin{construction}\label{construction:power set map}
Choose an embedding $\varphi:\coprod_{k\geq0}B_k(M)\to \mathbb{R}^\infty$ such that $\varphi|_{B_k(M)}$ factors through $\mathbb{R}^{r_k}$. We define $\Phi:\Conf_X(M,M_0)\to \Conf_V(\mathbb{R}^\infty)$ by the formula \[\Phi\left(\sum_{a\in A}m_ax_a\right)=\sum_{\varnothing\neq I\subseteq A}\varphi\left(\{m_a\}_{a\in I}\right)\cdot\left[\sum_{a\in I}m_ax_a\right]_{|I|},\] where $[-]_k:\Conf_X(M,M_0)_{\leq k}\to \Conf_k(M,M_0)\wedge_{\Sigma_k}X^{\wedge k}=V_k/V_{k-1}$ is the quotient map.
\end{construction}

\begin{lemma}
The map $\Phi$ is well-defined.
\end{lemma}
\begin{proof}
First, we note $\varphi(\{m_a\}_{a\in I}\neq \varphi(\{m_a\}_{a\in J}$ for $I\neq J$, so the set $\left\{\varphi(\{m_a\}_{a\in I})\right\}_{\varnothing\neq I\subseteq A}$ does in fact define an element of $B_{2^{|A|}-1}(\mathbb{R}^\infty)$. Indeed, if $I\neq J$, then $\{m_a\}_{a\in I}\neq \{m_a\}_{a\in J}$, since $m_a\neq m_b$ if $a\neq b$, and $\varphi$ is injective.

Second, we verify that the definition of $\Phi$ is independent of the choice of representative of a labeled configuration by an expression of the form $\sum_{a\in A}m_ax_a$. Indeed, any two representatives differ by a term of the form $m_bx_b$ with $m_b\in M_0$ or $x_b=x_0$, and $\left[\sum_{a\in I} m_ax_a\right]_{|I|}=0$ whenever $b\in I$; therefore, we have \begin{align*}\Phi\left(\sum_{a\in A}m_ax_a\right)&=\sum_{\varnothing\neq I\subseteq A\setminus \{b\}}\varphi\left(\{m_a\}_{a\in I}\right)\cdot\left[\sum_{a\in I}m_ax_a\right]_{|I|}+\sum_{b\in I\subseteq A}\varphi\left(\{m_a\}_{a\in I}\right)\cdot0\\&=\sum_{\varnothing\neq I\subseteq A\setminus \{b\}}\varphi\left(\{m_a\}_{a\in I}\right)\cdot\left[\sum_{a\in I}m_ax_a\right]_{|I|}\\
&=\Phi\left(\sum_{b\neq a\in A}m_ax_a\right),
\end{align*} as claimed.
\end{proof}

Note that, by construction, $\Phi|_{\Conf_k(M,M_0)_{\leq k}}$ factors through a map $\Phi_k:\Conf_k(M,M_0)_{\leq k}\to \Conf_{V_k}(\mathbb{R}^{r_k})$ as in the diagram displayed above.

\begin{lemma}\label{lem:phi compatibility}
For each $k\geq1$, the diagram \[\xymatrix{
\Conf_X(M,M_0)_{\leq k}\ar[d]_-{[-]_k}\ar[rr]^-{\Phi_k}&&\Conf_{V_k}(\mathbb{R}^{r_k})\ar[d]\\
V_k/V_{k-1}\ar@{=}[r]&\Conf_{V_k/V_{k-1}}(\mathbb{R}^0)\ar[r]&\Conf_{V_k/V_{k-1}}(\mathbb{R}^{r_k})
}\] commutes up to homotopy. In the case $k=1$, we adopt the convention that $V_0$ is the basepoint, so that $V_1/V_0=V_1$.
\end{lemma}
\begin{proof}
The clockwise composite is the map \[\sum_{a\in A}m_ax_a\mapsto \begin{cases}
\varphi\left(\{m_a\}_{a\in A}\right)\cdot\left[\sum_{a\in A}m_ax_a\right]_k&\quad |A|=k\text{ and }m_a\notin M_0,\, x_a\in x_0 \text{ for } a\in A\\
0&\quad\text{otherwise.}
\end{cases}\] Since $\varphi$ is homotopic to the constant map to the origin, the claim follows.
\end{proof}

The last ingredient is the following.

\begin{lemma}\label{lem:labeled cofibrations}
For any manifold $M$, submanifold $M_0$, and pointed CW complex $X$, the inclusion $\Conf_X(M,M_0)_{\leq k-1}\to \Conf_X(M,M_0)_{\leq k}$ is a Hurewicz cofibration.
\end{lemma}

Before proving this fact, we pause to deduce the theorem.

\begin{proof}[Proof of Theorem \ref{thm:stable splitting}]
By Lemma \ref{lem:phi compatibility}, we have the homotopy commutative diagram
\[\xymatrix{
\Conf_X(M,M_0)_{\leq 1}\ar@{=}[d]\ar[rr]^-{\Phi_1}&&\Conf_{V_1}(\mathbb{R}^{r_1})\ar@{=}[d]\\
V_1\ar@{=}[r]&\Conf_{V_1}(\mathbb{R}^0)\ar[r]&\Conf_{V_1}(\mathbb{R}^{r_1}),
}\] so the composite $\Conf_X(M,M_0)_{\leq 1}\to \Conf_{V_1}(\mathbb{R}^{r_1})\simeq \Omega^{r_1}\Sigma^{r_1}V_1$ is homotopic to the unit map. It follows that the induced map $\Conf_X(M,M_0)_{\leq 1}\to V_1$ in the Spanier--Whitehead category is the identity and in particular an isomorphism. Next, using the same lemma, together with the commutativity of the large diagram above, we find that the induced diagram \[\xymatrix{
\Conf_X(M,M_0)_{\leq k-1}\ar[d]\ar[r]&\Conf_X(M,M_0)_{\leq k}\ar[d]\ar[r]&V_k/V_{k-1}\ar@{=}[d]\\
V_{k-1}\ar[r]&V_k\ar[r]&V_k/V_{k-1}
}\] in the Spanier--Whitehead category commutes. The lefthand vertical arrow is an isomorphism by induction, so Lemma \ref{lem:five lemma} will allow us to conclude that the middle arrow is an isomorphism as long as we are assured that the top row is a cofiber sequence, which follows from Lemma \ref{lem:labeled cofibrations}. Thus, the adjoints of the maps $\Conf_X(M,M_0)_{\leq k}\xrightarrow{\Phi_k} \Conf_{V_k}(\mathbb{R}^{r_k})\simeq \Omega^{r_k}\Sigma^{r_k}V_k$ constitute the data of a filtered stable weak equivalence, as claimed.
\end{proof}

\begin{remark}
The proof of Theorem \ref{thm:stable splitting} given here is the classical one following \cite{Boedigheimer:SSMS} and \cite{CohenMayTaylor:SCSC}. More modern arguments have since become available. For one such agument, which uses hypercover-type techniques like those of \S\ref{section:covering theorems}, see \cite{Bandklayder:SSNPVNPD}.
\end{remark}

We turn now to the proof of Lemma \ref{lem:labeled cofibrations}. Recall that the condition for $A\to B$ to be a Hurewicz cofibration is the existence of a solution to all lifting problems of the form \[\xymatrix{
A\ar[d]\ar[r]&Y^I\ar[d]\\
B\ar@{-->}[ur]\ar[r]& Y,
}\] where $Y$ is an arbitrary topological space.

\begin{definition}
A pair of spaces $(B,A)$ is a \emph{neighborhood deformation retract} (NDR) pair if there exist maps $f:B\to I$ and $H:B\times I\to B$ such that \begin{enumerate}
\item $f^{-1}(0)=A$,
\item $H|_{A\times I}=\id_{A\times I}$, and
\item $H(b,1)\in A$ if $f(b)<1$.
\end{enumerate}
\end{definition}

We record the following standard facts about NDR pairs---see \cite[A]{May:GILS}, for example. As a matter of notation, given a subspace $A\subseteq B$, we write \[Z_k(A):=\left\{(x_1,\ldots, x_k)\in B^k: \{x_i\}_{i=1}^k\cap A\neq \varnothing\right\}\subseteq B^k.\]

\begin{proposition}\label{prop:NDR facts}
\begin{enumerate}
\item If $B$ is Hausdorff, then the inclusion of the closed subspace $A$ is a Hurewicz cofibration if and only if $(B,A)$ is an NDR pair.
\item If $(B,A)$ and $(B',A')$ are NDR pairs, then so is $(B\times B', B\times A'\cup A\times B')$.
\item If $(B,A)$ is an NDR pair, then $(B^k, Z_k(A))$ is a $\Sigma_k$-equivariant NDR pair.
\end{enumerate}
\end{proposition}

\begin{proof}[Proof of Lemma \ref{lem:labeled cofibrations}]
Since $M_0\subseteq M$ is a submanifold and $x_0\in X$ is a 0-cell, both $(M,M_0)$ and $(X,x_0)$ are NDR pairs. Thus, by Proposition \ref{prop:NDR facts}(3), $(M^k, Z_k(M_0))$ and $(X^k, x_0^k)$ are $\Sigma_k$-equivariant NDR pairs, and, by means of a tubular neighborhood, we may further arrange that the map $H_t:M^k\to M^k$ is injective for each $0\leq t<1$; therefore, $(\Conf_k(M), Z_k(M_0)\cap \Conf_k(M))$ is a $\Sigma_k$-equivariant NDR pair. Thus, by point (2), the pair \[\left(\Conf_k(M)\times X^k, \Conf_k(M)\times Z_k(x_0)\cup \left(Z_k(M_0)\cap \Conf_k(M)\right)\times X^k\right)\] is an NDR pair. It follows that the top arrow in the pushout diagram \[\xymatrix{
\Conf_k(M)\times_{\Sigma_k} Z_k(x_0)\cup \left(Z_k(M_0)\cap \Conf_k(M)\right)\times_{\Sigma_k} X^k\ar[r]\ar[d]&\Conf_k(M)\times_{\Sigma_k}X^k\ar[d]\\
\Conf_X(M,M_0)_{\leq k-1}\ar[r]&\Conf_X(M,M_0)_{\leq k}
}\] is the inclusion of an NDR pair and hence a Hurewicz cofibration, which implies that the bottom arrow is a Hurewicz cofibration as well.
\end{proof}

\subsection{Homology decomposition}

Having completed the proof of the theorem on stable splittings, our next task is to examine its implications for the homology of configuration spaces. Fix a field $\mathbb{F}$ and write $H_*=H_*(-;\mathbb{F})$. We will prove the following theorem of \cite{BoedigheimerCohenTaylor:OHCS}.

\begin{theorem}[B\"{o}digheimer-Cohen-Taylor]\label{thm:labeled conf homology}
Let $M$ be a compact manifold, possibly with boundary. For any $r>1$, there is an isomorphism of bigraded vector spaces \[H_*(\Conf_{S^r}(M))\cong \bigotimes_{i=0}^nH_*(\Omega^{n-i}S^{n+r})^{\otimes\dim H_i(M)}\]
provided $r+n$ is odd or $\mathbb{F}$ is of characteristic 2.
\end{theorem}

The auxiliary grading referred to in the statement is determined on the lefthand side by the isomorphism \[H_*(\Conf_{S^r}(M))\cong\bigoplus_{k\geq1} H_*(\Conf_k(M)_+\wedge_{\Sigma_k}S^{rk})\] and on the righthand side by the isomorphisms \begin{align*}H_*(\Omega^{n-i}S^{n+r})&\cong H_*(\Conf_{S^r}(D^{i}\times D^{n-i}, \partial D^i\times D^{n-i}))\\&\cong \bigoplus_{k\geq1}H_*(\Conf_k(S^i\wedge D^{n-i}_+, *)\wedge_{\Sigma_k}S^{rk})\end{align*} provided by scanning and stable splitting. 

\begin{remark}
The obvious relative statement relating $\Conf_{S^r}(M,M_0)$ and $H_*(M,M_0)$ is also true, but we will not need to use it. For the necessary alterations in the argument, see \cite[3.5]{BoedigheimerCohenTaylor:OHCS}.
\end{remark}

\begin{remark}
By a different argument, the theorem can also be shown to hold in the case $r=1$---see \cite[4.5]{BoedigheimerCohenTaylor:OHCS}. The theorem is false for $r=0$.
\end{remark}

For the remainder of the lecture, we implicitly assume that the hypotheses of the theorem are satisfied.

The strategy will be to proceed by induction on a handle deccomposition of $M$ using exactness and the Serre spectral sequence. We may assume that $M$ is connected; then, in the base case $M=D^n$, we have $\Conf_{S^r}(D^n)\simeq\Omega^nS^{n+r},$ and the theorem holds without assumption on the parity of $r+n$ or the characteristic of the field. For the induction step, we write $\overline M=M\cup_\varphi D^i\times D^{n-i}$ with $i>0$ and consider the following two cases.

Assume first that $H_i(\overline M)$ surjects onto $H_i(\overline M,M)\cong \widetilde H_i(S^i)$. Since $r>0$, the sphere $S^r$ is connected, so exactness provides us with the homotopy pullback square
\[\xymatrix{\Conf_{S^r}(M)\ar[d]\ar[r]&\Conf_{S^r}(\overline M)\ar[d]\\
\pt\ar[r]&\Conf_{S^r}(\overline M, M).
}\] In the lower left corner, we have \begin{align*}
\Conf_{S^r}(\overline M, M)&\cong \Conf_{S^r}(D^i\times D^{n-i}, \partial D^i\times D^{n-i})\\
&\simeq \Map\left(\mathring{D}^i\times D^{n-i}, \mathring{D}^i\times \partial D^{n-i}, S^{n+r}\right)\\
&\simeq \Omega^{n-i}S^{n+r},
\end{align*} which is simply connected, since \[\pi_1(\Omega^{n-i}S^{n+r})\cong\pi_{n-i+1}(S^{n+r})=0\] for $n+r\geq n+1>n-i+1$. Thus, the action on the homology of the homotopy fiber is trivial. Since we are working over a field, we conclude that, in the homology Serre spectral sequence for this homotopy pullback, we have \begin{align*}
E^2&\cong H_*(\Omega^{n-i}S^{n+r})\otimes H_*(\Conf_{S^r}(M))\\
&\cong H_*(\Omega^{n-i}S^{n+r})\otimes\bigotimes_{j=0}^nH_*(\Omega^{n-j}S^{n+r})^{\otimes\dim H_j(M)}\\
&\cong \bigotimes_{i=0}^nH_*(\Omega^{n-i}S^{n+r})^{\otimes\dim H_i(\overline{M})},
\end{align*} where the second isomorphism uses the inductive hypothesis and the third our assumption on the homology of $\overline M$. Thus, in order to prove the theorem for $\overline M$, it will suffice to show that the Serre spectral sequence collapses at $E^2$.

Assume instead that the map $H_i(\overline M)\to H_i(\overline M, M)$ is trivial. Proceeding as before, and looping the resulting homotopy pullback once, we obtain the homotopy pullback square \[\xymatrix{
\Omega\,\Conf_{S^r}(\overline M, M)\ar[d]\ar[r]&\Conf_{S^r}(M)\ar[d]\\
\pt\ar[r]&\Conf_{S^r}(\overline M).
}\] We now use the assumption that $r>1$.

\begin{lemma}
For $r>1$, $\Conf_{S^r}(\overline M)$ is simply connected.
\end{lemma}
\begin{proof}
We again proceed by induction over a handle decomposition. Assuming the claim is known for $M$, the exact sequence
\[\cdots\to \pi_1(\Conf_{S^r}(M))\to \pi_1(\Conf_{S^r}(\overline M))\to \pi_1(\Omega^{n-i}S^{n+r})\to \cdots\] implies the claim for $\overline M$. For the base case, we note that \[\pi_1(\Conf_{S^r}(D^n))\cong \pi_1(\Omega^nS^{n+r})\cong \pi_{n+1}(S^{n+r})=0,\] since $r>1$.
\end{proof} 

Thus, in this second case as well, the action on the homology of the homotopy fiber is trivial, and we have \begin{align*}
E^2\cong H_*(\Conf_{S^r}(\overline M))\otimes H_*(\Omega^{n-i+1}S^{n+r})\implies H_*(\Conf_{S^r}(M)\cong \bigotimes_{j=0}^nH_*(\Omega^{n-j}S^{n+r})^{\otimes\dim H_j(M)},
\end{align*} where the inductive hypothesis has been used in identifying the $E^\infty$ page. Thus, in order to prove the theorem for $\overline M$, it will again suffice to show that the Serre spectral sequence collapses at $E^2$.

Assuming the collapse of these spectral sequences, the final step in the proof of Theorem \ref{thm:labeled conf homology} is to identify the extra gradings on each side. In order to do so, we note that the maps \[\Conf_{S^r}(M)\to \Conf_{S^r}(\overline M)\to \Conf_{S^r}(\overline M,M)\] are compatible with the respective filtrations by cardinality, so this filtration gives rise to an third grading at the level of the Serre spectral sequence. Since this grading is precisely the auxiliary grading referred to in the statement of the theorem, and since the isomorphism in question was established by observing the spectral sequence to collapse, it follows that the isomorphism is compatible with this auxiliary grading.

In summary, we have reduced the theorem to the following result.

\begin{lemma}\label{lem:collapse lemma}
Let $\overline M=M\cup_\varphi D^i\times D^{n-i}$. \begin{enumerate}
\item If $H_i(\overline M)\to H_i(\overline M,M)$ is surjective, then the homological Serre spectral sequence for the fiber sequence \[\Conf_{S^r}(M)\to \Conf_{S^r}(\overline M)\to \Conf_{S^r}(\overline M,M)\] collapses at $E^2$.
\item If $H_i(\overline M)\to H_i(\overline M,M)$ is zero, then the homological Serre spectral sequence for the fiber sequence \[ \Omega\,\Conf_{S^r}(\overline M,M)\to\Conf_{S^r}(M)\to \Conf_{S^r}(\overline M)\] collapses at $E^2$.
\end{enumerate}
\end{lemma}

We will prove part (1) only, as the proof of part (2) is nearly identical---see \cite[3.4]{BoedigheimerCohenTaylor:OHCS} for full details. Our strategy in proving this result will be to relate the spectral sequence in question to a second spectral sequence, whose collapse will be easier to prove.

Recall from the proof of Theorem \ref{thm:stable splitting} that we have the map \[\eta_M:\Conf_{S^r}(M)\xrightarrow{\Phi}\Conf_V(\mathbb{R}^\infty)\to \Conf_{M_+\wedge S^r}(\mathbb{R}^\infty)\simeq \Omega^\infty\Sigma^\infty(M_+\wedge S^r),\] whose adjoint we might think of conceptually as comparing the ``homology theory'' of labeled configuration spaces to ordinary homology. In order to exploit this map, we require the following simple observation.

\begin{lemma}\label{lem:Q and fiber sequences}
If $A\to B\xrightarrow{f} C$ is a cofiber sequence, then \[\Omega^\infty\Sigma^\infty A\to \Omega^\infty\Sigma^\infty B\to \Omega^\infty\Sigma^\infty C\] is a fiber sequence.
\end{lemma}
\begin{proof}
From our previous observation that $\pi_i\circ \Omega^\infty\Sigma^\infty\cong\pi_i^s$, the long exact sequence in stable homotopy associated to a cofiber sequence, and the long exact sequence in homotopy associated to a fiber sequence, we obtain the commuting diagram \[\xymatrix{
\cdots\ar[r]&\pi_{i+1}(\Omega^\infty\Sigma^\infty C)\ar@{=}[d]\ar[r]&\pi_i(\Omega^\infty\Sigma^\infty A)\ar[d]\ar[r]& \pi_i(\Omega^\infty\Sigma^\infty B)\ar@{=}[d]\ar[r]&\cdots\\
\cdots\ar[r]&\pi_{i+1}(\Omega^\infty\Sigma^\infty C)\ar[r]&\pi_i(\mathrm{hofiber}\,f)\ar[r]&\pi_i(\Omega^\infty\Sigma^\infty B)\ar[r]&\cdots
}\] with exact rows. It follows from the five lemma that the canonical map from $\Omega^\infty\Sigma^\infty A$ to the homotopy fiber is a weak equivalence, as claimed.
\end{proof}

Therefore, since $M_+\to \overline M_+\to \overline M/M$ is a cofiber sequence, we obtain a map of fiber sequences as depicted in the following commuting diagram \[\xymatrix{
\Conf_{S^r}(M)\ar[d]\ar[r]&\Omega^\infty\Sigma^\infty(M_+\wedge S^r)\ar[d]\\
\Conf_{S^r}(\overline M)\ar[d]\ar[r]&\Omega^\infty\Sigma^\infty(\overline M_+\wedge S^r)\ar[d]\\
\Conf_{S^r}(\overline M,M)\ar[r]&\Omega^\infty\Sigma^\infty(\overline M/M\wedge S^r),
}\] and thereby a map at the level of Serre spectral sequences \[E(\Conf)\to E(\Omega^\infty\Sigma^\infty).\] This map is the desired comparison map. The following lemma asserts that its target has the desired property.

\begin{lemma}\label{lem:surjective collapse}
If $H_*(\overline M)\to H_*(\overline M,M)$ is surjective, then the spectral sequence $E(\Omega^\infty\Sigma^\infty)$ collapses at $E^2$.
\end{lemma}

In order to prove this auxiliary collapse result, we make use of the following observation.

\begin{lemma}\label{lem:Q and surjections}
The functor $\Omega^\infty\Sigma^\infty$ preserves homology surjections.
\end{lemma}

This result, in turn, relies on the following basic fact.

\begin{lemma}\label{lem:conf classifying space}
The space $\Conf_k(\mathbb{R}^\infty)$ is contractible for every $k\geq0$.
\end{lemma}
\begin{proof}
We proceed by induction on $k$, the base case $k\in \{0,1\}$ being obvious. By Fadell--Neuwirth, we have a homotopy pullback square \[\xymatrix{\mathbb{R}^n\setminus\{x_1,\ldots, x_k\}\ar[d]\ar[r]&\Conf_{k+1}(\mathbb{R}^n)\ar[d]\\
\pt\ar[r]&\Conf_k(\mathbb{R}^n)
}\] for each $n\geq0$, and thus, in the limit, we obtain the homotopy pullback square \[\xymatrix{\mathbb{R}^\infty\setminus\{x_1,\ldots, x_k\}\ar[d]\ar[r]&\Conf_{k+1}(\mathbb{R}^\infty)\ar[d]\\
\pt\ar[r]&\Conf_k(\mathbb{R}^\infty).
}\] The space in the upper left corner is homotopy equivalent to $\bigvee_kS^\infty$, which is contractible, and the claim follows from the long exact sequence in homotopy.
\end{proof}

\begin{proof}[Proof of Lemma \ref{lem:Q and surjections}]
We give the proof under the assumption that the source and target of the homology surjection are both connected, which is the only case that we will use. The general case may be treated in the same way with group completion techniques. 

Let $f:X\to Y$ be a homology surjection between connected pointed spaces. We have the chain of isomorphisms \begin{align*}
H_*(\Omega^\infty\Sigma^\infty X)&\cong H_*(\Conf_X(\mathbb{R}^\infty))\\
&\cong \bigoplus_{k\geq0} H_*(\Conf_k(\mathbb{R}^\infty)_+\wedge_{\Sigma_k}X^{\wedge k})\\
&\cong \bigoplus_{k\geq0} H_*(\Sigma_k; \widetilde H(X)^{\otimes k}),
\end{align*} and similarly for $Y$, where the first isomorphism uses connectivity and scanning, the second uses stable splitting, and the third uses Lemma \ref{lem:conf classifying space}, the fact that the action of $\Sigma_k$ on $\Conf_k(\mathbb{R}^\infty)$ is contractible, and the K\"{u}nneth isomorphism. Since we are working over a field, the surjection $\widetilde H_*(X)\to \widetilde H_*(Y)$ splits, and this splitting induces an equivariant splitting of $\widetilde H_*(X)^{\otimes k}\to \widetilde H_*(Y)^{\otimes k}$ and hence a splitting at the level of group homology.
\end{proof}

\begin{proof}[Proof of Lemma \ref{lem:surjective collapse}]
In the commuting diagram \[\xymatrix{
H_*(\Omega^\infty\Sigma^\infty(\overline M_+\wedge S^r))\ar@{->>}[d]\ar[r]&H_*(\Omega^\infty\Sigma^\infty(\overline M/M\wedge S^r))\ar@{=}[d]^-\wr\\
E^\infty_{*,0}\ar[r]&E^2_{*,0},
}\] the bottom arrow is an injection, since $E^\infty_{*,0}=\bigcap_{r}\ker d^r_{*,0}$ (recall that the differential $d^r$ in the homological Serre spectral sequence has bidegree $(-r,r-1)$). Since the quotient map $\overline M\to \overline M/M$ induces a surjection on homology by assumption, it does so after smashing with $S^r$; therefore, the top map is a surjection by Lemma \ref{lem:Q and surjections}. It follows that the bottom map is a surjection and therefore an isomorphism, and we conclude that $d_r|_{E^r_{*,0}}$ vanishes for every $r\geq2$.

Now, the functor $\Omega^\infty\Sigma^\infty$ takes values in homotopy associative and commutative $H$-spaces, so the spectral sequence in question is a spectral sequence of graded commutative algebras. By induction and Leibniz rule, it now follows that $d^r$ is identically zero for all $r\geq2$, as desired.
\end{proof}

The other relevant piece of information about this comparison is the following fact.

\begin{lemma}\label{lem:spectral sequence injective}
The map $E^2(\Conf)\to E^2(\Omega^\infty\Sigma^\infty)$ is injective.
\end{lemma}

Before demonstrating this injectivity, we use it to prove the desire result collapse.

\begin{proof}[Proof of Lemma \ref{lem:collapse lemma}]
By induction with base case covered by Lemma \ref{lem:spectral sequence injective}, $E^2(\Conf)\cong E^r(\Conf)$ and $E^r(\Conf)\to E^r(\Omega^\infty\Sigma^\infty)$ is injective. By Lemma \ref{lem:surjective collapse}, $d^r(x)$ becomes zero in $E^r(\Omega^\infty\Sigma)$ for any $r\geq2$ and $x\in E^r(\Conf)$; therefore, we conclude that $d^r$ is identically zero, so $E^2(\Conf)\cong E^{r+1}(\Conf)$ and $E^{r+1}(\Conf)\to E^{r+1}(\Omega^\infty\Sigma^\infty)$ is inejctive. Thus, the spectral sequence collapses at $E^2$.
\end{proof}

\subsection{Loop space calculations} We will give the proof of Lemma \ref{lem:spectral sequence injective} under the assumption that $r+n$ is odd and $\mathbb{F}$ has characteristic 0. The argument in finite characteristic proceeds along similar lines, leveraging a more complicated computation of the homology of iterated loop spaces carried out in \cite{CohenLadaMay:HILS}. The simpler computation that we will use is the following.

\begin{proposition}\label{prop:loop space odd sphere}
If $m$ is odd and $\mathbb{F}$ is of characteristic 0, then, for any $0\leq k<m$, there is an isomorphism of graded vector spaces \[\Sym_\mathbb{F}(\alpha)\xrightarrow{\simeq}H_*(\Omega^kS^m),\] where $\alpha$ is the image of the class of the identity under the composite map \[\pi_m(S^m)\cong \pi_{m-k}(\Omega^k S^m)\to H_*(\Omega^k S^m).\] Moreover, for $k>0$, this isomorphism is an isomorphism of algebras.
\end{proposition}

We will return to this calculation in the next lecture.

\begin{proof}[Proof of Lemma \ref{lem:spectral sequence injective}]
In light of the commuting diagram \[\xymatrix{
E^2(\Conf)\ar[d]\ar@{=}[r]^-\sim&H_*(\Conf_{S^r}(\overline M,M))\otimes H_*(\Conf_{S^r}(M))\ar[d]^-{H_*(\eta_{(\overline M,M)})\otimes H_*(\eta_M)}\\
E^2(\Omega^\infty\Sigma^\infty)\ar@{=}[r]^-\sim&H_*(\Omega^\infty\Sigma^\infty(\overline M/M\wedge S^r))\otimes H_*(\Omega^\infty\Sigma^\infty(M_+\wedge S^r)),
}\] it suffices to show that $\eta_{(\overline M,M)}$ and $\eta_M$ are each injective on homology. 

For the map $\eta_{(\overline M,M)}$, a brief consideration of the definition of the map $\Phi$, along the lines of earlier arguments, shows that the composite \[
\Omega^{n-i}\Sigma^{n-i}S^{r+i}\simeq \Conf_{S^r}(\overline M,M)\xrightarrow{\eta_{(\overline M,M)}} \Conf_{\overline M/M\wedge S^r}(\mathbb{R}^\infty)\simeq\Omega^\infty\Sigma^\infty S^{r+i}
\] is homotopic to the canonical inclusion. Thus, the claim in this case follows from Proposition \ref{prop:loop space odd sphere}, since \begin{align*}
H_*(\Omega^\infty\Sigma^\infty S^{r+i})&\cong \colim_\ell H_*(\Omega^\ell \Sigma^\ell S^{r+i})\\
&\cong \colim_s H_*(\Omega^{n-i+2s}\Sigma^{n-i+2s}S^{r+i})\\
&\cong \colim_s H_*(\Omega^{n-i+2s}S^{n+r+2s})\\
&\cong \colim_s\Sym_\mathbb{F}(\alpha_s)\\
&\cong\Sym_\mathbb{F}(\alpha_\infty),
\end{align*} where $|\alpha_s|=|\alpha_\infty|=r+i$, and we have used that $n+r+2s$ is odd.

The case of the map $\eta_M$ follows by induction on a handle decomposition, the base case of $D^n$ following as before from Proposition \ref{prop:loop space odd sphere}.
\end{proof}

Note that, in the characteristic 0 case, the injection $E^2(\Conf)\to E^2(\Omega^\infty\Sigma^\infty)$ is in fact an isomorphism.

We abstract the general features of the calculation behind Proposition \ref{prop:loop space odd sphere} in the following two lemmas.

\begin{lemma}
Let $\mathbb{F}$ be any field and $\{E^r\}$ a first-quadrant, homological, multiplicative spectral sequence over $\mathbb{F}$. If\begin{enumerate}
\item $E^2_{*,*}\cong E^2_{*,0}\otimes E^2_{0,*}$,
\item $E^2_{*,0}\cong\wedge_\mathbb{F}(x)$ with $|x|$ odd, and
\item $E^\infty=E^\infty_{0,0}\cong\mathbb{F}$,
\end{enumerate} then $E^2_{0,*}\cong\mathbb{F}[y]$ with $|y|=|x|-1$.
\end{lemma}
\begin{proof}
We note first that $d^r|_{E_{*,0}}=0$ for $r<|x|$ by (1); therefore, by the Leibniz rule, $d^r\equiv0$ for $r<|x|$. Next, we note that $y:=d^{|x|}(x)\neq 0$, for otherwise $x$ is a permanent cycle, contradicting (3). By the Leibniz rule, we compute that \begin{align*}
d^{|x|}(x\otimes y^\ell)&=d^{|x|}(x)y^\ell+(-1)^{|x|}xd^{|x|}(y^\ell)\\
&=y^{\ell+1}.
\end{align*}
Note that $y^{\ell+1}\neq 0$ for each $\ell$, since otherwise $x\otimes y^\ell$ is a permanent cycle, and the result follows, since there can be no further differentials and hence no further elements in $E^2_{0,*}$.
\end{proof}

\begin{lemma}
Let $\mathbb{F}$ be a field of characteristic zero and $\{E^r\}$ a first-quadrant, homological, multiplicative spectral sequence over $\mathbb{F}$. If\begin{enumerate}
\item $E^2_{*,*}\cong E^2_{*,0}\otimes E^2_{0,*}$,
\item $E^2_{*,0}\cong \mathbb{F}[y]$ with $|y|$ even, and
\item $E^\infty=E^\infty_{0,0}\cong\mathbb{F}$,
\end{enumerate} then $E^2_{0,*}\cong\wedge_\mathbb{F}(x)$ with $|x|=|y|-1$.
\end{lemma}
\begin{proof}
As in the previous argument, we find that $d^r\equiv0$ for $r<|y|$, and that $x:=d^{|y|}(y)\neq0$. By the Leibniz rule and induction, we have that \begin{align*}
d^{|y|}(y^\ell)&=d^{|y|}(y)y^{\ell-1}+(-1)^{|y|} yd^{|y|}(y^{\ell-1})\\
&=\ell y^{\ell-1}\otimes x\\
&\neq 0,
\end{align*} since $\mathbb{F}$ has characteristic zero. Since there can be no further differentials, the claim follows.
\end{proof}

\begin{proof}[Proof of Proposition \ref{prop:loop space odd sphere}]
We proceed by induction on $k$, the base case of $k=0$ being the observation that $H_*(S^m)\cong \wedge_\mathbb{F}(\alpha)$. For the induction step, we use the Serre spectral sequence for the homotopy pullback square \[\xymatrix{
\Omega^k S^m\ar[r]\ar[d]&P\,\Omega^{k-1}S^m\ar[d]\\
\pt\ar[r]&\Omega^{k-1}S^m,
}\] where $P$ denotes the path space functor. Since the path space is contractible, and since, for $k>1$, this spectral sequence is multiplicative, one of the two lemmas applies, depending on the parity of $k$. To identify the generator $\alpha$, it suffices to note that $\alpha\neq0$ by the Hurewicz theorem, and that $|\alpha|=m-k$.

In the case $k=1$, the spectral sequence is not multiplicative; rather, via naturality and the commutative diagram \[\xymatrix{
\Omega S^m\times\Omega S^m\ar[d]\ar[r]&\Omega S^m\ar[d]\\
\Omega S^m\times PS^m\ar[r]\ar[d]&PS^m\ar[d]\\
S^m\ar@{=}[r]&S^m
}\]it is a spectral sequence of $H_*(\Omega S^m)$-modules. This structure is sufficient to imply the necessary equation \[d^m(x\otimes y^\ell)=d^m(x)y^\ell=y^{\ell+1}\] in this case as well.
\end{proof}

\begin{example}
In finite characteristic $p$, the same calculation shows that $H_*(\Omega S^m)\cong \mathbb{F}[y]$. For $\Omega^2S^m$, however, the argument breaks down, for now\[d^{m-1}(y^p)=py^{p-1}=0,\] implying the existence of further elements \begin{align*}
Q^1(y)&:=d^{p(m-1)}(y^p)\\
\beta Q^1(y)&=d^{(p-1)(m-1)}(y^{p-1}\otimes x)
\end{align*} of degree $p(m-1)-1$ and $p(m-1)-2$, respectively. A systematic approach to these and other nontrivial higher differentials, and the resulting profusion of homology classes, is provided by the framework of homology operations for iterated loop spaces. The symbol $Q^1$ used above refers to a certain \emph{Dyer-Lashof} operation and the letter $\beta$ to the Bockstein operator, and it turns out that the homology of iterated loop spaces of spheres can be completely described in terms of composites of Dyer-Lashof operations and the Bockstein---see \cite{CohenLadaMay:HILS} for further details.
\end{example}

Our next move, having completed the proof of Theorem \ref{thm:labeled conf homology}, will be to exploit it in understanding the homology of ordinary configuration spaces.

\subsection{Odd dimensional homology} Our strategy in proving Theorem \ref{thm:odd homology} will be to reinterpret the two sides of the isomorphism of Theorem \ref{thm:labeled conf homology}, keeping track of the extra grading.

On the lefthand side of the isomorphism, as bigraded vector spaces, we have \begin{align*}H_*(\Conf_{S^r}(M))&\cong \mathbb{F}\oplus\widetilde H_*(\Conf_{S^r}(M))\\
&\cong \mathbb{F}\oplus \bigoplus_{k\geq1}\widetilde H_*\left(\Conf_k(M)_+\wedge_{\Sigma_k}S^{rk}\right)\\
&\cong\bigoplus_{k\geq0}\widetilde H_*\left(\Conf_k(M)_+\wedge_{\Sigma_k}S^{rk}\right)
\end{align*} by stable splitting. Writing $\pi_{k,r}$ for the natural projection \[\pi_{k,r}:\Conf_k(M)\times_{\Sigma_k}\mathbb{R}^{rk}\to B_k(M),\] we have the following observation.

\begin{lemma}\label{lem:thom space}
The map $\pi_{k,r}$ is a vector bundle of rank $rk$, and there is a pointed homeomorphism \[\mathrm{Th}(\pi_{k,r})\cong \Conf_k(M)_+\wedge_{\Sigma_k} S^{rk}.\]
\end{lemma}
\begin{proof}
For an open subset $U\subseteq M$ with $k$ connected components, each Euclidean, a choice of ordering $\tau:\{1,\ldots, k\}\cong \pi_0(U)$ gives rise to the commuting diagram
\[\xymatrix{\displaystyle\coprod_{\sigma:\{1,\ldots, k\}\cong \pi_0(U)}\prod_{i=1}^kU_{\sigma(i)}\times\mathbb{R}^{rk}\ar[d]&\displaystyle\coprod_{\sigma:\{1,\ldots, k\}\cong \pi_0(U)}\Conf_k^0(U,\sigma)\times\mathbb{R}^{rk}\ar[l]_-{\simeq}\ar[d]\ar[r]&\Conf_k(M)\times\mathbb{R}^{rk}\ar[d]\\
\displaystyle\prod_{i=1}^kU_{\tau(i)}\times\mathbb{R}^{rk}\ar[d]&\displaystyle\left(\coprod_{\sigma:\{1,\ldots, k\}\cong \pi_0(U)}\Conf_k^0(U,\sigma)\times\mathbb{R}^{rk}\right)_{\Sigma_k}\ar[d]\ar[l]_-\simeq\ar[r]&\Conf_k(M)\times_{\Sigma_k}\mathbb{R}^{rk}\ar[d]^-{\pi_{k,r}}\\
\displaystyle\prod_{i=1}^k U_{\tau(i)}&B_k^0(U)\ar[l]_-\simeq\ar[r]&B_k(M),
}\] in which the two righthand squares are pullback squares and the upper three vertical maps are projections to spaces of $\Sigma_k$-coinvariants. Since $B_k(M)$ is covered by open subsets of the form $B_k^0(Y)$, it follows that $\pi_{k,r}$ is locally trivial. Since addition and scalar multiplication in $\mathbb{R}^{rk}$, the linear structure of $\Conf_k(M)\times\mathbb{R}^{rk}$ descends to the quotient.

For the second claim, we have \begin{align*}
\mathrm{Th}(\pi_{k,r})&\cong D(\pi_{k,r})/S(\pi_{k,r})\\
&\cong \frac{\Conf_k(M)\times_{\Sigma^k}D^{rk}}{\Conf_k(M)\times_{\Sigma_k}S^{rk-1}}\\
&\cong \Conf_k(M)_+\wedge_{\Sigma_k}S^{rk}.
\end{align*}
\end{proof}

In order to apply the Thom isomorphism, we require the following result.

\begin{lemma}\label{lem:orientable}
If $r$ is even, then $\pi_{k,r}$ is orientable.
\end{lemma}
\begin{proof}
Since $r$ is even, any orientation of $\mathbb{R}^r$ induces a $\Sigma_k$-invariant orientation of $\mathbb{R}^{rk}$, whence of the trivial bundle $\Conf_k(M)\times \mathbb{R}^{rk}$. By $\Sigma_k$-invariance, this orientation descends to the quotient.
\end{proof}

\begin{remark}
Again, characterstic 2 is exceptional.
\end{remark}

\begin{corollary}\label{cor:lhs}
For $n$ odd and $r>1$ even, there is an isomorphism of bigraded vector spaces \[H_*(\Conf_{S^r}(M))\cong\bigoplus_{k\geq0} H_*(B_k(M))[rk].\]
\end{corollary}

As for the righthand side of the isomorphism of Theorem \ref{thm:labeled conf homology}, stable splitting gives us the following commuting diagram of isomorphisms \[\xymatrix{
H_*(\Omega^{n-i}S^{n+r})\ar@{=}[d]_-\wr&  \displaystyle\bigoplus_{k\geq0}\widetilde H_*\left(\Conf_k(D^i\times D^{n-i}, \partial D^i\times D^{n-i})\wedge_{\Sigma_k}S^{rk}\right)
\ar[l]_-\simeq\ar@{-->}[d]\\
H_*(\Omega^{n-i}\Sigma^{n-i}S^{r+i})&\displaystyle\bigoplus_{k\geq0}\widetilde H_*\left(\Conf_k(\mathbb{R}^{n-i})_+\wedge_{\Sigma_k}S^{k(r+i)}\right).\ar[l]_-\simeq
&
}\] All four terms of this diagram are naturally bigraded, and the horizontal morphisms are compatible with these bigradings. It is possible, by direct consideration of the scanning map, to show that the vertical arrows are also compatible with the bigradings. Rather than proceed in this manner, however, we invoke the following result, which likewise assures us that we may work with the lower left and upper right corners interchangeably. 

\begin{lemma}
For any manifold $N$ and $i,k\geq0$, there is a canonical, $\Sigma_k$-equivariant weak equivalence \[\Sigma^{ik}_+\Conf_k(N)\xrightarrow{\sim} \Conf_k\left(D^i\times N,\partial D^i\times N\right).\]
\end{lemma}
\begin{proof}
The map is defined via the inclusion \[\Sigma_+^{ik}\Conf_k(N)\cong \frac{(D^i)^k\times\Conf_k(N)}{\partial (D^i)^k\times \Conf_k(N)}\xrightarrow{\simeq} U\subseteq \Conf_k\left(D^i\times N, \partial D^i\times N\right)\] of the open subspace $U$ consisting of the basepoint together with all nontrivial configurations whose coordinates in $N$ are distinct. Letting $V$ denote the open subspace consisting of the basepoint together with all nontrivial configurations with at least two distinct coordinates in $D^i$, it is clear that $\Conf_k(D^i\times N,\partial D^i\times N)=U\cup V$. Now, both $V$ and $U\cap V$ are contractible by radial expansion, so the inclusion $i$ of $U$ induces the weak equivalences \begin{align*}
i^{-1}(U)&\xrightarrow{=} U\\
i^{-1}(V)=U\cap V&\xrightarrow{\sim} V\\
i^{-1}(U\cap V)&\xrightarrow{=}U\cap V.
\end{align*} It follows from Proposition \ref{prop:mayer-vietoris} that $i$ is a weak equivalence.
\end{proof}

The final ingredient in the proof will be the following calculation.

\begin{lemma}
For $n$ odd, $r>1$ even, $i\geq0$, and $\mathbb{F}$ of characteristic 0, \[\widetilde H_*\left(\Conf_k(\mathbb{R}^{n-i})_+\wedge_{\Sigma_k}S^{k(r+i)}\right)\cong\begin{cases}
\mathbb{F}[k(r+i)]&\quad i \text{ even or } k\in\{0,1\}\\
0&\quad \text{otherwise.}
\end{cases}\]
\end{lemma}
\begin{proof}
If $i$ is even, then $\pi_{r+i,k}$ is orientable, since $r$ is even, so the homology group in question is identified with \[H_*(B_k(\mathbb{R}^{n-i}))[k(r+i)]\cong\mathbb{F}[k(r+i)]\] by the Thom isomorphism, where we have used that $n-i$ is odd when $n$ is odd and $i$ is even. On the other hand, if $i$ is odd, then by the K\"{u}nneth theorem and the assumption on the charateristic of $\mathbb{F}$, we have \begin{align*}\widetilde H_*\left(\Conf_k(\mathbb{R}^{n-i})_+\wedge_{\Sigma_k}S^{k(r+i)}\right)
&\cong H_*(\Conf_k(\mathbb{R}^{n-i}))\otimes_{\Sigma_k}\widetilde H_*(S^{k(r+i)})\\
&\cong H_*(\Conf_k(\mathbb{R}^{n-i}))\otimes_{\Sigma_k}\mathbb{F}^{\mathrm{sgn}}[k(r+i)],
\end{align*} where $\mathbb{F}^\mathrm{sgn}$ denotes the sign representation of $\Sigma_k$. Earlier, we showed that the homology group $H_*(\Conf_k(\mathbb{R}^{n-i}))$ has a spanning set indexed by $k$-forests, where switching adjacent leaves of a tree introduces a sign $(-1)^\epsilon$, where $\epsilon$ is the degree of the antipodal map on $S^{n-i-1}$. Since $n$ and $i$ are both odd, it follows that $\epsilon=0$. After tensoring with the sign representation, it follows that any transposition in $\Sigma_k$ acts by $-1$. Thus, in this case, \[H_*(\Conf_k(\mathbb{R}^{n-i})\otimes_{\Sigma_k}\mathbb{F}^\mathrm{sgn}[k(r+i)]\cong\begin{cases}
\mathbb{Q}[k(r+i)]&\quad k\in\{0,1\}\\
0&\quad\text{otherwise,}
\end{cases}\] which completes the proof.
\end{proof}

\begin{proof}[Proof of Theorem \ref{thm:odd homology}]
Combining what we have done so far, we obtain the chain of isomorphisms of bigraded vector spaces \begin{align*}
\bigoplus_{k\geq0} H_*(B_k(M))[kr]&\cong \bigotimes_{i=0}^n\left(\bigoplus_{\ell\geq0}H_*\left(\Conf_\ell(\mathbb{R}^{n-i})_+\wedge_{\Sigma_\ell}S^{\ell(r+i)}\right)\right)^{\otimes\dim H_i(M)}\\
&\cong\bigotimes_{i\text{ even}}\left(\bigoplus_{\ell\geq0}\mathbb{F}[\ell(r+i)]\right)^{\otimes \dim H_i(M)}\otimes\bigotimes_{i\text{ odd}}\left(\bigoplus_{\ell\in\{0,1\}}\mathbb{F}[\ell(r+i)]\right)^{\otimes\dim H_i(M)}\\
&\cong \bigotimes_{i\text{ even}}\Sym(\mathbb{F}[r+i])^{\otimes \dim H_i(M)}\otimes\bigotimes_{i\text{ odd}}\Sym(\mathbb{F}[r+i])^{\otimes\dim H_i(M)}\\
&\cong\Sym(H_\mathrm{even}(M)[r])\otimes\Sym(H_\mathrm{odd}(M)[r])\\
&\cong \Sym\left(H_*(M)[r]\right),
\end{align*} where we have repeatedly used that $\Sym$ sends direct sums to tensor products. It follows that \[H_*(B_k(M))[kr]\cong \Sym^k(H_*(M))[kr],\] which completes the proof.
\end{proof}

\section{Mod $p$ cohomology}\label{section:mod p cohomology}

In the previous section, we expressed the mod $p$ (co)homology of the unordered configuration of a manifold of odd dimension in terms of the (co)homology of the spaces $\Conf_k(\mathbb{R}^n)_+\wedge_{\Sigma_k}S^{rk}$ for varying $k$, $n$, and $r$. Therefore, reformulating a little using the K\"{u}nneth isomorphism, we see that this computation amounts to essentially two computations of cohomology with local coefficients; indeed, we have \begin{align*}
H^*\left(\Conf_k(\mathbb{R}^n)_+\wedge_{\Sigma_k}S^{rk}\right)&\cong H\left(\Map_{\Sigma_k}\left(C_*(\Conf_k(\mathbb{R}^n)), \widetilde{C}^*(S^{rk})\right)\right)\\
&\cong H\left(\Map_{\Sigma_k}\left(C_*(\Conf_k(\mathbb{R}^n)), \mathbb{F}_p(r)\right)\right)[kr]\\
&=: H^*(B_k(\mathbb{R}^n);\mathbb{F}_p(r))[kr],
\end{align*} where $\mathbb{F}_p(r)$ is the $\Sigma_p$-module defined by \[
\mathbb{F}_p(r)=\begin{cases}\mathbb{F}_p^\mathrm{triv}&\quad r\text{ even}\\
\mathbb{F}_p^\mathrm{sgn}&\quad r\text{ odd}.
\end{cases}\] Our present goal is to describe the method behind this computation, which was first carried out in \cite[III]{CohenLadaMay:HILS}. For ease of exposition, we will focus on the case $r$ even, but the other case may be treated in an entirely parallel fashion. We also restrict to the case $k=p$, from which the other cases may be derived in a rather roundabout manner.

\subsection{Outline of argument} The strategy is to leverage our complete understanding of cohomology of the covering space $\Conf_k(\mathbb{R}^n)$ using the spectral sequence of Corollary \ref{cor:ss of a cover}. From the commutative diagram of covering spaces \[\xymatrix{
\Conf_p(\mathbb{R}^n)\ar[d]\ar@{=}[r]&\Conf_p(\mathbb{R}^n)\ar[d]\ar[r]&\Conf_p(\mathbb{R}^\infty)\ar[d]\\
\Conf_p(\mathbb{R}^n)_{C_p}\ar[r]&B_p(\mathbb{R}^n)\ar[r]&B_p(\mathbb{R}^\infty),
}\] we obtain the commutative diagram of fiber sequences \[\xymatrix{
\Conf_p(\mathbb{R}^n)\ar[d]\ar@{=}[r]&\Conf_p(\mathbb{R}^n)\ar[d]\ar[r]&\pt\ar[d]\\
\Conf_p(\mathbb{R}^n)_{C_p}\ar[d]\ar[r]&B_p(\mathbb{R}^n)\ar[d]\ar[r]&B\Sigma_p\ar@{=}[d]\\
BC_p\ar[r]&B \Sigma_p\ar@{=}[r]&B\Sigma_p,
}\] and hence the commuting diagram of spectral sequences \[\xymatrix{
H^s(\Sigma_p)\ar[d]\ar@{=}[r]&H^s(\Sigma_p)\ar[d]\\
H^s(\Sigma_p; H^t(\Conf_p(\mathbb{R}^n)))\ar[d]\ar@{=>}[r]&H^{s+t}(B_p(\mathbb{R}^n))\ar[d]\\
H^s(C_p;H^t(\Conf_p(\mathbb{R}^n)))\ar@{=>}[r]&H^{s+t}(\Conf_p(\mathbb{R}^n)_{C_p}).
}\] The utility of this maneuver lies in the following basic observation.

\begin{lemma}\label{lem:principal bundle transfer}
If $\pi:P\to X$ is a principal $G$-bundle and $H<G$ is a $p$-Sylow subgroup, then the induced map of spectral sequences is injective on $E_2$ and on $E_\infty$.
\end{lemma}
\begin{proof}
On $E_\infty$, the map is induced by the projection $\pi_H:P/H\to X$, which is a $[G:H]$-fold cover; therefore, the composite of $\pi^*_H$ followed by the transfer map is multiplication by $[G:H]$, which is injective, since $H$ is $p$-Sylow. On $E_2$, the map is induced by the map $BH\to BG$, which is also a $[G:H]$-fold cover, and a similar argument in cohomology with local coefficients applies.
\end{proof}

Thus, the spectral sequence of immediate interest is a summand of the spectral sequence obtained by restriction to $C_p$. Regarding this spectral sequence, we have the following 

\begin{vanishingtheorem}
In the spectral sequence for the cover $\Conf_p(\mathbb{R}^n)\to \Conf_p(\mathbb{R}^n)_{C_p}$, we have $E_2^{s,t}=0$ if $s>0$ and $0<t<(n-1)(p-1)$.
\end{vanishingtheorem}

By Lemma \ref{lem:principal bundle transfer}, then, the same vanishing range holds for the spectral sequence of primary interest. Moreover, as we calculated earlier, $H^*(\Conf_p(\mathbb{R}^n))=0$ for $*>(n-1)(p-1)$, so the spectral sequence is concentrated in the $s=0$ column in degree at most $(n-1)(p-1)$ and in the two rows $t=0$ and $t=(n-1)(p-1)$.

As for the differentials, it follows from what has already been said that the only possible nonzero differentials are $d_r$ on the $s=0$ column for $2\leq r\leq (n-1)(p-1)$ and on the $t=(n-1)(p-1)$ column for $r=(n-1)(p-1)+1$. Moreover, we have the following simple, but crucial, observation:

\begin{lemma}\label{lem:high differential}
In the spectral sequence for the cover $\Conf_p(\mathbb{R}^n)\to B_p(\mathbb{R}^n)$, the differential $d_{(n-1)(p-1)+1}:E_{(n-1)(p-1)+1}^{s,(n-1)(p-1)}\to E_{(n-1)(p-1)+1}^{s+(n-1)(p-1)+1,0}$ is an isomorphism for $s>n+p-2$.
\end{lemma}
\begin{proof}
Assuming otherwise, it follows that $E_\infty^{s+(n-1)(p-1)+1,0}\neq 0$, and hence that $B_p(\mathbb{R}^n)$ has non-vanishing homology in some degree strictly greater than $(n+p-2)+(n-1)(p-1)+1=np$. Since $B_p(\mathbb{R}^n)$ is a manifold of dimension $np$, this is a contradiction.
\end{proof}

Together with the following classical calculation, this observation will allow us to populate much of the $t=(n-1)(p-1)$ row.

\begin{theorem}[Nakaoka]
There is a commuting diagram \[\xymatrix{
H^*(C_p)\ar@{=}[r]^-\sim&\wedge_{\mathbb{F}_p}[u]\otimes \mathbb{F}_p[\beta u]&u^{2(p-1)-1}\\
H^*(\Sigma_p)\ar@{=}[r]^-\sim\ar[u]&\wedge_{\mathbb{F}_p}[v]\otimes \mathbb{F}_p[\beta v]\ar[u]&v,\ar@{|->}[u]
}\] where $|u|=1$ and $\beta$ denotes the mod $p$ Bockstein.
\end{theorem}

In order to populate the remainder of this row, we make use of the following result, the proof of which is discussed in Appendix \ref{appendix:periodicity}.

\begin{periodicitytheorem}
Let $M$ be a $C_p$-module and $N$ a $\Sigma_p$-module.
\begin{enumerate}
\item Cup product with $\beta u$ induces an isomorphism \[H^s(C_p;M)\xrightarrow{\simeq} H^{s+2}(C_p; M).\]
\item Cup product with $\beta v$ induces an isomorphism \[H^s(\Sigma_p;N)\xrightarrow{\simeq} H^{s+2(p-1)}(\Sigma_p; N).\]
\end{enumerate}
\end{periodicitytheorem}

The final step is the identification of the $s=0$ column, which is simply the module of invariants for the action $\Sigma_p$. For simplicity, we do not state the analogous result for $p=3$.

\begin{invariantstheorem}
For any prime $p>3$,\[I:=H^*(\Conf_p(\mathbb{R}^n))^{\Sigma_p}\cong\begin{cases}
\wedge_{\mathbb{F}_p}(\alpha_{n-1})&\quad \text{$n$ even}\\
\mathbb{F}_p&\quad \text{$n$ odd}.
\end{cases}\]
\end{invariantstheorem}

Thus, the only possible differential aside from $d^{(n-1)(p-1)+1}$, whose effect is determined by Lemma \ref{lem:high differential} and periodicity, is $d^n$. After checking that this differential vanishes, the calculation follows.

\begin{theorem}[Cohen]
There is an isomorphism \[H^*(B_p(\mathbb{R}^n))\cong I\times_{\mathbb{F}_p} \frac{H^*(\Sigma_p)}{H^{>(n-1)(p-1)}(\Sigma_p)}.\]
\end{theorem}

In the remainder of this section, we now give this outline flesh.

\begin{warning}
This argument is the one given in the original reference \cite{CohenLadaMay:HILS}, but we adopt our own notational conventions and give different proofs of some result.
\end{warning}

\subsection{Invariants} We begin with a few sample calculations.

\begin{example}
For $p=3$, a basis for the degree $2n-2$ cohomology is given by the set $\left\{\alpha_{12}\alpha_{13}, \alpha_{12}\alpha_{23}\right\}$. We compute that \begin{align*}
\tau_{12}^*(\alpha_{12}\alpha_{13})&=(-1)^n\alpha_{12}\alpha_{23}\\
\tau_{23}^*(\alpha_{12}\alpha_{13})&=(-1)^{(n-1)^2}\alpha_{12}\alpha_{13}\\
\tau_{12}^*(\alpha_{12}\alpha_{23})&=(-1)^n\alpha_{12}\alpha_{13}\\
\tau_{23}^*(\alpha_{12}\alpha_{23})&=(-1)^{n+1}\alpha_{12}\alpha_{13}+(-1)^{(n-1)^2+1}\alpha_{12}\alpha_{23}.
\end{align*} Thus, if $\beta=\lambda_1\alpha_{12}\alpha_{13}+\lambda_2\alpha_{12}\alpha_{23}$ is a fixed point, we have the equations \begin{align*}
\lambda_1&=(-1)^n\lambda_2\\
\lambda_1&=(-1)^{(n-1)^2}\lambda_1+(-1)^{n+1}\lambda_2,
\end{align*} which imply $\beta=0$ unless $2\equiv(-1)^{(n-1)^2}\mod 3$, which occurs precisely when $n$ is even. This special case is the reason for the restriction $p>3$ in the statement of the theorem.
\end{example}

\begin{example}\label{example:fixed point}
Assuming that $n$ is even and setting $\alpha=\sum_{1\leq a<b\leq p}\alpha_{ab}$, we compute that \begin{align*}
\sigma^*\alpha&=\sum_{1\leq a<b\leq p}\alpha_{\sigma(a)\sigma(b)}\\
&=\sum_{\sigma(a)<\sigma(b)}\alpha_{\sigma(a)\sigma(b)}+(-1)^n\sum_{\sigma(b)<\sigma(a)}\alpha_{\sigma(a)\sigma(b)}\\
&=\alpha,
\end{align*} so the fixed point set is at least as large as claimed in the theorem above. Thus, it remains to show that $\alpha$ is the only possible nontrivial fixed point.
\end{example}

In order to prove the theorem, we pass from cohomology to homology in order to btain a more convenient basis. This passage is justified by the tensor/hom adunction; for a finite group $G$ and $G$-module $M$, we have \begin{align*}
H^*(G; M^\vee)&\cong \mathrm{Ext}^*_{\mathbb{F}[G]}\left(\mathbb{F}, M^\vee\right)\\
&\cong \mathrm{Ext}^*_{\mathbb{F}[G]}\left(\mathbb{F}, \mathrm{Ext}^*_\mathbb{F}\left(M,\mathbb{F}\right)\right)\\
&\cong  \mathrm{Ext}^*_{\mathbb{F}}\left(\mathrm{Tor}_*^{\mathbb{F}[G]}\left(\mathbb{F}, M\right),\mathbb{F}\right)\\
&\cong H_*(G; M)^\vee.
\end{align*} In particular, in the case of interest, we have the isomorphisms \begin{align*}
H^*(\Conf_p(\mathbb{R}^n))^{\Sigma_p}&\cong \left(H_*(\Conf_p(\mathbb{R}^n))_{\Sigma_p}\right)^\vee\\
H^*(C_p;H^*(\Conf_p(\mathbb{R}^n)))&\cong H_*(C_p;H_*(\Conf_p(\mathbb{R}^n)))^\vee.
\end{align*}

Recall from \S\ref{section:cohomology} that $H_*(\Conf_p(\mathbb{R}^n))$ is spanned by classes indexed by $p$-forests modulo the Jacobi identity and graded antisymmetry, with a preferred basis given by the tall forests. The advantage of this perspective is that $\Sigma_p$ acts on a forest by permuting the leaves, which are independent of one another, whereas the indices of a monomial such as $\alpha_{ab}\alpha_{bc}\alpha_{ac}$ are far from independent.
 
The first of the theorems in question is now almost immediate; indeed, we essentially gave the argument previously when computing $H_*(B_k(\mathbb{R}^n);\mathbb{Q})$.

\begin{proof}[Proof of Invariants Theorem]
Let $\alpha$ be the class of a tall forest with at least three leaves; thus, $|\alpha|\geq 2(n-1)$. By the Jacobi identity, $3[\alpha]=0$ in $H_*(\Conf_p(\mathbb{R}^n))_{\Sigma_p}$; therefore, since $p>3$, the map from $H_*(\Conf_p(\mathbb{R}^n))$ to the module of coinvariants is zero, since it annihilates a basis. Since this map is also surjective, the claim follows in degree $2(n-1)$ and higher. The argument in degree $0$ is trivial. In degree $n-1$, there are two cases to consider. If $n$ is odd, then every basis element is equivalent to its negative in coinvariants, which must therefore be trivial. If $n$ is even, we note that $\Sigma_p$ acts transitively on our preferred homology basis, so $H_{n-1}(\Conf_p(\mathbb{R}^n))_{\Sigma_p}$ has dimension at most 1. Therefore, by Example \ref{example:fixed point}, the dimension is exactly 1.
\end{proof}

\begin{remark}
This argument differs from that given in \cite[III.9]{CohenLadaMay:HILS}. We were unable to justify some of the claims made in the course of that argument.
\end{remark}

\subsection{Vanishing}
We turn now to the vanishing theorem. As in the previous section, we may reformulate this result in homological terms as the assertion  \[H_s(C_p; H_t(\Conf_p(\mathbb{R}^n))=0\] for $s>0$ and $0<t<(n-1)(p-1)$.

\begin{remark}
We emphasize that this theorem requires no restriction on $p$ (note that the case $p=2$ is vacuous).
\end{remark}

We will use the following vanishing criterion.

\begin{proposition}\label{prop:cyclic decomposition}
Let $V$ be a $C_p$-module over $\mathbb{F}$ and $\sigma\in C_p$ a fixed generator. If $V$ admits a decomposition of the form \[V\cong \bigoplus_{i=1}^{p}V_{\sigma^i}\] such that $\sigma\left(V_{\sigma^i}\right)\subseteq V_{\sigma^{i+1}}$, then $H_s(C_p;V)=0$ for $s>0$.
\end{proposition}
\begin{proof}
The trivial $C_p$-module $\mathbb{F}$ admits the so-called periodic resolution \[\cdots\to \mathbb{F}[C_p]\xrightarrow{N}\mathbb{F}[C_p]\xrightarrow{\sigma-1}\mathbb{F}[C_p]\xrightarrow{\epsilon}\mathbb{F},\] where $N=1+\sigma+\cdots+\sigma^{p-1}$ and $\epsilon$ denotes the augmention. Thus, the group homology of $V$ is computed by the complex \[\cdots \to V\xrightarrow{N} V\xrightarrow{\sigma-1} V,\] so it suffices to show that the inclusion $\mathrm{im}(N)\subseteq \ker(\sigma-1)$ is an equality. 

Suppose that $\sigma(v)=v$. Our assumption on $V$ provides the unique decomposition \[\sum_{i=1}^pv_{\sigma^i}=v=\sigma(v)=\sum_{i=1}^p\sigma(v_{\sigma^i}).\] Since $\sigma(v_{\sigma^i})\in V_{\sigma^{i+1}}$, it follows by induction that $v_{\sigma^i}=v_\sigma$ for all $1\leq i\leq p$, so $v=N v_\sigma.$
\end{proof}

In order to apply this observation to our situation, we recall that a $p$-forest is simply an ordered partition of $\{1,\ldots, p\}$ together with a binary parenthesization of each block of the partition, and that changing the order of the partition introduces an overall sign. Since the Jacobi identity and antisymmetry do not change the partition of a forest, we have the direct sum decomposition \[H_*(\Conf_p(\mathbb{R}^n))\cong \bigoplus_{1\leq \ell\leq p}\bigoplus_{[\pi]\in \mathrm{Surj}(p,\ell)_{\Sigma_\ell}}F_{[\pi]},\] where $F_{[\pi]}$ denotes the subspace spanned by the forests with underlying unordered partition $[\pi]$. We now make three observations.
\begin{enumerate}
\item The degree 0 subspace is exactly the $\ell=p$ summand. Thus, we disregard this summand.
\item The degree $(n-1)(p-1)$ subspace is exactly the $\ell=1$ summand. Thus, we disregard this summand.
\item The symmetric group acts via the action on $\mathrm{Surj}(p,\ell)_{\Sigma_\ell}$ given by pre-composition. In particular, the $\ell$th summand above is closed under the action of $\Sigma_p$.
\end{enumerate}

We conclude that the Vanishing Theorem is equivalent to the claim that \[H_{s}\left(C_p;\bigoplus_{\mathrm{Surj}(\ell,p)_{\Sigma_\ell}}F_{[\pi]}\right)=0\] for $s>0$ and $1<\ell<p$. The essential observation in establishing this claim is the following.

\begin{lemma}\label{lem:partitions}
For any $1<\ell<p$ and $[\pi]\in \mathrm{Surj}(p,\ell)_{\Sigma_\ell}$, \[[\pi\circ\sigma]\neq[\pi].\]
\end{lemma}
\begin{proof}
There are numbers $1\leq i,j\leq \ell$ such that $|\pi^{-1}(i)|\neq |\pi^{-1}(j)|$; indeed, assuming otherwise implies that $\ell\mid p$, which contradicts our assumption that $\ell\notin\{1,p\}$. Now, if $\rho$ is a cyclic permutation taking any element of $\pi^{-1}(i)$ to $\pi^{-1}(j)$, then $[\pi]\neq [\pi\circ\rho]=[\pi\circ\sigma^i]$, which leads to a contradiction under the assumption that $[\pi\circ\sigma]=[\pi]$.
\end{proof}

\begin{proof}[Proof of Vanishing Theorem]
Since $F_{[\pi]}\cap F_{[\pi']}=0$ for $[\pi]\neq [\pi']$, Lemma \ref{lem:partitions} implies that $F_{[\pi]}\cap F_{[\pi\circ\sigma]}=0$ for any $1<\ell<p$ and $[\pi]\in \mathrm{Surj}(p,\ell)_{\Sigma_\ell}$. Thus, for fixed $[\pi]$, the submodule $\bigoplus_{i=1}^pF_{[\pi\circ\sigma^i]}$ satisfies the hypotheses of Proposition \ref{prop:cyclic decomposition}. By induction $\bigoplus_{\mathrm{Surj}(\ell,p)_{\Sigma_\ell}}F_{[\pi]}$ decomposes as a direct sum of submodules of this form. The proposition now implies the claim.
\end{proof}

\begin{remark}
The reader may compare the complexity of this argument in homology to the argument in cohomology of \cite[III.10]{CohenLadaMay:HILS}.
\end{remark}

\section{Postscript: Lie algebra methods}\label{section:lie algebra methods}

In this short final portion of these notes, we give an informal discussion of an approach to computing the rational homology of configuration spaces of general manifolds, which is premised on Lie algebras. 

\subsection{Lie algebras and their homology} We begin with a few reminders. Fix a field $\mathbb{F}$ of characteristic zero. 

\begin{definition}
A \emph{graded Lie algebra} is a graded vector space $\mathfrak{g}$ equipped with a map \[[-,-]:\mathfrak{g}^{\otimes 2}\to \mathfrak{g},\] called the \emph{Lie bracket} of $\mathfrak{g}$, satisfying the following two identities:
\begin{enumerate}
\item $[x,y]=(-1)^{|x||y|+1}[y,x]$
\item $(-1)^{|x||z|}\left[x,[y,z]\right]+(-1)^{|z||y|}\left[z,[x,y]\right]+(-1)^{|y||x|}\left[y,[z,x]\right]=0.$
\end{enumerate}
\end{definition}

\begin{example}
An ordinary Lie algebra may be viewed as a graded Lie algebra concentrated in degree 0.
\end{example}

\begin{example}
Given a graded vector space $V$, there is a free Lie algebra $\mathcal{L}(V)$ satisfying the obvious universal property. In particular, it follows easily from the definition that \[\mathcal{L}\left(v_r\right)=\begin{cases}
\mathbb{F}\langle v\rangle&\quad r\text{ even}\\
\mathbb{F}\langle v\rangle\oplus\mathbb{F}\langle [v,v]\rangle &\quad r\text{ odd.}
\end{cases}\]
\end{example}

\begin{example}
If $\mathfrak{g}$ is a graded Lie algebra and $A$ a graded commutative (possibly nonunital) algebra, then $A\otimes\mathfrak{g}$ is canonically a graded Lie algebra with bracket determined by the formula \[[\alpha\otimes x,\beta\otimes y]=(-1)^{|x||\beta|}\alpha\beta\otimes[x,y].\]
\end{example}

The first identity of the definition above, which is usually called \emph{graded antisymmetry}, ensures that the bracket descends to a map $\Sym^2\left(\mathfrak{g}[1]\right)[-1]\to \mathfrak{g}[1]$. The second identity, known as the \emph{Jacobi identity}, ensures that the appropriate composite \[\Sym^3\left(\mathfrak{g}[1]\right)[-2]\to \Sym^2\left(\mathfrak{g}[1]\right)[-1]\to \mathfrak{g}[1]\] is zero. We now systematize these observations.

\begin{recollection}
For a graded vector space $V$, write $\Sym(V)=\mathbb{F}[V_{\mathrm{even}}]\otimes\Lambda_\mathbb{F}[V_{\mathrm{odd}}]$. This object is a graded cocommutative coalgebra under the \emph{shuffle coproduct} \[\Delta(x_1\cdots x_k)=\sum_{i+j=k}\sum_{\Sigma_k/\Sigma_i\times\Sigma_j}\epsilon(\sigma; x_1,\ldots, x_k)\, x_{\sigma(1)}\cdots x_{\sigma(i)}\otimes x_{\sigma(i+1)}\cdots x_{\sigma(k)},\] where $\epsilon(-;x_1,\ldots, x_k):\Sigma_k\to \{\pm 1\}$ is the group homomorphism determined by the formula  $\epsilon(\tau_j;x_1,\ldots, x_k)=(-1)^{|x_j||x_{j+1}|}$. 

Recall that a \emph{coderivation} of a graded coalgebra $(C,\Delta)$ is a self-map $\delta$ satisfying the ``co-Leibniz rule'' $\Delta\circ\delta=(\id\otimes \delta+\delta\otimes \id)\circ\Delta$. Coderivations of $\Sym(V)$ of fixed degree $m$ decreasing word length by $n$ are in bijection with graded maps \[\Sym^{n+1}(V)\to V\] of degree $m$ \cite[22(a)]{FelixHalperinThomas:RHT}.
\end{recollection}

From this universal property and graded antisymmetry, we conclude that the bracket of a Lie algebra $\mathfrak{g}$ determines a coderivation $d_{[-,-]}=d$ of $\Sym\left(\mathfrak{g}[1]\right)$ of degree $-1$. Moreover, since the square of an odd coderivation is again a coderivation, the Jacobi identity implies that $d^2$ is the unique coderivation determined by 0, which, of course, is zero. Thus, the following definition is sensible.

\begin{definition}
Let $\mathfrak{g}$ be a graded Lie algebra. The \emph{Chevalley-Eilenberg complex} of $\mathfrak{g}$ is the chain complex \[\mathrm{CE}(\mathfrak{g}):=\left(\Sym(\mathfrak{g}[1]),\,d_{[-,-]}\right).\] The \emph{Lie algebra homology} of $\mathfrak{g}$ is \[H^\mathcal{L}(\mathfrak{g}):=H\left(\mathrm{CE}(\mathfrak{g})\right).\]
\end{definition}

\begin{remark}
The formula given above for the coproduct $\Delta$ on $\Sym(V)$ is determined by requiring that \begin{enumerate}
\item $\Delta(x)=1\otimes x+x\otimes 1$ and
\item $\Delta(xy)=\Delta(x)\Delta(y)$,
\end{enumerate}
or, in other words, by requiring that the elements of $V$ be primitive and that $\Delta$ be a map of algebras. Note, however, that the Chevalley-Eilenberg complex is a differential graded coalgebra but \emph{not} a differential graded algebra.
\end{remark}

\begin{remark}
The differential in the Chevalley-Eilenberg complex may be computed explicitly as \[d(x_1\cdots x_k)=\sum_{1\leq i<j\leq k}(-1)^{|x_i|} \epsilon(\sigma_{ij};x_1,\ldots, x_k)\, [x_i,x_j]\,x_1\cdots \hat x_i\cdots \hat x_j\cdots x_k,\] where $\sigma_{ij}$ is the unique $(2,k-2)$-shuffle sending $i$ to $1$ and $j$ to $2$.
\end{remark}

\begin{remark}
As the terminology suggests, there is an isomorphism \[H^\mathcal{L}(\mathfrak{g})\cong \mathrm{Tor}_*^{U(\mathfrak{g})}(\mathbb{F},\mathbb{F}),\] where $U(\mathfrak{g})$ is the universal enveloping algebra of $\mathfrak{g}$.
\end{remark}

\begin{remark}
If $\mathfrak{g}$ is the Lie algebra associated to a compact Lie group, we have a composite quasi-isomorphism \[\mathrm{CE}(\mathfrak{g})^\vee\xrightarrow{\simeq}\Omega(G)^G\xrightarrow{\sim}\Omega(G),\] where $\Omega(G)^G$ is the space of left-invariant differential forms on $G$ \cite{ChevalleyEilenberg:CTLGLA}.
\end{remark}

\subsection{Lie algebra homology and configuration spaces}

The relevance of Lie algebra homology from our point of view is the following result.

\begin{theorem}[{\cite{Knudsen:BNSCSVFH}}]
Let $M$ be an orientable $n$-manifold. There is an isomorphism \[\bigoplus_{k\geq0}H_*(B_k(M))\cong H^\mathcal{L}\left(\mathfrak{g}_M\right)\] of bigraded vector spaces, where $\mathfrak{g}_M=H_c^{-*}(M)\otimes \mathcal{L}(v_{n-1,1})$.
\end{theorem}

\begin{remark}
Explicitly, we have
\[\mathfrak{g}_M=\begin{cases}
H_c^{-*}(M)\otimes v&\quad n \text{ odd}\\
H_c^{-*}(M)\otimes v\oplus H_c^{-*}(M)\otimes[v,v]&\quad n\text{ even,}
\end{cases}\] where in the first case all brackets vanish and in the second the bracket is determined up to sign by the cup product.
\end{remark}

\begin{remark}
An analogous statement for nonorientable manifolds is also valid---see \cite{Knudsen:BNSCSVFH}.
\end{remark}

\begin{example}
If $n$ is odd, then the Lie bracket in $\mathfrak{g}_M$ is identically zero, so the differential in $\mathrm{CE}(\mathfrak{g}_M)$ vanishes. Thus, equating auxiliary gradings, we find in this case that \begin{align*}
H_*(B_k(M))\cong\Sym^k\left(H_c^{-*}(M)[n-1][1]\right)\cong \Sym^k\left(H_*(M)\right)
\end{align*} by Poincar\'{e} duality, recovering the computation of Theorem \ref{thm:odd homology}.
\end{example}

\begin{example}
Let $M_1=T^2\setminus \pt$ and $M_2=\mathbb{R}^2\setminus S^0$. Then, since $M_1^+\cong T^2$ and $M_2^+\simeq S^1\vee S^1\vee S^2$, we have \[H_c^{-*}(M_j)\cong\begin{cases}
\mathbb{F}\left\langle x_{-1}, y_{-1}, z_{-2}\mid xy=z\right\rangle&\quad j=1\\
\mathbb{F}\left\langle x_{-1}, y_{-1}, z_{-2}\right\rangle&\quad j=2.
\end{cases}\] Thus, as a bigraded vector space, we have \[\mathfrak{g}_{M_j}\cong\mathbb{F}\left\langle x_{0,1}, y_{0,1}, z_{-1,1}, \tilde x_{1,2}, \tilde y_{1,2}, \tilde z_{0,2}\right\rangle,\] and the bracket is determined by \[[x,y]=\begin{cases}
\tilde z&\quad j=1\\
0&\quad j=2.
\end{cases}\] It follows that the weight 2 subcomplex of the Chevalley-Eilenberg complex is given additively as \[\xymatrix{\mathbb{F}\left\langle \tilde x, \tilde y, xy\right\rangle \ar[rr]^-{xy\mapsto [x,y]}&& \mathbb{F}\left\langle xz, yz, \tilde z\right\rangle\ar[r]^-{0}& \mathbb{F}\left\langle z^2\right\rangle.}\] It follows that \[\dim H_2(B_2(M_j))=\begin{cases}
2&\quad j=1\\
3&\quad j=2.
\end{cases}\] In particular, $B_2(T^2\setminus \pt)\not\simeq B_2(\mathbb{R}^2\setminus S^0)$.
\end{example}

\begin{remark}
This type of computation can be pushed much further; indeed, \cite{DrummondColeKnudsen:BNCSS} determines $H_i(B_k(\Sigma_g))$ for every degree $i$, cardinality $k$, and genus $g$. For example, \[\dim H_{101}(B_{100}(\Sigma_3))=28,449,499.\]
\end{remark}

\subsection{Homological stability}
In the previous section, we saw that the rational homology of unordered configuration spaces may be computed as Lie algebra homology. Our present goal is to leverage this information in order to understand some of the qualitative behavior of these homology groups. To begin, recall that, if $\partial M\neq \varnothing$, there is a stabilization map \[B_k(M)\to B_{k+1}(M)\] defined by inserting a point in a collar neighborhood of the boundary, and this map induces an isomorphism on integral homology \cite{McDuff:CSPNP}. Surprisingly, although the stabilization map may fail to exist, stability actually holds in general, at least rationally.

\begin{theorem}[Church]
Let $M$ be a connected $n$-manifold with $n>2$. There is a map \[H_i(B_{k+1}(M);\mathbb{Q})\to H_i(B_k(M);\mathbb{Q})\] that is an isomorphism for $i\leq k$.
\end{theorem}

Although our approach to this result will differ from \cite{Church:HSCSM}, it will be motivationally useful to recall the definition of the map used therein. The idea is that, although one does not have a stabilization map in general, there is always a map going the other direction, at least at the level of ordered configurations, namely the projection of the Fadell-Neuwirth fibration $\Conf_{k+1}(M)\to \Conf_k(M)$. This map is only $\Sigma_k$-equivariant and so fails to descend to the unordered configuration spaces, but this can be remedied by remembering that there are in fact $k+1$ different such projections. The desired map is then obtained as the dashed filler in the commuting diagram \[\xymatrix{
H_*(\Conf_{k+1}(M))\ar[d]\ar[r]&H_*(B_{k+1}(M))\ar@{-->}[dd]^-{\pi}\\
H_*(\Conf_k(M))^{\oplus k+1}\ar[d]_-\Sigma\\
H_*(\Conf_k(M))\ar[r]&H_*(B_k(M)),
}\] where the $i$th coordinate of the upper left map is the projection away from the $i$th coordinate.

The key to our argument is the observation that this map is a piece of a larger structure. Indeed, whenever $i+j=k$, we have a $\Sigma\times\Sigma_j$-equivariant map \begin{align*}
\Conf_k(M)&\to \Conf_i(M)\times\Conf_j(M)\\
(x_1,\ldots, x_k)&\mapsto \left(\left(x_1,\ldots, x_i\right), \left(x_{i+1},\ldots, x_k\right)\right),
\end{align*} which induces the commuting diagram \[\xymatrix{
H_*(\Conf_k(M))\ar[d]\ar[r]&H_*(\Conf_i(M))\otimes H_*(\Conf_j(M))\otimes_{\Sigma_i\times\Sigma_j}\Sigma_k\ar[d]\\
H_*(B_k(M))\ar@{-->}[r]&H_*(B_i(M))\otimes H_*(B_j(M)),
}\] which, after summing over $k$ and $i+j=k$, produces a map \[\Delta:H_*(B(M))\to H_*(B(M))\otimes H_*(B(M)),\] where $B(M):=\coprod_{k\geq0}B_k(M)$. This map is the comultiplication of a cocommutative coalgebra structure on $H_*(B(M))$. Morally speaking, this comultiplication is given by the formula \[\text{``}\Delta(x_1,\ldots, x_k)= \sum_{i+j=k}\sum_{\Sigma_k/\Sigma_i\times\Sigma_j}(x_{\sigma(1)},\ldots, x_{\sigma(i)})\otimes (x_{\sigma(i+1)}, \ldots, x_{\sigma(k)})\text{''}\]

From this point of view, Church's map $\pi$ is obtained by using the comultiplication to split points apart and then discarding all summands not of the form $i=1$ and $j=k-1$. This type of operation is familiar in the theory of coalgebras.

\begin{definition}
Let $(C,\Delta)$ be a differential graded coalgebra and $\lambda\in C^\vee$ an $r$-cocycle. The \emph{cap product} with $\lambda$ is the composite \[\lambda\frown(-):C\cong\mathbb{Q}\otimes C\xrightarrow{\lambda\otimes\Delta}C^\vee[r]\otimes C\otimes C\cong C^\vee\otimes C\otimes C[r]\xrightarrow{\langle-,-\rangle,\otimes \id_C}\mathbb{Q}[r]\otimes C\cong C[r].\]
\end{definition}

Explicitly, if $\Delta(c)=\sum_{i} c_i\otimes c_i'$, then \[\lambda\frown c=\sum_i\langle \lambda, c_i\rangle c_i'.\] Since $\lambda$ is assumed to be closed, $\lambda\frown(-)$ is a chain map and hence induces a map of the same name at the level of homology.
	
The moral of our discussion so far is that Church's map $\pi$ is given by the cap product with the unit in $H^0(M)$, which is dual to the homology class of a single point, since $M$ is connected.

Now, as we saw last time, $H_*(B(M))$ may be computed using the Chevalley-Eilenberg complex of the Lie algebra $\mathfrak{g}_M$, which is a cocommutative coalgebra whose comultiplication obeys a very similar formula to that given above. This observation leads us to formulate the following version of Church's theorem, which is the one that we will prove. Let $x\in H_c^{-*}(M)[n]\subseteq \mathfrak{g}_M[1]\subseteq \mathrm{CE}(\mathfrak{g}_M)$ denote the Poincar\'{e} dual of the class of a point in $M$, and let $\lambda$ be the dual functional to $x$. Finally, write $\mathrm{CE}(\mathfrak{g}_M)_k$ for the summand of the Chevalley-Eilenberg complex of weight $k$.

\begin{theorem}\label{thm:my stability}
Let $M$ be a connected $n$-manifold with $n>2$. Cap product with $\lambda$ induces an isomorphism \[H_i(\mathrm{CE}(\mathfrak{g}_M)_{k+1})\xrightarrow{\simeq} H_i(\mathrm{CE}(\mathfrak{g}_M)_{k})\] for $i\leq k$.
\end{theorem}

\begin{remark}
The same conclusion holds in the case $n=2$ with a slightly worse stable range. The theorem is false for trivial reasons for $n=1$ and vacuous for $n=0$.
\end{remark}

In fact, we will show that the chain level cap product is a chain isomorphism in a range. In order to do so, it will be useful to have a formula for the cap product.

\begin{lemma}
$\lambda\frown(-)=\frac{d}{dx}$.
\end{lemma}
\begin{proof}
The claim is equivalent to the claim that $\lambda\frown x^ry=rx^{r-1}y$ for every monomial $y$ not divisible by $x$. We compute that \begin{align*}
\Delta(x^ry)&=\Delta(x)^r\Delta(y)\\
&=(1\otimes x+x\otimes 1)^r\sum_j y_j\otimes y_j'\\
&=\sum_{i,j}\binom{r}{i}x^i y_j\otimes x^{r-i}y_j',
\end{align*} whence \[\lambda\frown x^ry=\sum_{i,j}\binom{r}{i}\langle \lambda, x^i y_j\rangle x^{r-i}y_j'\]Since $\lambda$ is the dual functional to $x$, the $(i,j)$ term of this term vanishes unless $i=1$ and $y_j$ is a scalar. 

There are now two cases. If $y$ itself is a scalar, then $\Delta(y)=y\otimes y$, and the claim follows easily. If $y$ is not a scalar, then $\Delta(y)\equiv 1\otimes y+y\otimes1$ modulo elementary tensors in which neither factor is a scalar, and the claim again follows easily.
\end{proof}

\begin{corollary}
The chain map $\lambda\frown(-)$ is surjective with kernel spanned by the monomials not divisible by $x$.
\end{corollary}

Thus, it remains to determine which monomials are divisible by $x$. For simplicity, we assume $n>2$, but 

\begin{lemma}
If $n>2$, then any nonzero monomial $y$ with $\mathrm{wt}(y)>|y|$ is divisible by $x$.
\end{lemma}
\begin{proof}
Write $y=y_1\cdots y_r$ with $y_j\in \mathfrak{g}_M[1]$. Since $\mathrm{wt}(y)>|y|$, we conclude that $\mathrm{wt}(y_j)>|y_j|$ for some $j$. Since $y_j\in \mathfrak{g}_M[1]$, the weight of $y_j$ is either 1 or 2, and we treat these cases separately.

If $\mathrm{wt}(y_j)=1$, then $y_j\in H_c^{-*}(M)[n]$, which is concentrated in degrees $0\leq *\leq n$, so the assumption $|y_j|<1$ implies that $y_j=0$. Since $M$ is connected, it follows that $y_j$ is a multiple of $x$.

If $\mathrm{wt}(y_j)+2$, then $y_j\in H_c^{-*}(M)[2n-1]$, which is concentrated in degrees $n-1\leq *\leq 2n-1$. Since $|y_j|<2$ and $n>2$, it follows that $y_j=0$, a contradiction.
\end{proof}

\begin{proof}[Proof of Theorem \ref{thm:my stability}]
What we have shown so far is that the chain map \[\mathrm{CE}(\mathfrak{g}_M)_{k+1}\to \mathrm{CE}(\mathfrak{g}_M)_k\] induced by $\lambda\frown(-)$ is surjective and an isomorphism through degree $k$. An easy exercise shows that any chain map with these properties induces an isomorphism in homology through degree $k$.
\end{proof}

\begin{appendix}

\section{Split simplicial spaces}\label{appendix:split simplicial spaces}
In this appendix, we develop a criterion guaranteeing that a degreewise weak homotopy equivalence of simplicial spaces induces a weak homotopy equivalence after geometric realization. We follow \cite{DuggerIsaksen:THAR}, but similar results may be found in \cite[11.15]{May:GILS} and \cite[A.5]{Segal:CCT}.
\subsection{Split simplicial spaces}\label{section:split}

\begin{definition}
A simplicial space $\op X$ is \emph{split} if there are subspaces $N_m(\op X)\subseteq \op X_m$ for each $m\geq0$, called the \emph{non-degenerate part} in degree $n$, such that the map \[\coprod_{[n]\twoheadrightarrow[m]}N_m(\op X)\to \op X_n\] induced by the degeneracies is a homeomorphism for every $n\geq0$.
\end{definition}

\begin{proposition}[Dugger-Isaksen]\label{prop:split criterion}
Let $f:\op X\to \op Y$ be a map between split simplicial spaces. If $f_n:\op X_n\to \op Y_n$ is a weak equivalence for every $n\geq0$, then $|f|$ is a weak equivalence.
\end{proposition}

The strategy of the proof is simple. First, we argue that $f$ induces a weak equivalence on geometric realizations of $n$-skeleta for every $n$; second, we argue that every element in homotopy of the full realization is captured by some skeleton. In order to put this plan into action, we need to have control over skeleta. 

\begin{lemma}\label{lem:split pushout}
Let $\op X$ be a split simplicial space. The diagram \[\xymatrix{
N_n(\op X)\times\partial \Delta^n\ar[r]\ar[d]&|\sk_{n-1}(\op X)|\ar[d]\\
N_n(\op X)\times\Delta^n\ar[r]&|\sk_n(\op X)|
}\] is a pushout square.
\end{lemma}
\begin{proof}
Recall that the \emph{tensor} of a space $X$ with a simplicial space $\op Z$ is the simplicial space $(X\otimes\op Z)_n=X\times \op Z_n$, with simplicial structure maps induced by those of $\op Z$, together with the identity on $X$. Since geometric realization, as a left adjoint, preserves colimits, it suffices to produce a pushout square in simplicial spaces of the form
\[\xymatrix{
N_n(\op X)\otimes\partial \Delta^n\ar[r]\ar[d]&\sk_{n-1}(\op X)\ar[d]\\
N_n(\op X)\otimes\Delta^n\ar[r]&\sk_n(\op X),
}\] where we have indulged in the traditional abuse of using the same notation $\Delta^n$ for the representable simplicial set $\Hom_\Delta(-,[n])$ and its geometric realization, and similarly for $\partial\Delta^n$. To verify that this diagram is a pushout, it suffices to check in each level. Now, it is direct from the definitions that \[\sk_n(\op X)_m=\sk_{n-1}(\op X)_m\amalg\left(\coprod_{[m]\twoheadrightarrow[n]}N_n(\op X)\right),\] so $\sk_n(\op X)_m$ is the pushout of $\sk_{n-1}(\op X)_m$ and $N_n(\op X)\times \Delta^n_m$ over a coproduct of copies of $N_n(\op X)$ indexed by the set of maps $f:[m]\to [n]$ that fail to be surjective, which is exactly $\partial\Delta^n_m$.
\end{proof}

This fact will only be useful once we are assured that such pushouts are homotopically well-behaved. With regularity assumptions on the spaces involved, the following type of result is common knowledge, but in fact it holds in complete generality.

\begin{lemma}\label{lem:pushout invariant}
If $f:A\to A'$ and $g:B\to B'$ are weak homotopy equivalences, and if the front and back faces in the commuting diagram
\[\xymatrix{&A\times \partial\Delta^n\ar[rr]\ar[dd]_>>>>>>>{}|!{[d]}\hole\ar[dl]_-{f\times\id_{\partial\Delta^n}}&&B\ar[dd]\ar[dl]_-g\\
A'\times\partial\Delta^n\ar[rr]\ar[dd]&& B'\ar[dd]\\
&A\times\Delta^n\ar[dl]_{f\times\id_{\partial\Delta^n}}\ar[rr]^<<<<<<<<{}|!{[r]}\hole&&C\ar[dl]^-h
\\
A'\times\Delta^n\ar[rr]&&C'
}\] are pushout squares, then $h:C\to C'$ is a weak homotopy equivalence.
\end{lemma}
\begin{proof}
We cover $C$ by two open sets, the first being $U_1\times A\times D$, where $D\subseteq\mathring{\Delta}^n$ is a Euclidean neighborhood of the barycenter, and the second $U_2=B\coprod_{A\times\partial\Delta^n}(A\times P)$, where $P\subseteq \Delta^n$ is the complement of the barycenter. Similarly, we cover $C'$ by $U_1'$ and $U_2'$. Clearly, $h^{-1}(U_j')=U_j$ for $j\in\{0,1\}$. Consider the commuting diagrams \[\xymatrix{
U_1\ar[d]_-{h|_{U_1}}\ar[r]&A\ar[d]^-f&& U_2\ar[d]_-{h|_{U_2}}&B\ar[l]\ar[d]^-g&& U_1\cap U_2\ar[d]_-{h|_{U_1\cap U_2}}&A\times(D\cap P)\ar@{=}[l]_-{\simeq}\ar[d]^-{f\times \id_{D\cap P}}\\
U_1'\ar[r]&A'&& U_2'&\ar[l]B'&&U_1'\cap U_2'\ar@{=}[r]^-{\simeq}&A'\times(D\cap P),
}\] where the horizontal arrows in the leftmost diagram are the projections onto the first factor, and the horizontal arrows in the middle idagram are induced by the inclusion $\partial\Delta^n\subseteq P$. Both horizontal arrows in the leftmost diagram are homotopy equivalences, and $f$ is a weak homotopy equivalence by assumption; both horizontal arrows in the middle diagram are inclusions of deformation retracts, and $g$ is a weak homotopy equivalence by assumption; and $f\times\id_{D\cap P}$ is a weak homotopy equivalence by assumption. Thus, by two-out-of-three, all three restrictions of $h$ are weak homotopy equivalences, so $h$ itself is a weak homotopy equivalence.
\end{proof}

In verifying that elements in the homotopy groups of $|\op X|$ are all captured by skeleta, we must be assured that the inclusions among skeleta are not too pathological. This assurance takes the form of a relative separation axiom.

\begin{definition}
A subspace $A\subseteq B$ is \emph{relatively $T_1$} if any open set $U\subseteq A$ may be separated from any point $b\in B\setminus U$ by an open set $U\subseteq V\subseteq B$. An inclusion map is \emph{relatively $T_1$} if its image is so.
\end{definition}

This terminology is motivated by the observation that a space is $T_1$ if and only if each of its points is relatively $T_1$. Since finite intersections of open sets are open, we have the following immediate consequence:

\begin{lemma}
If $A\subseteq B$ is relatively $T_1$, then any open set $U\subseteq A$ may be separated from any finite subset of $B\setminus U$ by an open set $U\subseteq V\subseteq B$. 
\end{lemma}

The importance of this notion for our purposes is the following result.

\begin{proposition}\label{prop:factor through colimit}
Let $Y_i\subseteq Y_{i+1}$ be a relatively $T_1$ inclusion for $i\geq 1$. If $K$ is compact, then any map $f:K\to \colim_\mathbb{N} Y_i$ factors through the inclusion of some $Y_i$. 
\end{proposition}
\begin{proof}
If $f$ does not factor as claimed, then, without loss of generality, we may assume the existence of $x_i\in \mathrm{im}(f)\cap Y_i$ for each $i\geq1$. Recall that a subset of the colimit is open precisely when its intersection with each stage is open; thus, for each $j\geq1$, we may define an open subset $U_j\subseteq\colim_\mathbb{N} Y_i$ by the following prescription:
\begin{enumerate}
\item for $1\leq i<j$, set $U_{ij}=\varnothing$;
\item for $i=j$, set $U_{ij}=Y_j$;
\item for $i>j$, take $U_{ij}$ to be an open subset of $Y_i$ separating $U_{i-1,j}$ from the set $\{x_{j+1},\ldots, x_i\}$;
\item finally, set $U_j=\colim_\mathbb{N}U_{ij}$.
\end{enumerate}
Then $U_j\cap Y_i=U_{ij}$, so $U_j$ is an open subset the colimit, and, since $Y_j\subseteq U_{j}$, the collection $\{U_j\}_{j\in\mathbb{N}}$ is an open cover of $\colim_\mathbb{N}Y_i$. Since $K$ is compact, $\mathrm{im}(f)$ is compact, so it is contained in some finite subcover $\{U_{j_r}\}_{r=1}^N$. But, by construction, $U_{j_r}$ does not contain $x_i$ for $i>j_r$, so $\bigcup_{r=1}^N U_{j_r}$ does not contain $x_i$ for $i>\max\{j_r:1\leq r\leq N\}$, a contradiction.
\end{proof}

This fact will only be useful once we are able to identify relatively $T_1$ maps, a task that is made easier by the following observation.

\begin{lemma}\label{lem:stable under pushout}
Relatively $T_1$ inclusions are stable under finite products and pushouts along arbitrary continuous maps.
\end{lemma}
\begin{proof}
For the first claim, it suffices by induction to show that $A_1\times A_2\subseteq B_1\times B_2$ is relatively $T_1$ if each $A_j\subseteq B_j$ is so. Fix an open subset $U\subseteq A_1\times A_2$ and a point $(x_1,x_2)\in B_1\times B_2\setminus U$. By the definition of the product topology, we have $U=\bigcup_{i\in I}U_{i1}\times U_{i2}$ for open subsets $U_{ij}\subseteq A_j$. By our assumption on the inclusions of the $A_j$, we may find open subsets $U_{ij}\subseteq W_{ij}\subseteq B_j$ for each $i\in I$ such that $x_j\notin W_{ij}$. Then $U\subseteq W:=\bigcup_{i\in I}W_{i1}\times W_{i2}$ is open in $B_1\times B_2$, and $(x_1,x_2)\notin W$, as desired.

For the second claim, suppose that the diagram \[\xymatrix{
A\ar[r]^-f\ar[d]_-i&Y\ar[d]\\
B\ar[r]^-g&Z
}\] is a pushout square and that $i$ is a relatively $T_1$ inclusion. Fix an open subset $U\subseteq Y$ and a point $z\in Z\setminus U$ (here, in a small abuse, we identity $Y$ with its image in $Z$, since the pushout of an inclusion is an inclusion). There are two cases to consider. 

Assume first that $z\in Y$. Since $f^{-1}(U)$ is open in $A$ and $i$ is an inclusion, there is an open subset $W\subseteq B$ with $W\cap A=f^{-1}(U)$, and $W\amalg_{f^{-1}(U)} U\subseteq Z$ is open. To see that $z$ is not contained in this subset, it suffices to show that $z\notin g(W)$, since $z\notin U$ by assumption. But $z\in Y$, and $Y\cap g(B)=f(A)$, so $Y\cap g(W)=f(W\cap A)=U$, and the claim follows. 

On the other hand, suppose that $z\notin Y$; in particular, $z=g(b)$ for a unique $b\in B$. Then $b\notin i(f^{-1}(U))$ and $f^{-1}(U)\subseteq A$ is open, so, since $i$ is relatively $T_1$, there is an open subset $i(f^{-1}(U))\subseteq W\subseteq B$ with $b\notin W$. As before, $W\coprod_{f^{-1}(U)} U$ is open in $Z$ and clearly does not contain $z$. 
\end{proof}

\begin{corollary}\label{cor:pushout is T1}
For any pushout diagram of the form\[\xymatrix{
A\times \partial \Delta^n\ar[r]\ar[d]_-{\id_A\times(\partial\Delta^n\subseteq\Delta^n)}&Y\ar[d]^-i\\
A\times \Delta^n\ar[r]& Z
}\] the inclusion $Y\to Z$ is relatively $T_1$.
\end{corollary}

Finally, we will need the following, essentially obvious, observation.

\begin{lemma}\label{lem:disjoint weak equivalence}
If $f:W\amalg X\to Y\amalg Z$ is a weak homotopy equivalence such that $f|_W$ factors through $Y$ as a weak homotopy equivalence, then $f|_X$ factors through $Z$ as a weak homotopy equivalence.
\end{lemma}
\begin{proof}
The claim that $f|_X$ factors through $Z$ is obvious after applying $\pi_0$ and considering the analogous claim for bijections of sets, since $\pi_0(f)$ is a bijection. The claim that this factorization is a weak homotopy equivalence follows in the same way after applying $\pi_n$.
\end{proof}

\begin{proof}[Proof of Proposition \ref{prop:split criterion}]
Fix $N_n(\op X)\subseteq \op X_n$ and $N_n(\op Y)\subseteq\op Y$ witnessing $\op X$ and $\op Y$ as split. We claim that the restriction of $f_n$ to $N_n(\op X)$ factors through $N_n(\op Y)$ as a weak homotopy equivalence for every $n\geq0$. Having established this, it will follow by induction and Lemmas \ref{lem:split pushout} and \ref{lem:pushout invariant} above that the induced map $|\sk_n(\op X)|\to |\sk_n(\op Y)|$ is a weak homotopy equivalence for every $n\geq0$. 

To establish the claim, we proceed by induction on $n$, the base case following from our assumption, since $N_0(\op X)=\op X_0$ and similarly for $\op Y$. For the induction step, we note that the inductive assumption implies that the dashed filler in the diagram \[\xymatrix{
N_m(\op X)\ar@{-->}[d]\ar[r]&\op X_m\ar[d]^-{f_m}\ar[r]^-{s}&\op X_n\ar[d]^-{f_n}\\
N_m(\op Y)\ar[r]&\op Y_m\ar[r]^{s}&\op Y_n
}\] exists and is a weak homotopy equivalence for every $m<n$ and every degeneracy $s$. Thus, the dashed filler in the diagram \[\xymatrix{
\displaystyle\coprod_{[n]\twoheadrightarrow[m]\neq[n]}N_m(\op X)\ar[r]\ar@{-->}[d]&\op X_n\ar[d]^-{f_n}\\
\displaystyle\coprod_{[n]\twoheadrightarrow[m]\neq[n]}N_m(\op Y)\ar[r]&\op Y_n
}\] exists and is a weak homotopy equivalence. Since the righthand map is also a weak homotopy equivalence, the claim follows from Lemma \ref{lem:disjoint weak equivalence}.

Now, from Lemma \ref{lem:split pushout} and Corollary \ref{cor:pushout is T1}, it follows by induction that each of the inclusions $|\sk_n(\op X)|\to |\sk_{n+1}(\op X)|$ is relatively $T_1$, and similarly for $\op Y$. Since $|\op X|=\colim_\mathbb{N}|\sk_n(\op X)|$, and likewise for $\op Y$, we conclude from Proposition \ref{prop:factor through colimit} that any map $(D^{m}, S^{m-1})\to (|\op Y|,|\op X|)$ factors as in the solid commuting diagram \[\xymatrix{
S^{m-1}\ar[d]\ar[r]&|\sk_n(\op X)|\ar[d]\ar[r]&|\op X|\ar[d]^-{|f|}\ar[d]\\
D^m\ar@{-->}[ur]\ar[r]&|\sk_n(\op Y)|\ar[r]&|\op Y|.
}\] We have already shown the middle arrow to be a weak homotopy equivalence, so the dashed filler exists making the upper triangle commute and the lower triangle commute up to homotopy relative to $S^{m-1}$. Thus, $\pi_m(f)=0$ for every $m\geq0$, and the claim follows.
\end{proof}

\subsection{Some examples}\label{section:loose ends}

In this section, we identify a few classes of split simplicial spaces used in the main text. We begin by noting two easy consequences of the simplicial identities. First, for any simplicial space $\op X$, each degeneracy $s_i:\op X_{n-1}\to \op X_n$ is injective. Second, the intersection $s_i(\op X_{n-1})\cap s_j(\op X_{n-1})$ is contained in the union of the images of the various iterated degeneracies $\op X_{n-2}\to \op X_n$.

\begin{lemma}\label{lem:covering map split}
Let $f:\op X\to \op Y$ be a degreewise covering map. If $\op Y$ is split, then $\op X$ is split.
\end{lemma}
\begin{proof}
Setting $N_0(\op X)=\op X_0$, assume for $m<n$ that $N_m(\op X)\subseteq \op X_m$ has been constructed with the desired property. The degeneracies of $\op X$ induce a map \[s:\coprod_{\pi:[n]\twoheadrightarrow[m]\neq [n]}N_m(\op X)\to \op X_n,\] and the observations above imply that $s$ is injective. We claim that $s$ is a local homeomorphism and hence the inclusion of a collection of connected components, which is enough to imply the claim, for in this case we may take $N_n(\op X)=\op X_n\setminus\mathrm{im}(s)$. To see that $s$ is a local homeomorphism, we may restrict our attention to the component indexed by $\pi$, in which case it suffices to show that the top map in the commuting diagram \[\xymatrix{
\op X_n\ar[d]_-{f_n}& \op X_m\ar[l]_-{\pi^*}\ar[d]^-{f_m}\\
\op Y_n& \op Y_m\ar[l]^-{\pi^*}
}\] is a local homeomorphism. Since $\op Y$ is split, the bottom map is a local homeomorphism, and the righthand map, as a covering map, is a local homeomorphism. Thus, by commutativity, given $x\in \op X_m$, there is a connected open neighborhood $x\in U$ such that $(f_n\circ\pi^*)|_{U}$ is a homeomorphism onto its image. Since $U$ is connected and $f_n$ is a covering map, there is a subspace $V\subseteq \op X_n$ such that $\pi^*(U)\subseteq V$ and $f_n|_V$ is a homeomorphism onto its image. It follows that $\pi^*|_{U}$ is a homeomorphism onto its image, as desired. 
\end{proof}

\begin{corollary}\label{cor:hypercovers are split}
If $\op X$ is a hypercover, then $\op X$ is split.
\end{corollary}
\begin{proof}
Since the condition of being split is a condition that, in each degree, involves only a truncation of the simplicial space in question, and since a hypercover $\op X$ coincides with some bounded hypercover through any simplicial degree, it suffices to assume that $\op X$ is bounded of height at most $N$. We proceed by induction on $N$, the base case being the observation that \v{C}ech covers are split, with degree $n$ non-degenerate part given by the union of the $(n+1)$-fold intersections of cover elements in which no two adjacent indices coincide. The inductive step follows from Lemma \ref{lem:covering map split}, since $\op X\to \cosk_{N-1}(\op X)$ is a degreewise covering map with split target by Lemma \ref{lem:coskeleta} and induction.
\end{proof}

We also have the following simple observation.

\begin{lemma}\label{lem:split levelwise}
If $f:\op X\to \op Y$ is a levelwise weak equivalence and $\op Y$ is split, then $\op X$ is split.
\end{lemma}

\begin{corollary}\label{cor:horizontal cech cover split}
Let $f:\op X\to \op Y$ be a degreewise covering map and $\op W$ the bisimplicial space obtained by forming the degreewise \v{C}ech nerve. If $\op Y$ is split, then $|\op W|_h$ is split.
\end{corollary}

Finally, we have the following observation, which is immediate from the fact that the singular functor and geometric realization each preserve coproducts.

\begin{lemma}\label{lem:realization split}
If $\op X$ is split, then the simplicial space $|\mathrm{Sing}(\op X)|$, obtained by applying the singular functor and geometric realization in each degree, is split.
\end{lemma}

\section{Homotopy colimits}\label{section:homotopy colimits}

\subsection{Bar construction}
Recall that, if $M$ and $N$ are differential graded modules over a ring $R$, then the homology of the relative tensor product $M\otimes_R N$ may fail to be invariant under quasi-isomorphisms in the factors. This defect can be eliminated by resolving $M$, say, by nice $R$-modules before computing the tensor product, and a canonical choice of such a resolution with free entries is given by the \emph{bar complex} \[\cdots\to M\otimes R^{\otimes n+1}\to \cdots\to M\otimes R\to M,\] where the differentials are defined as alternating sums of maps defined in terms of the multiplication in $R$ and the module structure of $M$. 

In order to apply this intuition in more general contexts, we note that the bar complex arises from a simplicial $R$-module whose construction depended only on the existence of a free/forgetful adjunction relating differential graded $R$-modules to chain complexes. The corresponding simplicial \emph{monadic bar construction} is available whenever such adjunction data is present.

In our context, the role of the $R$-module $M$ is played by a functor $F:I\to \Top$. Since a ring is roughly like a category with one object, we guess that the corresponding forgetful functor should be the functor that only remembers the values of a functor on objects. In other words, writing $I_0$ for the discrete category with the same objects as $I$, we should consider the forgetful functor \[\iota^*:\Top^I\to \Top^{I_0}\] given by restriction along the inclusion $\iota:I_0\to I$. This functor does admit a left adjoint $\iota_!$, given by left Kan extension. The standard formula for the left Kan extension gives us \begin{align*}
\iota_!\iota^*F(i)=\colim\left((\iota\downarrow i)\to I_0\xrightarrow{\iota} I\xrightarrow{F} \Top\right)\cong \coprod_{i_1\to i}F(i_1),
\end{align*} and, more generally, \[(\iota_!\iota^*)^{n+1}F(i)\cong\coprod_{i_{n+1}\to \cdots \to i_1\to i}F(i_{n+1}).\] We obtain in this way a simplicial functor, whose geometric realization maps to $F$, which we wish to think of as a kind of resolution of $F$. Our guess, then, is that the homotopy colimit should be the colimit of this geometric realization. Since colimits commute with geometric realization, and since \[\colim_I(\iota_!\iota^*)^{n+1} F\cong \colim_{I_0}\iota^*(\iota_!\iota^*)^nF\cong \coprod_{i_0\in I}(\iota_!\iota^*)^nF(i_0)\cong \coprod_{i_n\to \cdots \to i_0}F(i_n),\] this guess may be rephrased in terms of the following object.

\begin{definition}
Let $F:I\to \Top$ be a functor. The \emph{(simplicial) bar construction} on $F$ is the simplicial space $\mathrm{Bar}_\bullet(F)$ with \[\mathrm{Bar}_n(F)=\coprod_{i_n\to \cdots \to i_0}F(i_n)\] and with face and degeneracy maps given by the composition and identities in $I$, respectively.
\end{definition}

\begin{hypothesis}
Let $F:I\to \Top$ be a functor. The homotopy colimit of $F$ is the space \[\hocolim_I F=|\mathrm{Bar}_\bullet(F)|.\]
\end{hypothesis}

The first check of this hypothesis is to verify the following.

\begin{proposition}
Let $F,G:I\to \Top$ be functors and $\varphi:F\to G$ a natural transformation. If $\varphi(i)$ is a weak homotopy equivalence for every $i\in I$, then $|\mathrm{Bar}_\bullet(\varphi)|$ is also a weak homotopy equivalence. 
\end{proposition}
\begin{proof}
It is easy to see that the simplicial space $\mathrm{Bar}_\bullet(F)$ is split, so the claim follows from Proposition \ref{prop:split criterion}.
\end{proof}

Since there is a natural map $|\mathrm{Bar}_\bullet(F)|\to \colim_I F$ supplied by the observation that \[\colim_IF\cong \mathrm{coeq}\left(\coprod_{i_1\to i_0}F(i_1)\rightrightarrows \coprod_{i_0}F(i_0)\right),\] we may summarize our progress so far as having exhibited one homotopy invariant approximation to the colimit. On the other hand, another homotopy invariant approximation to the colimit is the constant functor with value $\varnothing$! Why should we think that our proposed construction of the homotopy colimit is any better than this functor? 

\subsection{Derived functors} We now formulate precisely what it means to be the best approximation by a homotopy invariant functor.

\begin{definition}
A \emph{category with weak equivalences} is a pair $(\op C, \weq(\op C))$ of a category and a collection of morphisms that contains the isomorphisms and has the property that if, in the commuting diagram \[\xymatrix{
A\ar[dr]\ar[rr]&&B\ar[dl]\\
&B
}\] in $\op C$, any two arrows lie in $\weq(\op C)$, then so does the third.
\end{definition}

The arrows in $\weq(\op C)$ are called \emph{weak equivalences}, and the closure property is called \emph{two-out-of-three}.

\begin{example} 
Two cases of interest are the category of topological spaces with weak homotopy equivalences and the category of chain complexes with quasi-isomorphisms.
\end{example}

\begin{example}
If $(\op C, \weq(\op C))$ is a category with weak equivalences and $I$ is a category, then the functor category $\op C^I$ is again a category with weak equivalences when equipped with the \emph{pointwise} weak equivalences, i.e., a natural transformation is a weak equivalence if and only if each of its components is so.
\end{example}

\begin{definition}
Let $(\op C, \weq(\op C))$ be a category with weak equivalences. A \emph{homotopy category} for $\op C$ is a category $\Ho(\op C)$ equipped with a functor $\gamma=\gamma_{\op C}:\op C\to \Ho(\op C)$ with the following properties.
\begin{enumerate}
\item If $f\in\weq(\op C)$, then $\gamma(f)$ is an isomorphism.
\item Any functor $F:\op C\to \op D$ sending weak equivalences in $\op C$ to isomorphisms in $\op D$ factors uniquely through $\gamma$.
\end{enumerate}
\end{definition}

If $\op C$ has a homotopy category, then it is unique up to a unique equivalence of categories, so there is no harm in referring to \emph{the} homotopy category of $\op C$.

Often it will be the case, as with the colimit functor, that one is given a functor $F$ that does not send weak equivalences to isomorphisms. In this case, one can ask for the best approximation to $F$ by a functor having this property.

\begin{definition}
Let $(\op C, \weq(\op C))$ be a category with weak equivalences and $F:\op C\to \op D$ a functor. A \emph{left derived functor} of $F$ is a functor $\mathbb{L}F:\Ho(\op C)\to \op D$ equipped with a natural transformation $\mathbb{L}F\circ\gamma\to F$ that is final among functors $T:\Ho(\op C)\to \op D$ equpped with natural transformations $T\circ \gamma\to F$.
\end{definition}

Dually, we have the notion of a \emph{right} derived functor $\mathbb{R}F$. Note that, categorically speaking, the left derived functor $\mathbb{L}F$ is the right Kan extension of $F$ along $\gamma$.

\begin{definition}
Let $(\op C,\weq(\op C))$ be a category with weak equivalences, and assume that $\op C$ admits colimits indexed by $I$. The \emph{homotopy colimit} functor for $I$-shaped diagrams, if it exists, is the left derived functor of the composite $\gamma\circ\colim_I$, as depicted in the following digram: \[\xymatrix{
\op C^I\ar[d]_-\gamma\ar[rr]^-{\colim_I}&&\op C\ar[d]^-\gamma\\
\Ho(\op C^I)\ar[rr]^-{\hocolim_I}\ar@{=>}[urr]&&\Ho(\op C)
}\]
\end{definition}

One says that $\hocolim_I$ is the \emph{total left derived functor} of $\colim_I$.

Thus, the homotopy colimit is the closest approximation to the colimit by a homotopy invariant construction.

\subsection{Model structures} Our next task is to make this definition into something useable in practice. Our approach, following \cite{Quillen:HA}, will be to impose extra structure on our category, motivated by the structure observed ``in the wild'' in the homotopy theory of spaces, enabling us to make these abstract notions concrete. It should be emphasized, however, that this structure is scaffolding, and that the fundamental objects of interest are all at the level of the bare category with weak equivalences.

\begin{definition}
Let $(\op C, \weq(\op C))$ be a category with weak equivalences, and assume that $\op C$ has small limits and colimits. A \emph{model structure} on $\op C$ is a pair $(\cof(\op C),\fib(\op C))$ of classes of morphisms in $\op C$ satisfying the following axioms.
\begin{enumerate}
\item Both $\cof(\op C)$ and $\fib(\op C)$ are closed under retracts in the arrow category $\op C^{\Delta^1}$.
\item If $i\in\cof(\op C)$ and $p\in \fib(\op C)$, then the dashed filler exists in the commuting diagram \[\xymatrix{
A\ar[d]_-i\ar[r]&B\ar[d]^-p\\
C\ar@{-->}[ur]\ar[r]&D
}\] provided either $i$ or $p$ is a weak equivalence.
\item Any morphism $f:A\to B$ in $\op C$ may be factored as
\begin{enumerate}
\item $f=p\circ i$ with $i\in \cof(\op C)\cap \weq(\op C)$ and $p\in \fib(\op C)$ and as
\item $f=q\circ j$ with $j\in \cof(\op C)$ and $q\in \fib(\op C)\cap \weq(\op C)$.
\end{enumerate}
\end{enumerate}
A \emph{model category} is a category with weak equivalences equipped with a model structure. The morphisms in $\cof(\op C)$, $\fib(\op C)$, $\cof(\op C)\cap \weq(\op C)$, and $\fib(\op C)\cap \weq(\op C)$ are the \emph{cofibrations}, the \emph{fibrations}, the \emph{trivial cofibrations}, and the \emph{trivial fibrations}, respectively. An object is \emph{cofibrant} if the unique morphism from the initial object of $\op C$ is a cofibration (resp. \emph{fibrant}, final, fibration).
\end{definition}

In the situation of (2), we say that $i$ has the \emph{left lifting property} with respect to $p$, and that $p$ has the \emph{right lifting property} with respect to $i$.

\begin{exercise}[{See \cite[7]{Hirschhorn:MCL}}]
Derive the following consequences of the model category axioms.
\begin{enumerate}
\item Weak equivalences are closed under retracts in the arrow category.
\item (Trivial) cofibrations are closed under coproducts and pushouts along arbitrary morphisms.
\item (Trivial) fibrations are closed under products and pullbacks along arbitrary morphisms.
\item Cofibrations are exactly those morphisms with the left lifting property with respect to every trivial fibration, and similarly for the other classes of morphisms.
\end{enumerate}
\end{exercise}

From the point of view of our motivating questions, the benefit of the presence of a model structure on a category with weak equivalences is that it allows for explicit control over the homotopy category and derived functors. 

\begin{proposition}[{\cite[8.3.5, 8.4.4]{Hirschhorn:MCL}}]
Let $\op C$ be a model category and $F:\op C\to \op D$ a functor.
\begin{enumerate}
\item The homotopy category $\Ho(\op C)$ exists.
\item If $F$ sends weak equivalences between cofibrant objects to isomorphisms in $\op D$, then $\mathbb{L} F$ exists, and the natural map $\mathbb{L}F(\gamma(C))\to \gamma(F(C))$ is an isomorphism for $C$ cofibrant (resp. fibrant, $\mathbb{R}F$).
\end{enumerate}
\end{proposition}

From what we have said so far, we draw the following strategy for working with homotopy colimits:
\begin{enumerate}
\item endow the functor category $\Top^I$ with a model structure;
\item verify that, in this model structure, the total left derived functor of $\colim_I:\Top^I\to \Top$ exists; and
\item understand cofibrant replacement in this model structure.
\end{enumerate} With these steps accomplished, we may compute the homotopy colimit by cofibrantly replacing and computing the ordinary colimit.

We will be aided in accomplishing the second step of this strategy by the following result---see \cite[8.5.3, 8.5.18]{Hirschhorn:MCL}.

\begin{theorem}[Quillen]
Let $\op C$ and $\op D$ be model categories and \[\adjunct{\op C}{\op D}{F}{G}\] an adjunction. The following two conditions are equivalent: \begin{enumerate}
\item $F(\cof(\op C))\subseteq \cof(\op D)$ and $F(\cof(\op C)\cap \weq(\op C))\subseteq \cof(\op D)\cap \weq(\op D)$; 
\item $G(\fib(\op D))\subseteq \fib(\op C)$ and $G(\fib(\op D)\cap \weq(\op D))\subseteq \fib(\op C)\cap \weq(\op C)$.
\end{enumerate} Moreover, if these conditions hold, there is an adjunction \[\adjunct{\Ho(\op C)}{\Ho(\op D).}{\mathbb{L}(\gamma_{\op D}\circ F)}{\mathbb{R}(\gamma_{\op C}\circ G)}\] In particular, both total derived functors exist.
\end{theorem}

Such an adjunction is commonly known as a \emph{Quillen adjunction} or \emph{Quillen pair}, and the functors $F$ and $G$ are left and right \emph{Quillen functors}, respectively.

Thus, if the putative model structure on $\Top^I$ is chosen to be compatible with some fixed model structure on $\Top$ in the sense that the diagonal functor $\Top\to \Top^I$ is a right Quillen functor, then the existence of the homotopy colimit will be assured.

We begin by specifying our choice of model structure on $\Top$.

\begin{theorem}[Quillen]\label{thm:quillen model structure}
The following specifications define a model structure on the category $\Top$ of topological spaces:
\begin{enumerate}
\item the weak equivalences are the weak homotopy equivalences;
\item the fibrations are the Serre fibrations;
\item the cofibrations are the retracts of relative cell complexes.
\end{enumerate}
\end{theorem}
\begin{proof}
We highlight only one aspect of the proof, namely the verification of the factorization axiom via the so-called \emph{small object argument}. For a detailed proof of the full result, see \cite{Hirschhorn:QMCTS}.

Let $f:X\to Y$ be any map, which we wish to factor as $f=p\circ i$ with $p$ a fibration and $i$ a trivial cofibration. We define $X_1$ as the pushout in the diagram \[\xymatrix{
\displaystyle\coprod\Lambda_k^n\ar[d]\ar[r]& X\ar[d]^-{i_1}\\
\displaystyle\coprod\Delta^n\ar[r]&X_1,
}\] where the coproduct in each case is indexed by the set \[\coprod_{n\geq0}\coprod_{0\leq k\leq n}\Hom(\Delta^n,Y)\times_{\Hom(\Lambda_k^n, Y)}\Hom(\Lambda_k^n, X).\] The simplicial set $\Lambda_k^n$ is the $k$th \emph{horn} of $\Delta^n$---see \cite[I.1]{GoerssJardine:SHT}. Setting $X_0=X$ and assuming $i_r:X_{r-1}\to X_r$ to have been defined, we apply the same procedure to the induced map $X_r\to Y$ to obtain $i_{r+1}:X_r\to X_{r+1}$, and we set $X_f=\colim_{r\geq0} X_r$. We claim that, in the commuting diagram \[\xymatrix{
X\ar[dr]_-f\ar[rr]^-i&& X_f\ar[dl]^-p\\
&Y,
}\] $i$ is a trivial cofibration and $p$ a fibration. The former claim involves finding a lift in the diagram \[\xymatrix{
X\ar[d]_-{i_1}\ar[d]\ar[rrr]&&&E\ar[ddd]^-q\\
X_1\ar[d]_-{i_2}\\
\vdots\ar[d]\\
X_f\ar[rrr]&&&B
}\] with $q$ a fibration, which may be accomplished inductively using the fact that each $i_r$ is a trivial cofibration, since each inclusion $\Lambda_k^n\subseteq \Delta^n$ is so. For the latter claim, we note that any map of pairs $(\Delta^n,\Lambda_i^n)\to (Y, X_f)$ factors as in the diagram \[\xymatrix{
\Lambda_i^n\ar[dd]\ar[r]&X_r\ar[d]_-{i_{r+1}}\ar[r]& X_f\ar[dd]^-p\\
&X_{r+1}\ar[ur]\ar[dr]\\
\Delta^n\ar@{-->}[ur]\ar[rr]&&Y,
}\] and the dashed filler exists by our definition of $X_{r+1}$ as the pushout.

In order to factor $f$ as a cofibration followed by a trivial fibration, we apply the same argument with horn inclusions replaced by the inclusion $\partial \Delta^n\subseteq \Delta^n$.
\end{proof}

\begin{remark}
Note that the factorization produced by the small object argument is even functorial.
\end{remark}

\subsection{Cofibrant generation} The only essential features of the category $\Top$ used in the argument of Theorem \ref{thm:quillen model structure} was the existence of a collection $S=\{\Lambda_k^n\to \Delta^n\}$ of ``test maps'' for fibrations, and similarly for trivial fibrations, with the property that the domain of an element of $S$ is \emph{small} with respect to composites of iterated pushouts of members of $S$. These features can be axiomatized.

\begin{definition}
Let $\op C$ be a category and $S$ a set of morphisms in $\op C$. A morphism in $\op C$ is 
\begin{enumerate}
\item $S$-\emph{injective} if it has the right lifting property with respect to every element of $S$;
\item an $S$-\emph{cofibration} if it has the left lifting property with respect to every $S$-injective;
\item a \emph{relative $S$-cell complex} if it is a composite of pushouts of elements of $S$.
\end{enumerate} We say that $S$ \emph{permits the small object argument} if the domains of elements of $S$ are small with respect to the collection of relative $S$-cell complexes.
\end{definition}

\begin{example}
Let $\op C=\Top$. If $S$ is the collection of inclusions $\partial\Delta^n\subseteq\Delta^n$, then the $S$-injectives are the trivial fibrations, and the $S$-cofibrations are the cofibrations. On the other hand, if $S$ is the collection of inclusions $\Lambda_k^n\subseteq\Delta^n$, then the $S$-injectives are fibrations, and the $S$-cofibrations are the trivial cofibrations. Both of these collections permit the small object argument, since in both cases every element of $S$ is a relatively $T_1$ inclusion with compact domain.
\end{example}

Motivated by this example, we think of the elements of $S$ as ``generators'' for a set of cofibrations.

\begin{theorem}[Kan]
Let $(\op C, \weq(\op C))$ be a category with weak equivalences with $\weq(\op C)$ closed under retracts, and let $S_\cof$ and $S_\triv$ be sets of morphisms in $\op C$. Assume that \begin{enumerate}
\item $S_\cof$ and $S_\triv$ each permit the small object argument,
\item every $S_\triv$-cofibration is both an $S_\cof$-cofibration and a weak equivalence,
\item every $S_\cof$-injective is both an $S_\triv$-injective and a weak equivalence, and
\item one of the following conditions holds: 
\begin{enumerate}
\item any map that is an $S_\cof$-cofibration and a weak equivalence is also an $S_\triv$-cofibration, or
\item any map that is an $S_\triv$-injective and a weak equivalence is also an $S_\cof$-injective.
\end{enumerate}
\end{enumerate}
Then $(\op C,\weq(\op C))$ extends to a model structure with cofibrations the $S_\cof$-cofibrations and fibrations the $S_\triv$-injectives.
\end{theorem}
\begin{proof}
Given $f:X\to Y$, we apply the same procedure as above to factor $f$ as $p \circ i$, where $p$ is an $S_\triv$-injective and $i$ is a relative $S_\triv$-cell complex. The latter is in particular an $S_\triv$-cofibration, so point (2) applies to show that this factorization is of the desired form. The dual argument, using point (3), furnishes the other factorization. Point (4) is used to demonstrate the lifting axiom. For a detailed account, see \cite[11.3.1]{Hirschhorn:MCL}.
\end{proof}

One can show that, in the situation of this theorem, the cofibrations are exactly the retracts of relative $S_\cof$-cell complexes, and similarly for trivial cofibrations and $S_\triv$---see \cite[10.5.22]{Hirschhorn:MCL}.

For obvious reasons, it is standard to refer to such collections $S_\cof$ and $S_\triv$ as \emph{generating} cofibrations and trivial cofibrations, respectively, and to refer to the resulting model category as \emph{cofibrantly generated}. One of the excellent features of cofibrantly generated model structures is that they are very portable.

\begin{corollary}
Let $\op C$ be a cofibrantly generated model category with generators $S_\cof$ and $S_\triv$, $\op D$ a category with small (co)limits, and \[\adjunct{\op C}{\op D}{F}{G}\] an adjunction. If \begin{enumerate}
\item $F(S_\cof)$ and $F(S_\triv)$ permit the small object argument, and
\item $G$ sends relative $F(S_\triv)$-cell complexes to weak equivalences,
\end{enumerate} then the category with weak equivalences $(\op D, G^{-1}(\weq(\op C)))$ extends to model structure cofibrantly generated by $F(S_\cof)$ and $F(S_\triv)$. Moreover, $F$ and $G$ are a Quillen pair with respect to this model structure.
\end{corollary}
\begin{proof}
We apply the previous theorem. Prerequisitely, we note that $G^{-1}(\weq(\op C))$ satisfies two-out-of-three and is closed under retracts, and point (1) is true by assumption. For point (2), we note that $F(S_\triv)$-cofibrations are retracts of relative $F(S_\triv)$-cell complexes, which are sent to weak equivalences by assumption, so $F(S_\triv)$-cofibrations are weak equivalences. To see that they are also $F(S_\cof)$-cofibrations are also $F(S_\triv)$-cofibrations, it suffices by definition to show that $F(S_\cof)$-injectives are also $F(S_\triv)$-injectives. For this, we note that the two lifting problems \[\xymatrix{
F(A)\ar[d]_-{F(i)}\ar[r]&B\ar[d]^-p&&A\ar[d]_-i\ar[r]&G(B)\ar[d]^-{G(p)}\\
F(C)\ar@{-->}[ur]\ar[r]&D&&C\ar@{-->}[ur]\ar[r]&G(D)
}\] are equivalent, and that $S_\cof$-injectives are also $S_\triv$-injectives.

The remaining assumptions are verified in a similar manner---see \cite[11.3.2]{Hirschhorn:MCL} for a detailed account.
\end{proof}

The resulting model structure on $\op D$ is called the \emph{transferred} model structure.

\subsection{Projective model structure} In our example of interest, we obtain the \emph{projective} model structure on functors:

\begin{corollary}[{\cite[11.6.1]{Hirschhorn:MCL}}]
Let $\op C$ be a cofibrantly generated model category and $I$ any small category. The following specifications define a model structure on the functor category $\op C^I$:
\begin{enumerate}
\item the weak equivalences are the pointwise weak equivalences;
\item the fibrations are the pointwise fibrations;
\item the cofibrations are the natural transformations with the left lifting property with respect to every pointwise fibration.
\end{enumerate}
\end{corollary}

Turning now to the question of cofibrant replacement, we write $\mathrm{Bar}_\bullet(F,-)$ for the augmented simplicial functor given in degree $n$ by \[\mathrm{Bar}_\bullet(F,-)=(\iota_!\iota^*)^{n+1}F(i)\cong\coprod_{I_0^{n+1}}F(i_{n+1})\times \Hom(i_{n+1}, i_n)\times\cdots\times\Hom(i_2,i_1)\times \Hom(i_1,-)\] and with the face and degeneracy maps induced by composition and insertion of identities, respectively. As we saw in the previous lecture, we have $\colim_I\mathrm{Bar}_\bullet(F,-)=\mathrm{Bar}_\bullet(F)$.

\begin{proposition}\label{prop:bar is cofibrant}
The augmentation $|\mathrm{Bar}_\bullet(F,-)|\to F$ is a pointwise weak homotopy equivalence. Moreover, if $F$ is pointwise cofibrant, then $|\mathrm{Bar}_\bullet(F,-)|$ is projective cofibrant.
\end{proposition}
\begin{proof}
Since $\iota^*$ reflects weak equivalences and commutes with colimits, it suffices to prove the first claim instead for the geometric realization of the augmented simplicial object $\iota^*\mathrm{Bar}_\bullet(F, -)$ in $\Top^{I_0}$. But this augmented simplicial object has an extra degeneracy, which is given by the unit of the $(\iota_!,\iota^*)$-adjunction.

For the second claim, it will suffice to show that each of the maps $|\sk_{n-1}(\mathrm{Bar}_\bullet(F,-))|\to |\sk_n(\mathrm{Bar}_\bullet(F,-))|$ is a cofibration between cofibrant objects. We proceed by induction on $n$, the base case of the cofibrancy of the $0$-skeleton following from our assumption that $F$ is pointwise cofibrant. For the induction step, we write \[N_n(-)=\coprod_{I_0^{n+1}}F(i_{n+1})\times \Hom'(i_{n+1}, i_n)\times\cdots\times \Hom'(i_2,i_1)\times \Hom(i_1,-)\subseteq \mathrm{Bar}_\bullet(F,-),\] where \[\Hom'(i,j)=\begin{cases}
\Hom(i,j)&\quad i\neq j\\
\Hom(i,i)\setminus\{\id_i\}&\quad i=j.
\end{cases}
\] After evaluating at an object $i$, these spaces witness $\mathrm{Bar}_\bullet(F, i)$ as split, so we have the pushout square \[\xymatrix{
\partial\Delta^n\times N_n(i)\ar[r]\ar[d]&|\sk_{n-1}(\mathrm{Bar}_\bullet(F,i))|\ar[d]\\
\Delta^n \times N_n(i)\ar[r]&|\sk_{n}(\mathrm{Bar}_\bullet(F,i))|,
}\] which is moreover natural in $i$. Since the $(n-1)$-skeleton is known to be cofibrant by induction, it suffices to show that the map $\partial\Delta^n \times N_n(-)\to \Delta^n\times N_n(-)$ is a projective cofibration. For every $j\in I_0$, define a functor $N_n^{j}:I_0\to \Top$ by \[N_n^{j}(i):=\begin{cases}
\displaystyle\coprod_{I_0^{n}}F(i_{n+1})\times \Hom'(i_{n+1}, i_{n})\times\cdots\times \Hom'(i_2,j)&\quad i=j\\
\varnothing&\quad\text{otherwise.}
\end{cases}\] Then we evidently have the commuting diagram \[\xymatrix{
\partial\Delta^n\times N_n(-)\ar[r]&\Delta^n\times N_n(-)\\
\displaystyle\coprod_{I_0}\iota_!\left(\partial\Delta^n\times N_n^{j}\right)\ar@{=}[u]^-\wr\ar[r]&\displaystyle\coprod_{I_0}\iota_!\left(\Delta^n\times N_n^{j}\right)\ar@{=}[u]_-\wr
}\] of functors, where the components of the bottom arrow are each induced by the inclusion $\partial \Delta^n\subseteq \Delta^n$. Since this map is a cofibration in $\Top$, since products with cofibrant objects preserve cofibrations, since $\iota_!$ is a left Quillen functor, and since cofibrations are closed under coproducts, it follows that the bottom arrow, and hence the top arrow, in this diagram is a projective cofibration.
\end{proof}

\begin{corollary}
For any functor $F:I\to \Top$, the canonical map $|\mathrm{Bar}_\bullet(F)|\to \colim_I F$ exhibits a homotopy colimit of $F$.
\end{corollary}
\begin{proof}
Since the values of the functor $\mathrm{Bar}_\bullet(-):\Top^I\to \Top^{\Delta^{op}}$ are all split, the functor $|\mathrm{Bar}_\bullet(-)|$ descends to a functor at the level of homotopy categories. Thus, from the existence of the canonical map to the colimit and the universal property of the derived functor, we obtain the commuting diagram 
\[\xymatrix{
\gamma\left(|\mathrm{Bar}_\bullet(Q(F))|\right)\ar[d]\ar[r]&\gamma\left(|\mathrm{Bar}_\bullet(F)|\right)\ar[d]\\
\hocolim_IQ(F)\ar[r]&\hocolim_IF
}\] in $\Ho(\Top)$, where $Q:\Top \to \Top$ is any cofibrant replacement functor (for example, we may take $Q(X)$ to be the geometric realization of the singular simplicial set of $X$). The bottom map is an isomorphism, since weak equivalences in $\Top^I$ are pointwise; the top map is a weak equivalence because both simplicial spaces are split; and the lefthand map is an isomorphism by Proposition \ref{prop:bar is cofibrant}. It follows that the righthand map is an isomorphism, as claimed.
\end{proof}

\begin{remark}
Most of what we have said carries over verbatim to the setting of a general cofibrantly generated simplicial model category, but the ability to forego pointwise cofibrant replacement seems to be an accident specific to $\Top$, which is connected to the existence of the \emph{Str{\o}m model structure}---see \cite[A]{DuggerIsaksen:THAR} for further discussion.
\end{remark}

\begin{recollection}
To a category $I$, we may associate its \emph{nerve}, which is the simplicial set $NI$ given in degree $n$ by $NI_n=\Fun([n], I)$, the set of composable $n$-tuples of morphisms in $I$. The \emph{classifying space} of $I$ is the space $BI:=|NI|$, and we say that $I$ is \emph{contractible} if its classifying space is weakly contractible. We say that a functor $T:I\to J$ is \emph{homotopy final} if the overcategory $(j\downarrow T)$ is contractible for every object $j\in J$ (resp. \emph{homotopy initial}, $(T\downarrow j)$). Homotopy final functors induce weak homotopy equivalences on classifying spaces, and, since $BI\simeq BI^{op}$, the same holds for homotopy initial functors.
\end{recollection}

\begin{example}
A category with an initial or final object is contractible, as is a (co)filtered category.
\end{example}

We record the following standard facts about homotopy colimits.

\begin{proposition}[{\cite[6.7,\,20.3]{Dugger:PHC}}]\label{prop:hocolim facts}
Let $F:J\to \Top$ be a functor.
\begin{enumerate}
\item If $T:I\to J$ is homotopy final, then the induced map \[\hocolim_I T^*F\to \hocolim_J F\] is a weak equivalence.
\item If $J=\Delta^{op}$ and $F$ is split, then \[\hocolim_{\Delta^{op}}F\simeq |F|.\]
\end{enumerate}
\end{proposition}

\subsection{Relative homotopy colimits} We will also have use for a relative version of the homotopy colimit. Note that, if $\lambda:I\to J$ is a functor, then the restriction $\lambda^*$ preserves fibrations and weak equivalences in the respective projective model structures, since both are pointwise. Thus, $(\lambda_!,\lambda^*)$ is a Quillen pair, and we may contemplate the \emph{homotopy left Kan extension} $\hoLan_\lambda:=\mathbb{L}\lambda_!$. In order to understand this functor, we recall that, from the commuting diagram of categories \[\xymatrix{
(\lambda\downarrow j)\ar[d]_-\pt\ar[r]^-{\pi_j} &I\ar[d]^-\lambda\\
\pt \ar[r]^-{\iota_j}&J,
}\] there is an induced \emph{base change isomorphism} \[\lambda_!F(j)=\iota_j^*\lambda_!F\xrightarrow{\simeq}\pt_!\pi_j^*F=\colim_{(\lambda\downarrow j)}\pi_j^*F.\] Our next results asserts that the analgous result holds for the homotopical versions of these functors. 

\begin{corollary}
For any $\lambda:I\to J$ and $F:I\to \Top$, there is a natural isomorphism \[\hoLan_\lambda F(j)\cong \hocolim_{(\lambda\downarrow j)}\pi_j^*F\] in $\Ho(\Top^J)$.
\end{corollary}
\begin{proof}
We may assume that $F$ is objectwise cofibrant. In this case, by Proposition \ref{prop:bar is cofibrant}, we have the isomorphisms in $\Ho(\Top^J)$ \[\hoLan_\lambda F(j)\cong \lambda_!|\mathrm{Bar}_\bullet(F, -)|\cong \colim_{(\lambda\downarrow j)}|\mathrm{Bar}_\bullet(F,\pi_j(-))|,\] whereas \[\hocolim_{(\lambda\downarrow j)}\pi_j^*F\cong \colim_{(\lambda\downarrow j)}|\mathrm{Bar}_\bullet(\pi_j^*F, -)|.\] By inspection, the simplicial functors $\mathrm{Bar}_\bullet(F,\pi_j(-))$ and $\mathrm{Bar}_\bullet(\pi_j^*F,-)$ are isomorphic, and the claim follows.
\end{proof}

\subsection{Quillen's Theorem B}

For any functor $F:I\to \Top$, there is a natural map \[\hocolim_I F\to BI,\] which is induced on geometric realizations by the simplicial map $\mathrm{Bar}_\bullet(F)\to \mathrm{Bar}_\bullet(\underline \pt)$ arising from the unique natural transformation $F\to \underline \pt$ (note that we have used the isomorphism $BI\cong BI^{op}$, since the latter bar construction is the nerve $NI^{op}$). Since the bar construction is split, this map is a weak equivalence whenever $F$ is pointwise contractible. 

When $F$ is not pointwise contractible, it is often useful to be able to understand the homotopy fiber of this map. The result that will guide is in this task is due to Quillen and usually referred to as ``Theorem B.'' We follow the treatment of \cite[IV]{GoerssJardine:SHT}, beginning with the following preliminary result, which is interesting in its own right.

\begin{lemma}\label{lem:baby unstraightening}
If $F:I\to \mathrm{sSet}$ is a functor sending each morphism in $I$ to a weak equivalence, then the diagram \[\xymatrix{
F(i)\ar[r]\ar[d]&d^*\mathrm{Bar}_\bullet(F)\ar[d]\\
\Delta^0\ar[r]^-i&NI^{op}
}\] of simplicial sets is homotopy Cartesian for every object $i\in I$.
\end{lemma}
\begin{proof}
We will produce a diagram of the form
\[\xymatrix{
F(i)\ar[r]\ar[d]&K\times_{NI^{op}}d^*\mathrm{Bar}_\bullet(F)\ar[d]\ar[r]&\mathrm{Bar}_\bullet(F)\ar[d]\\
\Delta^0\ar[r]^-j& K\ar[r]^-p&NI^{op},
}\] in which $j$ is a trivial cofibration and $p$ a fibration, and we will show that the upper left map is a weak equivalence. 

We take $K=\colim K_r$ to be the output of the small argument applied to $\Delta^0\to NI^{op}$; that is, $K_r$ is defined as the pushout in the diagram \[\xymatrix{
\displaystyle \coprod \Lambda_k^n\ar[dd]\ar[rr]&&K_{r-1}\ar[dl]\ar[dd]\\
&K_r\ar@{-->}[dr]\\
\displaystyle\coprod \Delta^n\ar[ur]\ar[rr]&& NI^{op},
}\] where $K_0=\Delta^0$. It now follows that each of the diagrams \[\xymatrix{
\displaystyle \coprod \Lambda_k^n\times_{NI^{op}}d^*\mathrm{Bar}_\bullet(F)\ar[d]\ar[r]&K_{r-1}\times_{NI^{op}}d^*\mathrm{Bar}_\bullet(F)\ar[d]\\
\displaystyle\coprod \Delta^n\times_{NI^{op}}d^*\mathrm{Bar}_\bullet(F)\ar[r]&K_r\times_{NI^{op}}d^*\mathrm{Bar}_\bullet(F),
}\] is a pushout, since colimits in $(\mathrm{sSet}\downarrow NI^{op})$ are computed in simplicial sets, and since the fiber product with $d^*\mathrm{Bar}_\bullet(F)$ admits a right adjoint on this category. Thus, since $K_0\times_{NI^{op}}d^*\mathrm{Bar}_\bullet(F)=F(i),$ it will suffice to show that each of the maps \[\Lambda_k^n\times_{NI^{op}}d^*\mathrm{Bar}_\bullet(F)\to \Delta^n\times_{NI^{op}}d^*\mathrm{Bar}_\bullet(F)\] is a weak equivalence. This map is obtained by applying the diagonal to the bottom map in the diagram \[\xymatrix{
\displaystyle\coprod_{(k_0\leq\cdots\leq k_r)\in (\Lambda_k^n)_r}F(\sigma(n))\ar[r]\ar[d]&\displaystyle\coprod_{(k_0\leq\cdots\leq k_r)\in (\Delta^n)_r}F(\sigma(n))\ar[d]\\
\displaystyle\coprod_{(k_0\leq\cdots\leq k_r)\in (\Lambda_k^n)_r}F(\sigma(k_r))\ar[r]&\displaystyle\coprod_{(k_0\leq\cdots\leq k_r)\in (\Delta^n)_r}F(\sigma(k_r))
}\] of bisimplicial sets, where $\sigma:[n]\to I^{op}$ is the given simplex of $NI^{op}$. The vertical maps, which are weak equivalences by our assumption on $F$, are supplied by the unique morphism to the final object $n\in [n]$, and the diagonal of the top map is the weak equivalence $\Lambda_k^n\times F(\sigma(n))\to \Delta^n\times F(\sigma(n))$. It follows by two-out-of-three that the bottom map becomes a weak equivalence after applying $d^*$, which was to be shown.
\end{proof}

\begin{corollary}\label{cor:hocolim quasifibration}
If $F:I\to \mathrm{Top}$ is a functor sending each morphism in $I$ to a weak equivalence, then the diagram \[\xymatrix{
F(i)\ar[r]\ar[d]&\hocolim_IF\ar[d]\\
\pt\ar[r]^-i&BI
}\] of spaces is homotopy Cartesian for every object $i\in I$.
\end{corollary}
\begin{proof}
Applying Lemma \ref{lem:baby unstraightening} to the functor $\mathrm{Sing}(F)$ yields the homotopy pullback square \[\xymatrix{
\mathrm{Sing}(F(i))\ar[d]\ar[r]&d^*\mathrm{Bar}_\bullet(\mathrm{Sing}(F))\ar[d]\\
\Delta^0\ar[r]^-i&NI^{op}.
}\] The claim now follows after applying geometric realization, since \[|d^*\mathrm{Bar}_\bullet(\mathrm{Sing}(F))|\cong \big||\mathrm{Bar}_\bullet(\mathrm{Sing}(F))|\big|\cong |\mathrm{Bar}_\bullet(|\mathrm{Sing}(F)|)|\simeq |\mathrm{Bar}_\bullet(F)|\simeq \hocolim_IF,\] and since geometric realization, as part of the Quillen equivalence between topological spaces and simplicial sets, preserves homotopy pullbacks. 
\end{proof}

\begin{remark}
In case $I$ is the category associated to a group $G$, the statement becomes the familiar fact that the homotopy colimit of a $G$-space $X$, thought of as a functor, is given by the Borel construction on $X$, which forms a bundle over $BG$ with fiber $X$.

In general, this corollary encourages us to think of a functor as defining a sort of bundle over the classifying space with total space the homotopy colimit, fibers given by the functor itself, and, as we will see, space of sections given by the homotopy limit. Taking this kind of idea to its logical conclusion leads to the \emph{unstraightening} construction---see \cite[3.2]{Lurie:HTT}.
\end{remark}

\begin{corollary}[Quillen's Theorem B]\label{cor:quillen b}
Let $T:I\to J$ be a functor such that, for every morphism $\alpha:j\to j'$ in $J$, the map $B(j'\downarrow T)\to B(j\downarrow T)$ is a weak equivalence. The diagram \[\xymatrix{
B(j\downarrow T)\ar[r]\ar[d]&BI\ar[d]^-{BT}\\
\pt\ar[r]&BJ
}\] is homotopy Cartesian.
\end{corollary}
\begin{proof}
By Lemma \ref{lem:baby unstraightening}, it suffices to show that \[\hocolim_{J^{op}}B(-\downarrow T)\xrightarrow{\sim} BI.\] To see why this is so, we note that the lefthand side arises from the bisimplicial set given in bidegree $(m,n)$ by \begin{align*}
\coprod_{j_0\to\cdots \to j_n}N(j_n\downarrow T)_m&\cong \coprod_{j_0\to\cdots\to j_n\to T(i_0)\to \cdots \to T(i_m)}\pt\\
&\cong \coprod_{i_0\to \cdots \to i_m}N(J\downarrow T(i_0))_n. 
\end{align*} Thus, \[\hocolim_{J^{op}}B(-\downarrow T)\simeq\hocolim_{I^{op}}B(J\downarrow T(-))\simeq \hocolim_{I^{op}}\underline{\pt}\simeq BI,\] where for the second equivalence we have used that $(J\downarrow T(i))$ is contractible for every $i\in I$, having the initial object $\id_i$.
\end{proof}

\section{The Spanier--Whitehead category}\label{appendix:Spanier--Whitehead}

\subsection{Stable homotopy}

We begin by recalling the following classical fact, which asserts that homotopy behaves like homology in a certain ``stable'' range.

\begin{theorem}[Homotopy excision]
Let $(B,A)$ be a $q$-connected CW pair. If $A$ is $p$-connected, then the map induced on $\pi_i$ by the map $(B,A)\to (B/A, *)$ is an isomorphism for $i\leq p+q$ and a surjection for $i=p+q+1$.
\end{theorem}

Recall that the suspension functor $A\mapsto \Sigma A$ is homotopy invariant if $A$ is a CW complex. Thus, we obtain in this case a \emph{suspension homomorphism}
\[\xymatrix{
\pi_i(A)=[S^i, A]_*\xrightarrow{\Sigma} [\Sigma S^i, \Sigma A]_*\cong [S^{i+1}, \Sigma A]_*=\pi_{i+1}(\Sigma A).
}\]

\begin{corollary}[Freudenthal suspension theorem]
If $A$ is a $p$-connected CW complex, then the suspension homomorphism \[\pi_i(A)\to \pi_{i+1}(\Sigma A)\] is an isomorphism for $i\leq 2p$ and a surjection for $i=2p+1$.
\end{corollary}
\begin{proof}
The suspension map coincides with the dashed map in the commuting diagram
\[\xymatrix{
\pi_{i+1}(CA)\ar[d]\ar[r]&\pi_{i+1}(CA,A)\ar[d]^-{(\star)}\ar[r]^-\simeq&\pi_i(A)\ar@{-->}[ddl]\\
\pi_{i+1}(SA)\ar@{=}[d]_-\wr\ar@{=}[r]&\pi_{i+1}(SA)\ar@{=}[d]^-\wr\\
\pi_{i+1}(\Sigma A)\ar@{=}[r]&\pi_{i+1}(\Sigma A)
}\] Since $A$ is $p$-connected by assumption, the pair $(CA,A)$ is $(p+1)$-connected, so the starred map is an isomorphism for $i+1\leq 2p+1$ and a surjection in the next degree, implying the claim.
\end{proof}

In light of this result, it is reasonable to make the following definition. 

\begin{definition}
Let $A$ be a pointed CW complex. The $i$th \emph{stable homotopy group} of $A$ is \[\pi_i^s(A):=\colim_{r} \pi_{i+r}^s(\Sigma^rA)\cong \pi_{2i+2}(\Sigma^{i+2}A).\] A \emph{stable map} from $A$ to $B$ is a map $f:\Sigma^rA\to \Sigma^rB$ for some $r$, regarded as an element of $\colim_r[\Sigma^rA,\Sigma^rB]_*$. 
\end{definition}

We may compose the stable maps $f:\Sigma^rA\to \Sigma^rB$ and $g:\Sigma^sB\to \Sigma^sC$ to obtain the stable map \[\Sigma^rg\circ\Sigma^s f:\Sigma^{r+s}A\to \Sigma^{r+s}C.\] Since an element of $\pi_i^s(A)$ is nothing other than a stable map from $S^i$ to $A$, it follows that stable maps induce morphisms at the level of stable homotopy groups.

\begin{definition}
A \emph{stable weak equivalence} is a stable map that induces an isomorphism on all stable homotopy groups. 
\end{definition}

We shall use the notation $f:A\xrightarrow{\sim_s} B$ to indicate a stable weak equivalence.

\begin{recollection}
If $f:A\to B$ is a continuous map, the \emph{mapping cylinder} of $f$ is the pushout in the diagram \[\xymatrix{
A\ar[d]_-{\{1\}}\ar[r]^-f&B\ar[d]\\
A\times[0,1]\ar[r]& \mathrm{Cyl}(f).
}\] The \emph{mapping cone} of $f$ is the quotient \[C(f):=\frac{\mathrm{Cyl}(f)}{A\times\{0\}}.\] The diagram \[\xymatrix{
A\ar[d]\ar[r]^-f&B\ar[d]\\
\pt\ar[r]&C
}\] is a \emph{cofiber sequence} if the induced map \[C(f)\to B/f(A)\to C\] is a weak equivalence.
\end{recollection}

\begin{example}
If $f$ is a Hurewicz cofibration, then $A\xrightarrow{f} B\to B/f(A)$ is a cofiber sequence.
\end{example}

For our purposes, the key fact about stable weak equivalences is the following.

\begin{lemma}\label{lem:five lemma}
In the commuting diagram \[\xymatrix{
A\ar[r]\ar[d]_-{f_1}&B\ar[d]_-{f_2}\ar[r] &C\ar[d]_-{f_3}\\
A'\ar[r]&B'\ar[r]&C',
}\] if both rows are cofiber sequences and $f_1$ and $f_3$ are stable weak equivalences, then $f_2$ is also a stable weak equivalence.
\end{lemma}

It will be useful to have a means of producing stable maps.

\begin{lemma}\label{lem:stable maps adjunction}
Let $A$ be a finite CW complex. The set of stable maps from $A$ to $B$ is a in natural bijection with $[A,\,\Omega^\infty\Sigma^\infty B]_*$, where $\Omega^\infty\Sigma^\infty B:=\colim_r \Omega^r\Sigma^rB$.
\end{lemma}
\begin{proof}
By finiteness and adjunction, we have \begin{align*}
[A,\Omega^\infty\Sigma^\infty B]_*&=[A,\colim_r\Omega^r\Sigma^r B]_*\\
&\cong \colim_r[A, \Omega^r\Sigma^rB]_*\\
&\cong \colim_r[\Sigma^r A, \Sigma^rB]_*.
\end{align*}
\end{proof}

\subsection{Spanier--Whitehead category} In order to see why this lemma should be true, we locate the notion of stable weak equivalence within a convenient categorical context, which is a slight variant of that introduced in \cite{SpanierWhitehead:FAHT}.

\begin{definition}
The \emph{Spanier--Whitehead category} is the category $\SW$ in which an object is a pair $(A,m)$ of a pointed CW complex and an integer, the set of morphisms are given by \[\SW\left((A,m), (B,n)\right)=\colim_r \left[\Sigma^{r+m}A, \Sigma^{r+n}B\right]_*,\] where the colimit is taken over the set of natural numbers $r$ such that $r+m$ and $r+n$ are both nonnegative, and composition is defined in the same manner as composition of stable maps.
\end{definition}

We begin with a few basic observations on this category.
\begin{enumerate}
\item The full subcategory of objects of the form $(A,0)$ is the category stable maps between pointed CW complexes from the previous lecture. Note, however, that this subcategory is not closed under isomorphism in $\SW$.
\item The assignment $A\mapsto (A,0)$ extends to a functor $\Ho(\Top_*)\to \SW$ fitting into the commuting diagram \[\xymatrix{
\Ho(\Top_*)\ar[d]_-\Sigma\ar[r]& \SW\ar[d]^-\Sigma&(A,m)\ar@{|->}[d]\\
\Ho(\Top_*)\ar[r]&\SW&(\Sigma A, m).
}\] 
\item The class of the isomorphism \[\{\Sigma^m\Sigma A\cong\Sigma^{m+1}A\}\in [\Sigma^m\Sigma A,\Sigma^{m+1}A]_*\to\SW\left((\Sigma A,m),(A,m+1)\right)\] defines a natural isomorphism $(\Sigma A, m)\cong (A,m+1)$, from which we conclude that $\Sigma:\SW\to \SW$ is an equivalence of categories with quasi-inverse $\Sigma^{-1}(A,m)=(A,m-1)$. In fact, the pair $(\SW,\Sigma)$ is universal with respect to this property in an appropriate sense---see \cite{DellAmbrogio:SWCAT}.
\item Any finite diagram in $\SW$, after finitely many applications of the functor $\Sigma$, may be realized by a homotopy commutative diagram of CW complexes.
\item Since $(A,m)\cong(\Sigma A,m-1)$, the functor $\SW((A,m),-)$ is naturally valued in groups; moreover, since $(A,m)\cong(\Sigma A,m-2)$, these groups are Abelian. This extra structure witnesses $\SW$ as a \emph{preadditive} category, which is to say a category enriched in Abelian groups.
\item Fix $(A,m)$ and $(B,n)$, and let $N\geq 0$ be large enough so that $m+N$ and $n+N$ are both nonnegative. There is a natural chain of isomorphisms \begin{align*}
\SW\left((\Sigma^{m+N}A\vee \Sigma^{n+N}B, -N), (C,p)\right)&= \colim_{r}\left[\Sigma^{r-N}(\Sigma^{m+N}A\vee\Sigma^{n+N}B), \Sigma^{p+r}C\right]_*\\
&\cong\colim_{r}\left[\Sigma^{m+r}A\vee \Sigma^{n+r}B, \Sigma^{p+r}C\right]_*\\
&\cong \colim_{r}\left(\left[\Sigma^{m+r}A, \Sigma^{p+r}C\right]_*\times  \left[\Sigma^{n+r}B, \Sigma^{p+r}C\right]_*\right)\\
&\cong \colim_{r}\left[\Sigma^{m+r}A, \Sigma^{p+r}C\right]_*\times  \colim_r\left[\Sigma^{n+r}B, \Sigma^{p+r}C\right]_*\\
&\cong \SW\left((A,m), (C,p)\right)\times\SW\left((B,n), (C,p)\right),
\end{align*} which exhibits a coproduct of $(A,m)$ and $(B,n)$. In the same way, we see that $\SW$ has finite coproducts, and, since $\SW$ is preadditive, finite biproducts \cite[VIII:2]{MacLane:CWM}; in other words, $\SW$ is an \emph{additive} category.
\end{enumerate}

\subsection{Triangulated structure}
In fact, there is more structure to be uncovered.

\begin{definition}
Let $\op C$ be an additive category with an additive self-equivalence $\Sigma:\op C\to \op C$. A \emph{triangulation} of $\op C$ is a class $\op T$ of triples $(f,g,h)$ of morphisms of the form \[A\xrightarrow{f} B\xrightarrow{g} C\xrightarrow{h} \Sigma A,\] which satisfy the following axioms.
\begin{enumerate}
\item For every $A\in\op C$, $(\id_A,0,0)\in \op T$.
\item For each morphism $f$, there exists $(f,g,h)\in \op T$.
\item The class $\op T$ is closed under isomorphism.
\item If $(f,g,h)\in \op T$, then $(g,h,-\Sigma f)\in\op T$.
\item Given the solid commuting diagram \[\xymatrix{
A\ar@{=}[d]\ar[r]& B\ar[d]\ar[r]&C\ar@{-->}[d]\ar[r]&\Sigma A\ar@{=}[d]\\
A\ar[r]& B'\ar[d]\ar[r]&C'\ar@{-->}[d]\ar[r]&\Sigma A\\
&D\ar[d]\ar@{=}[r]&D\ar@{-->}[d]\\
&\Sigma B\ar[r]&\Sigma C
}\] in which the rows and lefthand column lie in $\op T$, the dashed fillers exist making the entire diagram commute, and the righthand column also lies in $\op T$.
\end{enumerate}
\end{definition}

\begin{remark}
Various equivalent combinations of axioms are possible. We follow \cite{May:ATC}.
\end{remark}

\begin{remark}
The elements of $\op T$ are typically referred to as \emph{distinguished triangles}, and the crucial fifth axiom is known variously as the ``octahedral axiom,'' for one of its visual representations, and ``Verdier's axiom.'' One way to understand this axiom is as enforcing a kind of third isomorphism theorem in $\op C$, i.e., that the ``quotient'' of $B'/A$ by $B/A$ should coincide with $B'/B$.
\end{remark}

We now seek to apply this formalism to the Spanier--Whitehead category.

\begin{definition}
A sequence $(A,m)\to (B,n)\to (C,p)$ in $\SW$ is a \emph{cofiber sequence} if, after applying $\Sigma^N$ for some $N\in\mathbb{Z}$, it becomes isomorphic to the image of a cofiber sequence in $\Ho(\Top_*)$.
\end{definition}

Recall that, for any map $f:A\to B$, we obtain a canonical map $C(f)\to \Sigma A$ by collapsing $B\subseteq C(f)$. Thus, any cofiber sequence in $\Ho(\Top_*)$ extends canonically to a sequence of the form \[A\xrightarrow{f} B\xrightarrow{g} C\xrightarrow{h} \Sigma A.\]

We only sketch the proof of the following fundamental result in stable homotopy theory. For a detailed proof in an expanded context, see \cite[A.12]{Schwede:TTC}, for example.

\begin{theorem}[Puppe]
The collection of cofiber sequences is a triangulation of $\SW$.
\end{theorem}
\begin{proof}[Sketch proof]
The first and third axioms are obvious, and the second follows from cellular approximation and the observation that the mapping cone of a cellular map between CW complexes is again a CW complex. After suspending, the fourth axiom follows from the standard fact that the rotation \[B\xrightarrow{g} C\xrightarrow{h} \Sigma A\xrightarrow{-\Sigma f} \Sigma B\] of a cofiber sequence is again a cofiber sequence \cite[8.4]{May:CCAT}. Finally, again after suspending, the fifth axiom reduces to checking that, given maps $f:A\to B$ and $f':B\to B'$, the natural sequence \[C(f)\to C(f'\circ f)\to C(f')\] is a cofiber sequence. After replacing $f$ and $f'$ by cofibrations, this claim follows from the homeomorphism \[\frac{B'/A}{B/A}\cong \frac{B'}{B}.\]
\end{proof}

\subsection{Consequences}
An important consequence of this structure is the following result, which is valid in any triangulated category---see \cite[13.4.2]{Stacks}, for example.

\begin{corollary}
For any $(A,m)\in \SW$, the functors $\SW\left((A,m),-\right)$ and $\SW\left(-,(A,m)\right)$ each send cofiber sequences to long exact sequences of Abelian groups.
\end{corollary}

Since $\pi_i^s(A)=\SW\left((S^0, i), (A,0)\right)$, Lemma \ref{lem:five lemma} now follows by the five lemma.

We also have the following appealing interpretation of stable weak equivalences.

\begin{corollary}\label{cor:finite stable weak equivalence}
If $A$ and $B$ are finite CW complexes and $f:(A,m)\to (B,n)$ induces an isomorphism \[\SW\left((S^0,i),(A,m)\right)\xrightarrow{\simeq}\SW\left((S^0,i),(B,n)\right)\] for every $i\in \mathbb{Z}$, then $f$ is an isomorphism. In particular, two finite CW complexes are stably weakly equivalent if and only if they become isomorphic in $\SW$.
\end{corollary}
\begin{proof}
By induction on the skeletal filtration of $C$, using the fact that cofiber sequences induce long exact sequences, one shows that \[\SW\left((C,p),(A,m)\right)\xrightarrow{\simeq}\SW\left((C,p),(B,n)\right)\] for any finite CW complex $C$ and $p\in\mathbb{Z}$. The claim now follows from the Yoneda lemma for the subcategory of $\SW$ on the finite CW complexes.
\end{proof}

\begin{corollary}\label{cor:stable weak equivalence homology}
Stable weak equivalences induce isomorphisms on homology.
\end{corollary}
\begin{proof}
By the suspension isomorphism, homology descends to a functor on $\SW$. Since any functor sends isomorphisms to isomorphisms, the previous corollary implies the claim for stable weak equivalences between finite CW complexes. Since homology commutes with filtered colimits, the general case follows.
\end{proof}

\begin{remark}
It is a somewhat paradoxical fact that, although, as we have seen, the Spanier--Whitehead category effectively captures the stable phenomenon of homotopy behaving homologically, we in fact lose almost all Eilenberg-MacLane spaces upon passage to $\SW$, since suspending destroys the property of being a $K(G,n)$. On the other hand, the Freudenthal suspension theorem guarantees that the homotopy groups remain correct in a range, and the map \[\Sigma K(G,n)\to K(G,n+1)\] adjoint to the weak equivalence $K(G,n)\xrightarrow{\sim} \Omega K(G,n+1)$ exhibits a sequence of objects in $\SW$ that we might think should converge to the missing Eilenberg-MacLane object. This line of reasoning is one way to motivate the enlargement of the Spanier--Whitehead category to the full stable homotopy category or homotopy category of spectra---see \cite{Puppe:OSHC}, for example. Another motivation is taken up below.
\end{remark}

\subsection{Filtered stable weak equivalences} The purpose of this section is to address and partially remove the finiteness assumption in Corollary \ref{cor:finite stable weak equivalence}. We begin by noting that the assumption would have been unnecessary had we been assured that $(C,p)$ is the colimit in $\SW$ of the objects $(\sk_k(C),p)$. Unfortunately, we have the computation \begin{align*}
\SW\left((C,p), (A,m)\right)&=\textstyle{\colim_r}\left[\Sigma^{p+r}C, \Sigma^{m+r}A\right]_*\\
&\cong{\textstyle\colim_r}\,\pi_0\,{\textstyle\Map_*}\left(\Sigma^{p+r}C, \Sigma^{m+r}A\right)\\
&\cong{\textstyle\colim_r}\,\pi_0\left({\textstyle{\holim_k}}\,{\textstyle{\Map_*}}\left(\Sigma^{p+r}\sk_k(C),\Sigma^{m+r}A\right)\right),
\end{align*} and the formation of homotopy groups fails in general to commute with sequential homotopy limits; indeed, this failure is measured by the \emph{Milnor exact sequence} \[0\to \textstyle\lim^1_k\pi_{i+1}(X_k)\to \pi_i\left(\holim_k X_k\right)\to \lim_k\pi_i(X_k)\to0.\] To summarize the problem, the category $\SW$ lacks certain filtered colimits that we might naively expect it to have.

In these notes, we adopt a somewhat ad hoc solution to this problem, which is nevertheless sufficient for our needs (but see Remark \ref{rmk:spectra} below). We make the following (non-standard) definition.

\begin{definition}
Let $X=\bigcup_{k\geq1} X_k$ and $Y=\bigcup_{k\geq1} Y_k$ be filtered pointed spaces. A \emph{filtered stable weak equivalence} is a collection of stable weak equivalences $f_k:X_k\xrightarrow{\sim_s} Y_k$ fitting into a commuting diagram \[\xymatrix{
X_{k-1}\ar[d]\ar[r]^-{f_{k-1}}&Y_{k-1}\ar[d]\\
X_k\ar[r]^-{f_k}&Y_k
}\] of stable maps.
\end{definition}

Since homology commutes with filtered colimits, we have the following consequence of Corollary \ref{cor:stable weak equivalence homology}.

\begin{corollary}\label{cor:filtered stable weak equivalence homology}
Filtered stable weak equivalences induce isomorphisms on homology.
\end{corollary}

\begin{remark}\label{rmk:spectra} A more principled solution to the problem of finiteness is provided by the category $\Sp$ of spectra, which may be constructed in many inequivalent ways, all of which yield the same homotopy category, typically called the \emph{stable homotopy category}. Using the subscript $\mathrm{fin}$ to indiate full subcategories spanned by finite spaces, the relationships among the various categories in question may be summarized in the following commuting diagram \[\xymatrix{
\Top_*\ar[r]\ar[dd]_-{\Sigma^\infty}&\Ho(\Top_*)\ar[d]&\Ho(\Top_{*,\mathrm{fin}})\ar[l]\ar[d]\\
&\SW\ar@{-->}[d]&\SW_\mathrm{fin}\ar[l]\ar@{-->}[dl]\\
\Sp\ar[r]&\Ho(\Sp).
}\] The vertical dashed functor is badly behaved, being neither fully faithful nor essentially surjective, but the diagonal dashed functor is the inclusion of the full subcategory of compact objects in $\Ho(\Sp)$; indeed, this fact may be used as one of a set of axioms characterizing the homotopy theory of spectra. 

The category of spectra has many wonderful properties, but only two are directly relevant to our discussion: first, $\Ho(\Sp)$ is again a triangulated category, so arguments in $\SW$ translate directly to this new context; and $\Sigma^\infty C\cong\colim_k\Sigma^\infty \sk_k(C)$, repairing the deficiency of the Spanier--Whitehead category that hindered us before. 
\end{remark}

\section{Tate cohomology and periodicity}\label{appendix:periodicity}

The goal of this appendix is to give an indication of the proof of the periodicity theorem used in the computation of the mod $p$ cohomology of $B_p(\mathbb{R}^n)$ for $p$ odd. The idea that we will pursue is that, if multiplication by the class in question is to be an isomorphism, then it should have an inverse, which we might imagine to be given by ``multiplication by an element of negative degree.'' Somewhat surprisingly, this idea is not nonsense, and the framework that gives it sense is the framework of \emph{Tate cohomology}, which we review briefly. For a more complete survey, see \cite[XII]{CartanEilenberg:HA}.

\subsection{Tate cohomology}

\begin{definition}
Let $G$ be a finite group and $V$ a $\mathbb{Z}[G]$-module. The \emph{norm map} for $V$ is the dashed filler in the commuting diagram \[\xymatrix{
V\ar[d]\ar[r]^-{\sum_{G} g}&V\\
H_0(G;V)\ar@{-->}[r]^-N& H^0(G;V).\ar[u]
}\]
\end{definition}

The norm map permits the definition of a cohomology theory combining ordinary group cohomology and group homology.

\begin{definition}
Let $G$ be a finite group and $V$ a $\mathbb{Z}[G]$-module. The \emph{Tate cohomology} of $G$ with coefficients in $V$ is \[
\hat H^n(G;V)=\begin{cases}
H^n(G;V)&\quad n>0\\
\mathrm{coker}\,N&\quad n=0\\
\ker N&\quad n=-1\\
H_{-n-1}(G;V)&\quad n<-1.
\end{cases}
\]
\end{definition}

We now summarize a few facts about Tate cohomology.

\begin{enumerate}
\item Tate cohomology is the cohomology associated to a certain type of resolution \cite[XII.3.2]{CartanEilenberg:HA}. By definition, a \emph{complete resolution} is a commuting diagram of free $\mathbb{Z}[G]$-modules \[\xymatrix{
\cdots\ar[r]&X_1\ar[r]^-{\delta_1}&X_0\ar[rr]^-{\delta_0}\ar[dr]_-{\epsilon}&&X_{-1}\ar[r]^-{\delta_{-1}}&X_{-2}\ar[r]&\cdots\\
&&&\mathbb{Z}\ar[ur]_-{\eta}
}\] in which the infinite row is exact, and Tate cohomology may be computed as \[\hat H^n(G;V)\cong H^n\left(\Hom_{\mathbb{Z}[G]}(X_\bullet, V)\right).\]
\item An injective group homomorphism induces a restriction on Tate cohomology \cite[XII.8]{CartanEilenberg:HA}. Tate cohomology is not functorial for general group homomorphisms.
\item Tate cohomology is multiplicative in the sense that there is a unique $C_2$-equivariant graded homomorphism \[\hat H(G;V_1)\otimes \hat H(G; V_2)\to \hat H(G;V_1\otimes V_2)\] extending the cup product in group cohomology and compatible with connecting homomorphisms \cite[XII.4]{CartanEilenberg:HA}. In particular, $\hat H(G;\mathbb{Z})$ is a graded commutative ring and $\hat H(G;V)$ is a module over this ring.
\item Tate cohomology enjoys a self-duality \cite[XII.6.6]{CartanEilenberg:HA} of the form \[\hat H^n(G;\mathbb{Z})\cong \Hom_\mathbb{Z}\left(\hat H^{-n}(G;\mathbb{Z}), C_{|G|}\right).\]
\end{enumerate}

\begin{remark}
According to (3), there are multiplication maps of the form \[H_m(BG)\otimes H_n(BG)\to H_{m+n+1}(BG),\] which may be interpreted from the point of view of singular chains in terms of the join of simplices \cite{Tene:PNTCFG}.
\end{remark}

\begin{example}
The periodic resolution for the cyclic group $C_k$ extends to a complete resolution \[\xymatrix{
\cdots\ar[r]&\mathbb{Z}[C_k]\ar[r]^-{\sigma-1}&\mathbb{Z}[C_k]\ar[rr]^-{\sum_{i=1}^k\sigma^i}\ar[dr]_-{\epsilon}&&\mathbb{Z}[C_k]\ar[r]^-{\sigma-1}&\mathbb{Z}[C_k]\ar[r]&\cdots\\
&&&\mathbb{Z}\ar[ur]_-{\sum_{i=1}^k\sigma^i}
}\] and it follows that \[
\hat H^n(C_k;\mathbb{Z})\cong\begin{cases}
C_k&\quad n \text{ even}\\
0&\quad n \text{ odd.}
\end{cases}
\] As for multiplicative structure, one can show that $\hat H(C_k;\mathbb{Z})\cong C_k[\beta,\beta^{-1}]$, where $\beta$ is the class represented by the 2-cocycle in $\Hom_{\mathbb{Z}[C_k]}\left(\mathbb{Z}[C_k], \mathbb{Z}\right)\cong \Hom_\mathbb{Z}\left(\mathbb{Z},\mathbb{Z}\right)$ corresponding to the identity \cite[XII.7]{CartanEilenberg:HA}. 
\end{example}

\subsection{Periods}

The periodicity evident in the previous calculation is a special case, and also the source, of a more general phenomenon. For a $\mathbb{Z}[G]$-module $V$, write $\hat H(G;V)_p$ for the $p$-primary component of $\hat H(G;V)$ and $|G|_p$ for the largest power of $p$ dividing $G$.

\begin{proposition}
Multiplication by $\alpha\in \hat H^n(G;\mathbb{Z})$ induces an isomorphism on $\hat H(G;V)_p$ for every $\mathbb{Z}[G]$-module $V$ if and only if $\alpha$ has exact order $|G|_p$.
\end{proposition}

The idea of the argument is that such an element $\alpha$ determines a homomorphism $\hat H^n(G;\mathbb{Z})\to C_{|G|}$ and thus an element $\alpha^{-1}\in \hat H^{-n}(G;\mathbb{Z})$ by the duality referenced above, and this element will act as a multiplicative inverse to $\alpha$. For details, see \cite[XII.11.1, Exercise XII.11]{CartanEilenberg:HA}.

\begin{definition}
The $p$-\emph{period} of $G$, if it exists, is the least $n>0$ such that $\hat H^n(G;\mathbb{Z})$ contains an element of exact order $|G|_p$.
\end{definition}

\begin{remark}
Most groups do not have a $p$-period; indeed, if $p$ is odd, then $G$ has a $p$-period if and only if the $p$-Sylow subgroup of $G$ is cyclic \cite[Exercsie XII.11]{CartanEilenberg:HA}.
\end{remark}

\begin{example}
The cyclic group $C_{p^k}$ has $p$-period $2$.
\end{example}

The following theorem of \cite{Swan:PPFG} will allow us to determine the $p$-period of $\Sigma_p$. For a subgroup $H\leq G$, we write $C_G(H)$ and $N_G(H)$ for the centralizer and normalizer of $H$ in $G$, respectively.

\begin{theorem}[Swan]
Let $G$ be a finite group, $p$ an odd prime, and $H\leq G$ a $p$-Sylow subgroup. If $H$ is cyclic, then the $p$-period of $G$ is equal to $2\left|\frac{N_G(H)}{C_G(H)}\right|$.
\end{theorem}

Before discussing the proof of this theorem, we use it to prove the Periodicity Theorem. First, we recall a few standard facts about cyclic groups.

\begin{lemma} Fix $k\geq0$.
\begin{enumerate}
\item $C_{\Sigma_k}(C_k)=C_k$
\item $N_{\Sigma_k}(C_k)\cong C_k\rtimes \mathrm{Aut}(C_k)$
\item $\mathrm{Aut}(C_k)\cong C_k^\times$
\end{enumerate}
\end{lemma}

\begin{proof}[Proof of Periodicity Theorem]
Since $V$ is defined over $\mathbb{F}_p$, we have $\hat H(G;V)_p=\hat H(G;V)$. Thus, the first claim is implied by our previous calculation of the $p$-period of $C_p$, while the second claim is implied by Swan's theorem and the calculation \[2\left|\frac{N_{\Sigma_p}(C_p)}{C_{\Sigma_p}(C_p)}\right|=2\frac{|C_p|\cdot|C_p^\times|}{|C_p|}=2(p-1).\] 
\end{proof}

\subsection{Proof of Swan's theorem}

In order to deduce the Swan's theorem from our calculation in the case of a cyclic group, we will need to understand the relationship between the Tate cohomology of $G$ and that of its $p$-Sylow subgroup.

\begin{definition}
Let $H\leq G$ be a subgroup. An element $\alpha\in \hat H(H;V)$ is \emph{stable} if, for every $g\in G$, $\alpha$ equalizes the two homomorphisms \[\hat H(H; V)\to \hat H(H\cap H^g;V)\] induced by the inclusion $H\cap H^g\leq H$ and the injection $(-)^{g^{-1}}:H\cap H^g\to H$, respectively.
\end{definition}

A transfer argument, one proves the following \cite[XII.10.1]{CartanEilenberg:HA}.

\begin{proposition}
If $H\leq G$ is a $p$-Sylow subgroup, then, for every $\mathbb{Z}[G]$-module $V$, the natural map $\hat H(G;V)_p\to \hat H(H;V)$ is injective with image the stable classes.
\end{proposition}

Thus, in order to deduce the period of $G$ from the period of $H$, we must understand the stable classes.

\begin{lemma}
If the $p$-Sylow subgroup $H\leq G$ is Abelian, then $\alpha\in \hat H^*(H;\mathbb{Z})$ is stable if and only if $\alpha$ is fixed by $N_G(H)$.
\end{lemma}
\begin{proof}
The ``only if'' direction is obviously true without assumption on $H$, so assume that $\alpha$ is fixed by $N_G(H)$, and choose $g\in G$. Since $H$ is Abelian, both $H$ and $H^g$ are contained in $C_G(H\cap H^g)$, and, since $H$ is a $p$-Sylow subgroup of $G$, it follows that both are $p$-Sylow subgroups of this centralizer. Since all such are conjugate, it follows that $H=H^{tg}$ for some $t\in C_G(H\cap H^g)$, whence $tg\in N_G(H)$. We therefore have the commuting diagram \[\xymatrix{
H\cap H^g\ar@{=}[dd]_-{(-)^{t^{-1}}}\ar[rr]&&H\ar[dd]^-{(-)^{(tg)^{-1}}}\\\\
H\cap H^g\ar[rr]^-{(-)^{g^{-1}}}&&H,
}\] and the maps induced on Tate cohomology by the righthand map fixes $\alpha$ by assumption. It follows that $\alpha$ equalizes the two horizontal maps, as desired.
\end{proof}

\begin{lemma}
If the $p$-Sylow subgroup $H\leq G$ is cyclic, then $n$ is a multiple of the $p$-period of $G$ if and only if $n$ is even and $N_G(H)$ fixes $\hat H^n(H;\mathbb{Z})$ pointwise.
\end{lemma}
\begin{proof}
By assumption, the $p$-period of $H$ is $2$, so we may assume that $n$ is even \cite[XII.11.3]{CartanEilenberg:HA}. Then $n$ is a multiple of the $p$-period of $G$ if and only if $\hat H^n(G;\mathbb{Z})$ contains an element of exact order $|G|_p$. By our computation in the case of a cyclic group and the previous proposition, this statement is equivalent to the statement that the map $\hat H^n(G;\mathbb{Z})_p\to \hat H^n(H;\mathbb{Z})$ is an isomorphism, which is to say that every class in the target is stable. Since $H$ is in particular Abelian, the previous lemma shows that this condition is equivalent to the condition that $N_G(H)$ act trivially in degree $n$.
\end{proof}

\begin{proof}[Proof of Swan's theorem]
Suppose that $n$ is a multiple of the $p$-period of $G$; then, by the lemma, $n$ is even and every class in degree $n$ Tate cohomology is fixed by $N_G(H)$. Now, $N_G(H)$ acts on $H$ via the composite \[N_G(H)\twoheadrightarrow \frac{N_G(H)}{C_G(H)}\hookrightarrow \mathrm{Aut}(H)\cong H^\times.\] By our assumption on $H$, the target is cyclic, so the intermediate quotient is also cyclic, and the action on $\hat H^n(H;\mathbb{Z})$ is by multiplication by $r^{n/2}$ with $(r,p)=1$ \cite[Lemma 3]{Swan:PPFG}. Since this action is trivial, $n/2$ is a multiple of the order of the cyclic group $N_G(H)/C_G(H)$.
\end{proof}

\end{appendix}

\bibliographystyle{amsalpha}
\bibliography{references}

\providecommand{\bysame}{\leavevmode\hbox to3em{\hrulefill}\thinspace}
\providecommand{\MR}{\relax\ifhmode\unskip\space\fi MR }
\providecommand{\MRhref}[2]{%
  \href{http://www.ams.org/mathscinet-getitem?mr=#1}{#2}
}
\providecommand{\href}[2]{#2}
\begin{thebibliography}{ADCK17}

\bibitem[Abr00]{Abrams:CSBGG}
A.~Abrams, \emph{Configuration spaces of braid groups of graphs}, Ph.D. thesis,
  UC Berkeley, 2000.

\bibitem[ADCK17]{AnDrummond-ColeKnudsen:SSGBG}
B.~H. An, G.~C. Drummond-Cole, and B.~Knudsen, \emph{Subdivisional spaces and
  graph braid groups}, arXiv:1708.02351, 2017.

\bibitem[AF14]{AyalaFrancis:PKD}
David Ayala and John Francis, \emph{Poincar\'e/{K}oszul duality}, Preprint,
  \url{http://arxiv.org/abs/1409.2478}, 2014.

\bibitem[AF15]{AyalaFrancis:FHTM}
D.~Ayala and J.~Francis, \emph{Factorization homology of topological
  manifolds}, J. Topol. \textbf{8} (2015), 1045--1084.

\bibitem[Arn69]{Arnold:CRCBG}
V.~I. Arnol'd, \emph{The cohomology ring of the colored braid group}, Math.
  Notes \textbf{5} (1969), no.~2, 138--140.

\bibitem[Art47]{Artin:TB}
E.~Artin, \emph{Theory of braids}, Ann. of Math. (2) \textbf{48} (1947), no.~1,
  101--126.

\bibitem[AS94]{AxelrodSinger:CSPT}
S.~Axelrod and I.~Singer, \emph{Chern-simons perturbation theory {$II$}}, J.
  Differential Geom. \textbf{39} (1994), no.~1, 173--213.

\bibitem[Ban]{Bandklayder:SSNPVNPD}
L.~Bandklayder, \emph{Stable splittings of mapping spaces via nonabelian
  poincar\'{e} duality}, arXiv:1705.03090.

\bibitem[BCT89]{BoedigheimerCohenTaylor:OHCS}
Carl-Friedrich B{\"o}digheimer, Fred Cohen, and Laurence~R. Taylor, \emph{On
  the homology of configuration spaces}, Topology \textbf{28} (1989), 111--123.

\bibitem[Bir75]{Birman:BLMCG}
J.~Birman, \emph{Braids, links, and mapping class groups}, Ann. of Math. Stud.,
  vol.~82, Princeton University Press, 1975.

\bibitem[BLZ15]{BlagojevichLueckZiegler:ETCS}
P.~V.~M. Blagojevic, W.~L{\"{u}}ck, and G.~M. Ziegler, \emph{Equivariant
  topology of configuration spaces}, J. Topol. \textbf{8} (2015), 414--456.

\bibitem[B{\"{o}}d87]{Boedigheimer:SSMS}
C.-F. B{\"{o}}digheimer, \emph{Stable splittings of mapping spaces}, Algebraic
  Topology, Lecture Notes in Math., vol. 1286, Springer, 1987.

\bibitem[BS84]{BrownRazekSalleh:VKTUNS}
R.~Brown and A.~Razak Salleh, \emph{A van kampen theorem for unions of
  non-connected spaces}, Arch. Math. \textbf{42} (1984).

\bibitem[BV73]{BoardmanVogt:HIASTS}
J.M Boardman and R.~M. Vogt, \emph{Homotopy invariant algebraic structures on
  topological spaces}, Lecture Notes in Math., vol. 347, Springer, 1973.

\bibitem[CE48]{ChevalleyEilenberg:CTLGLA}
C.~Chevalley and S.~Eilenberg, \emph{Cohomology theory of lie groups and lie
  algebras}, Trans. Amer. Math. Soc. \textbf{63} (1948), 85--124.

\bibitem[CE56]{CartanEilenberg:HA}
H.~Cartan and S.~Eilenberg, \emph{Homological algebra}, Princeton University
  Press, 1956.

\bibitem[CEF15]{ChurchEllenbergFarb:FIMSRSG}
T.~Church, J.~S. Ellenberg, and B.~Farb, \emph{{FI}-modules and stability for
  representations of symmetric groups}, Duke Math. J. \textbf{164} (2015),
  no.~9, 1833--1910.

\bibitem[Chu12]{Church:HSCSM}
T.~Church, \emph{Homological stability for configuration spaces of manifolds},
  Invent. Math. \textbf{188} (2012), no.~465-504.

\bibitem[CJ98]{CrabbJames:FHT}
M.~Crabb and I.~James, \emph{Fiberwise homotopy theory}, Springer Monographs in
  Mathematics, Springer, 1998.

\bibitem[CLM76]{CohenLadaMay:HILS}
F.~R. Cohen, T.~J. Lada, and J.~P. May, \emph{Homology of iterated loop
  spaces}, Lecture Notes in Math., no. 533, Springer, 1976.

\bibitem[CMT78]{CohenMayTaylor:SCSC}
F.~R. Cohen, J.~P. May, and L.~R. Taylor, \emph{Splitting of certain spaces
  {$CX$}}, Math. Proc. Camb. Phil. Soc. \textbf{84} (1978).

\bibitem[DCK17]{DrummondColeKnudsen:BNCSS}
G.~Drummond-Cole and B.~Knudsen, \emph{Betti numbers of configuration spaces of
  surfaces}, J. London Math. Soc. \textbf{96} (2017), no.~2, 367--393.

\bibitem[Del04]{DellAmbrogio:SWCAT}
I.~Dell'Ambrogio, \emph{The spanier-whitehead category is always triangulated},
  Master's thesis, ETH Z{\"{u}}rich, 2004.

\bibitem[DI04]{DuggerIsaksen:THAR}
D.~Dugger and D.~Isaksen, \emph{Topological hypercovers and
  {$\mathbb{A}^1$}-realizations}, Math. Z. \textbf{246} (2004), no.~4,
  667--689.

\bibitem[DK01]{DavisKirk:LNAT}
J.~Davis and P.~Kirk, \emph{Lecture notes in algebraic topology}, Graduate
  Studies in Mathematics, vol.~35, American Mathematical Society, 2001.

\bibitem[Dug]{Dugger:PHC}
D.~Dugger, \emph{A primer on homotopy colimits}, available at
  \url{http://pages.uoregon.edu/ddugger/hocolim.pdf.}

\bibitem[Far]{Farb:RS}
B.~Farb, \emph{Representation stability}, Contribution to the proceedings of
  the ICM 2014, Seoul, available as arXiv:1404.4065.

\bibitem[Far02]{Farber:TBPI}
M.~Farber, \emph{Topology of billiard problems, i}, Duke Math. J. \textbf{115}
  (2002), no.~3, 559--585.

\bibitem[Far17]{Farber:CSRMPA}
\bysame, \emph{Configuration spaces and robotmotion planning algorithms},
  arXiv:1701.02083, 2017.

\bibitem[FH12]{FadellHusseini:GTCS}
E.~Fadell and S.~Husseini, \emph{Geometry and topology of configuration
  spaces}, Springer Monographs in Mathematics, Springer Science and Business
  Media, 2012.

\bibitem[FHT00]{FelixHalperinThomas:RHT}
Y.~F{\'{e}}lix, S.~Halperin, and J.-C. Thomas, \emph{Rational homotopy theory},
  Springer, 2000.

\bibitem[FM94]{FultonMacPherson:CCS}
W.~Fulton and R.~MacPherson, \emph{A compatification of configuration spaces},
  Ann. of Math. (2) \textbf{139} (1994), no.~1, 183--225.

\bibitem[FN62a]{FadellNeuwirth:CS}
E.~Fadell and L.~Neuwirth, \emph{Configuration spaces}, Math. Scand.
  \textbf{10} (1962), 111--118.

\bibitem[FN62b]{FoxNeuwirth:BG}
R.~Fox and L.~Neuwirth, \emph{The braid groups}, Math. Scand. \textbf{10}
  (1962), 119--126.

\bibitem[Ghr02]{Ghrist:CSBGGR}
R.~Ghrist, \emph{Configuration spaces and braid groups on graphs in robotics},
  Knots, braids, and mapping class groups---papers dedicated to {J}oan {S}.
  {B}irman (New York, 1998), AMS/IP Stud. Adv. Math., vol.~24, Amer. Math.
  Soc., 2002, pp.~29--40.

\bibitem[GJ09]{GoerssJardine:SHT}
P.~G. Goerss and J.~F. Jardine, \emph{Simplicial homotopy theory},
  Birkh\"auser, 2009.

\bibitem[GK15]{GoodwillieKlein:MDSSE}
T.~Goodwillie and J.~Klein, \emph{Multiple disjunction for spaces of smooth
  embeddings}, J. Topol. \textbf{8} (2015), 651--674.

\bibitem[Gra75]{Gray:HT}
B.~Gray, \emph{Homotopy theory}, Academic Press, 1975.

\bibitem[GW99]{GoodwillieWeiss:EPVIT}
T.~Goodwillie and M.~Weiss, \emph{Embeddings from the point of view of
  immersion theory: part {II}}, Geom. Topol. \textbf{3} (1999), 103--118.

\bibitem[Hat01]{Hatcher:AT}
A.~Hatcher, \emph{{Algebraic Topology}}, Cambridge University Press, 2001.

\bibitem[Hil55]{Hilton:HGUS}
P.~J. Hilton, \emph{On the homotopy groups of the union of spheres}, J. London
  Math. Soc. \textbf{s1-30} (1955), 154--172.

\bibitem[Hir]{Hirschhorn:QMCTS}
P.~Hirschhorn, \emph{The quillen model category of topological spaces}, Notes
  available at \url{http://www-math.mit.edu/~psh/notes/modcattop.pdf}.

\bibitem[Hir02]{Hirschhorn:MCL}
P.~Hirschhorn, \emph{Model categories and their localizations}, Mathematical
  Surveys and Monographs, vol.~99, Amer. Math. Soc., Providence, Rhode Island,
  2002.

\bibitem[JS86]{JoyalStreet:BMC}
A.~Joyal and R.~Street, \emph{Braided monoidal categories}, Macquarie
  mathematics reports (1986).

\bibitem[Kas95]{Kassel:QG}
C.~Kassel, \emph{Quantum groups}, Grad. Texts in Math., vol. 155, Springer,
  1995.

\bibitem[Knu17]{Knudsen:BNSCSVFH}
Ben Knudsen, \emph{Betti numbers and stability for configuration spaces via
  factorization homology}, Alg. Geom. Topol. \textbf{17} (2017), no.~5,
  3137--3187.

\bibitem[Kon99]{Kontsevich:OMDQ}
Maxim Kontsevich, \emph{Operads and motives in deformation quantization}, Lett.
  Math. Phys. \textbf{48} (1999), 35--72.

\bibitem[LS05]{LongoniSalvatore:CSNHI}
R.~Longoni and P.~Salvatore, \emph{Configuration spaces are not homotopy
  invariant}, Topology \textbf{44} (2005), 375--380.

\bibitem[Lur03]{Lurie:HA}
Jacob Lurie, \emph{Higher algebra}, Preprint,
  \url{http://www.math.harvard.edu/~lurie/papers/higheralgebra.pdf}, 2003.

\bibitem[Lur09]{Lurie:HTT}
\bysame, \emph{Higher topos theory}, Ann. of Math. Stud., Princeton University
  Press, 2009.

\bibitem[LV13]{LambrechtsVolic:FLNDO}
P.~Lambrechts and I.~Voli\'{c}, \emph{Formality of the little {$N$}-disks
  operad}, Memoirs of the American Mathematical Society, Amer. Math. Soc.,
  2013.

\bibitem[{Mac}69]{MacLane:CWM}
S.~{Mac Lane}, \emph{Categories for the working mathematician}, Grad. Texts in
  Math., vol.~5, Springer, 1969.

\bibitem[May]{May:ATC}
J.~P. May, \emph{The axioms for triangulated categories}, Notes available at
  \url{http://www.math.uchicago.edu/~may/MISC/Triangulate.pdf}.

\bibitem[May72]{May:GILS}
J.~P. May, \emph{{The Geometry of Iterated Loop Spaces}}, Lecture Notes in
  Math., vol. 271, Springer-Verlag, Berlin, Germany, 1972.

\bibitem[May99]{May:CCAT}
J.~P. May, \emph{{A Concise Course in Algebraic Topology}}, Chicago Lectures in
  Mathematics, The University of Chicago Press, 1999.

\bibitem[McC00]{McCleary:UGSS}
J.~McCleary, \emph{{A User's Guide to Spectral Sequences}}, Cambridge Studies
  in Advances Mathematics, vol.~58, Cambridge University Press, 2000.

\bibitem[McD75]{McDuff:CSPNP}
Dusa McDuff, \emph{Configuration spaces of positive and negative particles},
  Topology \textbf{14} (1975), 91--107.

\bibitem[MKS04]{MagnusKarrassSolitar:CGT}
W.~Magnus, A.~Karrass, and D.~Solitar, \emph{Combinatorial group theory}, Dover
  Books on Mathematics, Dover Publications, 2004.

\bibitem[MS76]{McDuffSegal:HFGCT}
D.~McDuff and G.~Segal, \emph{Homology fibrations and the "group-completion"
  theorem}, Invent. Math. \textbf{31} (1976), 279--284.

\bibitem[OT92]{OrlikTerao:AH}
P.~Orlik and H.~Terao, \emph{Arrangements of hyperplanes}, Grundlehren Math.
  Wiss., vol. 300, Springer, 1992.

\bibitem[Pup72]{Puppe:OSHC}
D.~Puppe, \emph{On the stable homotopy category}, Topol. Appl. (1972).

\bibitem[Qui67]{Quillen:HA}
D.~Quillen, \emph{Homotopical algebra}, Lecture Notes in Math., vol.~43,
  Springer, 1967.

\bibitem[Qui72]{Quillen:HAKTI}
\bysame, \emph{Higher algebraic k-theory: I}, Algebraic K-theory, I: Higher
  K-theories, Lecture Notes in Math., vol. 341, Springer-Verlag, 1972,
  pp.~85--147.

\bibitem[RW13]{RandalWilliams:HSUCS}
O.~Randal-Williams, \emph{Homological stability for unordered configuration
  spaces}, Q. J. Math. \textbf{64} (2013), 303--326.

\bibitem[Sal01]{Salvatore:CSSl}
P.~Salvatore, \emph{Configuration spaces with summable labels}, Cohomological
  Methods in Homotopy Theory, Progr. Math., vol. 196, Birkh\"auser, 2001.

\bibitem[Sch]{Schwede:TTC}
S.~Schwede, \emph{Topological triangulated categories}, available as
  arXiv:1201.0899.

\bibitem[Seg68]{Segal:CSSS}
G.~Segal, \emph{Classifying spaces and spectral sequences}, Publ. Math. Inst.
  Hautes \'{E}tudes Sci. \textbf{34} (1968), 105--112.

\bibitem[Seg73]{Segal:CSILS}
\bysame, \emph{Configuration-spaces and iterated loop-spaces}, Invent. Math.
  \textbf{21} (1973), 213--222.

\bibitem[Seg74]{Segal:CCT}
\bysame, \emph{Categories and cohomology theories}, Topology \textbf{13}
  (1974), no.~3, 293--312.

\bibitem[Ser51]{Serre:HSEF}
J.-P. Serre, \emph{Homologie singuli{\`{e}}re des espaces fibr{\'{e}}s}, Ann.
  of Math. (2) \textbf{54} (1951), no.~3, 425--505.

\bibitem[Sin04]{Sinha:MTCCS}
D.~Sinha, \emph{Manifold-theoretic compactifications of configuration spaces},
  Selecta Math. \textbf{10} (2004), no.~3, 391--428.

\bibitem[Sin06]{Sinha:HLDO}
\bysame, \emph{The homology of the little disks operad}, Available as
  arXiv:0610236, 2006.

\bibitem[Sta]{Stacks}
\emph{Stacks project. {\rm{\url{https://stacks.math.columbia.edu/}}}}.

\bibitem[SW53]{SpanierWhitehead:FAHT}
E.~H. Spanier and J.~H.~C. Whitehead, \emph{A first approximation to homotopy
  theory}, Proc. Natl. Acad. Sci. U.S.A. \textbf{39} (1953), no.~7, 655--660.

\bibitem[Swa60]{Swan:PPFG}
R.~G. Swan, \emph{The {$p$}-period of a finite group}, Illinois J. Math.
  \textbf{4} (1960), no.~3, 341--346.

\bibitem[Ten]{Tene:PNTCFG}
H.~Tene, \emph{On the product in negative tate cohomology for finite groups},
  available as arXiv:0911.3014.

\bibitem[Tot96]{Totaro:CSAV}
Burt Totaro, \emph{Configuration spaces of algebraic varieties}, Topology
  \textbf{35} (1996), no.~4, 1057--1067.

\bibitem[Tuf02]{Tuffley:FSSS}
C.~Tuffley, \emph{Finite subset spaces of {$S^1$}}, Alg. Geom. Topol.
  \textbf{2} (2002), 1119--1145.

\bibitem[Val14]{Vallette:AHO}
B.~Vallette, \emph{Algebra {$+$} homotopy {$=$} operad}, Symplectic, Poisson
  and Noncommutative Geometry, vol.~62, MSRI Publications, 2014.

\bibitem[Wei99]{Weiss:EPVIT}
M.~Weiss, \emph{Embeddings from the point of view of immersion theory: part
  {I}}, Geom. Topol. \textbf{3} (1999), 67--101.

\bibitem[ZVC80]{ZieschangVogtColdeway:SPDG}
H.~Zieschang, E.~Vogt, and H.-D. Coldeway, \emph{Surfaces and planar
  discontinuous groups}, Lecture Notes in Math., vol. 835, Springer, 1980.

\end{thebibliography}
\end{document}